\documentclass[11pt]{amsart}
\usepackage{verbatim, latexsym, amssymb, amsmath,color, mathabx}
\usepackage{enumitem}
\usepackage{epsfig}
\usepackage{tikz-cd}

\usepackage{geometry}
\usepackage[hidelinks]{hyperref}

\newtheorem{thm}{Theorem}[section]
\newtheorem{prop}[thm]{Proposition}
\newtheorem{lem}[thm]{Lemma}
\newtheorem{cor}[thm]{Corollary}
\theoremstyle{definition}

\newtheorem{rem}[thm]{Remark}
\newtheorem{convention}[thm]{Convention}

\newtheorem{exe}{Example}[section]

\DeclareMathOperator{\rank}{rank}

\newcommand{\PFH}{\text{PFH}}
\newcommand{\TWPFH}{\text{TwPFH}}
\newcommand{\PFC}{\text{PFC}}
\newcommand{\TWPFC}{\text{TwPFC}}

\newcommand{\HM}{\text{HM}}

\newcommand{\HMbar}{\overline{\text{HM}}}
\newcommand{\TWHM}{\text{TwHM}}
\newcommand{\TWHMhat}{\widehat{\text{TwHM}}}
\newcommand{\TWHMcheck}{\widecheck{\text{TwHM}}}
\newcommand{\TWHMbar}{\overline{\text{TwHM}}}
\newcommand{\CM}{\text{CM}}

\newcommand{\CMbar}{\overline{\text{CM}}}
\newcommand{\TWCM}{\text{TwCM}}
\newcommand{\TWCMhat}{\widehat{\text{TwCM}}}
\newcommand{\TWCMcheck}{\widecheck{\text{TwCM}}}
\newcommand{\TWCMbar}{\overline{\text{TwCM}}}

\newcommand{\cl}{\text{cl}}
\newcommand{\SF}{\text{SF}}

\newcommand{\coexact}{\Omega^1_{\text{co-exact}}}
\newcommand{\Diff}{\text{Diff}}
\newcommand{\Ham}{\text{Ham}}

\newcommand{\wt}{\widetilde}
\newcommand{\dvol}{\text{dvol}}

\newcommand{\fa}{\mathfrak{a}}
\newcommand{\fc}{\mathfrak{c}}
\newcommand{\fC}{\mathfrak{C}}
\newcommand{\fd}{\mathfrak{d}}
\newcommand{\fe}{\mathfrak{e}}
\newcommand{\fg}{\mathfrak{g}}
\newcommand{\fp}{\mathfrak{p}}
\newcommand{\fq}{\mathfrak{q}}
\newcommand{\fs}{\mathfrak{s}}
\newcommand{\fS}{\mathfrak{S}}
\newcommand{\fcs}{\mathfrak{cs}}

\newcommand{\bD}{\mathbb{D}}
\newcommand{\bR}{\mathbb{R}}
\newcommand{\bZ}{\mathbb{Z}}

\newcommand{\bfA}{\mathbf{A}}
\newcommand{\bfS}{\mathbf{S}}

\newcommand{\cG}{\mathcal{G}}
\newcommand{\cH}{\mathcal{H}}
\newcommand{\cJ}{\mathcal{J}}
\newcommand{\cL}{\mathcal{L}}
\newcommand{\cS}{\mathcal{S}}
\newcommand{\cP}{\mathcal{P}}

\newcommand{\oJ}{\overline{J}}

\allowdisplaybreaks

\begin{document}
\title[The smooth closing lemma]{Periodic Floer homology and the smooth closing lemma for area-preserving surface diffeomorphisms}
\begin{abstract}

We prove a very general Weyl-type law for Periodic Floer Homology, estimating the action of twisted Periodic Floer Homology classes over essentially any coefficient ring in terms of the grading and the degree, and recovering the Calabi invariant of Hamiltonians in the limit.  We also prove a strong non-vanishing result, showing that under a monotonicity assumption which holds for a dense set of maps, the Periodic Floer Homology has infinite rank.  An application of these results yields that a $C^{\infty}$--generic area-preserving diffeomorphism of a closed surface has a dense set of periodic points.  This settles Smale's tenth problem in the special case of area-preserving diffeomorphisms of closed surfaces. 
\end{abstract}

\author{Dan Cristofaro-Gardiner}
\author{Rohil Prasad}
\author{Boyu Zhang}

\maketitle 

\setcounter{tocdepth}{1}
\tableofcontents

\section{Introduction} \label{sec:intro}

There are two aims of this paper.  The first is to establish the following theorem.

\begin{thm}[Smooth closing lemma for area-preserving diffeomorphisms] \label{thm:closingLemma}
Let $\Sigma$ be a closed, oriented surface equipped with an area form $\omega$, and let $\Diff(\Sigma, \omega)$ be the space of $C^\infty$--diffeomorphisms on $\Sigma$ preserving $\omega$.  Suppose $\phi \in \Diff(\Sigma, \omega)$. Then for any open set $U \subset \Sigma$ and any neighborhood $V$ of $\phi$ in the $C^\infty$ topology of $\text{Diff}(\Sigma, \omega)$, there is $\phi' \in V$ such that $\phi'$ has a periodic point in $U$.
\end{thm}

By a Baire category theorem argument that is now standard, we obtain the following corollary:

\begin{cor}[The generic density theorem for area-preserving diffeomorphisms] \label{cor:genden} 
For a $C^{\infty}$--generic element of $\Diff(\Sigma, \omega)$, the union of periodic points is dense.  More precisely, the set of elements of $\Diff(\Sigma, \omega)$ without dense periodic points forms a meager subset in the $C^{\infty}$--topology.
\end{cor}

Questions like this have attracted considerable interest.  For example, Franks and Le Calvez \cite{frankslecalvez} previously described the question of whether or not the generic density theorem holds as ``perhaps the most important" question about ``the topological picture of the dynamics of $C^r$ generic area preserving diffeomorphisms of surfaces"; Xia \cite[Conj. 1]{xia} attributes the conjecture that the generic density theorem holds to Poincar\'e \cite{poincare}.  A general form of Theorem \ref{thm:closingLemma} (for closed manifolds in all dimensions at any non-wandering point, without the volume-preserving assumption) is the subject of the tenth problem in Smale's 
list \cite{smaleProblemList}
of eighteen problems for the $21^{st}$ century.  

\begin{rem} (The $C^r$-topology)
It is known that the set of area-preserving maps of regularity $C^{\infty}$ is dense in the set of such maps of regularity $C^r$ for all $r \ge 1$ \cite{ZehnderSmoothing}.  From this it is straightforward to deduce analogues of
Theorem~\ref{thm:closingLemma} and Corollary~\ref{cor:genden} in the $C^r$-topology from Theorem~\ref{thm:closingLemma}, for any $r \ge 1$.  The case $r = 1$ was previously proved by Pugh--Robinson \cite{PughRobinson}, building on work of Pugh \cite{PughClosing, PughGenericDensity}.   
\end{rem}

A major breakthrough related to Theorem~\ref{thm:closingLemma} and Corollary~\ref{cor:genden} occurred in celebrated work of Asaoka-Irie \cite{asaokairie}, which established Theorem~\ref{thm:closingLemma} and Corollary~\ref{cor:genden} for Hamiltonian diffeomorphisms of surfaces; this work itself built on beautiful work of Irie \cite{irie}.  A key idea in \cite{irie, asaokairie} was to establish the desired closing lemmas by applying a kind of Weyl law, proved in \cite{CGHR15}, relating the lengths of certain periodic orbits to the volume of three-manifolds with contact forms.  The precise statement of this law used a Floer homology for three-manifolds called ``embedded contact homology" (ECH).     
Since the groundbreaking work in \cite{asaokairie, irie}, it has generally been expected 
that Theorem~\ref{thm:closingLemma} and Corollary~\ref{cor:genden} could follow from a strong enough Weyl law on an appropriate Floer homology for surfaces. 

Establishing such a law is the second aim of this work, and the bulk of the paper is about this.  We establish essentially the most general possible statement, which includes a strong non-vanishing result for the relevant Floer homology that makes our Weyl law broadly applicable.

\subsection*{The PFH Weyl Law}

We now explain the statement of the aforementioned Weyl law and non-vanishing theorem, keeping our overview as non-technical as possible and deferring a more detailed summary to \S \ref{subsec:pfh}.  Let $\phi$ be an area-preserving diffeomorphism of $(\Sigma,\omega)$.  Let $M_\phi=[0,1]\times \Sigma/(1,p)\sim (0,\phi(p))$ be the mapping torus of $\phi$.  Let $V$ denote the vertical tangent bundle of $M_\phi$. The pull-back of $\omega$ to $[0,1]\times \Sigma$ defines a smooth $2$--form on $M_\phi$ because $\phi$ is area-preserving, and we denote this $2$--form by $\omega_\phi$. A class $\Gamma\in H_1(M_\phi;\mathbb{Z})$ has a well-defined {\bf degree} defined by pushing forward to $H_1(S^1) \cong \mathbb{Z}.$  

There is an invariant $\PFH(\phi)$ called the {\bf periodic Floer homology} of $\phi$; after a choice of coefficient ring $R$ it depends only on the Hamiltonian isotopy class of $\phi$.   Essentially, the PFH is the homology of a chain complex generated by certain sets of periodic points of $\phi$, relative to a differential that counts certain pseudoholomorphic curves in $\mathbb{R} \times M_{\phi}$.  In general, $\PFH(\phi)$ only has the structure of a graded $R$--module; however, there is a closely related homology $\TWPFH(\phi,\Theta)$
after fixing an additional data set of a {\bf reference cycle} $\Theta$ (see \S \ref{subsec:pfhDefinition}).
By definition, $\Theta$ is an embedded oriented closed $1$--manifold in $M_{\phi}$ transverse to the bundle $V$, such that a trivialization of $V$ along $\Theta$ has been fixed. 
The homology group $\TWPFH(\phi,\Theta)$ has an $\mathbb{R}$-valued filtration associated to the ``symplectic action functional" 
defined by $\omega_\phi$ (see \S \ref{subsubsec:pfhSpectralInvariantsDefn}).
We will suppress the coefficient ring $R$ from the notation most of the time, but in a handful of cases we will specify it and write $\TWPFH(\phi, \Theta; R)$. 

One can then attempt to compare the actions, gradings, and degrees of elements in $\TWPFH$, and this is the genesis of the questions that are taken up here.
We first state our non-vanishing result, which will guarantee that there 
exist interesting classes in $\TWPFH$ for which one can attempt to make these kinds of comparisons.  To set the stage, note first that there exist maps $\phi$, for example irrational translations of the two-torus, with no periodic points at all.  Thus, we can not expect an interesting non-vanishing result for $\TWPFH$ to hold for every $\phi$.  To state the relevant condition on $\phi$ for us, denote by $[\Theta] \in H_1(M_\phi;\mathbb{Z})$ the homology class induced by $\Theta$, and define the degree of $\Theta$ to be the degree of $[\Theta]$.  We say that a class $\Gamma \in H_1(M_\phi;\mathbb{Z})$ is {\bf monotone} if 
\begin{equation} \label{eq:cGamma}
2\text{PD}(\Gamma) + c_1(V) = - \rho [\omega_\phi]
\end{equation}  
for some constant $\rho \in \mathbb{R}$,
where $\text{PD}(\Gamma)$ denotes the Poincar\'e dual of $\Gamma$, 
and we say that $\Theta$ is monotone if $[\Theta]$ is.  Equation \eqref{eq:cGamma} guarantees that 
the change in grading is proportional to the change in action for different liftings of a generator from $\PFH$ to $\TWPFH$ (see \S \ref{subsec:pfh} for details).

\begin{thm}[Non-triviality of PFH]
\label{thm:nonvan}
Let $\left(\Sigma,\omega\right)$ and $\phi$ be given and let $R$ be a commutative ring with finite global dimension.   Then for any monotone reference cycle $\Theta$ in sufficiently high degree,
\[ \TWPFH(\phi,\Theta;R) \ne 0.\]
\end{thm}

Monotone reference cycles exist in abundance, as the following shows.

\begin{exe}
By definition, the map $\phi$ admits monotone classes if and only
if $[\omega_\phi] \in H^2(M_\phi;\mathbb{R})$ is proportional to an element in $H^2(M_\phi;\mathbb{Q})$.
When $\phi$ admits monotone classes, 
there are infinitely many monotone classes on $M_\phi$ with degrees tending to $+\infty$.  In this case, Theorem~\ref{thm:nonvan} implies that there are nonvanishing twisted PFH groups with degrees tending to $+\infty$.
\end{exe}

We will call a map $\phi$  {\bf monotone} if $[\omega_\phi] \in H^2(M_\phi;\mathbb{R})$ is proportional to an element in $H^2(M_\phi;\mathbb{Q})$. We will prove in Lemma \ref{lem:genericallyMonotone}  below that monotone maps are dense in $\Diff(\Sigma, \omega)$.  Our Weyl law will be about the asymptotics of actions on $\TWPFH(\phi,\Theta)$ when the degrees of $\Theta$ tend to $+\infty$.

Let us now explain the Weyl law that is the heart of our paper.   Fix $\phi$, let $H$ be any Hamiltonian, let $\phi^1_H$ be the associated time--$1$ Hamiltonian flow, and consider $\phi' = \phi\circ \phi^1_H$.  Then $\phi$ and $\phi'$ are Hamiltonian isotopic and so, as mentioned above, $\TWPFH(\phi)$ and $\TWPFH(\phi')$ are isomorphic.  In fact, as we will explain in \S \ref{subsec:pfhCobordismMaps}, this isomorphism is canonical. If $\sigma \in \TWPFH(\phi)$, we denote the corresponding element of $\TWPFH(\phi')$ by $\sigma^H$; a reference cycle $\Theta$ for $\phi$ also induces a reference cycle for $\phi'$ which we denote by $\Theta^H$.  As we also mentioned above, $\TWPFH$ has a grading, which we denote by $I$, and an action.  We denote the minimum action required to represent a class $\sigma$ for $(\phi,\Theta)$ by $c_{\sigma}(\phi,\Theta)$, deferring the precise definition to \S\ref{subsubsec:pfhSpectralInvariantsDefn}, and we call this the {\bf PFH spectral invariant} associated to $\sigma$.  Let $\pi: \Sigma \times [0,1] \to M_\phi$ denote the quotient map that defines the mapping torus.

\begin{thm}[The Weyl law] \label{thm:hutchingsConjectureFullintro}
Let $\Sigma$ be a closed
connected surface of genus $G$ with area form $\omega$ of area $1$. 
Fix $\phi,$ fix any 
 Hamiltonian $H \in C^\infty([0,1] \times \Sigma)$, fix any sequence of monotone reference cycles $\Theta_m$ with degrees $d_m$ tending to $+\infty$, and fix any two sequences of nonzero classes 
$\sigma_m, \tau_m \in \TWPFH_*(\phi,\Theta_m)$.  Then
\begin{equation*}
\lim_{m \to \infty} \frac{c_{\tau^{H}_{m}}(\phi', \Theta^H_m) - c_{\sigma_m}(\phi, \Theta_m) + \int_{\pi^{-1}(\Theta_m)} H dt}{(d_m + 1 - G)} - \frac{I(\tau_m) - I(\sigma_m)}{2(d_m + 1 - G)^2} 
 = \int_{\Sigma \times [0,1]} H \omega \wedge dt.
\end{equation*}
\end{thm}

This gives an essentially complete description of the asymptotic relationship between grading and action on PFH, with any coefficients, at the leading order.  
The proof of Theorem \ref{thm:hutchingsConjectureFullintro} can also yield an asymptotic bound for the subleading term (see  \eqref{eq:hutchingsProof10} and \eqref{eq:hutchingsProof11} below).
Exploring the subleading asymptotics
is an interesting question for further study, but is not the focus of the present work. 

We now present several examples to help the reader get a feel for what this Weyl law is saying.  

\begin{exe}
Let $\phi =id$ be the identity and $\phi' = \phi^1_H$ be the time-$1$ flow of a Hamiltonian on the two-disc that vanishes near the boundary.  We can embed this disc in the northern hemisphere of $\mathbb{S}^2$, extend $H$ by $0$, extend $\phi'$ as the time-$1$ flow of $H$ and extend $\phi$ as the identity.  
Then our Weyl Law recovers the classical Calabi invariant of $\phi'$ via the asymptotics of PFH spectral invariants.    For example, some particular choices\footnote{For these choices, we take $\tau_m = \sigma_m$, with each having grading $d_m$, and choose the reference cycles $\Theta_m$ to be in the Southern Hemisphere; then, the terms involving $I$, the $\int H dt$ term, and the $c_{\sigma_m}$ term all vanish.}
of the data in the Weyl Law lead to the formula
\[ \lim_{m \to \infty} \frac{c_{\tau^H_m}(\phi',\Theta^H_m)}{d_m + 1} = \text{CAL}(\phi'),\]
which is equivalent to a conjecture of Hutchings \cite[Rmk. 1.12]{simplicity20}; the resolution of a special case of this conjecture was a key step in the recent proof in \cite{simplicity20} of the longstanding  ``Simplicity Conjecture".  
\end{exe}

\begin{exe}
More generally, if $\phi'$ is any Hamiltonian diffeomorphism
of a compact surface with non-trivial boundary that is the identity near the boundary, then we can embed this surface in a closed surface and mimic the choices from the previous example.  One expects in this case to be able to recover the Calabi invariant via our PFH spectral invariants and this is planned for future work.  We are guaranteed infinitely many nonzero PFH classes by our Theorem~\ref{thm:nonvan}.  

In general, it is an interesting question to try to recover classical invariants from invariants of a more Floer-homological nature; this can illustrate, for example, that one has a multitude of sensitive invariants.  The above considerations show that this can be done for the Calabi invariant using our Weyl law and our non-vanishing result.  In fact, the Calabi invariant is essentially the only homomorphism out of the group of compactly supported Hamiltonian diffeomorphisms, in the sense that Banyaga has shown that the kernel of Calabi is a simple group \cite{banyaga1978structure, banyaga2013structure}. 
\end{exe}

\begin{exe}
\label{ex:zero}
Let $H = 0$ in Theorem~\ref{thm:hutchingsConjectureFullintro}.  Then $\phi = \phi'$ and all terms in Theorem~\ref{thm:hutchingsConjectureFullintro} involving $H$ vanish.     Theorem~\ref{thm:hutchingsConjectureFullintro} in this case now asserts that PFH classes of approximately equal grading have approximately equal action.   Comparisons between grading and action play a key role in some applications of ECH to dynamics \cite{CGHutchings16, CGHP19, CGHHL21, CGMazzucchelli20}.
\end{exe}

The perspective,
then, is that our theorems are quite general regarding PFH and its quantitative structure.  One expects this addition to 
the current toolkit available for studying area-preserving surface diffeomorphisms to have various implications.  Its cousin, the Weyl law for ``embedded contact homology", has had many applications regarding the study of Reeb vector fields on three-manifolds, see \cite{irie, CGHutchings16, CGHP19, irieEquidistribution, CGHHL21, CGMazzucchelli20}; it has been called ``perhaps the deepest property of ECH" \cite[Sec. 1.2]{CGHP19}.
A first example of an application appears in the work \cite{rpEquidistribution}, which uses the Weyl law and the nonvanishing result to quantitatively refine Corollary \ref{cor:genden} and show a ``generic equidistribution'' theorem for periodic points. 

For the application in the present work, namely the smooth closing lemma, the key philosophical point is that if $U \subset \Sigma$ is any open set and $H$ is supported in $U$, then the Weyl law guarantees that the PFH spectral invariants are sensitive enough to detect $H$.  This is a kind of ``locality" property that forces an abundance of periodic points for generic choice of data.

\subsubsection*{Further properties of PFH spectral invariants}

To actually deduce Theorem~\ref{thm:closingLemma} from the Weyl Law, we prove some additional properties of PFH spectral invariants.  These have been expected at least since the work \cite{simplicity20}, and our proofs are standard extensions of the ideas in \cite{simplicity20}, but we briefly mention these results here since they are likely to be useful more generally.
To elaborate, we show in \S\ref{subsec:pfhSpectralInvariantProperties}  that these PFH spectral invariants satisfy ``spectrality" and ``Hofer continuity" properties. The spectrality axiom guarantees that the spectral invariants are actions of periodic points.  The Hofer continuity property allows one to estimate the difference between spectral invariants in terms of the Hofer norm.
We defer the precise statements to \S\ref{subsec:pfhSpectralInvariantProperties}.

\subsection{Ideas of the proofs}
We now give an overview of the ideas and tools needed to prove our main theorems.  

\subsubsection{SWF spectral invariants and the basic idea behind the proof}

As in the proof of the ECH Weyl law \cite{CGHR15}, a key role is played by Seiberg--Witten invariants.
In the PFH case, this is packaged via a 
beautiful theory due to Hutchings and Lee--Taubes which shows that surface dynamics is related to Seiberg--Witten theory via an isomorphism on Floer homology \cite{LeeTaubes12}.  

The Seiberg--Witten--Floer cohomology is a cohomology theory for smooth $3$--manifolds developed by Kronheimer--Mrowka \cite{monopolesBook}. As with PFH, we will work with a ``twisted'' variant of Seiberg--Witten--Floer cohomology; the $3$--manifold in question will always be the mapping torus of an area-preserving surface diffeomorphism.   As with PFH, we give only an impressionistic summary here, deferring the details to \S\ref{subsec:swf}.  Essentially, the twisted Seiberg--Witten--Floer cohomology is the homology of a chain complex generated by solutions to the ``$r$-perturbed" Seiberg-Witten equations on $M_\phi$, relative to a differential counting solutions to the four-dimensional equations on $\mathbb{R} \times M_\phi$.   Here, $r > 0$ is a positive real number which is a very important parameter in our story.   The data of a reference connection $\fc_\Gamma$ endows the twisted Seiberg--Witten--Floer cohomology 
with a quantitative structure defined by tracking the values of a functional called the \emph{Chern--Simons--Dirac functional}, which is denoted by $\fa_{r, \fg}(-, \fc_\Gamma)$.   The quantitative data is packaged into what we call ``SWF spectral invariants''. These are real numbers 
associated to nonzero classes of the twisted Seiberg--Witten--Floer cohomology.
Similarly to the PFH case, the spectral invariant $c^{\HM}_\sigma(\phi,r; \fc_\Gamma)$ is defined, informally, as the largest value of $r^{-1}\fa_{r, \fg}(-, \fc_\Gamma)$ required to represent $\sigma$ as a cocycle.

Lee--Taubes \cite{LeeTaubes12} established an isomorphism 
$$\TWPFH_* \cong \TWHM^{-*}$$
between twisted PFH and twisted SWF which reverses the relative $\bZ$--gradings.
It stands to reason that there should be a correspondence between the associated spectral invariants of PFH and SWF as well. Indeed, in \S\ref{sec:twistedIso} we show that the following result, stated in full generality in Proposition \ref{prop:spectralInvariantIdentity}, holds.  Given a reference cycle $\Theta$ for PFH, we can choose a one-parameter family of base configurations $\{\fc_{\Gamma}(r)\}_{r > -2\pi\rho}$ which ``concentrate'' around the reference cycle $\Theta$ such that the associated SWF spectral invariants recover the PFH spectral invariants in the limit as $r \to \infty$:
\begin{equation} \label{eq:introSwfRecoversPfh} \lim_{r \to \infty} c_\sigma^{\HM}(\phi, r; \fc_\Gamma(r)) = -\pi c_\sigma(\phi, \Theta).\end{equation}

We can now explain the idea behind the proof of the Weyl law.
We will use the equation (\ref{eq:introSwfRecoversPfh}) together with some continuity properties of the SWF spectral invariants as one varies $r$
to estimate the PFH spectral invariants with a high degree of precision.  
We show that there is a continuous, piecewise-smooth family of SWF spectral invariants $c_\sigma^{\HM}(\phi, r; \fc_{\Gamma})$
for $r \in (-2\pi\rho, \infty)$. For large $r$, \eqref{eq:introSwfRecoversPfh} shows that these approximate the PFH spectral invariants well.
On the other hand, the aforementioned continuity properties as one varies $r$ and some smoothness properties allow us to estimate the large $r$ behavior from the behavior at smaller values of $r$; 
it turns out that at a particular $r$, the spectral invariant is carried by a ``reducible solution", see \S\ref{subsubsec:swEquations}.  Critically, 
the reducible locus can be described rather explicitly.  Putting this all together with some extra work allows us to obtain the estimate on the PFH spectral invariants asserted by our Weyl Law.

\subsubsection{Some more details via challenges and comparisons with previous works}

As mentioned above, Periodic Floer homology has a cousin called Embedded Contact Homology (ECH), and an analogous Weyl law for ECH has been known since $2012$ \cite{CGHR15}, so one might ask why the Weyl law for PFH was not proved around the same time; such a Weyl law has long been expected to have interesting applications.  Indeed, an expert reader might note that, in broad strokes, the proof strategy outlined above is similar to the proof in \cite{CGHR15}.  More precisely, the paper \cite{CGHR15} used a Seiberg--Witten approach for the key upper bound, and it was shown in \cite{WeifengMinMax} that one can in fact recover the entire main theorem of \cite{CGHR15} with Seiberg--Witten theory. 

To shed light on this, and to further explain some important points in the proof, we now explain several main challenges in adapting the Seiberg-Witten approach to the present context, together with their resolution.  The language in this section is necessarily more involved than in our impressionistic overview, and is meant for a more specialist reader; it can be skipped and returned to later if desired.

First of all, the analogue of the action on ECH on the Seiberg--Witten side is the ``energy functional" $\frac{i}{2\pi} \int \lambda \wedge F_B,$ where $\lambda$ is a contact form on a three-manifold and $F_B$ is the curvature of a Seiberg--Witten solution: roughly speaking, this is approximately the action of an orbit set in the isomorphism between $ECH$ and $HM$, see \cite[Prop. 2.6]{CGHR15} for the precise statement.  The analogue of the energy in the Lee--Taubes paper is the functional $\frac{i}{2\pi} \int dt \wedge F_B. $
However, since $dt$ is closed, this is a purely cohomological quantity, in particular an integer, and does not generate interesting spectral invariants.  

Second, 
the ``reducible" locus, mentioned above, 
plays a central role in the proofs of both \cite{CGHR15} and \cite{WeifengMinMax}: roughly speaking, a key idea in both proofs is to reduce to some computation about reducibles that can be carried out explicitly.  However, as written, the Seiberg--Witten equations in Lee--Taubes have no reducible solutions at all. 

Third, the Weyl law for ECH only involves a fixed spin-c structure and the corresponding Seiberg--Witten equations are defined on a fixed bundle.  In the PFH case, however, the Weyl law is stated by taking the degrees of $\Theta_m$ to infinity. So for PFH, one has to study the asymptotic behavior of the Seiberg--Witten equations on a sequence of vector bundles, and prove estimates which are robust against these changes.

To explain our resolution of the first issue, first fix a base connection $B_0$.  Then, by Stokes' theorem, the data contained in the integrand is essentially the same as the integral $\int d \lambda \wedge b$, where $F_B = B_0 + b$.  This suggests a formulation which it turns out adapts to the PFH case: we can think of the pair $(\lambda,d\lambda)$ as a stable Hamiltonian structure in the contact case, and then in analogy study the integrand $\int \omega_\phi \wedge b$
for the stable Hamiltonian structure $(dt,\omega_{\phi})$ on the mapping torus.  This is essentially the approach that we take. 
This energy functional is encoded in the rescaled action $r^{-1}\fa_{r, \fg}(-, \fc_\Gamma)$. 

For the second issue, we consider instead a variant of the Lee--Taubes equations that do admit reducible solutions.  
Even when reducible solutions exist, it is not so clear how relevant they are to the spectral invariants, since 
in principle the Seiberg-Witten spectral invariants could always be carried by irreducibles.  In the ECH case, the relationship was given by an argument essentially due to Taubes in his celebrated proof of the Weinstein conjecture: one compares the spectral flow of reducibles and irreducibles, see \cite[Section $2$]{WeifengMinMax} or \cite{CGSavaleSpectralFlow}.  However, this argument does not work in our case, and we argue in Proposition \ref{prop:maxMinReducibles} via a completely different approach, which interestingly involves considering the completed Seiberg--Witten groups.  

For the third issue, we will endow different spin-c structures with different base connections.  This introduces new analytic difficulties but we show that we can choose the base connections so that there are some uniform bounds given by the degree $d$. This allows us to derive uniform estimates for all spin-c structures.

The above ideas introduce some further complications, but for brevity we stop here.  On a technical level, we should note that the Lee--Taubes' isomorphism \cite{LeeTaubes12} between periodic Floer homology and Seiberg--Witten--Floer cohomology is not quantitative, compared to work in the ECH case, so we have to build a lot of this quantitative structure from the ground up.

As a historical remark, one should also note that, a priori, it is not so clear that such a Weyl law should even exist: for example, it was perhaps conceivable that the contact structure in the ECH case plays a crucial role.  In this regard, key evidence for a Weyl law came from a computation of Hutchings, who computed spectral invariants for rotations of the two-sphere and made conjectures along the lines mentioned here, and \cite{simplicity20}, which established the Weyl law for the large class of ``monotone twist" maps on $S^2$, 

As for applying the Weyl law to prove the closing lemma, the new conceptual difficulty here is to get a strong enough nonvanishing result for PFH: ECH is always nonvanishing, whereas as explained above PFH can essentially vanish.

\subsubsection{The rest of the proofs} 
\label{sec:compare}

It remains to explain the idea behind the nonvanishing theorem (Theorem \ref{thm:nonvan}) and the closing lemma.  The nonvanishing theorem is also proved by consideration of the reducible locus; we use a variant of the ``third root of unity trick" from \cite[Cor. 35.1.3]{monopolesBook} to prove the nonvanishing of a corresponding Seiberg--Witten group, and then apply an algebraic argument to extend the non-vanishing result to the twisted case; the non-vanishing of twisted PFH then follows from the Lee--Taubes isomorphism \cite{LeeTaubes12}. The proof of Theorem \ref{thm:closingLemma}, the closing lemma, uses a now standard argument pioneered by Irie, adapted to our purposes and powered by the previous results: Let $U$ and $V$ be as in the statement of the theorem. We can assume after a small perturbation that $\phi$ is nondegenerate and admits monotone classes.  Then the nonvanishing theorem Theorem~\ref{thm:nonvan} guarantees an abundance of PFH spectral invariants.  If Theorem \ref{thm:closingLemma} is false, then for any small Hamiltonian $H$ supported in $U$, the perturbation $\phi' = \phi \circ \phi^1_H$ will have the exact same set of periodic points as $\phi$ itself. This implies, using properties of PFH spectral invariants,
that the PFH spectral invariants do not change under any such perturbation. However, using the Weyl law Theorem \ref{thm:hutchingsConjectureFullintro}, the difference of the PFH spectral invariants of $\phi$ and $\phi'$ recovers the Calabi invariant of $H$, which yields a contradiction because $\phi$ and $\phi'$ may have different Calabi invariants.

\subsubsection{Comparison with the work of Edtmair-Hutchings} 
In the course of the final preparations of the first version of this article, we became aware of work of Oliver Edtmair and Michael Hutchings \cite{edtmairHutchings}, who, simultaneously and independently of us, use a different method to apply PFH spectral invariants to prove, in certain cases, the smooth closing lemma and a Weyl law.
For the convenience of the reader, we now provide a summary of the differences between the two works, and we also mention some relevant developments that occurred after our work and the work of Edtmair--Hutchings first appeared. 

To be precise, Edtmair--Hutchings introduce a PFH theoretic ``$U$--cycle condition", on Hamiltonian isotopy classes, and they prove that for any Hamiltonian isotopy class satisfying this condition, a $C^{\infty}$ closing lemma holds; they prove that rational isotopy classes on $T^2$ and the unique isotopy class on $S^2$ satisfies this condition.  Note that to prove the $C^{\infty}$ closing lemma for a surface, it suffices to prove it on rational isotopy classes.  We later showed in \cite{CGPPZ21}, building on the  ideas for the nonvanishing result in this work, that the $U$--cycle condition holds for all rational isotopy classes.  

As for a Weyl law, Edtmair--Hutchings prove a Weyl law \cite[Thm. 8.1]{edtmairHutchings} that computes
the asymptotic change in spectral invariant for $\phi$ and $\phi' = \phi \circ \phi^1_H$ associated to a sequence of $U$-cyclic PFH classes; the PFH grading plays no role in their law.
Our Weyl law is more general in that it has no $U$-cycle requirement and applies to arbitrary pairs of sequences of PFH classes, rather than a single sequence, see Example~\ref{ex:zero} for a simple illustration of this distinction.   As first remarked to us by Edtmair, our argument in \cite{CGPPZ21} can be used to show that over some coefficient rings, for example $\mathbb{Z}/2$, every PFH class is $U$--cyclic.  On the other hand, we proved in \cite{CGPPZ21} that over some rings, for example $\mathbb{Z}$, there are many classes which are not $U$--cyclic. 

The work of Edtmair--Hutchings also includes an interesting quantitative result for area-preserving diffeomorphisms in rational isotopy classes,  estimating the time it takes orbits of certain periods to appear in a fixed open set along a Hamiltonian perturbation.  This question is not taken up in the present work, which is focused on the PFH Weyl law and its application to the resolution of Smale's question, at all.  

We also note that some time after the first version of our paper appeared, Edtmair \cite{edtmair_elementary} gave a new holomorphic curve-based approach that recovers the results of Edtmair--Hutchings \cite{edtmairHutchings} for the case of Hamiltonian diffeomorphisms. 

\vspace{3 mm}

\textbf{Acknowledgements: } The authors would like to thank Daniel Pomerleano and Clifford Taubes for various helpful discussions. D.C-G. would like to thank Vincent Humili{\`e}re, Kei Irie, Peter Kronheimer, Nikhil Savale and Sobhan Seyfaddini for useful conversations.  R.P. would like to thank his advisor, Helmut Hofer, as well as Mike Miller Eismeier. B.Z. would like to thank Zhouli Xu for helpful conversations regarding homological algebra.  We also thank Oliver Edtmair and Michael Hutchings for correspondence regarding their work that was discussed above. 

The research of D.C-G. is supported by the NSF under Awards \#1711976 and \#2105471.  Part of this research took place when D. C-G. was supported by the Institute for Advanced Study under a von Neumann fellowship.  He thanks the Institute for this support.  Very useful conversations with Humili{\`e}re and Seyfaddini occurred when D. C-G. was an ``FSMP Distinguished Professor” at the
Institut Math{\'e}matiques de Jussieu-Paris Rive Gauche (IMJ-PRG) and he thanks the Fondation Sciences Math{\'e}matiques de Paris for their support as well.  The research of R.P. is supported by the NSF (Graduate Research Fellowship) under Award \#DGE-1656466.

\section{Preliminaries} \label{sec:prelim}

We now review the relevant background and notation that we will need for our proofs.

\subsection{Area-preserving and Hamiltonian diffeomorphisms of surfaces} \label{subsec:areaPreservingMaps}

Here we collect some basic definitions and fix notation regarding area-preserving and Hamiltonian diffeomorphisms of surfaces. 

Let $(\Sigma,\omega)$ denote a closed, oriented surface equipped with an area form $\omega$.
We let $A_\Sigma = \int_\Sigma \omega$ 
denote the area of $(\Sigma,\omega)$; 
we let $G$ denote its genus.
Let  $\text{Ham}(\Sigma, \omega)$ 
denote the subgroup of \textbf{Hamiltonian diffeomorphisms}: these are the time-$1$ maps of the time-dependent vector field 
$X_{H_t}$ defined uniquely at each $t \in \mathbb{R}$ by the equation $\omega(X_{H_t}, -) = -dH_t.$
Here, $H \in C_c^\infty((0,1)\times \Sigma)$; in other words, $H=0$ near $t = 0$ and $t = 1$.  
For brevity, we will often drop the time index and write the Hamiltonian vector field as $X_H$.  The flow of $X_H$ is written as $\{\phi^s_H\}_{s \in \mathbb{R}}$. We say that two diffeomorphisms $\phi$ and $\phi'$ in $\text{Diff}(\Sigma, \omega)$ are \textbf{Hamiltonian isotopic} if there is a Hamiltonian $H \in C^{\infty}_c( (0,1) \times \Sigma )$ such that $\phi' = \phi \circ \phi^1_H$. 

We will write 
$ \widetilde{\Ham}(\Sigma, \omega)$
for the space of smooth paths of Hamiltonian diffeomorphisms
starting at the identity map.
It tends to be convenient to require that the paths are constant near the endpoints, and we will always assume this unless otherwise stated; this can always be achieved by a reparametrization. 
It is a standard fact that any such isotopy is generated by a Hamiltonian $H \in C^{\infty}_c( (0,1) \times \Sigma )$.  
The quotient of $\widetilde{\Ham}(\Sigma, \omega)$ by homotopy of paths relative to endpoints yields the universal cover
$\overline{\Ham}(\Sigma, \omega)$
of ${\Ham}(\Sigma, \omega)$.

There is a natural norm on $\Ham(\Sigma,\omega_{\Sigma})$, called the {\bf Hofer norm}, that is in the background of much of what we do here.  In the present work, we only need to recall the Hofer norm for Hamiltonians: for any $H \in C^{\infty}(\mathbb{R}/\mathbb{Z} \times \Sigma)$, we define 
$\|H\|_{1,\infty} := \int^1_0 \text{max}(H(t,\cdot)) - \text{min}(H(t,\cdot)) dt.$
This can be used to define a norm on $\Ham(\Sigma,\omega_{\Sigma})$, but we do not need this here. There is also an important real-valued invariant of Hamiltonian diffeomorphisms featured in this work. 
Fix any union $\bD \subset \Sigma$ of embedded disks and fix any Hamiltonian $H \in C^{\infty}_c( (0,1) \times \Sigma )$, we define the \textbf{Calabi invariant} of $H$ as $\text{CAL}(H) = \int_0^1 \big( \int_{\bD} H \omega \big) dt.$
Unpublished work of Fathi gives a geometric interpretation of $\text{CAL}$ as the ``average rotation" of the Hamiltonian diffeomorphism associated to $H$, see \cite[Section $2.1$]{Ghys07}.
 It turns out \cite{Calabi70, McduffSalamonBook} that the Calabi invariant only depends on the time-one map $\phi^1_H \in \Diff_c(\Sigma, \omega, \bD)$ of such $H$
 and therefore descends to a homomorphism on the group of area-preserving diffeomorphisms that are compactly supported in $\bD$.

\subsection{Periodic Floer homology} \label{subsec:pfh}

We now give a brief overview of periodic Floer homology (PFH), filling in the details that we omitted in the introduction. We also define PFH spectral invariants, following the approach of \cite{simplicity20}. 

\subsubsection{Stable Hamiltonian structures}
\label{subsec_stable_hamiltonia_structure}
Let $M$ be an oriented three-manifold. A \textbf{stable Hamiltonian structure} on $M$ is a pair $(\lambda, \omega)$ of a one-form $\lambda$ and a closed two-form $\omega$ such that $\lambda \wedge \omega > 0$ and $\ker(\omega) \subset \ker(d\lambda)$.
A stable Hamiltonian structure determines a two-plane field
$\xi = \ker(\lambda)$
and a \textbf{Reeb vector field} $R$, which is the unique vector field such that
$\omega(R, -) \equiv 0, \lambda(R) \equiv 1.$  An \textbf{embedded Reeb orbit} $\gamma$ is a closed leaf of the one-dimensional foliation spanned by $R$. A general \textbf{Reeb orbit} is a pair $(\gamma, m)$ of an embedded Reeb orbit $\gamma$ and a number $m \in \mathbb{N}$, called its \textbf{multiplicity}. 
 A \textbf{parameterization} of a Reeb orbit $(\gamma, m)$ is a smooth map 
$f: \mathbb{R}/mT\mathbb{Z} \to M$
such that $f$ is an $m$-fold cover of $\gamma \subset M$ and $\dot f(t) = R(f(t))$ for all $t \in \mathbb{R}/mT\mathbb{Z}$.  
A parametrized Reeb orbit $f: \mathbb{R}/mT\mathbb{Z} \to M$ is \textbf{nondegenerate} if $1$ is not an eigenvalue of the linearized Reeb flow; this does not depend on the choice of the parameterization. 

Given a stable Hamiltonian structure $(\lambda, \omega)$ on $M$,  let $\cJ(\lambda, \omega)$ denote the space of smooth vector bundle endomorphisms
$J: \xi \to \xi$
with square $-1$ such that $\omega(v, Jv) > 0$ for all nonzero $v \in \xi$. The space $\cJ(\lambda, \omega)$ is nonempty and contractible.
An almost-complex structure $\oJ$ on $\mathbb{R}_s \times M$ is \textbf{admissible} with respect to $\omega$ if it is invariant under translation in the $\mathbb{R}$ direction, sends $\partial_s$ to $R$, preserves $\xi$, and  $\omega(v, \oJ v) > 0$ for all nonzero $v \in \xi$.  
It is evident that admissible almost-complex structures 
are in bijection with elements of $\cJ(\lambda, \omega)$ via restrictions to $\xi$, and we will use this to sometimes refer to an element $J \in \cJ(\lambda, \omega)$ as an admissible almost-complex structure.

\subsubsection{Definition of periodic Floer homology} \label{subsec:pfhDefinition}

Let $(\Sigma, \omega)$ be a closed, connected surface equipped with an area form. Let $\phi \in \Diff(\Sigma, \omega)$ be an area-preserving diffeomorphism. 
 Recall that the area form $\omega$ pulls back to a closed 2-form on the mapping torus $M_\phi$, which is denoted by $\omega_\phi$. Let $t$ be the coordinate for the interval component of $[0,1] \times \Sigma$. Then $dt$ pushes forward to a smooth 1-form on $M_\phi$. The pair $(dt,\omega_\phi)$ forms a stable Hamiltonian structure on $M_\phi$ and we will denote the Reeb vector field by $R$. The associated two-plane bundle $\ker(dt)$ is equal to the vertical tangent bundle of the fibration $M_\phi \to S^1$, which we denote by $V$.  
 Note that periodic orbits of $R$ coincide with periodic points of the surface diffeomorphism $\phi$.  A periodic point $p \in \Sigma$ of $\phi$, of period $k$ is \textbf{nondegenerate} if the derivative of $\phi^k$ at $p$ does not have $1$ as an eigenvalue.  We say $\phi$ is \textbf{$d$--nondegenerate} if all periodic points of period at most $d$ are nondegenerate. We say $\phi$ is \textbf{nondegenerate} if $\phi$ is $d$--nondegenerate for all $d$. It is a basic fact that the set of $d$--nondegenerate diffeomorphisms in $\Diff(\Sigma, \omega)$ is open and dense with respect to the $C^\infty$ topology. 
Now take some nonzero homology class $\Gamma \in H_1(M_\phi; \bZ)$. Recall from the introduction that the \textbf{degree} of $\Gamma$ is defined to be the degree of the image of $\Gamma$ via the projection $M_\phi \to S^1$. We assume that the degree of $\Gamma$ is positive and greater than the genus $G$ of the surface $\Sigma$. The main results of this paper only use PFH for $\Gamma$ of high degree, so this is not a restrictive assumption. It is made to rule out certain types of holomorphic curve bubbling, see \cite[Theorem $1.8$]{HutchingsECH02} and the surrounding discussion. 

The \textbf{PFH generators} are
 finite sets
$\Theta = \{(\gamma_i, m_i)\}$
of pairs of embedded Reeb orbits $\gamma_i$ and multiplicities $m_i \in \mathbb{N}$ which satisfy the following three conditions: (1) the orbits $\gamma_i$ are distinct, (2) the multiplicity $m_i$ is $1$ whenever $\gamma_i$ is a hyperbolic orbit, (3) $\sum_i m_i[\gamma_i] = \Gamma$. 
The chain complex $\PFC_*(\phi, \Gamma, J)$ is the free module generated by the set of all PFH generators. For the proof of the closing lemma, it suffices to take the coefficient ring to be $\bZ/2$, but the Weyl law in Theorem \ref{thm:hutchingsConjectureFullintro} holds for general coefficient rings. From now on, we will fix an arbitrary coefficient ring and omit it from the notation unless otherwise specified. 
 
The differential on $\PFC_*(\phi, \Gamma, J)$ counts ``ECH index 1" curves in $\mathbb{R} \times M_\phi$. More precisely,  define 
$\mathbf{M}_1(\Theta_+, \Theta_-; J)$
to be the space of $J$-holomorphic currents $C$ in $\mathbb{R}_s \times M_\phi$, modulo translation in the $\mathbb{R}$-direction with ECH index $I(\Theta_+, \Theta_-, [C]) = 1$, which are asymptotic to $\Theta_\pm$ as the $\mathbb{R}$-coordinate $s$ limits to $\pm \infty$; we refer the reader to \cite[\S$3.4$]{ECHNotes} for the definition of the ECH index.  There is a generic subset 
$ \cJ^\circ(dt, \omega_\phi) \subset \cJ(dt, \omega_\phi) $
of admissible almost-complex structures such that for any $J \in \cJ^\circ(dt, \omega_\phi)$ and any pair of PFH generators $\Theta_{\pm}$, the space
$\mathbf{M}_{1}(\Theta_+, \Theta_-; J)$
is a (possibly empty) compact oriented  $0$-dimensional manifold and we define the PFH differential by the rule
$\langle \partial\Theta_+, \Theta_-\rangle = \#\mathbf{M}_{1}(\Theta_+, \Theta_-; J)$,
where $\#$ denotes the cardinality counted with signs.  The set $\cJ^\circ(dt, \omega_\phi)$ can be chosen so that $\partial^2 = 0$ (see \cite{HutchingsTaubes07}, \cite{HutchingsTaubes09}). It follows that, for $\phi$ nondegenerate, $\Gamma$ negative monotone of degree greater than $G$, and $J \in \cJ^\circ(dt, \omega_\phi)$, the periodic Floer homology 
$\PFH_*(\phi, \Gamma, J)$
is well-defined; it has a relative $\bZ/\ell$--grading induced by the ECH index, where $\ell$ is the divisibility of $c_\Gamma$ defined by equation \eqref{eq:cGamma} regarded as an element of $H^2(M_\phi;\bZ)$. 

To actually define spectral invariants, we need to instead work with the variant \textbf{twisted periodic Floer homology}, as defined by Hutchings and as introduced in \cite{simplicity20} in the case where $M_\phi = S^1 \times S^2$. Define a \textbf{(twisted) PFH parameter set} 
$\mathbf{S} = (\phi, \Theta_{\text{ref}}, J)$
to be a tuple consisting of the following terms. 
The term $\phi$ in $\mathbf{S}$ is a $d$-nondegenerate area-preserving diffeomorphism of $(\Sigma, \omega)$.  The term $J$ is an admissble almost-complex structure. The term $\Theta_{\text{ref}}$ refers to a choice of \textbf{trivialized reference cycle}. This is itself the datum of two objects. The first is a set of embedded loops with multiplicity $\Theta_{\text{ref}}$ in $M_\phi$, such that the class $[\Theta_{\text{ref}}] \in H_1(M_\phi; \mathbb{Z})$ is negative monotone with degree greater than $G$. The second is a homotopy class of symplectic trivializations of the bundle $V$ over $\Theta_{\text{ref}}$, which we suppress from the notation. It will also be useful in our discussion of the Lee--Taubes isomorphism in \S\ref{sec:twistedIso} to require that $\Theta_{\text{ref}}$ 
is \textbf{separated}, which means that the the embedded loops in $\Theta_{\text{ref}}$ are pairwise disjoint and have multiplicity $1$, $\Theta_{\text{ref}}$ is transverse to $V$, and  $\Theta_{\text{ref}}$ is disjoint from the union of all embedded Reeb orbits of degree less than or equal to $d$ in the mapping torus $M_\phi$.
It is always possible to find separated reference cycles for every positive degree $d$ such that $\phi$ is $d$--nondegenerate.

The \textbf{twisted PFH}, denoted by $\TWPFH_*(\mathbf{S}) = \TWPFH(\phi, \Theta_{\text{ref}}, J)$,
is the homology of the chain complex $\TWPFC_*(\mathbf{S}),$
which we now define.  The complex $\TWPFC_*(\phi, \Theta_{\text{ref}}, J)$ is the free module generated by formal sums of pairs $(\Theta, W)$, where $\Theta$ is a generator of the untwisted complex $\PFC_*(\phi, [\Theta_{\text{ref}}], J)$ and $W$ is an element of the relative homology group $H_2(M_\phi, \Theta, \Theta_{\text{ref}}; \bZ)$. 
The differential is defined as follows. Let $\Theta_-$ and $\Theta_+$ be two PFH generators, and let $C$ be a $J$--holomorphic current in $\bR \times M_\phi$ asymptotic to $\Theta_\pm$ as $s \to \pm \infty$.  The projection of $C$ to $M_\phi$ determines a relative homology class $[C] \in H_2(M_\phi, \Theta_+, \Theta_-; \bZ)$. 
For any class $W \in H_2(M_\phi, \Theta_+, \Theta_-; \bZ)$, define 
$\mathbf{M}_1(\Theta_+, \Theta_-, W; J) \subset \mathbf{M}_1(\Theta_+, \Theta_-; J)$
to be the space of such currents $C$ such that $[C] = W$. 
We define the twisted PFH differential $\widetilde{\partial}$ by setting
$$\langle \widetilde{\partial}(\Theta_+, W_+), (\Theta_-, W_-) \rangle = \#\mathbf{M}_1(\Theta_+, \Theta_-, W_+ - W_-; J).$$
 The arguments of \cite{HutchingsTaubes07, HutchingsTaubes09} apply in this setting as well to show that $\widetilde{\partial}^2 = 0$.   Note that the twisted PFH chain complexes for two distinct choices of trivialized reference cycles $\Theta_{\text{ref}}$, $\Theta_{\text{ref}}'$ are (non-canonically) isomorphic via addition of a class $W' \in H_2(M_\phi, \Theta_{\text{ref}}', \Theta_{\text{ref}}; \mathbb{Z})$. The ECH index induces an absolute $\mathbb{Z}$-grading on $\TWPFC_*(\mathbf{S})$ which we will denote by $I$, see \cite[\S$3.3$]{simplicity20}.

\subsubsection{PFH spectral invariants} \label{subsubsec:pfhSpectralInvariantsDefn}

Fix a PFH parameter set
$\bfS = (\phi, \Theta_{\text{ref}}, J).$
The complex 
$\TWPFC_*(\bfS)$
admits a natural filtration by the \textbf{PFH action functional} defined by
$\bfA(\Theta, W) = \langle \omega_\phi, W \rangle.$
We define the submodule 
$\TWPFC_*^L(\bfS) = \TWPFC_*^L(\phi, \Theta_{\text{ref}}, J)$
of $\TWPFC_*(\bfS)$ to be the submodule generated by the twisted PFH generators $(\Theta, W)$ satisfying
$\bfA(\Theta, W) \leq L.$ Since the restriction of $\omega$ on a $J$--holomorphic curve for admissible $J$ is always pointwise non-negative, $\TWPFC_*^L(\phi, \Theta_{\text{ref}}, J)$ is a subcomplex of $\TWPFC_*(\phi, \Theta_{\text{ref}}, J)$.  We write
$\TWPFH_*^L(\bfS) = \TWPFH_*^L(\phi, \Theta_{\text{ref}}, J)$
to denote the homology of the complex $\TWPFC_*^L(\bfS)$. The inclusion
$\TWPFC_*^L(\bfS) \hookrightarrow \TWPFC_*(\bfS)$
induces a map
$\iota_{\TWPFH}^L(\bfS): \TWPFH_*^L(\bfS) \to \TWPFH_*(\bfS).$  

Let $\sigma \in \TWPFH_*(\bfS)$ be a non-zero homology class. The \textbf{PFH spectral invariant} of $\sigma$, denoted by 
$c_\sigma(\bfS) = c_\sigma(\phi, \Theta_{\text{ref}}, J)$,
is the infimum of all $L$ such that $\sigma$ lies in the image of the map $\iota_{\TWPFH}^L(\bfS)$. 

When $\Gamma$ is monotone, at a given grading, there are only finitely many possible values for the PFH action functional on the generators. Therefore, the PFH spectral invariant of every non-zero class is finite.

When $\phi$ is degenerate, we can still define spectral invariants.  We will state the definition here, but defer some details to later.   Consider a tuple $\bfS = (\phi, \Theta_{\text{ref}}, J)$, where the elements of $\bfS$ are as in the definition of a PFH parameter set, except we no longer require any nondegeneracy from $\phi$.  We define 
\begin{equation}
\label{eqn:defdeg}
c_\sigma(\bfS) := \lim_{n  \to \infty} c_\sigma(\phi \circ \phi^1_{H_n}, \Theta_{\text{ref}}^{H_n}, J^{H_n}),
\end{equation}
where the $\phi \circ \phi^1_{H_n}$ are $d$-nondegenerate,  and $\|H_n\|_{1,\infty} \to 0$.  Here, $\| \cdot \|_{1,\infty}$ denotes the Hofer norm, and in referring to the same class $\sigma$ with respect to the different parameter sets $(\phi \circ \psi^1_{H_n}, \Theta_{\text{ref}}^{H_n}, J^{H_n})$, we are using the identification of the corresponding twisted PFH groups via the canonical cobordism maps that we defer to \S\ref{subsec:pfhCobordismMaps}.  We also defer the proof that this is well-defined to Proposition \ref{prop:hoferContinuity1}.

Later we will show that the PFH spectral invariants 
$c_\sigma(\bfS) = c_\sigma(\phi, \Theta_{\text{ref}}, J)$
 are ``independent of $J$'' in a suitable sense. This requires discussing the canonical identification of twisted PFH complexes with distinct almost-complex structures, and so we will defer this to \S\ref{subsec:pfhSpectralInvariantProperties}.  However, note that the PFH spectral invariants are certainly dependent on the choice of trivialized reference cycle $\Theta_{\text{ref}}$.  
 
\subsubsection{Spectrality}
\label{subsubsec_spectrality}
 A fundamental property of PFH spectral invariants is \emph{spectrality}: the PFH spectral invariant is equal to the action of some PFH generator $(\Theta, W)$. 
 Fix a PFH parameter set $\bfS = (\phi, \Theta_{\text{ref}}, J),$ and fix a non-zero class $\sigma \in \TWPFH_*(\bfS)$. 
 	
 	\begin{prop}
 		\label{prop:pfhSpectrality}
 		There is a twisted PFH generator $(\Theta, W)$ such that
 		$\bfA(\Theta, W) = c_{\sigma}(\bfS).$
 	\end{prop}
 	
 	\begin{proof}
 		Suppose first that $\phi$ is nondegenerate. By definition of the spectral invariant, there is a sequence $(\Theta_k, W_k)$ of twisted PFH generators such that
 		$\lim_{k \to \infty} \int_{W_k} \omega_\phi = c_{\sigma}(\bfS)$. 
 		
 		Note that the orbit sets $\Theta_k$ all have the same degree $d$. Since $\phi$ is nondegenerate, there are finitely many orbit sets of degree $d$. Therefore, after passing to a subsequence we can assume that the orbit sets $\Theta_k$ are all equal to a single orbit set $\Theta$. 
 		
 		The monotonicity assumption implies that the image of the map
 		$\int_{(\cdot)} \omega_\phi: H_2(M; \mathbb{Z}) \to \mathbb{R}$ is discrete. The space $H_2(M, \Theta, \Theta_{\text{ref}}; \bZ)$ is an affine space over $H_2(M; \bZ)$.
 		It follows that the sequence
 		$\{\int_{W_k} \omega_\phi\}_{k \geq 1}$
 		lies in a discrete subset of $\bR$. Since it also converges to $c_{\sigma}(\bfS)$, for sufficiently large $k$ we must have
 		$\int_{W_k}\omega_\phi = c_{\sigma}(\bfS).$
 		
 		This concludes the proof of the proposition in the case where $\phi$ is nondegenerate. In the case where $\phi$ is degenerate, the proposition follows by taking nondegenerate approximations $\phi \circ \phi^1_{H_n}$ with $\|H_n\|_{1, \infty} \to 0$. 
 	\end{proof}
 
 \subsubsection{$(\delta, d)$-approximations} \label{subsubsec:dFlatApproximations}

Fix a pair $(\phi, J)$ of a diffeomorphism $\phi \in \Diff(\Sigma, \omega)$ and an admissible almost-complex structure $J \in \cJ^\circ(dt, \omega_\phi)$.  The Lee--Taubes isomorphism \cite{LeeTaubes12} requires taking a ``$(\delta, d)$''--approximation to $(\phi, J)$, as introduced in \cite{LeeTaubes12}; the idea of a $(\delta,d)$--approximation is that the local dynamics for such an approximation around periodic points of period less than $d$ have a particularly nice form; we defer the definition of the local form to equation \eqref{eqn:normalform} below.

We will need to know that we may pass to a $(\delta,d)$-approximation without changing the spectral invariants by too much; this should be expected since the approximation $\phi_*$ can be made arbitrarily close to $\phi$. To make this precise, let $(\phi,J)$ be given, and let $(\phi_*,J_*)$ be our $(\delta,d)$-approximation. Then there will be a Hamiltonian $H^1 \in C^{\infty}_c( (0,1) \times \Sigma)$ that generates an isotopy from $\phi$ to $\wt\phi_*(1)$. The construction of $H^1$ is given in \eqref{eqn:defnhfamily} in \S\ref{subsec:deltaDEstimates} below.
The following proposition is the key estimate we need.

\begin{prop}
\label{prop:dFlatSpectralInvariants} 
Fix PFH parameter sets
$\bfS = (\phi, \Theta_{\text{ref}}, J)$
and
$\bfS_* = (\phi_*, \Theta_{\text{ref}}^{H^1}, J_*^{H_1}).$
where the $(\delta, d)$--approximation $(\phi_*, J_*)$ and the Hamiltonian $H^1$ are as above.

If $[\Theta_{\text{ref}}] \in H_1(M_\phi; \bZ)$ has degree less than or equal to $d$, and $\sigma \in \TWPFH_*(\bfS)$ is a homology class with corresponding class $\sigma_* \in \TWPFH_*(\bfS_*)$ under the isomorphism induced by the canonical bijection of twisted PFH generators, then there is a constant $\kappa_{\ref{prop:dFlatSpectralInvariants}} \geq 1$ depending only on $\phi$, $J$, $d$, $\Theta_{\text{ref}}$ and the metric $g$ such that
$|c_{\sigma}(\bfS) - c_{\sigma_*}(\bfS_*)| \leq \kappa_{\ref{prop:dFlatSpectralInvariants}}\delta.$
\end{prop}

 We defer the proof, which is technical, to \S\ref{subsec:deltaDEstimates}.
 
 \subsection{Seiberg--Witten Floer cohomology} \label{subsec:swf}
 
In this section we collect various definitions, notation, and results regarding Kronheimer--Mrowka's theory of Seiberg--Witten--Floer cohomology.   

\subsubsection{Data for the equations } \label{subsubsec:swSetup}

The Seiberg--Witten cohomology requires a choice of metric and spin-c structure.  We now explain how the natural structures on $M_\phi$ give rise to this data.

The choice of admissible almost complex structure $J$ defines a Riemannian metric $g$ on $M_\phi$ by the formula
$g(v_1, v_2) = dt(v_1)dt(v_2) + \frac{1}{2}\omega_{\phi}(\pi_V(v_1), J\pi_V(v_2))$,
where $\pi_V$ denotes the projection 
$TM_\phi \to V$
of the tangent bundle to the vertical tangent bundle with kernel spanned by $R$. It follows from this that
$\|dt\|_{g} \equiv 1$
and
$\star^{g}\omega_{\phi} = 2dt$,
where $\star^{g}$ denotes the Hodge star operator associated to $g$. 

A \textbf{spin-c} structure $\fs$ on an oriented Riemannian 3-manifold $(M,g)$ is defined to be a rank-two Hermitian vector bundle 
$S \to M$
along with an injective bundle map $\cl: T^*M \to \mathfrak{su}(S)$ called the \textbf{Clifford multiplication}.
The bundle $S$ is called the \textbf{spinor bundle} and sections of the spinor bundle $S$ are called \textbf{spinors}.  
By identifying $TM$ with $T^*M$ using the metric $g$, the map $\cl$ also defines a Clifford multiplication on the tangent bundle. A class $\Gamma\in H_1(M_\phi;\bZ)$ determines a unique spin-c structure $\fs_\Gamma$ on $M_\phi$  such that 
\begin{enumerate}
	\item The spinor bundle of $\fs_\Gamma$ is given by
	$S_\Gamma = E_\Gamma \oplus (E_\Gamma \otimes V)$,
	where $V$ has the Hermitian line bundle structure induced by $J$ and $g$, and $E_\Gamma$ is the line bundle over $M_\phi$ with $c_1(E_\Gamma) = PD(\Gamma)$.
	\item $\cl(dt) = \begin{pmatrix} i & 0 \\ 0 & -i \end{pmatrix}$
	with respect to the splitting $S_\Gamma = E_\Gamma \oplus (E_\Gamma \otimes V)$.
\end{enumerate}
The reader may refer to \cite[Section 2(a)]{TaubesWeinstein1} for details on the construction of $\fs_\Gamma$.
	
For $\Psi \in C^\infty(S_\Gamma)$,  define 
$ \cl^\dagger(\Psi) 
= \langle \rho(e^i)\Psi, \Psi \rangle e^i$,
where $\{e^i\}_{i=1}^3$ is a local $g$-orthonormal coframe. 
We will also often write the spinor $\Psi$ as
$ \Psi = r^{1/2}\psi = r^{1/2}(\alpha, \beta),$
where $\psi$ is the rescaling of $\Psi$ by a factor of $r^{-1/2}$, and $\alpha$ and $\beta$ denote the components of $\psi$ with respect to the splitting $S_\Gamma = E_\Gamma \oplus (E_\Gamma \otimes V)$. 

A \textbf{spin-c connection} $\nabla_{\hat{B}}$ on the spinor bundle $S_\Gamma$ is a complex-linear connection such that $\cl$ is parallel with respect to $\nabla_{\hat{B}}$ and the Levi-Civita connection.
Given a spin-c connection $\nabla_{\hat{B}}$, there is a corresponding \textbf{Dirac operator}
$\nabla_{\hat{B}}$
defined by the formula
$D_{\hat{B}}\Psi = \sum_{i=1}^3 \cl(e^i)\nabla_{\hat{B},e_i}\Psi$,
where $\{e^i\}_{i=1}^3$ is any local orthonormal coframe.  When $\Gamma=0$, the spinor bundle $S_\Gamma=S_0$ is equal to $\underline{\mathbb{C}}\oplus V$. Let $\Psi_0$ be the spinor of $S_0$ given by the constant section of $\underline{\mathbb{C}}$ with value $1$. A quick computation identical to the argument in \cite[Lemma $10.1$]{HutchingsWeinsteinNotes} shows that there is a unique spin-c connection $\hat{B}_0$ on $S_0$ such that $D_{B_0}\Psi_0 = 0$.  
Let  
$\text{Conn}(E_\Gamma)$
be the space of smooth unitary connections on $E_\Gamma$, there is a bijection from $\text{Conn}(E_\Gamma)$ to the set of smooth spin-c connections on $S_\Gamma$ by taking $B\in \text{Conn}(E_\Gamma)$ to the connection $B+\hat{B}_0$ on $S_\Gamma=E_\Gamma \otimes S_0$.
In the following, we will abuse notation and use $\nabla_B$ and $D_B$ to denote the spin-c connection and Dirac operator on $S_\Gamma$ associated with $B\in \text{Conn}(E_\Gamma)$.

\subsubsection{Abstract perturbations} \label{subsubsec:abstractPerturbations}

For the purposes of ensuring transverality of the various moduli spaces used to define Seiberg--Witten--Floer cohomology, Kronheimer--Mrowka introduced in \cite[Chapter $11$]{monopolesBook} a large Banach space, denoted by $\cP$, consisting of \textbf{tame perturbations} for the Seiberg--Witten equations. 

A tame perturbation is a map from the $W^{k,2}$ Banach completion of $\text{Conn}(E_\Gamma) \times C^\infty(S_\Gamma)$ to $\mathbb{R}$, invariant under the action of the $W^{k+1,2}$ completion of the ``gauge group'' $\cG(M_\phi)$, defined below, that satisfies several estimates given in \cite[Definition $10.5.1$]{monopolesBook}. Here we set $k \geq 3$. 

The Banach space $\cP$ consists of a carefully chosen class of tame perturbations. Denote the $L^2$ formal gradient of a tame perturbation $\fg$, which by definition is a section of the bundle $iT^*M_\phi \oplus S_\Gamma$, by $(\fC_\fg, \fS_\fg)$, where $\fC_\fg$ denotes the component in $iT^*M_\phi$ and $\fS_\fg$ denotes the component in $S_\Gamma$.

A smooth 1-form $\mu$ on $M_\phi$ defines an element in Kronheimer--Mrowka's Banach space $\cP$ by the formula $\fe_\mu( (B, \Psi)) = \fe_\mu(B) = i\int_{M_\phi} \mu \wedge F_B.$
The formal gradient of $\fe_\mu$ is the pair $(i\star^{g}d\mu, 0)$. We will also write $\fe_\mu(B)-\fe_\mu(B_\Gamma)$ as $\fe_\mu(B,B_\Gamma)$. In the following, we will always take $\mu$ to be co-exact, and unless otherwise mentioned we will assume it has $C^3$-- and $\cP$--norms bounded by $1$.

\subsubsection{The Seiberg--Witten equations and the gauge group} \label{subsubsec:swEquations}

 We will now define a version of the Seiberg--Witten equations on $M_\phi$. 
 Define an \textbf{SW parameter set}
$\cS = (\phi, J, \Gamma, r, \fg = \fe_\mu + \fp)$
to be a tuple consisting of the following terms. 
The term $\phi$ in $\cS$ is an element of $\Diff(\Sigma, \omega)$. The term $J$ is an admissible almost-complex structure in the set $\cJ^\circ(dt, \omega_{\phi})$. The term $\Gamma$ is a class in $H_1(M_\phi, \bZ)$. The term $r$ is a nonnegative real number. The term $\mu$ is a co-exact $1$-form, and $\fp \in \cP$ is an element of the Banach space $\cP$ of tame perturbations. We will usually require $\|\fp\|_\cP$ to be small.

We can now write down the Seiberg--Witten equations associated to the parameter set $\cS$ which we will call the \textbf{$\cS$--Seiberg--Witten equations} for brevity.  Associate a metric and a spin-c structure to $\cS$ as explained above.  Then a pair $\fc = (B, \Psi) \in \text{Conn}(E_\Gamma) \times C^\infty(S_\Gamma)$ solves the $\cS$--Seiberg--Witten equations if and only if
\begin{equation}
\label{eq:sw}
\begin{split}
\star^{g}F_B &= \cl^\dagger(\Psi) - ir dt + i\star^{g}d\mu + i\varpi_V + \fC_\fp(B, \Psi), \\
D_B\Psi &= \fS_\fp(B, \Psi).
\end{split}
\end{equation}
Here $\varpi_V$ denotes the unique harmonic one-form such that $\star^{g}\varpi_V$ represents the class $\pi \cdot c_1(V)$. 

We call a pair $\fc = (B, \Psi) \in \text{Conn}(E_\Gamma) \times C^\infty(S_\Gamma)$
a \textbf{configuration}. We say a configuration is \textbf{reducible} if $\Psi \equiv 0$ and \textbf{irreducible} otherwise. The space $\text{Conn}(E_\Gamma) \times C^\infty(S_\Gamma)$ is acted on the the \textbf{group of gauge transformations}
$\cG(M_\phi) = C^\infty(M_\phi, U(1)),$
defined by
$u \cdot (B, \Psi) = (B - u^{-1}du, u \cdot \Psi).$  The gauge group action preserves the set of solutions to \eqref{eq:sw}.

The group $\cG(M_\phi)$ is disconnected. There is a canonical identification of the connected components of $\cG(M_\phi)$ with $H^1(M_\phi; \bZ)$ sending $u \in \mathcal{G}(M_\phi)$ to the cohomology class of the closed $1$-form $iu^{-1}du$. The \textbf{identity component} $\cG^\circ(M_\phi) \subset \cG(M_\phi)$ of the group of gauge transformations consists of maps $u: M_\phi \to U(1)$ that are null-homotopic, and we will refer to elements of $\cG^\circ(M_\phi)$ as ``null-homotopic gauge transformations''. 

\subsubsection{Linearized operators and spectral flow} \label{subsubsec:spectralFlowDefn}

We now review some natural linear differential operators associated to configurations. Spectral flows between these operators are used to define the grading on Seiberg--Witten cohomology.

To each solution $(B, \Psi = r^{1/2}\psi)$ of \eqref{eq:sw}, we associate to it a first-order linear differential operator
$$\cL_{(B, \Psi)}: C^\infty(iT^*M_\phi \oplus S_\Gamma \oplus i\bR) \to C^\infty(iT^*M_\phi \oplus S_\Gamma \oplus i\bR)$$
which sends a triple $(b, \eta, f)$ to the triple with the $iT^*M_\phi$, $S_\Gamma$, and $i\bR$ components described by
\begin{equation}
 \label{eq:extended3DHessian} 
     \begin{pmatrix}
    b
    \\
    \eta
    \\
    f
 \end{pmatrix}
\mapsto 
\begin{pmatrix}
    *db  -2^{-1/2}r^{1/2}(\eta^\dagger\tau\psi+\psi^\dagger\tau\eta)  -df -d{\fC_\fp}|_{(B,\Psi)}(b,\eta) \\
    2^{1/2}r^{1/2}\cl(b)\psi +   D_B\eta  + 2^{1/2} r^{1/2}f\psi - d \fS_\fp|_{(B, \Psi)}(b,\eta)
    \\
    -d^*b +  2^{-1/2}r^{1/2} (-\eta^\dagger \psi + \psi^\dagger \eta) .
\end{pmatrix}.
\end{equation}

The operator $\cL_{(B,\Psi)}$ is a version of the ``extended Hessian'' in \cite[Section $12.3$]{monopolesBook} and formula \eqref{eq:extended3DHessian} is the same as \cite[(3-1)]{TaubesWeinstein1}.

\begin{convention}
The operator $\cL_{(B, \Psi)}$ depends on a choice of parameter $r$ and also the abstract perturbation $\fg$, in the case where the perturbation is not of the form $\fg = \fe_\mu$. We fix the convention that if $(B, \Psi)$ is a solution of the $(\phi, J, \Gamma, r, \fg)$--Seiberg--Witten equations, then $\cL_{(B, \Psi)}$, unless otherwise stated, denotes the $r$--version of the operator in (\ref{eq:extended3DHessian}). If $(B, \Psi)$ is not explicitly stated to be a solution of the $\cS$--Seiberg--Witten equations, then $\cL_{(B, \Psi)}$, unless otherwise stated, denotes the version of the operator in (\ref{eq:extended3DHessian}) with $r$ set to equal $1$ and no abstract perturbation terms. 
\end{convention}

We say an \emph{irreducible} solution $(B, \Psi)$ to the $\cS$--Seiberg--Witten equations is \textbf{nondegenerate} if $\cL_{(B, \Psi)}$ has trivial kernel. Note that $\cL_{(B,\Psi)}$ extends to an unbounded, self-adjoint operator on the $L^2$ closure of its domain. Therefore, it has trivial kernel if and only if it is surjective. This implies that our definition of nondegeneracy coincides with the definition given in \cite[Lemma $12.4.1$]{monopolesBook}. In general, we will say any irreducible \emph{configuration} $\fc = (B, \Psi)$ is nondegenerate if $\cL_{(B, \Psi)}$ has no kernel. 

The operators $\cL_{(B,\Psi)}$ are used to define a relative $\bZ/\ell$--grading on the Seiberg--Witten--Floer cohomology, where $\ell$ denotes the largest integer dividing the class $c_\Gamma$ from (\ref{eq:cGamma}). For any two nondegenerate irreducible solutions $\fc_- = (B_-, \Psi_-)$ and $\fc_+ = (B_+, \Psi_+)$, define their relative degree as the reduction modulo $\ell$ of the \emph{spectral flow} from the operator $\cL_{(B_-, \Psi_-)}$ to $\cL_{(B_+, \Psi_+)}$.  We will use $\text{SF}(\fc_-, \fc_+)$ to denote the spectral flow between the operators $\cL_{\fc_-}$ and $\cL_{\fc_+}$ associated to two nondegenerate irreducible configurations $\fc_-$ and $\fc_+$, and use $\text{SF}_\ell(\fc_-, \fc_+)$ to denote the reduction of $\text{SF}(\fc_-, \fc_+)$ modulo $\ell$ when $\fc_\pm$ are solutions to the $\cS$--Seiberg--Witten equations.

We are only able to retrieve a relative $\bZ/\ell$ grading from the spectral flow because the spectral flow is not preserved under the action of $\cG(M_\phi)$ on $\fc_\pm$. If we modify either of $\fc_\pm$ by a gauge transformation, the spectral flow changes by an integer divisible by $\ell$. 
However, the spectral flow is invariant under the action of the group $\cG^\circ(M_\phi)$ of null-homotopic gauge transformations. 

\subsubsection{Instantons}

Fix a SW parameter set
$\cS = (\phi, J, \Gamma, r, \fg = \fe_\mu + \fp).$
The differential of Seiberg--Witten--Floer cohomology will be defined by \textbf{$\cS$--Seiberg--Witten instantons} on $\mathbb{R} \times M_\phi$. These are smooth, one-parameter families $\fd = (B_s, \Psi_s)$ in $\text{Conn}(E_\Gamma) \times C^\infty(S_\Gamma)$ that satisfy the equations
\begin{equation}
\label{eq:swInstanton}
\begin{split}
\frac{d}{ds}B_s &= -\star^{g}F_{B_s} + \cl^\dagger(\Psi_s) - ir dt + i\star^{g}d\mu + i\varpi_V + \fC_\fp(B_s, \Psi_s), \\
\frac{d}{ds}\Psi_s &= -D_{B_s}\Psi_s + \fS_\fp(B_s, \Psi_s),
\end{split}
\end{equation}
and limit as $s \to \pm \infty$ to solutions of (\ref{eq:sw}). If $\fd$ is a solution to (\ref{eq:swInstanton}) which limits as $s \to \pm\infty$ to nondegenerate irreducible solutions $\fc_\pm$ of (\ref{eq:sw}), we define its spectral flow $\SF(\fd) = \SF(\fc_-, \fc_+)$ to be the spectral flow between its endpoints. 

It will be useful to interpret $\SF(\fd)$ as the index of a Fredholm operator associated to the instanton $\fd$. Define the operator
\begin{equation} \label{eq:extended4DHessian} \cL_\fd = \frac{d}{ds} + \cL_{(B_s, \Psi_s)}. \end{equation}
Suppose that the limits $(B_\pm, \Psi_\pm)$ of $\fd$ as $s \to \pm \infty$ are irreducible and nondegenerate. Then the operator $\cL_\fd$ is Fredholm, and we have
\begin{equation} \label{eq:indexEqualsSF1} \text{ind}(\cL_\fd) = \SF(\fc_-, \fc_+) = \SF(\fd). \end{equation}

We can also use the operator $\cL_\fd$ to define a notion of nondegeneracy for instantons. An instanton $\fd$ is \textbf{nondegenerate} if the limits $(B_\pm, \Psi_\pm)$ of $\fd$ as $s \to \pm\infty$ are irreducible and nondegenerate and the operator $\cL_\fd$ has trivial cokernel. We also note that the Fredholm operator $\cL_\fd$ can be defined for \emph{any} path $\fd = (B_s, \Psi_s)$ of configurations which converge exponentially as $s \to \pm\infty$. As long as the limits $\fc_\pm = (B_\pm, \Psi_\pm)$ of $\fd$ as $s \to \pm \infty$ are irreducible, nondegenerate configurations (not necessarily solutions to (\ref{eq:sw})), the operator $\cL_\fd$ is Fredholm and the identity (\ref{eq:indexEqualsSF1}) holds. 

\subsubsection{The Chern-Simons-Dirac functional}

Fix an SW parameter set 
$\cS = (\phi, J, \Gamma, r, \fg = \fe_\mu + \fp),$ we define a variety of associated functionals on the configuration space $\text{Conn}(E_\Gamma) \times C^\infty(S_\Gamma)$. Choose a pair of configurations $\fc = (B, \Psi)$ and $\fc_\Gamma = (B_\Gamma, \Psi_\Gamma)$. Then we define the \textbf{energy} to be
$$ E_{\phi}(B, B_\Gamma) = i\int_{M_\phi} (B - B_\Gamma) \wedge \omega_{\phi}, $$
the \textbf{Chern-Simons functional} to be
$$ \fcs(B, B_\Gamma) = -\int_{M_\phi} (B - B_\Gamma) \wedge (F_B - F_{B_\Gamma}) - 2\int_{M_\phi} (B - B_\Gamma) \wedge (F_{B_\Gamma} - i\star^{g}\varpi_V) $$
and the \textbf{Chern-Simons-Dirac functional} to be
$$ \fa_{r, \fg}(\fc, \fc_\Gamma) = \frac{1}{2}(\fcs(B, B_\Gamma) - rE_{\phi}(B, B_\Gamma)) + \int_{M_\phi} \langle D_B\Psi, \Psi \rangle + \fg(B, \Psi) - \fg(B_\Gamma, \Psi_\Gamma). $$

Formally, the three-dimensional Seiberg--Witten equations are critical points of the Chern--Simons--Dirac functional, and the instantons are flow lines. 

The functionals $E_{\phi}$, $\fcs$ and $\fa_{r,\fg}$ are not fully gauge-invariant. They are only invariant under null-homotopic gauge transformations. The change of any of these quantities after pre-composition by a gauge transformation $u$ depends only on the associated class $i[u^{-1}du] \in H^1(M_\phi; \bZ)$, and can be easily computed.

\subsubsection{Seiberg--Witten--Floer cohomology in the monotone case}

Fix an SW parameter set
$\cS = (\phi, J, \Gamma, r, \fg = \fe_\mu + \fp).$
We will assume that the class $\Gamma$ is monotone with monotonicity constant $\rho$. Recall that $\Gamma$ is \textbf{positive monotone} if $\rho > 0$ and \textbf{negative monotone} if $\rho < 0$. Note that if $r \neq -2\pi\rho$, there are no reducible solutions to the $\cS$--Seiberg--Witten equations (\ref{eq:sw}). We will also assume that the abstract perturbation $\fp$ is small and generic (see \cite[Chapter $15$]{monopolesBook}). 

We start by reviewing the definition of Seiberg--Witten--Floer cohomology from \cite{monopolesBook} when $r\neq -2\pi\rho$. Suppose that all solutions to (\ref{eq:sw}) and (\ref{eq:swInstanton}) are nondegenerate. Define $\CM^{-*}(\cS) = \CM^{-*}(\phi, J, \Gamma, r, \fg)$
to be the free module generated by gauge-equivalence classes of solutions to the $\cS$--Seiberg--Witten equations. Given any configuration $\fc = (B, \Psi)$, we will denote its gauge equivalence class by $[\fc]$.  To define the differential $\partial$,
fix a pair of generators $[\fc_-]$ and $[\fc_+]$ and let
$ \mathcal{M}^1([\fc_-], [\fc_+])$
be the space of gauge-equivalence classes of instantons $\fd$ modulo pullback by translation in the $\bR$--direction in $\bR \times M_\phi$, such that $\SF(\fd) = 1$ and there is a gauge transformation $u$ such that $\lim_{s \to -\infty} \fd = u \cdot \fc_-$ and $\lim_{s \to \infty} \fd = \fc_+$. 
Given our nondegeneracy assumptions, $\mathcal{M}^1([\fc_-], [\fc_+])$ is a (possibly empty) oriented compact $0$-dimensional smooth manifold \cite{monopolesBook}. The differential $\partial$ is defined by the equation
$\langle \partial [\fc_+], [\fc_-] \rangle = \#\mathcal{M}^1([\fc_-], [\fc_+])$. The results of \cite[Parts V and VI]{monopolesBook} show $\partial^2 = 0$, and therefore
$\CM^{*}(\cS)$
is a cochain complex. We will denote its cohomology as 
$\HM^{*}(\cS) = \HM^{*}(\phi, J, \Gamma, r, \fg).$

As with PFH, to define spectral invariants we will want to instead use the \textbf{twisted Seiberg--Witten--Floer cohomology}.  Let
$\cS = (\phi, J, \Gamma, r, \fg = \fe_\mu + \fp)$ be an SW parameter set 
where $\Gamma$ is monotone with monotonicity constant $\rho$.  As before, we assume $r \neq -2\pi\rho$ for the moment.
The twisted Seiberg--Witten--Floer cohomology, denoted by 
$\TWHM(\cS),$
can be defined in two ways, both of which will be useful later. 

The first definition is more abstract. By \cite[Chapter $3$]{monopolesBook}, the space 
$\mathcal{B} = (\text{Conn}(E_\Gamma) \times C^\infty(S_\Gamma))/\cG(M_\phi)$
has fundamental group isomorphic to $H^1(M_\phi; \bZ)$. We can construct an explicit model for its universal cover in the following way. Let $\cG^\circ(M_\phi)$ be the identity component of the group of gauge transformations. Then the universal cover is the space
$\mathcal{B}^\circ = (\text{Conn}(E_\Gamma) \times C^\infty(S_\Gamma))/\cG^\circ(M_\phi)$. 
The cover $\mathcal{B}^\circ$ canonically defines a local system, which we denote by $\mathbb{B}^\circ$, over $\mathcal{B}$. The fiber of $\mathbb{B}^\circ$ over any particular point $[\fc] \in \mathcal{B}$ is merely the free module generated by the elements of the fiber of $\mathcal{B}^\circ$ over $[\fc]$. The monodromy of $\mathbb{B}^\circ$ is canonically induced by the monodromy of the cover $\mathcal{B}^\circ$. We define $\TWHM(\cS)$ to be the Seiberg--Witten--Floer homology of $M_\phi$ with respect to the parameter set $\cS$ with coefficients in the local system $\mathbb{B}^\circ$. 

The second definition is more concrete. Indeed, it can be derived by writing down explicitly the definition from \cite{monopolesBook} of the Seiberg--Witten--Floer cohomology with coefficients in the local system $\mathbb{B}^\circ$. Given a solution $\fc = (B, \Psi)$ to (\ref{eq:sw}), we write $[\fc]^\circ$ for its equivalence class modulo the action of $\cG^\circ(M_\phi)$. Then $\TWCM^*(\cS)$ is the free module generated by $\cG^\circ$--equivalence classes of solutions to the $\cS$--Seiberg--Witten equations. To define the differential, fix a pair of generators $[\fc_-]^\circ$ and $[\fc_+]^\circ$, represented by solutions $\fc_-$ and $\fc_+$ to (\ref{eq:sw}), respectively. We denote by
$ \widetilde{\mathcal{M}}^1([\fc_-]^\circ, [\fc_+]^\circ) $
the space of $\cG^\circ$-equivalence classes of instantons $\fd$ modulo pullback by translation in the $\bR$-direction in $\bR\times M_\phi$, such that $\SF(\fd) = 1$ and there is a \emph{null-homotopic} gauge transformation $u$ such that $\lim_{s \to -\infty} \fd = u \cdot \fc_-$ and $\lim_{s \to \infty} \fd = \fc_+$.
Given our nondegneracy assumptions, $\widetilde{\mathcal{M}}^1([\fc_-]^\circ, [\fc_+]^\circ)$ is a (possibly empty) compact oriented $0$--dimensional smooth manifold. Define the differential by
$\langle \widetilde{\partial}[\fc_+]^\circ, [\fc_-]^\circ \rangle = \#\widetilde{\mathcal{M}}^1([\fc_-]^\circ, [\fc_+]^\circ).$

The spectral flow defines a relative $\bZ$--grading on $\TWCM^*(\cS)$, as opposed to the relative $\bZ/\ell$--grading in the untwisted setting. For any two generators $[\fc_-]^\circ$, $[\fc_+]^\circ$, the relative grading $\text{SF}([\fc_-]^\circ, [\fc_+]^\circ)$
is defined to be $\SF(\fc_-, \fc_+)$. Since the spectral flow is invariant under nullhomotopic gauge transformations,  the relative $\mathbb{Z}$--grading is well-defined. Fix a nondegenerate irreducible base configuration $\fc_\Gamma = (B_\Gamma, \Psi_\Gamma)$. Then the spectral flow to the configuration $\fc_\Gamma$ defines an absolute $\bZ$--grading on $\TWCM^*(\cS)$. The differential $\widetilde{\partial}$ by definition increases the absolute grading by $1$. 

When $r = -2\pi\rho$, we can still define variants of Seiberg--Witten--Floer cohomology, and we now explain how this works.
Let
$\cS = (\phi, J, \Gamma, r, \fg = \fe_\mu + \fp)$ be an SW parameter set 
where $\Gamma$ is monotone with monotonicity constant $\rho$ and $r = -2\pi\rho$. 
In this setting, solutions of the $\cS$--Seiberg--Witten equations may admit reducibles. We define three versions of twisted Seiberg--Witten--Floer cohomology groups
$\TWHMcheck^*(\cS)$,  $\TWHMbar^*(\cS)$,  $\TWHMhat^*(\cS)$,
which again can be described in two equivalent methods.

The first method is to define the twisted Seiberg--Witten--Floer cohomology groups as the ``balanced'' Seiberg--Witten--Floer cohomology groups with coefficients in the local system $\mathbb{B}^\circ$, as described in \cite[Section $30.1$]{monopolesBook}. 
The second method again can be derived by writing out the chain complex explicitly. We will not give a fully detailed explanation, but instead content ourselves with writing down the generators, which is all that we need to know. 
The generators of $\TWCMcheck^*(\cS)$ are $\cG^\circ$--equivalence classes of irreducible solutions and ``boundary stable" reducible solutions. The generators of $\TWCMbar^*(\cS)$ are $\cG^\circ$--equivalence classes of both boundary stable and ``boundary unstable" reducible solutions. The generators of $\TWCMhat^*(\cS)$ are $\cG^\circ$--equivalence classes of irreducible solutions and boundary unstable reducible solutions. We write the associated cochain complexes as 
$\TWCMcheck^*(\cS)$, $\TWCMbar^*(\cS)$, $\TWCMhat^*(\cS).$ There are also ``completed'' versions obtained by taking the cohomology of positive completion of the associated cochain complexes (see \cite[Pages 597-598]{monopolesBook}). These are denoted by the three groups
$\TWHMcheck^\bullet(\cS)$, $\TWHMbar^\bullet(\cS)$,  $\TWHMhat^\bullet(\cS).$
We write the associated completed cochain complexes as 
$\TWCMcheck^\bullet(\cS),$  $\TWCMbar^\bullet(\cS),$ 
$\TWCMhat^\bullet(\cS).$  For more details, we refer the reader to \cite{monopolesBook}.

\subsubsection{Continuation maps} \label{subsec:continuationMaps}
Suppose $\phi_0$ is a fixed area-preserving map and suppose $\phi_-$ and $\phi_+$ are both Hamiltonian isotopic to $\phi_0$. Fix Hamiltonians $H_-, H_+ \in C^{\infty}_c( (0,1) \times \Sigma)$ such that $\phi_- = \phi_0 \circ \phi^1_{H_-}$, $\phi_+ = \phi_0 \circ \phi^1_{H_+}$. Let $K:(-\infty,+\infty)\to C^{\infty}_c( (0,1) \times \Sigma)$ be a smooth map such that there exists some $s_* > 0$ so that $K(s) = H_-$ on $(-\infty, -s_*]$ and $K(s) = H_+$ on $[s_*, \infty)$. 

There is a variant of the equations (\ref{eq:swInstanton}) associated with $K$ which we will have to consider in relating the complexes for $\phi_-$ and $\phi_+$.
The homotopy $K$ defines a $4$-manifold $\overline{X}_{K}$ with cylindrical ends modeled on $M_- = M_{\phi_-}$ and $M_+ = M_{\phi_+}$: we define $\overline{X}_{K}$ as the quotient of 
$\bR_s \times [0,1]_t \times \Sigma$
by the identification
$$(s, 1, p) \sim (s, 0, (\phi_0 \circ \phi^1_{K(s)})(p)).$$

Rather than write down the four-dimensional Seiberg--Witten equations, it is more efficient for our purposes to identify this symplectic manifold with the cylinder $\bR \times M_-$ and write down the equations on $\bR \times M_-$. The data of the homotopy $K$ define a diffeomorphism
$\bR_s \times [0,1]_t \times \Sigma \to \bR_s \times [0,1]_t \times \Sigma$
by the map
$(s, t, p) \mapsto (s, t, (\phi^t_{K(s)})^{-1}\phi^t_{H_-}(p)).$
This descends to a diffeomorphism
$\bR \times M_- \to \overline{X}_{K}.$ We will use this diffeomorphism to identify $\bR \times M_- $ with $\overline{X}_{K}.$

Now fix two SW parameter sets 
$\cS_\pm = (\phi_\pm, J_\pm, \Gamma_\pm, r_\pm, \fg_\pm = \fe_{\mu_\pm} + \fp_\pm)$ for $\phi_\pm$ such that $\Gamma_-=(M_{H_+}\circ M_{H_-}^{-1})_*(\Gamma_+)$.
A \textbf{SW continuation parameter set from $\cS_-$ to $\cS_+$} is a data set
$$\cS_s = (K, J_s, r_s, \fg_s = \fe_{\mu_s} + \fp_s)$$
interpolating between $\cS_-$ and $\cS_+$, rigorously defined as follows. Let $H_\pm,\phi_0,\phi_\pm,K(s)$ be as above, and suppose $s_*>0$ is a constant such that $K(s)=H_-$ for all $s\in(-\infty,-s_*]$ and $K(s)=H_+$ for all $s\in[s_*,+\infty)$. 
The notation $J_s$ denotes a smooth family $\{J_s\}_{s \in \bR}$ of almost-complex structures in $\cJ(dt, \omega_{\phi_-})$. The notation $r_s$ denotes a smooth family $\{r_s\}_{s \in \bR}$ of positive real numbers; if $r_+ = r_-$, then we require the family $\{r_s\}_{s \in \bR}$ to be constant. The notation $\mu_s$ denotes a smooth section of the Banach space bundle $\underline{\coexact(M_{\phi_0 \circ \phi^1_{K_s}})}$ over $\bR$, where each fiber is the space of co-exact one-forms with respect to the Riemannian metric $g_s = dt^2 + \frac{1}{2}\omega_K(-, J_s-).$
 The notation $\fg_s$ denotes a smooth section of the Banach space bundle $\underline{\cP}$ over $\bR$, where each fiber is the version of the Banach space $\cP$ of perturbations for the Riemannian metric $g_s$. We also require that $(J_s,r_s,\mu_s,\mathfrak{g}_s) = (J_-,r_-,\mu_-,\mathfrak{g}_-)$ when $s\in (-\infty, -s_*]$ and $(J_s,r_s,\mu_s,\mathfrak{g}_s) = (J_+,r_+,\mu_+,\mathfrak{g}_+)$ when $s\in [s_*,+\infty)$. It is straightforward to verify that once we fix the Hamiltonians $H_-$ and $H_+$, the space of SW continuation parameter set from $\cS_-$ to $\cS_+$ is contractible.

Let 
$\cS_s = (K, J_s, r_s, \fg_s)$
denote a SW continuation parameter set between two SW parameter sets $\cS_-$ and $\cS_+$. Then a smooth, one-parameter family $\fd = (B_s, \Psi_s)$ in $\text{Conn}(E_{\Gamma_-}) \times C^\infty(S_{\Gamma_-})$ is a \textbf{$\cS_s$--Seiberg--Witten instanton} if it satisfies the equations
\begin{equation}
\label{eq:swContinuation}
\begin{split}
\frac{d}{ds}B_s &= -\star^{g_s}F_{B_s} + \cl_s^\dagger(\Psi_s) - ir_s dt + i\star^{g_s} d\mu_s+i\varpi_{V,s} + \fC_{\fg_s}(B_s, \Psi_s), \\
\frac{d}{ds}\Psi_s &= -D_{B_s}\Psi_s + \fS_{\fg_s}(B_s, \Psi_s),
\end{split}
\end{equation}
and limits as $s \to \pm \infty$ to solutions of the $(\phi_\pm, J_\pm, \Gamma_\pm, r_\pm, \fg_\pm)$--Seiberg--Witten equations. 

Here are some more necessary clarifications regarding this definition. The $1$--form $\varpi_{V,s}$ is the unique $g_s$--harmonic $1$--form such that the harmonic $2$--form $\star^{g_s} \varpi_{V_s}$ represents $\pi c_1(V)$. The notation $\fC_{\fg_s}$ and $\fS_{\fg_s}$ denotes the $L^2$ formal gradient of $\fg_s$ with respect to the metric $g_s$. The notation $\cl_s$ denotes the Clifford multiplication on $S_{\Gamma_-} = E_{\Gamma_-} \oplus (E_{\Gamma_-} \otimes V)$ associated to the metric $g_s$. The Hermitian structure on $S_{\Gamma_-}$ is also determined here by the metric $g_s$, but we can assume that the Hermitian structure on $E_{\Gamma_-}$ remains independent of the metric $g_s$. The notion of $\fd$ limiting to a solution of the $(\phi_-, J_-, \Gamma_-, r_-, \fg_-)$--Seiberg--Witten equations as $s \to -\infty$ is straightforward. We say that $\fd$ limits to a solution of the $(\phi_+, J_+, \Gamma, r_+, \fg_+)$--Seiberg--Witten equations as $s \to \infty$ if and only if the limit of $\fd$ as $s \to \infty$ is a pullback of a solution of the $(\phi_+, J_+, \Gamma_+, r_+, \fg_+)$--Seiberg--Witten equations by the diffeomorphism $M_{H_+}\circ M_{H_-}^{-1}: M_- \to M_+$.  

There is a corresponding notion of ``spectral flow'' for solutions of (\ref{eq:swContinuation}), given as the index of the following Hessian. Let $\fd = (B_s, \Psi_s)$ be any smooth path in $\text{Conn}(E_{\Gamma_-}) \times C^\infty(S_{\Gamma_-})$ which has nondegenerate, irreducible limits $\fc_\pm = (B_\pm, \Psi_\pm)$ as $s \to \infty$. Note that $\fc_+$ is required to be nondegenerate in the sense that the version of the Hessian (\ref{eq:extended3DHessian}) with respect to the metric $g_+$ has no kernel. Then we can define a first-order differential operator
$ \cL_{\fd} = \frac{\partial}{\partial s} + \cL_{(B_s, \Psi_s)} $
which acts on a Hilbert space consisting of paths of sections of $iT^*M_- \oplus S_{\Gamma_-} \oplus i\bR$ with suitable exponential decay. The nondegeneracy of $\fc_\pm$ implies that this operator is Fredholm. Moreover, its index does not depend on perturbations of the path $\fd$. Given this, we will write the index of this operator as 
$\text{ind}(\cL_\fd) = \text{ind}(\fc_-, \fc_+).$

As shown by \cite[Chapter VII]{monopolesBook}, counting solutions to \eqref{eq:swContinuation} can be used to construct chain maps between the Seiberg–Witten–Floer complexes
associated with $\cS_+$ and $\cS_-$. We will call such maps \textbf{continuation maps} and denote them by $T(\cS_s)$.

\begin{rem}
	\label{rem:conventions}
	There are various ways of writing down the Seiberg--Witten equations with slightly different conventions. 
	Our equations (\ref{eq:sw}), (\ref{eq:swInstanton}), and (\ref{eq:swContinuation}) are equivalent to the analogues in \cite{TaubesWeinstein1, TaubesWeinstein2, ECHSWF} 
after rescaling $\Psi$ by a factor of $r^{1/2}$. Our equations are equivalent, as long as $r > 2\pi\rho$ and $\rho < 0$, to the analogues in \cite{LeeTaubes12} after rescaling $\Psi$ by a factor of $( 2(r + 2\pi\rho))^{1/2}$. They replace $dt$ with a one-form denoted by $\mathbf{a}$, but $dt = \mathbf{a}$ in our case since we are assuming that $J$ preserves the bundle $V$. We also note that they instead assume the Riemannian metric is fixed so that $\omega_{\phi}$ rather than $\frac{1}{2}\omega_{\phi}$ is Hodge dual to $dt$. Our equations are equivalent to the the equations in \cite{monopolesBook} after rescaling $\Psi$ by a factor of $\frac{1}{2}$. This is because they replace the term $\cl^\dagger(\Psi)$ with $\cl^{-1}$ of the traceless part of $\Psi\Psi^*$. This is equal to $\frac{1}{2}$ times $\cl^\dagger(\Psi)$. 
\end{rem}

\subsubsection{Identifying Seiberg--Witten--Floer cohomology for different parameter sets} \label{subsubsec:swfCobordismMaps}
Suppose
$$\cS_\pm = (\phi_\pm, J_\pm, \Gamma_\pm, r_\pm, \fg_\pm)$$ are SW parameter sets
with both $r_+$ and $r_-$ greater than $-2\pi\rho$ and $\phi_\pm = \phi_0\circ \phi_{H_\pm}^1$.
  Also assume that the classes $\Gamma_\pm$ are ``compatible'' in the sense that $(M_{H_+}\circ M_{H_-}^{-1})_*(\Gamma_-)=\Gamma_+$.
By \cite[Theorem $31.4.1$]{monopolesBook}, we have $\HM^{*}(\cS_+) \cong \HM^{*}(\cS_-)$.
We now use the construction from \S\ref{subsec:continuationMaps} to explicitly describe this isomorphism.

Let 
$\cS_s = (K, J_s, r_s, \fg_s)$
be a SW continuation parameter set between $\cS_-$ and $\cS_+$. If the smooth path $\{\fg_s\}$ is sufficiently generic, there is a quasi-isomorphism
$T(\cS_s): \CM^{*}(\cS_+) \to \CM^{*}(\cS_-)$
defined by counting solutions $\fd$ to the $\cS_s$--Seiberg--Witten instanton equations (\ref{eq:swContinuation}) with $\text{ind}(\cL_\fd) = 0$. This quasi-isomorphism preserves the relative $\bZ/\ell$--grading. 
In the language of \cite{monopolesBook}, this is the ``cobordism map'' defined by the completed cobordism $\overline{X}_{K}$, equipped with the canonical homology orientation as described below \cite[Definition $3.4.1$]{monopolesBook}. We refer the reader to \cite[Chapter VII]{monopolesBook} for the details of the construction of these quasi isomorphisms in the case where the path $\{r_s\}_{s \in \bR}$ is constant. These chain maps in the case where $\{r_s\}_{s \in \bR}$ is not constant are featured in the results of \cite[Chapter $31$]{monopolesBook}. 

The map $T(\cS_s)$ can change for different choices of $\cS_s$. However, since the space of all possible choices of $\cS_s$ is contractible (for a fixed choice of $H_-$ and $H_+$), the results of \cite[Chapters VII, VIII]{monopolesBook} imply that the induced maps on cohomology are independent of the choice of $\cS_s$ and may only depend on $H_-$ and $H_+$. 
Applying the above discussion to the case where $\phi_0=\phi_-$ and $H_-=0$, we conclude that 
for any pair $\cS_\pm$ of SW parameter sets as above, and
 for any choice $h \in \widetilde{\Ham}(\Sigma, \omega)$ of a path in $\Ham(\Sigma, \omega)$ from $\text{id}$ to $\phi_-^{-1}\phi_+$, there is an isomorphism
$$T(\cS_-, \cS_+; h): \HM^*(\cS_+) \to \HM^*(\cS_-).$$
The same argument as above also shows that $T(\cS_-, \cS_+; h)$ only depends on the homotopy class of $h$ relative to boundary, so the map $T(\cS_-, \cS_+; -)$ descends to $\overline{\Ham}(\Sigma, \omega)$.

The results of \cite[Chapter VII]{monopolesBook} also imply that these isomorphisms are functorial on the level of cohomology. Namely for any triple $\cS_1$, $\cS_2$, $\cS_3$ of SW parameter sets we have
$$T(\cS_1, \cS_2; h_1) \circ T(\cS_2, \cS_3; h_2) = T(\cS_1, \cS_3; h_2 \circ h_1)$$
where $h_2 \circ h_1$ denotes the composition of paths. 
Applying the above result when $h_1$ and $h_2$ are the constant path, we conclude that given $\phi$ and $\Gamma$, the groups 
$\HM^*(\cS)$
for any valid $\cS$ with $r > -2\pi\rho$  are canonically isomorphic to a single group 
$\HM^*(M_\phi, \Gamma, m_-).$

The same results hold in the twisted setting. This is because the results regarding cobordism maps in \cite{monopolesBook} extend to the case of the cohomology with coefficients in the local system $\mathbb{B}^\circ$. In general, the data of a certain ``morphism of local systems'' (See \cite[Page 458]{monopolesBook}) is required. In our setting, this is canonically determined by the cover $\mathcal{B}^\circ$ as well.
To summarize, an SW continuation parameter set
$\cS_s = (K, J_s, r_s, \fg_s)$
from $\cS_-$ to $\cS_+$, where
$\cS_\pm = (\phi_\pm, J_\pm, \Gamma_\pm, r_\pm, \fg_\pm),$
and $r_\pm > -2\pi\rho$, induces a quasi-isomorphism
$T(\cS_s): \TWCM^*(\cS_+) \to \TWCM^*(\cS_-)$
counting solutions to the $\cS_s$--Seiberg--Witten instanton equations (\ref{eq:swContinuation}) with $\text{ind}(\cL_\fd) = 0$. 

As before, the isomorphism on the level of cohomology also does not depend on the choice of $\cS_s$, apart from the homotopy class $h$ of the Hamiltonian isotopy from $\phi_-$ to $\phi_+$, and therefore descends to an isomorphism
$T(\cS_-, \cS_+; h): \TWHM^*(\cS_+) \to \TWHM^*(\cS_-).$

These isomorphisms satisfy the same naturality properties as in the untwisted case. 
Moreover, if $\phi_+ = \phi_-$, $J_+ = J_-$, and the family $\{J_s\}_{s \in \bR}$ is constant, the additivity of the index under gluing and the formula (\ref{eq:indexEqualsSF1}) imply that the isomorphisms $T(\cS_-, \cS_+; h)$ preserve the \emph{absolute} $\bZ$--gradings defined by a fixed choice of base configuration $\fc_\Gamma$. 
Therefore, similar to the untwisted case,
given $\phi$ and $\Gamma$, the groups $\TWCM(\cS)$ over all possible SW parameter sets with $r > -2\pi\rho$ are all canonically isomorphic to a fixed $\bZ$--graded group $\TWHM^*(M_\phi, \Gamma, m_-)$. Analogous statements hold for the groups
$\TWHMcheck^*(\cS),$ $\TWHMbar^*(\cS),$  $\TWHMhat^*(\cS)$
when $r = -2\pi\rho$ in the SW parameter set $\cS$. In particular, for every choice of SW parameter set $\cS$ when $r = -2\pi\rho$, they are canonically isomorphic to fixed groups
$\TWHMcheck^*(M_\phi, \Gamma, m_b),$ $\TWHMbar^*(M_\phi, \Gamma, m_b),$ $\TWHMhat^*(M_\phi, \Gamma, m_b).$

The cobordism maps on these groups also count solutions to the $\cS_s$--Seiberg--Witten equations, but to get a well-defined count in this case one needs to consider ``broken instantons'', which are sequences $\mathfrak{d}_1, \ldots, \mathfrak{d}_k$ such that the $s \to \infty$ limit of $\mathfrak{d}_i$ coincides with the $s \to -\infty$ limit of $\mathfrak{d}_{i+1}$. We refer the reader to \cite{monopolesBook} for details.

\subsubsection{Formal properties of twisted Seiberg--Witten--Floer cohomology}
\label{subsubsec_formal_properties_of_twHM}

Since twisted Seiberg--Witten--Floer cohomology can be interpreted as ordinary Seiberg--Witten--Floer cohomology with local coefficients, the theorems in \cite{monopolesBook} also apply to the twisted case, and we have the following results.

\begin{thm}
	\label{thm:exactTriangle} \cite[Chapter VIII]{monopolesBook}
	There is an exact triangle
	\begin{center}
		\begin{tikzcd}
			\TWHMcheck^*(M_\phi, \Gamma, m_b) \arrow{rr} & & \TWHMbar^*(M_\phi, \Gamma, m_b) \arrow{ld} \\
			& \TWHMhat^*(M_\phi, \Gamma, m_b) \arrow{lu} & 
		\end{tikzcd}
	\end{center}
	as well as an analogous exact triangle for the completed twisted monopole Floer homology groups.
\end{thm}

The following vanishing result is a variant of a well-known fact about Seiberg--Witten cohomology.

\begin{prop}
	\label{prop:checkVanishing} There is a constant $d_* \geq 1$ depending only on a fixed background metric on the manifold $M_\phi$ such that for every class $\Gamma$ of degree $d \geq d_*$, the cohomology group $\TWHMcheck^*(M_\phi, \Gamma, m_b)$ is zero. 
\end{prop}

\begin{proof}
	By \cite[Theorem $31.1.1$]{monopolesBook} the group $\TWHMcheck^*(M_\phi, \Gamma, m_b)$ is isomorphic to $\TWHM^*(M_\phi, \Gamma, m_+)$. Note that the theorem quoted asserts an isomorphism between the completed versions of the cohomology groups. However, both $\TWHMcheck^*(M_\phi, \Gamma, m_b)$ and $\TWHM^*(M_\phi, \Gamma, m_+)$ are equal to their completed versions, since their cochain complexes in this case do not change after completion. 
	
	Recall that $\TWHM^*(M_\phi, \Gamma, m_+)$ can be computed as $\TWHM^*(\cS)$, where $\cS = (\phi, J, 0, \Gamma, \fg)$
	is an SW parameter set with $r = 0$.
	It is a standard fact that the three-dimensional Seiberg--Witten equations with uniformly bounded perturbations have solutions only for finitely many spin-c structures. Therefore $\TWHM^*(M_\phi, \Gamma, m_+)$ vanishes when the degree of $\Gamma$ is sufficiently high.
\end{proof}

The following is a corollary of Theorem \ref{thm:exactTriangle} and Proposition \ref{prop:checkVanishing} that will be very useful to us.  In the untwisted setting, it was originally observed by Lee--Taubes in \cite{LeeTaubes12}. 

\begin{cor} \label{cor:barEqualsHat}
	For every $\Gamma\in H_1(M_\phi;\bZ)$ with sufficiently large degree, the map
	$\TWHMbar^*(M_\phi, \Gamma, m_b) \to \TWHMhat^*(M_\phi, \Gamma, m_b)$
	in the exact triangle of Theorem \ref{thm:exactTriangle} is an isomorphism. 
\end{cor}

\subsection{The Lee--Taubes isomorphism}
\label{sec:leetaubes}

In this section, we present Lee--Taubes' isomorphism between PFH and Seiberg--Witten--Floer cohomology in the untwisted case, following \cite{LeeTaubes12} and \cite[Paper V, Proposition $2.1$]{ECHSWF}.  

Fix a nondegenerate area-preserving diffeomorphism $\phi$ of a surface $(\Sigma, \omega)$ equipped with a choice of area form. Suppose that $\phi$ is negative monotone and $\Gamma$ is the corresponding negative monotone class of degree $d \geq \max(1, G)$ with monotonicity constant $\rho < 0$. Choose a generic almost-complex structure $J \in \cJ^\circ(dt, \omega_\phi)$. Fix a $(\delta, d)$-approximation $(\phi_*, J_*)$ of $(\phi, J)$ for $\delta > 0$ very small. We will see in \S\ref{subsec:deltaDEstimates} that $\phi_*$ is the endpoint of a specific Hamiltonian isotopy $\wt\phi_*$ from $\phi$ to $\phi_*$, and for every $s \in [0,1]$ there is a Hamiltonian $H^s$, supported in a union $\bD_\delta \subset \Sigma$ of embedded disks centered around the periodic points of period $\leq d$, generating an isotopy from $\phi$ to $\wt\phi_*(s)$. Pick $r > -2\pi\rho$ and define an SW parameter set $\cS_* = (\phi_*, J_*^{H^1}, \Gamma^{H^1}, r, \fg = \fe_\mu + \fp)$.  We assume that $\fe_\mu$ and $\fp$ have small $\cP$--norms and are generic, more precisely ``strongly admissible'' in the sense of \cite[\S$3$]{TaubesWeinstein1}.

Lee--Taubes \cite[Theorem $1.2$]{LeeTaubes12} defined for all sufficiently large $r > -2\pi\rho$ a chain-isomorphism
$$T_\Phi^r: \PFC_*(\phi_*, \Gamma^{H^1}, J_*^{H^1}) \to \CM^{-*}(\cS).$$
The map $T_\Phi^r$ defined above is induced by an injective map $\Phi^r$ from the set of PFH generators to the set of solutions of the $(\phi_*, J_*^{H^1}, \Gamma^{H^1}, r, \fe_\mu)$--Seiberg--Witten equations (\ref{eq:sw}). By a transversality result of Taubes (see \cite[Proposition $3.11$]{TaubesWeinstein1}) and the implicit function theorem, we find as long as $\mu$ is small and generic, and $\fp$ is sufficiently small, that these are in bijection with the set of solutions to the $\cS$--Seiberg--Witten equations. We compose the map of modules induced by $\Phi^r$ with this bijection to produce the desired chain map $T_\Phi^r$. The composition of $\Phi^r$ with the projection onto the gauge-equivalence class moreover defines a bijection between the set of PFH generators and the set of gauge-equivalence classes of solutions to the $\cS$--Seiberg--Witten equations. 

The following proposition lists some formal properties of the map $\Phi^r$ that we will find useful. 

\begin{prop}
\label{prop:untwistedIsomorphism} The map $\Phi^r$ satisfies the following properties: 
\begin{enumerate}
\item There are constants $\kappa_{\ref{prop:untwistedIsomorphism}} = \kappa_{\ref{prop:untwistedIsomorphism}}(\phi, J, \Gamma) \geq 1$ and $r_{\ref{prop:untwistedIsomorphism}} = r_{\ref{prop:untwistedIsomorphism}}(\phi, J, \Gamma) \geq 1$ such that for any PFH generator $\Theta$ and $r \geq r_{\ref{prop:untwistedIsomorphism}}$, the configuration $\Phi^r(\Theta) = (B(r), \Psi(r) = r^{1/2}(\alpha(r), \beta(r)))$ has $|\alpha(r)| \geq 999/1000$ on the complement of the tubular neighborhood around $\Theta$ of radius $\kappa_{\ref{prop:untwistedIsomorphism}}r^{-1/2}$. 
\item For a given PFH generator $\Theta$, denote the configuration $\Phi^r(\Theta)$ by $(B(r), \Psi(r))$. Then the two-forms $\frac{i}{2\pi}F_{B(r)}$ converge weakly as one-dimensional currents to the orbit set $\Theta$ as $r\to\infty$. 
\end{enumerate}
\end{prop}

\begin{proof}
The first item in Proposition \ref{prop:untwistedIsomorphism} is implicit in the construction of $\Phi^r$ in \cite[Paper II, Section $3$]{ECHSWF}, and is indeed mentioned explicitly in Part $2$ of the proof of Proposition $2.1$ in \cite[Paper V]{ECHSWF}. The second item follows from the a priori uniform energy bounds on the solutions $(B(r), \Psi(r))$ along with the argument in \cite[Paper IV, Theorem $1.1$]{ECHSWF}. 
\end{proof}

The Lee--Taubes isomorphism extends to the twisted setting; details of the construction can be found in \S\ref{sec:twistedIso}. Choose a reference cycle $\Theta_{\text{ref}}$ for the map $\phi$ discussed above which represents the class $\Gamma$, and let $\bfS_* = (\phi_*, \Theta_{\text{ref}}^{H^1}, J_*^{H^1})$ be a PFH parameter set given by a $(\delta, d)$-approximation. Then there is a bijection $\text{Tw}\Phi^r$ between generators of the complex $\TWPFC_*(\bfS_*)$ and the complex $\TWCM^{-*}(\cS_*)$, which induces a chain isomorphism $T_{\text{Tw}\Phi}^r: \TWPFC_*(\bfS_*) \to \TWCM^{-*}(\cS_*).$

\subsubsection{PFH cobordism maps} \label{subsec:pfhCobordismMaps}
Suppose
$\cS_\pm = (\phi_\pm, J_\pm, \Gamma_\pm, r_\pm, \fg_\pm)$ are SW parameter sets
with both $r_+$ and $r_-$ greater than $-2\pi\rho$ and $\phi_\pm = \phi_0\circ \phi_{H_\pm}^1$ for some Hamiltonians $H_\pm$. 
Assume that the classes $\Gamma_\pm$ are ``compatible'' in the sense that $(M_{H_+}\circ M_{H_-}^{-1})_*(\Gamma_-)=\Gamma_+$. Choose a trivialized reference cycle $\Theta_{\text{ref}} \subset M_{\phi_0}$ so that its pushforwards $\Theta_{\text{ref}}^{\pm} = M_{H_{\pm}}(\Theta_{\text{ref}})$ represent the classes $\Gamma_{\pm}$. 
Let 
$\cS_s = (K, J_s, r_s, \fg_s)$
be a SW continuation parameter set between $\cS_-$ and $\cS_+$.
The Lee--Taubes isomorphism identifies the Seiberg--Witten--Floer cohomology groups associated to $\cS_\pm$ with PFH groups associated to the parameter sets $\bfS_\pm = (\phi_\pm, \Theta^\pm_{\text{ref}}, J_\pm)$. 
Then the composition of the Lee--Taubes isomorphism and the Seiberg--Witten continuation map yields a chain isomorphism 
$$T^{\PFH}(\bfS_-, \bfS_+; \cS_s): \TWPFC_*(\bfS_+) \to \TWPFC_*(\bfS_-).$$

A ``pseudoholomorphic curve axiom'' proved by Chen \cite{Chen21} implies that this chain isomorphism preserves the $\bZ$-gradings on each side; see the discussion before Theorem $3.6$ in \cite{simplicity20} or \S\ref{subsec:pfhSpectralInvariantProperties} for a more detailed discussion of this result. 

The homotopy $K$ between the Hamiltonians $H_\pm$ induces a Hamiltonian isotopy from $\text{id}$ to $\phi_-\phi_+^{-1}$, which only depends on the particular choice of $K$ up to homotopy. From our discussion on invariance in \S\ref{subsubsec:swfCobordismMaps}, the induced map on PFH depends only on the choice of Hamiltonians $H_\pm$, and so we denote it by
$$T^{\PFH}(\bfS_-, \bfS_+; H_-, H_+): \TWPFH_*(\bfS_+) \to \TWPFH_*(\bfS_-).$$

We mostly consider the case where $\phi_0 = \text{id}$ and $H_- = 0$, and in this case for the sake of convenience we write $H_+ = H$ and denote the map on PFH by $T^{\PFH}(\bfS_-, \bfS_+; H)$. The functoriality properties on the Seiberg--Witten side imply that, via the cobordism maps for $\phi = \phi_- = \phi_+$ and $H_+ = H_- = 0$, the groups $\TWPFH_*(\phi, \Theta_{\text{ref}}, J)$ for varying $J \in \cJ(dt, \omega_\phi)$ are canonically isomorphic to a single $\bZ$-graded module which we write as $\TWPFH_*(\phi, \Theta_{\text{ref}})$.

\section{The Weyl law}
Having reviewed the extensive preliminaries, we now give the proof of the Weyl law, assuming some estimates whose proof we defer to \S\ref{sec:estimates}.  

From now on, we will call a constant ``geometric'' if it depends only on the underlying Riemannian metric. Unless otherwise specified, the notation $\kappa$ denotes a geometric constant greater than or equal to $1$, which can be taken to increase from line to line. For $x, y > 0$ positive, we will write $x \lesssim y$ or $x = O(y)$ if there exists a geometric constant $\kappa > 0$ such that $x \leq \kappa y$. 
\subsection{Seiberg--Witten--Floer spectral invariants}

We begin by defining the aforementioned ``Seiberg--Witten--Floer spectral invariants'', whose definition is analogous to the previously defined spectral invariants for PFH.  Fix an SW parameter set $\cS = (\phi, J, \Gamma, r, \fg = \fe_\mu + \fp)$ and a base configuration $\fc_\Gamma$. We define Seiberg--Witten--Floer spectral invariants in the case where $r > -2\pi\rho$. They will depend on our choice of base configuration $\fc_\Gamma = (B_\Gamma, \Psi_\Gamma)$.  

Recall from the preliminaries that the Chern--Simons--Dirac functional $\fa_{r, \fg}(-, \fc_\Gamma)$ is invariant under null-homotopic gauge transformations, so it is well-defined on generators of $\TWCM^*(\cS)$. 
For each $L\in \mathbb{R}$, define a submodule $\TWCM^*_L(\cS; \fc_\Gamma)$ of $\TWCM^*(\cS)$ to be the submodule generated by all of the twisted SWF generators for which $r^{-1}\fa_{r, \fg}(-, \fc_\Gamma) \geq L$.  By a standard monotonicity estimate for gradient flows of the Chern--Simons--Dirac functional, which can also be deduced as a special case of Lemma \ref{lem:csdContinuationGradient1} below, $\TWCM^*_L(\cS; \fc_\Gamma)$ is a subcomplex of $\TWCM^*(\cS)$. Write $\TWHM^*_L(\cS; \fc_\Gamma)$ for the cohomology of this subcomplex. The inclusion 
induces a map
$\iota^{\TWHM}_L(\cS; \fc_\Gamma): \TWHM^*_L(\cS; \fc_\Gamma) \to \TWHM^*(\cS)$. 

Fix a class $0\neq \sigma \in \TWHM^*(\cS)$. The \textbf{Seiberg--Witten--Floer spectral invariant} of $\sigma$, denoted by 
$c^{\HM}_{\sigma}(\cS; \fc_\Gamma)$,
is the supremum of all $L$ such that $\sigma$ lies in the image of the map $\iota^{\TWHM}_L(\cS; \fc_\Gamma)$. 

In the monotone case, there are only finitely many possible values for $\fa_{r, \fg}(\fc, \fc_\Gamma)$ on the generators of $\TWCM^*(\cS; \fc_\Gamma)$ at any given degree (see \cite[Proposition $29.2.1$]{monopolesBook}). Therefore, if $\sigma\neq 0$, then its Seiberg--Witten--Floer spectral invariant is finite.

We must extend the definition of the Seiberg--Witten spectral invariants to the case where $\fg \in \cP$ is \emph{any} abstract perturbation with $\|\fg\|_{\cP} \leq 1$; a priori they are only defined for those $\fg \in \cP$ for which the complex $\TWCM^*(\cS)$ is well-defined. This is necessary because it is more analytically tractable to work with the Seiberg--Witten spectral invariants where the abstract perturbation is of the special form $\fg = \fe_\mu$. 

The Seiberg--Witten spectral invariant
$$c^{\HM}_{\sigma}(\cS; \fc_\Gamma) = c^{\HM}_{\sigma}(\phi, J, \Gamma, r, \fg; \fc_\Gamma)$$
for general $\fg$ is defined as follows. Fix any sequence $\fp_k \to 0$ in $\cP$ such that, for every $k$, the chain complex
$\TWCM^*(\cS_k) = \TWCM^*(\phi, J, \Gamma, r, \fg_k)$
is well-defined. Then set
$$c^{\HM}_{\sigma}(\cS; \fc_\Gamma) = \lim_{k \to \infty} c^{\HM}_{\sigma}(\cS_k; \fc_\Gamma).$$

The limit on the right-hand side exists and is independent of the choice of sequence $\fp_k \to 0$. This is a direct consequence of the fact that Seiberg--Witten spectral invariants are locally Lipschitz-continuous with respect to the abstract perturbation term $\fg$, which is the content of the following lemma. The proof uses estimates from \S\ref{subsubsec:continuationEstimates}. 

\begin{lem}
\label{lem:maxMin1}
Let $\phi$, $J$, $\Gamma$, $\rho$, $\sigma$ and $\fc_\Gamma$ be as fixed above. Let $\cS_- = (\phi, J, \Gamma, r, \fg_-)$ and $\cS_+ = (\phi, J, \Gamma, r, \fg_+)$ be any pair of SW parameter sets with $r > -2\pi\rho$ so that the complexes $\TWCM^*(\cS_\pm)$ are defined. Suppose that $-2\pi\rho > 10r_{\ref{prop:continuationEstimates2}}$. Then there is a constant $\kappa_{\ref{lem:maxMin1}} \geq 1$ depending only on $\phi$ and $J$ such that if $\|\fg_+ - \fg_-\|_{\cP} \leq \kappa_{\ref{lem:maxMin1}}^{-1}$ then the following bound holds:
$$|c^{\HM}(\cS_+; \fc_\Gamma) - c^{\HM}(\cS_-; \fc_\Gamma)| \leq \kappa_{\ref{lem:maxMin1}}r^2\|\fg_+ - \fg_-\|_{\cP}.$$
\end{lem}

\begin{proof}
Pick some smooth path $\{\fg_s\}_{s \in \bR}$ in $\cP$ such that $\fg_s = \fg_-$ for $s < - 1$, $\fg_s = \fg_+$ for $s > 1$, and we have a uniform bound
$$\sup_{s \in [-1,1]} \|\frac{\partial}{\partial s} \fg_s\|_\cP \leq 4\|\fg_+ - \fg_-\|_\cP.$$

Also assume that $\|\fg_+ - \fg_-\|_\cP \leq \frac{1}{100}\kappa_{\ref{prop:continuationEstimates2}}^{-1}$. Define SW parameter sets
$\cS_- = (\phi, J, \Gamma, r, \fg_-)$
and
$\cS_+ = (\phi, J, \Gamma, r, \fg_+).$
Define the SW continuation parameter set
$\cS_s = (0, J, r, \fg_s)$
from $\cS_-$ to $\cS_+$. Suppose that $\{\fg_s\}_{s \in \bR}$ is chosen sufficiently generically, so that the chain quasi-isomorphism 
$T(\cS_s): \TWCM^*(\cS_+) \to \TWCM^*(\cS_-)$
is well-defined.

Let $\fc_+$ be a solution of the $\cS_+$--Seiberg--Witten equations, and let $\fc_-$ be any solution of the $\cS_-$--Seiberg--Witten equations such that $[\fc_-]^\circ$ appears with nonzero coefficient in $T(\cS_s)[\fc_+]^\circ$. Using the fact that the map $T(\cS_s)$ counts solutions to the $\cS_s$--Seiberg--Witten equations, we deduce that there is a solution $\fd$ of the $\cS_s$--Seiberg--Witten instanton equations (\ref{eq:swContinuation}) such that $\lim_{s \to \pm\infty} \fd = \fc_{\pm}$. 

It follows from the version of Proposition \ref{prop:continuationEstimates2} with $r = r_+ = r_-$ that
$$
\fa_{r, \fg_+}(\fc_+, \fc_\Gamma) - \fa_{r, \fg_-}(\fc_-, \fc_\Gamma) \leq \kappa_{\ref{prop:continuationEstimates2}}r^2\|\fg_+ - \fg_-\|_{\cP}.
$$
If we take any cochain $\widetilde{\sigma}_+$ in $\TWCM^*(\cS_+)$ representing $\sigma$, then we find that 
\begin{equation} \label{eq:maxMin1} \min_{[\fc_+] \in \widetilde{\sigma}_+}\fa_{r, \fg_+}(\fc_+, \fc_\Gamma) \leq \min_{[\fc_-] \in T(\cS_s) \cdot \widetilde{\sigma}_+} \fa_{r, \fg_-}(\fc_-, \fc_\Gamma) + \kappa_{\ref{prop:continuationEstimates2}}r^2\|\fg_+ - \fg_-\|_{\cP}. \end{equation}
Taking the maximum of (\ref{eq:maxMin1}) across all $\widetilde{\sigma}_+$ implies
$$c^{\HM}_{\sigma}(\cS_+; \fc_\Gamma) \leq c^{\HM}(\cS_-; \fc_\Gamma) + \kappa_{\ref{prop:continuationEstimates2}}r^2\|\fg_+ - \fg_-\|_\cP.$$
Reversing the roles of $\fg_\pm$ and applying the same argument yield a similar inequality with $\fg_+$ and $\fg_-$ switched, which proves the lemma. 
\end{proof}

We will also need a ``spectrality'' result for the Seiberg--Witten spectral invariants, which asserts that the Seiberg--Witten spectral invariants are equal to actions of genuine solutions of the Seiberg--Witten equations. 

\begin{lem} \label{lem:maxMin2}
Let $\phi$, $J$, $\Gamma$, $\rho$, $\sigma$, $\fc_\Gamma$ be as fixed above. Fix any $r > -2\pi\rho$ and any $\fg \in \cP$ such that $\|\fg\|_{\cP} \leq 1$. Then there is a solution $\fc$ of the $(\phi, J, \Gamma, r, \fg)$--Seiberg--Witten equations such that
$$c^{\HM}_{\sigma}(\phi, J, \Gamma, r, \fg; \fc_\Gamma) = r^{-1}\fa_{r, \fg}(\fc, \fc_\Gamma).$$
\end{lem}

\begin{proof}
We prove the lemma in two steps.

\textbf{Step 1:} The first step proves the lemma in the case where $\fg$ is such that the complex $\TWCM^*(\cS) = \TWCM^*(\phi, J, \Gamma, r, \fg)$ is well-defined. By the definition of the spectral invariant, there is a sequence of solutions $\{\fc_k\}_{k \in \mathbb{N}}$ of the $\cS$--Seiberg--Witten equations such that
$$\lim_{k \to \infty} r^{-1}\fa_{r, \fg}(\fc_k, \fc_\Gamma) = c^{\HM}_{\sigma}(\cS; \fc_\Gamma).$$

The compactness theory for the Seiberg--Witten equations implies that, after passing to a subsequence, there is a sequence of gauge transformations $\{u_k\}_{k \in \mathbb{N}}$ and a solution $\fc$ of the $\cS$--Seiberg--Witten equations such that
$\lim_{k \to \infty} u_k \cdot \fc_k = \fc$
in the $C^\infty$ topology. This implies, by the formula for the change in the Chern--Simons--Dirac functional under gauge transformations, that 
\begin{align*}
r^{-1}\fa_{r, \fg}(\fc, \fc_\Gamma) &= \lim_{k \to \infty} r^{-1}\fa_{r, \fg}(u_k \cdot \fc_k, \fc_\Gamma) \\
&= \lim_{k \to \infty} r^{-1}\fa_{r, \fg}(\fc_k, \fc_\Gamma) - i(\pi r^{-1} + \frac{1}{2\rho})([u_k^{-1}du_k] \cup c_\Gamma)[M_\phi] \\
&= c^{\HM}_{\sigma}(\phi, J, \Gamma, r, \fg; \fc_\Gamma) - \lim_{k \to \infty} i(\pi r^{-1} + \frac{1}{2\rho})([u_k^{-1}du_k] \cup c_\Gamma)[M_\phi].
\end{align*}

Here $c_\Gamma := 2\text{PD}(\Gamma) + c_1(V)$ is used for brevity. The set of real numbers $$\{([u^{-1}du] \cup c_\Gamma)[M_\phi]\,|\,u \in C^\infty(M_\phi, S^1)\}$$
is discrete, so there is a gauge transformation $v$ such that
$$\lim_{k \to \infty} -([u_k^{-1}du_k] \cup c_\Gamma)[M_\phi] = ([v^{-1}dv] \cup c_\Gamma)[M_\phi].$$
Hence
$\fa_{r, \fg}(v \cdot \fc, \fc_\Gamma) = r^{-1}c^{\HM}_{\sigma}(\phi, J, \Gamma, r, \fg; \fc_\Gamma)$
as desired. 

\textbf{Step 2:} The second step proves the lemma where $\fg$ is arbitrary. Then Step 1 and the definition of the Seiberg--Witten spectral invariants implies that there is a sequence $\fg_k \to \fg$ in $\cP$ and a sequence $\{\fc_k\}_{k \geq 1}$ of configurations such that for each $k$, $\fc_k$ is a solution of the $(\phi, J, \Gamma, r, \fg_k)$--Seiberg--Witten equations and $$r^{-1}\fa_{r, \fg_k}(\fc_k, \fc_\Gamma) = c^{\HM}_{\sigma}(\phi, J, \Gamma, r, \fg_k; \fc_\Gamma).$$
Therefore
$c^{\HM}_{\sigma}(\phi, J, \Gamma, r, \fg; \fc_\Gamma) = \lim_{k \to \infty} r^{-1}\fa_{r, \fg_k}(\fc_k, \fc_\Gamma).$

The compactness theory for the Seiberg--Witten equations implies that, after passing to a subsequence, there is a sequence of gauge transformations $\{u_k\}_{k \in \mathbb{N}}$ and a solution $\fc$ of the $(\phi, J, \Gamma, r, \fg)$--Seiberg--Witten equations such that
$\lim_{k \to \infty} u_k \cdot \fc_k = \fc$
in the $C^\infty$ topology. By an identical argument to Step 1, the fact that  $\{([u^{-1}du] \cup c_\Gamma)[M_\phi]\,|\,u \in C^\infty(M_\phi, S^1)\}$
is discrete implies that there is a gauge transformation $v$ such that 
$$r^{-1}\fa_{r, \fg}(v \cdot \fc, \fc_\Gamma) = c^{\HM}_{\sigma}(\phi, J, \Gamma, r, \fg; \fc_\Gamma)
\phantom\qedhere\makeatletter\displaymath@qed$$
\end{proof}

\subsection{Recovering the PFH spectral invariants from the Seiberg--Witten ones}
\label{sec:twistedIso}
\subsubsection{The Chern--Simons--Dirac functional and the PFH action}
 \label{subsec:twistedCompactness}
Fix a family of SW parameter sets
$\cS_r = (\phi, J, \Gamma, r, \fe_\mu)$
where $\Gamma$ has degree $d$ and is monotone, the map $\phi$ is $d$-nondegenerate, and $\mu$ is a small co-exact $1$-form. 
Fix a \emph{separated} reference cycle $\Theta_{\text{ref}}$ for $\phi$ representing $\Gamma$ as defined in \S\ref{subsec:pfhDefinition}.  We define a particular family of reference configurations $\fc^c(r) = (B^c(r), \Psi^c(r))$ which we call a \textbf{$\Theta_{\text{ref}}$--concentrated family}. The idea is that the functional $r^{-1}\fa_{r, \fe_\mu}(-, \fc^c(r))$ will approximate the twisted PFH action functional $\bfA$ as $r \to \infty$. 

To write down the construction of $\fc^c(r) = (B^c(r), \Psi^c(r))$, first observe that the bundle $E_\Gamma$, when restricted to the complement of $\Theta_{\text{ref}}$, admits a smooth section $\alpha^c$ such that $|\alpha^c| \equiv 1$ everywhere. Choose any such section $\alpha^c$. There is a unique flat, unitary connection $B^c$ on the restriction of $E_\Gamma$ to the complement of $\Theta_{\text{ref}}$ such that $\alpha^c$ is covariantly constant with respect to $B^c$. 

We choose the connections $B^c(r)$ to satisfy the following conditions:
\begin{enumerate}
\item $B^c(r)$ agrees with $B^c$ outside of a tubular neighborhood of $\Theta_{\text{ref}}$ of radius $r^{-1/2}$. 
\item There is a constant $\kappa \geq 1$ depending only on $\Gamma$, the ambient Riemannian metric, and the length of the reference cycle $\Theta_{\text{ref}}$ such that, for every $r$, $\|F^c(r)\|_{L^1(M_\phi)} \leq \kappa$ and $\|F^c(r)\|_{C^0} \leq \kappa r$. 
\item The two-forms $\frac{i}{2\pi}F^c(r)$, as $r \to \infty$, converge weakly as one-dimensional currents to the reference cycle $\Theta_{\text{ref}}$. 
\end{enumerate}

We choose the spinors $\Psi^c(r) = (\alpha^c(r), \beta^c(r))$ to satisfy the following conditions:
\begin{enumerate}
\item $\alpha^c(r)/|\alpha^c(r)|$ is equal to $\alpha^c$ outside of a tubular neighborhood of $\Theta_{\text{ref}}$ of radius $r^{-1/2}$. 
\item $\alpha^c(r)$ is transverse to the zero section of $E_\Gamma$ and has zero set equal to $\Theta_{\text{ref}}$.  
\item $\|\alpha^c(r)\|_{C^0(M_\phi)} \leq r^{-1}$ for every $r$. 
\item $\beta^c(r) \equiv 0$ for every $r$. 
\end{enumerate}

The following proposition can be thought of as an abstract compactness result. The statement is used later in the proof of Proposition~\ref{prop:dFlatSWSpectral}, but many of the computations in the proof can also be used in the proof of the upcoming Proposition~\ref{prop:spectralInvariantIdentity}, so we put it here to avoid repeating the same computations twice. 

\begin{prop} \label{prop:actionConvergence}
Fix a sequence $r_k \to +\infty$ and fix $\cS_{r_k} = (\phi, J, \Gamma, r_k, \fe_\mu)$. Fis for each $k$ a solution $\fc_k = (B_k, \Psi_k = r_k^{1/2}(\alpha_k, \beta_k))$ to the $\cS_{r_k}$--Seiberg--Witten equations. Fix a separated reference cycle $\Theta_{\text{ref}}$ for $\phi$ with $[\Theta_{\text{ref}}] = \Gamma$ and a $\Theta_{\text{ref}}$-concentrated family $\fc^c(r) = (B^c(r), \Psi^c(r))$. 
Suppose there is some $E > 0$ independent of $k$ such that 
$|r_k^{-1}\fa_{r_k, \fe_\mu}(\fc_k, \fc^c(r_k))| \leq E$ for every $k$.

Then there is a generator $(\Theta, W)$ of the twisted PFH complex 
$\TWPFC(\phi, \Theta_{\text{ref}}, J)$
such that after passing to a subsequence, we have
$\lim_{r \to \infty} r^{-1}\fa_{r, \fe_\mu}( \fc_k, \fc^c(r_k)) = -\pi \bfA(\Theta, W).$
\end{prop}

\begin{proof}
{\bf Step 1:}    We begin by recalling some compactness results from the work of Lee--Taubes.

The integrals
$\int_{M_\phi} F_{B_k} \wedge dt$
are uniformly bounded (in fact, constant) and so by the results of \cite{TaubesWeinstein1}, the following statements hold after passing to a subsequence:
\begin{enumerate}
\item The forms $\frac{1}{2\pi i}F_{B_k}$ converge as one-dimensional currents to an orbit set $\Theta$ of $\phi$ in the homology class $\Gamma$. 
\item The functions $|\alpha_k|$ converge uniformly to $1$ on compact subsets of $M_\phi \setminus \Theta$. 
\end{enumerate}

{\bf Step 2:}  
For every $k$, the function
$\fcs(-, B^c(r_k)) + 2\pi\rho E_\phi(-, B^c(r_k))$
is fully gauge-invariant.  
By the estimates from Proposition \ref{prop:3dActionEstimates} below, we have 
$$\lim_{k\to\infty} \underbrace{r_k^{-1}\bigg(\fcs(B_k, B^c(r_k)) + 2\pi\rho E_\phi(B_k, B^c(r_k))\bigg)}_{=:a_k}= 0.$$
By the assumption,
$b_k:=r_k^{-1}\bigg(\fcs(B_k, B^c(r_k)) - r_k E_\phi(B_k, B^c(r_k))\bigg)$
is bounded. Since 
$$
r_k^{-1} \fcs(B_k, B^c(r_k)) = \frac{r_k a_k + 2\pi\rho b_k}{r_k + 2\pi\rho},\quad E_\phi(B_k, B^c(r_k)) = \frac{a_k-b_k}{2\pi\rho/r_k+1},
$$
we conclude that $\lim_{k\to\infty} r_k^{-1} \fcs(B_k, B^c(r_k))=0,$ and $E_\phi(B_k, B^c(r_k))$ is bounded.
Hence
\begin{equation} \label{eq:dFlatSwSpectral1} 
	\lim_{k \to \infty} |r_k^{-1}\fa_{r_k, \fe_\mu}(\fc_k, \fc^c(r_k)) + \frac{1}{2}E_\phi(B_k, B^c(r_k))| =\lim_{k\to\infty} \big|\frac12 r^{-1}\fcs(B_k, B^c(r_k))\big|=0.
\end{equation}
Since $E_\phi(B_k, B^c(r_k))$ is bounded, we may pass to a subsequence so that
$\lim_{k \to \infty} E_\phi(B_k, B^c(r_k)) = E_*$
for some constant $E_*$.
By (\ref{eq:dFlatSwSpectral1}), we have
$\lim_{k \to \infty} r_k^{-1}\fa_{r_k, \fe_\mu}(\fc_k, \fc^c(r_k)) = -\frac{1}{2}E_*.$

{\bf Step 3:}  
Fix a smooth cutoff function $\chi: \bR \to [0,1]$ which is equal to $1$ on $(-\infty, 5/16]$ and $0$ on $[7/16, \infty)$.  Then write
$$\hat{B}_k = B_k - \frac{1}{2}\chi(1 - |\alpha_k|^2)|\alpha_k|^{-2}(\langle \nabla_{B_k}\alpha_k, \alpha_k \rangle - \langle \alpha_k, \nabla_{B_k}\alpha_k \rangle).$$
It will be useful in what follows to note that $\alpha_k/|\alpha_k|$ is covariantly constant with respect to $\hat{B}_k$ wherever $|\alpha(r_k)| \geq 3/4$ and that the two-forms $\frac{i}{2\pi}F_{\hat{B}_k}$ converge to $\Theta$ as currents.
We have 
\[
\begin{split}
 \int_{M_\phi} (B_k - \hat{B}_k) \wedge \omega_\phi 
&  = \frac{1}{2}\int_{M_\phi} \chi(1 - |\alpha_k|^2)|\alpha_k|^{-2}(\langle \nabla_{B_k}\alpha_k, \alpha_k \rangle - \langle \alpha_k, \nabla_{B_k}\alpha_k \rangle) \wedge \omega_\phi \\
&  = \frac{1}{2}\int_{M_\phi} \chi(1 - |\alpha_k|^2)|\alpha_k|^{-2}(\langle \nabla_{B_k,R}\alpha_k, \alpha_k \rangle - \langle \alpha_k, \nabla_{B_k,R}\alpha_k \rangle) \dvol_g.
\end{split}
\]
In the last line, $\nabla_{B_k, R}\alpha_k$ denotes the derivative of $\alpha_k$ along the Reeb vector field $R$. 

The Dirac equation $D_{B_k}\Psi_k = 0$ tells us that $\nabla_{B_k, R}$ is expressible as a sum of covariant derivatives of $\beta_k$. By standard estimates for solutions of the Seiberg--Witten equations, see Proposition \ref{prop:3dEstimates1} for the precise statement, this implies that for $k$ sufficiently large
$$|\nabla_{B_k, R}\alpha_k| \lesssim |1 - |\alpha_k|^2 + r_k^{-1}|^{1/2}.$$
Plugging in the inequality above and using the fact that $\chi(1 - |\alpha_k|^2)$ is supported where $|\alpha_k| \geq 3/4$, we find that for $k$ sufficiently large 
$$\int_{M_\phi} (B_k - \hat{B}_k) \wedge \omega_\phi \lesssim \big(\int_{M_\phi} |1 - |\alpha_k|^2 + r_k^{-1}|\dvol_g\big)^{1/2}.$$
By the first and last items in Proposition \ref{prop:3dEstimates1},  we have 
$\int_{M_\phi} (B_k - \hat{B}_k) \wedge \omega_\phi \lesssim r_k^{-1/2},$
and so
\begin{equation}
	\label{eqn:goalstep3}
	\lim_{k \to \infty} \int_{M_\phi} (B_k - \hat{B}_k) \wedge \omega_\phi = 0.
\end{equation}

{\bf Step 4:}  By \eqref{eqn:goalstep3} and the definition of $E_*$, we have
\begin{equation}
	\label{eqn:goalstep4}
E_* = \lim_{k \to \infty} i \int_{M_\phi} (\hat{B}_k - B^c(r_k)) \wedge \omega_\phi.
\end{equation}
Using this as motivation, we now define the desired class $W$, which is the content of the present step.

Recall that $\hat{B}_k$ is flat where $|\alpha_k| \geq 3/4$, while $B^c(r_k)$ is flat in the complement of a tubular neighborhood of $\Theta_{\text{ref}}$ of radius $r_k^{-1/2}$. The unit length sections $\alpha_k/|\alpha_k|$ and $\alpha^c(r_k)/|\alpha^c(r_k)|$ are respectively $\hat{B}_k$- and $B^c(r_k)$-covariantly constant on these respective neighborhoods. Write $U_k$ for the open region consisting of the intersection of $\{|\alpha_k| > 3/4\}$ and this tubular neighborhood complement.  Since $|\alpha_k|$ uniformly converges to $1$ on compact subsets of $M_\phi \setminus \Theta$, we can pass to a subsequence so that $U_k \subset U_{k+1}$ for every $k$ and also 
$\bigcup_k U_k = M_\phi \setminus (\Theta \cup \Theta_{\text{ref}}).$
Now fix a set $\Lambda_\Gamma = \{\gamma_1, \ldots, \gamma_{b_1(M_\phi)}\}$ of embedded loops disjoint from $\Theta$ and $\Theta_{\text{ref}}$ that form a basis of $H_1(M_\phi; \bR)$. 
We may assume without loss of generality that the loops $\gamma_j$ lie in $U_k$ for all $k$. Then we conclude 
$$i\int_{\gamma_j} (\hat{B}_k - B^c(r_k)) = 2\pi z_{j,k}$$
with $z_{j, k} \in \bZ$ for all $j,k$. 

{\bf Step 5:} 
For every $j$, let $\eta_j$ be a closed $2$--form that is Poincar\'e dual to $\gamma_j$ and is supported very close to $\gamma_i$. We can assume the support of $\eta_i$ lies in $U_k$ for all $k$. This implies
$$i \int_{M_\phi} (\hat{B}_k - B^c(r_k)) \wedge \eta_j = 2\pi z_{j,k}$$
for every $j$ and $k$.

Since $\phi$ is assumed to be monotone, there are integers $\{c_j\}$ and a real number $\lambda$ such that such that $\omega_\phi -\lambda \sum_j c_j \eta_j$ is exact. Write $\xi$ for a primitive of this exact two-form. 
Then we have
\begin{equation} \label{eqn:goalstep5}
i\int_{M_\phi} (\hat{B}_k - B^c(r_k)) \wedge \omega_\phi = 2\pi \lambda \sum_{j} c_j z_{j,k} + \int_{M_\phi} (F_{\hat{B}_k} - F_{B^c(r_k)})\wedge \xi.
\end{equation}
Recall that $F_{\hat{B}_k} - F_{B^c(r_k)}$ converges to $\Theta - \Theta_{\text{ref}}$ as a current. 
By \eqref{eqn:goalstep4}, we conclude that 
$$
\lim_{k\to\infty }\lambda \sum_j c_j z_{j,k} = \frac{1}{2\pi}\Big( E_* - \int_{\Theta-\Theta_{\text{ref}}}\xi \Big).
$$
Since all $c_j$'s are integers, we may choose a subsequence such that $\lambda\sum_j c_j z_{j,k}$ is independent of $k$.  Choose an arbitrary $k_0$ in this subsequence, and let $z_j = z_{j,k_0}$. Define W to be the class in $H_2(M_\phi;\Theta,\Theta_{\text{ref}};\mathbb{Z})$ such that the algebraic intersection of $\gamma_j$ and $W$ is $z_j$ for all $j$. Then by \eqref{eqn:goalstep5} we have
\begin{equation}
	\label{eqn_limit_Bhat_minus_Bc_wedge_omega}
\lim_{k \to \infty} i\int_{M_\phi} (\hat{B}_k - B^c(r_k)) \wedge \omega_\phi = 2\pi \int_W \omega_\phi.
\end{equation}
Hence
$$E_* = \lim_{k \to \infty} i\int_{M_\phi} (B_k - B^c(r_k)) \wedge \omega_\phi = 2\pi\bfA(\Theta, W),$$
which implies the proposition in view of \eqref{eq:dFlatSwSpectral1}.
\end{proof}

\subsubsection{Lee--Taubes' isomorphism in the twisted setting}\label{subsec:twistedIsoDefinition}

We now explain how to refine the Lee--Taubes isomorphism to an isomorphism in the twisted setting.
 
We retain the setup from above.  Fix a separated reference cycle $\Theta_{\text{ref}}$ for $\phi$ representing $\Gamma$. After taking $\delta$ very small, we can assume that the Hamiltonian $H^1$ from \S\ref{subsubsec:dFlatApproximations} 
generating an isotopy from $\phi$ to the $(\delta, d)$-approximation $\phi_*$ has support disjoint from $\Theta_{\text{ref}}$. Fix a PFH parameter set
$\bfS_* = (\phi_*, \Theta_{\text{ref}}^{H^1}, J_*^{H^1}).$
Then there is a corresponding Seiberg--Witten parameter set $\cS_*$: the induced data is just as in \S\ref{sec:leetaubes} , except that we will choose the reference connection to be a connection from the previous section.

The goal in this section is to establish the following:

\begin{prop}
	\label{prop_iso_twisted_PFH_SW}
There is a chain isomorphism
$T_{\text{Tw}\Phi}^r: \TWPFC_*(\bfS_*) \to \TWCM^{-*}(\cS_*).$
\end{prop}

The proof is quite similar to Taubes' proof in \cite[Paper V]{ECHSWF} regarding the isomorphism between twisted ECH and a twisted version of Seiberg--Witten Floer.  However, as the details and notation are important in what is coming, we do give a proof, although we are a little brief since our argument is standard in view of Taubes' work.  

\begin{proof}

We define the map $\text{Tw}\Phi^r$ as the composition of the map $\Phi^r$ from \S\ref{sec:leetaubes} and a choice of suitable gauge.  To define this gauge, we have to fix several additional pieces of data, which are similar to objects introduced before and during the proof of Proposition~\ref{prop:actionConvergence}:

\emph{Collection of loops:} Fix a collection of loops $\Lambda_\Gamma \subset M_\phi$ that are disjoint from the union of $\Theta_{\text{ref}}$ and all Reeb orbits of $\phi$ of degree less than or equal to $d$ and form a basis of $H_1(M_\phi; \bR)$. 

\emph{Scale:} Fix a scale $r_* \geq r_{\ref{prop:untwistedIsomorphism}}$ such that any of the loops in $\Lambda_\Gamma$, any of the Reeb orbits of $\phi$ of degree less than or equal to $d$, and $\Theta_{\text{ref}}$ are all pairwise of distance greater than or equal to $1000\kappa_{\ref{prop:untwistedIsomorphism}}r_*^{-1/2}$ from each other. We also require that $\delta \leq 10^{-2}\kappa_{\ref{prop:untwistedIsomorphism}}r_*^{-1/2}$. 

\emph{Concentrated family:} Fix a $\Theta_{\text{ref}}$-concentrated family of base configurations on $M_\phi$. This was already done in \S\ref{subsec:twistedCompactness}, but we repeat the construction here and make some slight tweaks to fit with the slightly different situation under discussion here.  

Observe that the bundle $E_\Gamma$, when restricted to the complement of $\Theta_{\text{ref}}$, admits a smooth section $\alpha^c$ such that $|\alpha^c| \equiv 1$ everywhere. Choose any such section $\alpha^c$. There is a unique flat, unitary connection $B^c$ on the restriction of $E_\Gamma$ to the complement of $\Theta_{\text{ref}}$ such that $\alpha^c$ is covariantly constant with respect to $B^c$.

We choose the connections $B^c(r)$ to satisfy the following conditions:
\begin{enumerate}
\item $B^c(r)$ agrees with $B^c$ outside of a tubular neighborhood of $\Theta_{\text{ref}}$ of radius $10^{-5}\kappa_{\ref{prop:untwistedIsomorphism}}r^{-1/2}$. 
\item There is a constant $\kappa \geq 1$ depending only on $\Gamma$, the ambient Riemannian metric, and the length of the reference cycle $\Theta_{\text{ref}}$ such that, for every $r$, $\|F^c(r)\|_{L^1(M_\phi)} \leq \kappa$ and $\|F^c(r)\|_{C^0} \leq \kappa r$. 
\item The two-forms $\frac{i}{2\pi}F^c(r)$, as $r \to \infty$, converge weakly as one-dimensional currents to the reference cycle $\Theta_{\text{ref}}$.
\end{enumerate}

We choose the spinors $\Psi^c(r) = (\alpha^c(r), \beta^c(r))$ to satisfy the following conditions:
\begin{enumerate}
\item $\alpha^c(r)/|\alpha^c(r)|$ is equal to $\alpha^c$ outside of a tubular neighborhood of $\Theta_{\text{ref}}$ of radius $10^{-5}\kappa_{\ref{prop:untwistedIsomorphism}}r^{-1/2}$.
\item $\alpha^c(r)$ is transverse to the zero section of $E_\Gamma$ and has zero set equal to $\Theta_{\text{ref}}$. 
\item $\|\alpha^c(r)\|_{C^0(M_\phi)} \leq r^{-1}$ for every $r$.
\end{enumerate}

\emph{Cutoff function:} Fix a smooth function $\chi: \bR \to [0,1]$ such that $\chi$ is equal to $1$ on $(-\infty, 7/16]$ and $0$ on $[9/16, \infty)$. 

\emph{Poincar\'e dual two-form:} It is not necessary for the definition of the isomorphism, but for the purposes of computations later, we fix for any loop $\gamma \in \Lambda_\Gamma$ a closed two-form $\eta_\gamma$ Poincar\'e dual to $\gamma$ and supported in a fixed tubular neighborhood of $\gamma$ of radius $10\kappa_{\ref{prop:untwistedIsomorphism}}r_*^{-1/2}$.  

\vspace{2 mm}
\textbf{Definition of the map:} We now define the promised choice of gauge. Let $(\Theta, W)$ be any twisted PFH generator. Then as long as $r$ is sufficiently large we write
$$\Phi^r(\Theta) = (B(r), \Psi(r) = r^{1/2}(\alpha(r), \beta(r))).$$

Similarly as in the proof of Proposition~\ref{prop:actionConvergence}, define the connection
$$\hat{B}(r) = B(r) - \frac{1}{2}\chi(1 - |\alpha(r)|^2)|\alpha(r)|^{-2}(\langle \nabla_{B(r)}\alpha(r), \alpha(r) \rangle - \langle \alpha(r), \nabla_{B(r)}\alpha(r) \rangle),$$
and observe that:
\begin{itemize}
\item $\hat{B}(r)$ is flat where $|\alpha(r)| > 3/4$. 
\item By Proposition \ref{prop:untwistedIsomorphism} and the first item, $\hat{B}(r)$ is flat outside of a tubular neighborhood of radius $\kappa_{\ref{prop:untwistedIsomorphism}}r^{-1/2}$ around $\Theta$. 
\item $\alpha(r)/|\alpha(r)|$ is covariantly constant with respect to $\hat{B}(r)$ where $|\alpha(r)| > 3/4$. 
\item As $r \to \infty$, the two-forms $\frac{i}{2\pi}F_{\hat{B}(r)}$ converge as currents to $\Theta$. 
\end{itemize}

Now choose as a base connection $B^c(r)$ from the \emph{$\Theta_{\text{ref}}$-concentrated family} from above.  We chose $B^c(r)$ to be a connection on the bundle $E_\Gamma \to M_\phi$, but as usual we view it as a connection on $M_{\phi_*}$ by pushing forward by the diffeomorphism $M_{H^1}$. Our setup in the previous step shows that the one-form $\hat{B}(r) - B^c(r)$ is closed on a neighborhood of any loop $\gamma \in \Lambda_\Gamma$. More precisely, near $\gamma$ we find that both $\hat{B}(r)$ and $B^c(r)$ are flat and admit covariantly constant sections. 

The connections $\hat{B}(r)$ and $B^c(r)$ therefore define an integer-valued functional on the set of loops $\Lambda_\Gamma$, defined by the map
$$\gamma \mapsto X_\gamma(\hat{B}(r), B^c(r)) = \frac{1}{2\pi i}\int_\gamma (\hat{B}(r) - B^c(r)).$$
In terms of the \emph{Poincar\'e dual two-forms} $\{\eta_\gamma\}_{\gamma \in \Lambda_\Gamma}$ fixed in the first step, we can write
$$X_\gamma(\hat{B}(r), B^c(r)) = \frac{1}{2\pi i}\int_{M_\phi} (\hat{B}(r) - B^c(r)) \wedge \eta_\gamma.$$

The functional $X_\gamma(-, B^c(r))$ is invariant under null-homotopic gauge transformations. For a general gauge transformation,
$$X_\gamma(u \cdot \hat{B}(r), B^c(r)) = X_\gamma(\hat{B}(r), B^c(r)) - \frac{1}{2\pi i}([u^{-1}du] \cup [\eta_\gamma])[M_\phi].$$
It follows that there is a unique gauge transformation  $u_W \in \cG(M_\phi)$ up to homotopy such that for any $\gamma \in \Lambda_\Gamma$, 
$$X_\gamma(u_W \cdot \hat{B}(r), B^c(r)) = \langle \gamma, W \rangle = \int_W \eta_\gamma.$$
The expression following the first equality is the algebraic intersection pairing and the second equality follows by Poincar\'e duality. 

Set $\text{Tw}\Phi^r(\Theta, W) = u_W \cdot (B(r), \Psi(r))$. It is immediate from the definition that $\text{Tw}\Phi^r$ is a bijection from the set of twisted PFH generators to the set of twisted SWF generators. An argument identical to the proof of \cite[Paper V, Proposition $2.1$]{ECHSWF} apart from cosmetic changes can be implemented to show that the induced map of modules
$$T_{\text{Tw}\Phi}^r: \TWPFC_*(\bfS_*) \to \TWCM^{-*}(\cS_*)$$
is a chain isomorphism. 
\end{proof}

\begin{rem} \label{rem:relativeWindingNumber}
There is a more topological interpretation of this isomorphism that we will find useful later. Let $\gamma \in M_\phi$ be any embedded loop and $\alpha$ and $\alpha'$ any two sections of $E_\Gamma$ that do not vanish along $\gamma$. Let $f: \gamma \to \mathbb{C}$ be the smooth function defined by
$\alpha' = f \cdot \alpha/|\alpha|$
at any point in $\gamma$. Note that by our assumptions $f$ is never equal to zero. We define the \textbf{relative winding number}
$\text{wind}(\gamma, \alpha, \alpha') \in \bZ$
to be the winding number of the function $f$. A local computation shows that the function $X_\gamma(B(r), B^c(r))$ is equal to $\text{wind}(\gamma, \alpha^c(r), \alpha(r))$. We also note that, by definition, for any $r_1$ and $r_2$ sufficiently large, we have
$\text{wind}(\gamma, \alpha^c(r_1), \alpha^c(r_2)) = 0.$
The gauge transformation $u_W$ defined above is the unique gauge transformation (up to homotopy) such that
$\text{wind}(\gamma, \alpha^c(r), u_W \cdot \alpha(r)) = \langle \gamma, W \rangle$
for every $\gamma \in \Lambda_\Gamma$. 
\end{rem}

\begin{rem}
Since the isomorphism class of a twisted PFH group is independent of the choice of reference cycle, the twisted Lee--Taubes isomorphism holds even when the reference cycle is not separated. However, when $\Theta_{\text{ref}}$ is separated, we obtain an explicit, geometric realization of the Lee--Taubes isomorphism as above, which will be useful for subsequent arguments. 
\end{rem}

\begin{rem} \label{rem:pfhToSwfCanonical}
It will be important to note how the isomorphism 
$$T_{\text{Tw}\Phi}^r: \TWPFC_*(\bfS_*) \to \TWCM^{-*}(\cS_*)$$
induces an isomorphism 
$$\TWPFH_*(\phi, \Theta_{\text{ref}}, J) \to \TWHM^{-*}(\phi, J, \Gamma, r, \fe_\mu + \fp)$$ 
of the Floer groups arising from $\phi$ rather than its $(\delta,d)$--approximation $\phi_*$. 

This follows from pre-composing the first isomorphism with the isomorphism of twisted PFH groups associated to the canonical bijection of the generators and then post-composing with the canonical isomorphism between the two Seiberg--Witten--Floer cohomology groups, induced by counting solutions to the $\cS_s$--Seiberg--Witten instanton equations. Moreover, the second isomorphism is canonical in the sense than any $(\delta, d)$-approximation of $\phi$ will induce the same isomorphism
$$\TWPFH_*(\phi, \Theta_{\text{ref}}, J) \to \TWHM^{-*}(\phi, J, \Gamma, r, \fe_\mu + \fp)$$ 
on the level of homology. This is verified explicitly in \cite[Section $5.2$]{Chen21}. 
\end{rem}

\subsubsection{Twisted SWF spectral invariants recover twisted PFH spectral invariants}

We conclude \S\ref{sec:twistedIso} with a proof that the Seiberg--Witten spectral invariants recover its twisted PFH spectral invariants in the limit $r \to \infty$. 

Let  $\Lambda_\Gamma$, $\fc^c(r) = (B^c(r), \Psi^c(r)= (\alpha^c(r), \beta^c(r)))$, $\{\eta_\gamma\}_{\gamma \in \Lambda_\Gamma}$, $r_*$, $\chi$ be as in \S\ref{subsec:twistedIsoDefinition}.  
Note 
that we retain the assumption that the reference cycle is separated. Write $\{c_\gamma\}_{\gamma \in \Lambda_\Gamma}$ for the unique set of real numbers such that 
$\omega_\phi - \sum c_\gamma\eta_\gamma$
is exact, and write $\xi_\phi$ for a choice of primitive of this exact two-form.

\begin{prop}
\label{prop:spectralInvariantIdentity}
Let $\sigma$ be any class in $\TWPFH_*(\bfS)$ and let $\sigma_\Phi \in \TWHM^{-*}(\cS_r)$ denote $T_{\text{Tw}\Phi}^r(\sigma)$. Then for any choice $\fc_\Gamma = (B_\Gamma, \Psi_\Gamma)$ of base configuration, we have the identity
$$\lim_{r \to \infty} \big(c^{\HM}_{\sigma_\Phi}(\cS_r; \fc_\Gamma) + \frac{i}{2}\int_{M_{\phi}} (B^c(r) - B_\Gamma) \wedge \omega_{\phi}\big) = -\pi c_{\sigma}(\bfS).$$
\end{prop}

\begin{proof}
The proof will proceed in four steps. 

\textbf{Step 1:} The first step reduces to proving the theorem in the case where $(\phi, J)$ is a $(\delta, d)$-approximation, i.e. a pair suitable for directly defining the Lee--Taubes isomorphism. Suppose that the statement of the theorem is true for any $(\delta, d)$-approximation $(\phi_*, J_*)$. Then the right-hand side, by Proposition \ref{prop:dFlatSpectralInvariants}, will differ from $-\pi c_\sigma(\bfS)$ by $\lesssim \delta$. 

Next, we show that the left-hand side also changes by $\lesssim \delta$. The first term on the left-hand side, the Seiberg--Witten spectral invariant, is known to change by $\lesssim \delta$ by Proposition \ref{prop:dFlatSWSpectral}. It remains only to conider the second term. Recall the notation for the Hamiltonian $H^1$ from \S\ref{subsubsec:dFlatApproximations} generating an isotopy from $\phi$ to $\phi_*$. Lemma \ref{lem:dFlatBounds} asserts that $\|H^1\|_{C^0} \lesssim \delta$. 

The two-form $\omega_{\phi_*}$ pulls back by $M_{H^1}$ to $\omega_\phi + dH^1 \wedge dt$, so the integral 
$$\frac{i}{2}\int_{M_\phi} (B^c(r) - B_\Gamma) \wedge \omega_\phi$$ 
changes by 
\begin{equation*}
\frac{i}{2}\int_{M_\phi} (B^c(r) - B_\Gamma) \wedge dH^1 \wedge dt = \frac{i}{2}\int_{M_\phi} (F^c(r) - F_\Gamma) \wedge H^1 dt \lesssim \delta \|F^c(r) - F_\Gamma\|_{L^1} \lesssim \delta. 
\end{equation*}

The final inequality follows from the fact that the two-form $F^c(r)$ have $L^1$ norms uniformly bounded in $r$. This shows that the left-hand side overall changes by $\lesssim \delta$ when taking a $(\delta, d)$-approximation. It then suffices to take $\delta \to 0$ to conclude the statement of the theorem for the original pair $(\phi, J)$. The remaining three steps will now prove the theorem directly assuming that $(\phi, J)$ is itself suitable for defining the Lee--Taubes isomorphism. 

\textbf{Step 2:} 
Fix any PFH generator $(\Theta, W)$. Write $\fc(r) = (B(r), \Psi(r) = r^{1/2}(\alpha(r), \beta(r)))$ for the SWF generator $\text{Tw}\Phi^r(\Theta, W)$.  We compute the $r \to \infty$ limits of the energies of these configurations: 
\begin{align*}
	-\frac12 \lim_{r \to \infty} E_{\phi}(B(r), B^c(r)) &= -\frac12  \lim_{r \to \infty} i\int_{M_{\phi}} (B(r) - B^c(r)) \wedge \omega_{\phi}  \\
&= -\frac12 \lim_{r \to \infty} i\int_{M_{\phi}} (\hat{B}(r) - B^c(r)) \wedge \omega_{\phi}  = -\pi \int_W \omega_\phi = -\pi \bfA(\Theta, W).
\end{align*}
The first equality follows from the definition of $E_\phi$ and the second equality follows from \eqref{eqn:goalstep3}. The third equality follows from the definition of $\text{Tw}\Phi^r(\Theta, W)$ and the fact that $F_{B(r)} - F_{B^c(r)}$ converges to $\Theta - \Theta_{\text{ref}}$ as $1$-dimensional currents as $r \to \infty$. Since $E_{\phi}(B(r), B^c(r))$ is bounded, the proof of \eqref{eq:dFlatSwSpectral1} can be carried out in reverse to conclude the identity
\begin{equation} \label{eq:spectralInvariantIdentity2} \lim_{r \to \infty} r^{-1}\fa_{r, \fe_\mu}(\fc(r), \fc^c(r)) = -\pi \bfA(\Theta, W). \end{equation}

\textbf{Step 3:} We apply (\ref{eq:spectralInvariantIdentity2}) to show
\begin{equation}
	\label{eq:spectralInvariantIdentity1} 
	\lim_{r \to \infty} c^{\HM}_{\sigma_\Phi}(\cS_r; \fc^c(r)) = -\pi c_{\sigma}(\bfS).
\end{equation} 
This is a formal argument using the definition of the spectral invariants and the fact that $\text{Tw}\Phi^r$ is a bijection between the set of twisted PFH generators and the set of twisted SWF generators. For any $\epsilon>0$, let $\wt\sigma \in \TWPFC_*(\bfS)$ be a chain representing $\sigma$ such that it has a generator with action in the interval
$(c_{\sigma}(\bfS) - \epsilon, c_{\sigma}(\bfS)].$

Then (\ref{eq:spectralInvariantIdentity2}) tells us that for sufficiently large $r$, there is a cochain
$\wt\sigma_\Phi(r) = T_{\text{Tw}\Phi}^r(\wt\sigma) \in \TWHM^{-*}(\cS_r)$
representing $\sigma_\Phi$ with a generator with action in the interval
$(-\pi c_{\sigma}(\bfS) -2 \epsilon, -\pi c_{\sigma}(\bfS) + \epsilon).$
By taking $\epsilon\to 0$, we obtain
$$\liminf_{r \to \infty} c^{\HM}_{\sigma_\Phi}(\cS_r; \fc^c(r)) \geq -\pi c_{\sigma}(\bfS).$$

Now assume for the sake of contradiction that
$$\limsup_{r \to \infty} c^{\HM}_{\sigma_\Phi}(\cS_r; \fc^c(r)) > -\pi c_{\sigma}(\bfS).$$
This implies that there is $\epsilon > 0$, a sequence $r_k \to \infty$, and a sequence of cochains 
$\wt\sigma_\Phi^k \in \TWHM^{-*}(\cS_{r_k})$
representing $\sigma_\Phi$ such that every generator in $\wt\sigma_\Phi^k$ has action greater than or equal to $-\pi c_{\sigma}(\bfS) + \epsilon$. 
 It follows that, after passing to a subsequence, there is a cochain $\wt\sigma \in \TWPFC_*(\bfS)$ representing $\sigma$ such that
every generator in $T_{\text{Tw}\Phi}^{r_k}(\wt\sigma)$ has action greater than or equal to $-\pi c_{\sigma}(\bfS) + \epsilon$
for every $k$. The identity (\ref{eq:spectralInvariantIdentity2}) then implies that every generator in $\wt\sigma$ will have action strictly less than $c_\sigma(\bfS)$. This is impossible by definition of the spectral invariant, which proves (\ref{eq:spectralInvariantIdentity1}). 

\textbf{Step 4:} 
Fix any base configuration $\fc_\Gamma = (B_\Gamma, \Psi_\Gamma)$. By Lemma \ref{lem:maxMin2}, for each $r > -2\pi\rho$ there is a solution $\fc(r) = (B(r), \Psi(r))$ to the $\cS_r$-Seiberg--Witten equations such that $r^{-1}\fa_{r, \fe_\mu}(\fc, \fc^c(r)) = c^{\HM}_{\sigma_\Phi}(\cS_r; \fc^c(r))$. We deduce the identity
\begin{equation} \label{eq:spectralInvariantIdentity3} \lim_{r \to \infty} \big(-\frac{1}{2}E_\phi(B(r), B_\Gamma) + \frac{i}{2}\int_{M_{\phi}} (B^c(r) - B_\Gamma) \wedge \omega_\phi\big) = \lim_{r \to \infty} -\frac{1}{2}E_\phi(B(r), B^c(r)) = -\pi c_\sigma(\bfS).\end{equation}

The first equality follows by definition of $E_\phi$ and the second follows from \eqref{eq:dFlatSwSpectral1} and \eqref{eq:spectralInvariantIdentity1}. The Chern--Simons--Dirac functional shifts by a constant when the base configuration is changed, so we also have $r^{-1}\fa_{r, \fe_\mu}(\fc, \fc_\Gamma) = c^{\HM}_{\sigma_\Phi}(\cS_r; \fc_\Gamma)$ for every $r > -2\pi\rho$. This and the analogue of \eqref{eq:dFlatSwSpectral1} for the base configuration $\fc_\Gamma$ shows that the left-hand side of \eqref{eq:spectralInvariantIdentity3} is equal to 
$$\lim_{r \to \infty} \big(c^{\HM}_{\sigma_\Phi}(\cS_r; \fc_\Gamma) + \frac{i}{2}\int_{M_{\phi}} (B^c(r) - B_\Gamma) \wedge \omega_\phi\big)$$
which proves the proposition. 
\end{proof}

\subsection{Estimating the Seiberg--Witten spectral invariants} \label{subsec:maxMin}

Having shown that we can recover PFH spectral invariants from the Seiberg--Witten ones, we now explain that the Seiberg--Witten spectral invariants can be estimated from the reducible solutions.  This section follows closely the ``max-min'' families technique from \cite{TaubesWeinstein1} and \cite{WeifengMinMax}. 
There are significant differences in the setup, a notable one being the lack of full gauge invariance of the Chern--Simons--Dirac functional.

\subsubsection{Max-min for $r > -2\pi\rho$}

Fix a monotone area-preserving map $\phi$, $J \in \cJ^\circ(dt, \omega_\phi)$, and a monotone $\Gamma \in H_1(M_\phi; \bZ)$ with monotonicity constant $\rho>0$. Fix an abstract perturbation $\fg \in \cP$ with $\|\fg\|_{\cP} \leq 1$. For any $r \in (-2\pi\rho, \infty)$, we define an SW parameter set $\cS_r = (\phi, J, \Gamma, r, \fg)$. Fix a choice of base configuration $\fc_\Gamma = (B_\Gamma, \Psi_\Gamma)$. Recall that the assoociated Seiberg--Witten--Floer cohomology groups $\TWHM^*(\cS_r)$ are all canonically isomorphic as $\bZ$-graded modules, where spectral flow to $\fc_\Gamma$ determines the $\bZ$-grading. Fix a nonzero homogeneous Seiberg--Witten--Floer cohomology class $\sigma$. 

Here, we study the behavior of the Seiberg--Witten spectral invariants $c^{\HM}_{\sigma}(\cS_r; \fc_\Gamma)$, considered as a function of the variable $r \in (-2\pi\rho, \infty)$. This relies on several estimates on solutions to the Seiberg--Witten equations which are proved in \S\ref{subsubsec:continuationEstimates}. Many of the results here make reference to constants fixed in \S\ref{subsubsec:continuationEstimates}. In particular, many of them assume that the degree $d$ and the monotonicity constant $-2\pi\rho$ of $\Gamma$ are large. These are not restrictive assumptions since the main Theorem \ref{thm:hutchingsConjectureFullintro} is concerned with the asymptotic behavior of spectral invariants as $d$ (and consequently $-2\pi\rho$) increases.

We begin with a proof that the function $c^{\HM}_{\sigma}(\cS_r; \fc_\Gamma) = c^{\HM}_\sigma(\phi, J, \Gamma, r, \fg; \fc_\Gamma)$
is continuous in $r$. 

\begin{lem}
\label{lem:maxMin3} Let $\cS_r = (\phi, J, \Gamma, r, \fg)$, $\rho$, $\sigma$, and $\fc_\Gamma$ be as fixed above. Suppose that $-2\pi\rho > 10r_{\ref{prop:continuationEstimates2}}$. Then $c^{\HM}_{\sigma}(\cS_r; \fc_\Gamma)$ is continuous as a function of $r \in (-2\pi\rho, \infty)$. 
\end{lem}

\begin{proof}
Pick $r_+ > r_- > -2\pi\rho$. Set $z_\pm = r_\pm + 2\pi\rho$. Choose any two perturbations $\fp_-$ and $\fp_+$ such that 
$$\TWHM(\phi, J, \Gamma, r_\pm, \fg_\pm = \fg + \fp_\pm)$$
are well-defined. Assume also that both $|r_+ - r_-|$ and $\|\fp_+ - \fp_-\|_{\cP}$ are bounded above by $\kappa_{\ref{prop:continuationEstimates2}}^{-1}/100$. Set
$$\cS_\pm = (\phi, J, \Gamma, r_\pm, \fg_\pm).$$

Fix an SW continuation parameter set
$\cS_s = (0, J, r_s, \fg_s)$
from $\cS_-$ to $\cS_+$. Suppose that 
$$\sup_{s \in \bR} \|\frac{\partial}{\partial s}\fg_s\|_\cP \leq 4\|\fg_+ - \fg_-\|_\cP \quad \text{and} \quad \sup_{s \in \bR} |r_s - r_-| \leq 4|r_+ - r_-|.$$
Suppose that $\{\fg_s\}_{s \in \bR}$ is chosen sufficiently generically, so that the chain isomorphism
$$T(\cS_s): \TWCM^*(\cS_+) \to \TWCM^*(\cS_-)$$
is well-defined. Let $\fc_+$ be a solution of the $\cS_+$--Seiberg--Witten equations, and let $\fc_-$ be any solution of the $\cS_-$--Seiberg--Witten equations such that $[\fc_-]^\circ$ appears with nonzero coefficient in $T(\cS_s)[\fc_+]^\circ$. Using the fact that $T(\cS_s)$ counts solutions to the $\cS_s$--Seiberg--Witten instanton equations, there is a solution $\fd$ of the $\cS_s$--Seiberg--Witten instanton equations (\ref{eq:swContinuation}) such that $\lim_{s \to \pm\infty} \fd = \fc_\pm$. 

By our earlier assumptions and Proposition \ref{prop:continuationEstimates2}, we have
\begin{equation*}
\frac{z_-}{z_+}\fa_{r_+, \fg_+}(\fc_+, \fc_\Gamma) - \fa_{r_-, \fg_-}(\fc_-, \fc_\Gamma) \leq \kappa_{\ref{prop:continuationEstimates2}}r_+^2(z_-^{-1} + z_+^{-1} + 1)|r_+ - r_-|  + \kappa_{\ref{prop:continuationEstimates2}}r_+^2\|\fg_+ - \fg_-\|_\cP.
\end{equation*}
Hence the same argument as in the proof of \eqref{eq:maxMin1} in Lemma \ref{lem:maxMin1} yields
\begin{equation}
\label{eq:maxMin6} 
\frac{z_-r_+}{z_+r_-}c^{\HM}_{\sigma}(\cS_+; \fc_\Gamma) \leq c^{\HM}_{\sigma}(\cS_-; \fc_\Gamma) + \kappa_{\ref{prop:continuationEstimates2}} r_+^2r_-^{-1}\|\fg_+ - \fg_-\|_\cP  + \kappa_{\ref{prop:continuationEstimates2}}r_+^2r_-^{-1}(z_-^{-1} + z_+^{-1} + 1)|r_+ - r_-|.
\end{equation}
Reversing the roles of $\cS_+$ and $\cS_-$, we also have
\begin{equation}
\label{eq:maxMin7} 
\frac{z_+r_-}{z_-r_+}c^{\HM}_{\sigma}(\cS_-; \fc_\Gamma) \leq c^{\HM}_{\sigma}(\cS_+; \fc_\Gamma) + \kappa_{\ref{prop:continuationEstimates2}} r_+\|\fg_+ - \fg_-\|_\cP 
 + \kappa_{\ref{prop:continuationEstimates2}}r_+(z_-^{-1} + z_+^{-1} +1)|r_+ - r_-|.
\end{equation}

Now fix some $r > -2\pi\rho$ and a sequence $\{r_k\}$ in $(-2\pi\rho, \infty)$ such that $r_k \to r$. It follows from (\ref{eq:maxMin6}) and (\ref{eq:maxMin7}) that 
$$\limsup_{k \to \infty} \frac{(2\pi\rho + r)r_k}{(2\pi\rho + r_k)r}c^{\HM}_{\sigma}(\cS_{r_k}; \fc_\Gamma) \leq c^{\HM}_{\sigma}(\cS_r; \fc_\Gamma) \leq \liminf_{k \to \infty} \frac{(2\pi\rho + r)r_k}{(2\pi\rho + r_k)r}c^{\HM}_{\sigma}(\cS_{r_k}; \fc_\Gamma).$$
Therefore
\begin{equation*}
c^{\HM}_{\sigma}(\cS_r; \fc_\Gamma) = \lim_{k \to \infty} \frac{(2\pi\rho + r)r_k}{(2\pi\rho + r_k)r}c^{\HM}_{\sigma}(\cS_{r_k}; \fc_\Gamma) \\
= \lim_{k \to \infty} c^{\HM}_{\sigma}(\cS_{r_k}; \fc_\Gamma),
\end{equation*}
and the desired result is proved. 
\end{proof}

Fix any small co-exact $1$-form $\mu$. A transversality result of Taubes implies that when $\mu$ is generic and the abstract perturbation term $\fg$ in the family $\cS_r$ is equal to $\fe_\mu$, the spectral invariants $c^{\HM}_\sigma(\cS_r; \fc_\Gamma)$ admit a particularly nice structure: there is a family of solutions to the Seiberg--Witten equations over the interval $(-2\pi\rho, \infty)$ that depends smoothly on $r$ on the complement of a locally finite set, such that the Chern--Simons--Dirac actions recover the spectral invariants.  The details are stated in the forthcoming lemma.

\begin{lem}
\label{lem:maxMin5} Let $\cS_r = (\phi, J, \Gamma, r, \fg)$, $\rho$, $\sigma$, $\fc_\Gamma$, and suppose $\fg = \fe_\mu$ where $\mu$ is generic. There is a locally finite set $U_\mu \subset (-2\pi\rho, +\infty)$ such that the following holds. First, the function $c^{\HM}_\sigma(\cS_r; \fc_\Gamma)$ is differentiable at any $r > -2\pi\rho$ that does not lie in $U_\mu$. Second, label the elements of the intersection of $U_\mu$ with $(-2\pi\rho, \infty)$ as a monotonically increasing sequence $\{s_k\}_{k \in \mathbb{N}}$. Then there is a  family of configurations
$$\{\fc(r,\mu,\sigma) = (B(r, \mu, \sigma), \Psi(r,\mu,\sigma))\}_{r > -2\pi\rho}$$
such that for any $r > -2\pi\rho$, $\fc(r,\mu,\sigma)$ solves the $\cS_r$--Seiberg--Witten equations, the family varies smoothly over the interval $(s_k, s_{k+1})$ for any $k \in \mathbb{N}$, and at every $r > -2\pi\rho$ we have
$$c^{\HM}(\cS_r; \fc_\Gamma) = \fa_{r, \fe_\mu}(\fc(r, \mu, \sigma), \fc_\Gamma).$$
\end{lem}

\begin{proof}
Lemma \ref{lem:maxMin2} states that for every $r$, the spectral invariant $c^{\HM}_{\sigma}(\cS_r; \fc_\Gamma)$ is the action of some solution to the $\cS_r$--Seiberg--Witten equations. Now \cite[Proposition $3.11$]{TaubesWeinstein1} implies that when $\mu$ is generic, all solutions of the $\cS_r$-Seiberg--Witten equations are nondegenerate apart from a locally finite set of $r > -2\pi\rho$, and moreover no two gauge-inequivalent solutions to the $\cS_r$-Seiberg--Witten equations such an $r$ will have the same Chern--Simons--Dirac action. The implicit function theorem implies that non-degenerate solutions to the Seiberg--Witten equations move in smooth families as the parameter set is varied, which allows us to conclude the lemma. See also the discussion in \cite[Section $5$]{WeifengMinMax} in the contact setting. 
\end{proof}

We will assume from now on that $\mu$ is chosen generically so that Lemma \ref{lem:maxMin5} is satisfied. A family of configurations $\{\fc(r, \mu, \sigma) = (B(r, \mu, \sigma), \Psi(r, \mu, \sigma))\}_{r > -2\pi\rho}$ as given by Lemma \ref{lem:maxMin5} is called a \textbf{$(\mu,\sigma)$--max-min family}. 

\begin{rem}
Changing the base configuration changes the Chern--Simons--Dirac functional $\fa_{r, \fe_\mu}$ by a constant. It follows that we can take the max-min family to be ``independent'' of the base configuration. That is, a max-min family for a base configuration $\fc_\Gamma$ will also be a max-min family for a different base configuration $\fc_\Gamma'$. 
\end{rem}

\subsubsection{Max-min for $r = -2\pi\rho$}

Let $\phi$, $J$, $\Gamma$, $\rho$, $\fg$, $\fc_\Gamma$, $\sigma$ be as fixed above. The SW parameter set $\cS_{-2\pi\rho} = (\phi, J, \Gamma, -2\pi\rho, \fg)$ has an associated Floer cohomology group $\TWHMhat^*(\cS_{-2\pi\rho})$. By a general property of Seiberg--Witten--Floer cohomology \cite[Theorem $31.5.1$]{monopolesBook}, it is canonically isomorphic to the groups $\TWHM^*(\cS_r)$ for $r > -2\pi\rho$. The isomorphism is, as in the prior cases considered, given by counting $\cS_s$-continuation instantons. Therefore, $\sigma$ can be interpreted as a class in $\TWHMhat^*(\cS_{-2\pi\rho})$, and there is a well-defined spectral invariant $c^{\HM}_{\sigma}(\cS_{-2\pi\rho}; \fc_\Gamma)$, defined in the same manner as the $r > -2\pi\rho$ case.  

The exact same argument in Lemma \ref{lem:maxMin1} works in this setting with minor modifications to show the following:
\begin{lem}
\label{lem:maxMin6}
The function
$c^{\HM}_{\sigma}(\phi, J, \Gamma, -2\pi\rho, \fg; \fc_\Gamma)$
extends continuously to any $\fg \in \cP$ with $\cP$-norm less than or equal to $1$. 
\end{lem}
\begin{proof}[Sketch of proof]
	The only difference from Lemma \ref{lem:maxMin1} is that in this case an SW continuation map $T(\cS_s)$ is given by counting \emph{broken} instantons. Lemma \ref{lem:maxMin6} follows by applying the estimate from Proposition \ref{prop:continuationEstimates2} to each instanton within the broken instanton from the differential map, and then proceeding with the same argument as the proof of Lemma \ref{lem:maxMin1}. 
\end{proof}
 
We now prove that the spectral invariants extend continuously to $r = -2\pi\rho$.  

\begin{lem}
\label{lem:maxMin7} Let $\cS_r = (\phi, J, \Gamma, r, \fg)$, $\rho$, $\sigma$ and $\fc_\Gamma$ be as fixed above. Suppose that $-2\pi\rho > 10r_{\ref{prop:continuationEstimates5}}$. Then 
$$\lim_{r \to -2\pi\rho} c^{\HM}_{\sigma}(\cS_r; \fc_\Gamma) = c^{\HM}_{\sigma}(\cS_{-2\pi\rho}; \fc_\Gamma).$$
\end{lem}

\begin{proof}
We proceed in a similar manner to the proof of Lemma \ref{lem:maxMin3}, using Proposition \ref{prop:continuationEstimates5} instead of Proposition \ref{prop:continuationEstimates2}. The proof will proceed in three steps. 

\textbf{Step 1:} The first step uses results of \S\ref{subsubsec:continuationEstimates} to bound the difference in the spectral invariants for $r_+ > -2\pi\rho$ and $r_- = -2\pi\rho$. Fix any two SW parameter sets
$$\cS_\pm = (\phi, J, \Gamma, r_\pm, \fg_\pm)$$
where $r_+ > r_- = -2\pi\rho$ and $\TWHMhat^*(\cS_-)$ and $\TWHM^*(\cS_+)$ are well-defined. We assume that $\|\fg_+ - \fg_-\|$ and $|r_+ - r_-|$ are both bounded above by $\kappa_{\ref{prop:continuationEstimates5}}^{-1}$. 

Fix an SW continuation parameter set $\cS_s$ from $\cS_-$ to $\cS_+$ and an SW continuation parameter set $\cS_s'$ from $\cS_+$ to $\cS_-$. Then the continuation maps $T(\cS_s): \TWCM^*(\cS_+) \to \TWCMhat^*(\cS_-)$ and $T(\cS_s'): \TWCMhat^*(\cS_-) \to \TWCM^*(\cS_+)$ define the canonical isomorphism between $ \TWCM^*(\cS_+)$ and $\TWCMhat^*(\cS_-)$.

By Proposition \ref{prop:continuationEstimates5} and the definition of $T(\cS_s)$, we conclude that for any cochain $\widetilde{\sigma}_+$ in $\TWCM(\cS_+)$ representing $\sigma$, and for any $\fc_+$ such that $[\fc_+]^\circ$ has nonzero coefficient in $\widetilde{\sigma}_+$, and for any $\fc_-$ such that $[\fc_-]^\circ$ has nonzero coefficient in $T(\cS_s) \cdot \widetilde{\sigma}_+$, that 
\begin{equation}
\label{eq:maxMin8}
\begin{split}
\fa_{r_-, \fg_+}(\fc_+, \fc_\Gamma) &\leq \fa_{r_-, \fg_-}(\fc_-) + \kappa_{\ref{prop:continuationEstimates5}} r_+^2(|r_+ - r_-| + \|\fg_+ - \fg_-\|_\cP).
\end{split}
\end{equation}
Similarly, for any cochain $\widetilde{\sigma}_-$ in $\TWCM(\cS_-)$ representing $\sigma$, and for any $\fc_-$ such that $[\fc_-]^\circ$ has nonzero coefficient in $\widetilde{\sigma}_-$, and for any $\fc_+$ such that $[\fc_+]^\circ$ has nonzero coefficient in $T(\cS_s') \cdot \widetilde{\sigma}_-$, we have 
\begin{equation}
	\label{eq:maxMin8'}
	\begin{split}
		\fa_{r_-, \fg_-}(\fc_-, \fc_\Gamma) &\leq \fa_{r_-, \fg_+}(\fc_+) + \kappa_{\ref{prop:continuationEstimates5}} r_+^2(|r_+ - r_-| + \|\fg_+ - \fg_-\|_\cP).
	\end{split}
\end{equation}

\textbf{Step 2:} The second step uses the inequality (\ref{eq:maxMin8}) from Step 1 to relate the spectral invariant $c^{\HM}_{\sigma}(\cS_-; \fc_\Gamma)$ at $r = -2\pi\rho$ to a variational quantity which is very close to the spectral invariant at $r = r_+$. Write 
$$f(r_+) = \max_{[\widetilde{\sigma}_+] = \sigma}\min_{[\fc_+]^\circ \in \widetilde{\sigma}_+}r_+^{-1}\fa_{r_-,\fg_+}(\fc_+).$$

This is similar to the spectral invariant at $r = r_+$, but we are taking a max-min of the Chern--Simons--Dirac functional with parameter $r = r_-$. 
Taking a max-min on \eqref{eq:maxMin8} and \eqref{eq:maxMin8'} yields
\begin{equation} \label{eq:maxMin9}
\begin{split}|r_+f(r_+) - r_-c^{\HM}_{\sigma}(\cS_-; \fc_\Gamma)| &\leq \kappa_{\ref{prop:continuationEstimates5}}r_+^2(|r_+ - r_-| + \|\fg_+ - \fg_-\|_\cP).
\end{split}
\end{equation}
and therefore
$$\lim_{r_+ \to -2\pi\rho} r_+f(r_+) = r_-c^{\HM}(\cS_-; \fc_\Gamma).$$

\textbf{Step 3:} This step shows that the quantity $f(r_+)$ and the spectral invariant at $r = r_+$ have the same limit as $r_+ \to -2\pi\rho$: 
\begin{equation} \label{eq:maxMin10} \lim_{r_+ \to -2\pi\rho} |f(r_+) - c^{\HM}(\cS_+; \fc_\Gamma)| = 0. \end{equation}

Let $\fc_+$ be a solution to the $\cS_+$--Seiberg--Witten equation such that $f(r_+)=r_+^{-1} \fa_{r_-,\fg_+}(\fc_+)$.
By the definition of the functional $\fa$, 
$$\fa_{r_+, \fg_+}(\fc_+)  - \fa_{r_-, \fg_+}(\fc_+) = \frac{1}{2}(r_- - r_+)E_{\phi}(\fc_+).$$
Notice that the function
$-2\pi\rho E_{\phi}(-, \fc_\Gamma) + \fcs(-, \fc_\Gamma)$
is gauge-invariant, so Proposition \ref{prop:3dActionEstimates} and Proposition \ref{prop:spectralFlow2} imply that $|E_{\phi}(\fc_+)| \leq \kappa r_+$, where $\kappa$ depends only on $\rho$ and the curvature $F_{B_\Gamma}$ of the base connection. 
Hence
$$|f(r_+) - c^{\HM}(\cS_+; \fc_\Gamma)| \leq \kappa (r_+ - r_-),$$
which implies (\ref{eq:maxMin10}). The lemma is then an immediate consequence of \eqref{eq:maxMin9} and \eqref{eq:maxMin10}.
\end{proof}

We now prove the following important proposition, which shows that the spectral invariant at $r = -2\pi\rho$ is the action of a reducible solution. This allows us to compute the spectral invariant at $r = -2\pi\rho$ later. 

\begin{prop}
\label{prop:maxMinReducibles} There is a constant $d_{\ref{prop:maxMinReducibles}} \geq 1$ depending only on the manifold $M_\phi$ such that the following holds. Fix an SW parameter set 
$\cS = (\phi, J, \Gamma, -2\pi\rho, \fe_\mu)$
such that the class $\Gamma$ has degree $d \geq d_{\ref{prop:maxMinReducibles}}$. Then for any homogeneous non-zero class $\sigma \in \TWHMhat^*(\cS)$ and base configuration $\fc_\Gamma$, there is a reducible solution $\fc_{\text{red}} = (B_{\text{red}}, 0)$ of the $\cS$--Seiberg--Witten equations such that 
$$c^{\HM}_{\sigma}(\cS; \fc_\Gamma) = \fa_{-2\pi\rho, \fe_\mu}(\fc_{\text{red}}, \fc_\Gamma).$$
\end{prop}

\begin{proof}
The proposition will follow in three steps. 

\textbf{Step 1:} The first step shows that when the degree of $\Gamma$ is sufficiently large, the existence of a reducible solution $\fc_{\text{red}} = (B_{\text{red}}, 0)$ such that 
$$c^{\HM}_{\sigma}(\phi, J, \Gamma, -2\pi\rho, \fe_\mu; \fc_\Gamma) \geq \fa_{-2\pi\rho, \fe_\mu}(\fc_{\text{red}}, \fc_\Gamma).$$

This is a consequence of Corollary \ref{cor:barEqualsHat}. In fact, the
isomorphism in Corollary \ref{cor:barEqualsHat} 
arises from a chain map
which counts (broken, blown-up) index zero solutions of the $\cS_b$--Seiberg--Witten instanton equations (\ref{eq:swInstanton}). Therefore, when the degree $d$ is sufficiently large, Corollary \ref{cor:barEqualsHat} tells us that for any solution $\fc$ of the $(\phi, J, \Gamma, -2\pi\rho, \fe_\mu)$--Seiberg--Witten equations such that
$c^{\HM}_{\sigma}(\phi, J, \Gamma, -2\pi\rho, \fe_\mu; \fc_\Gamma) = \fa_{-2\pi\rho, \fe_\mu}(\fc, \fc_\Gamma),$
there is a \emph{reducible} solution $\fc_{\text{red}}$ of the $(\phi, J, \Gamma, -2\pi\rho, \fe_\mu)$--Seiberg--Witten equations and a (possibly broken) instanton $\fd$ that limits to $\fc$ as $s \to -\infty$ and $\fc_{\text{red}}$ as $s \to \infty$. The monotonicity property of the Chern--Simons--Dirac functional (see Lemma \ref{lem:csdContinuationGradient1}) implies
$\fa_{-2\pi\rho, \fe_\mu}(\fc, \fc_\Gamma) \geq \fa_{-2\pi\rho, \fe_\mu}(\fc_{\text{red}}, \fc_\Gamma),$
which proves the desired lower bound. 

\textbf{Step 2:} The second step shows the existence of a reducible solution $\fc_{\text{red}}' = (B_{\text{red}}', 0)$ such that 
$$c^{\HM}_{\sigma}(\phi, J, \Gamma, -2\pi\rho, \fe_\mu; \fc_\Gamma) \leq \fa_{-2\pi\rho, \fe_\mu}(\fc_{\text{red}}').$$

It uses algebraic properties of both the completed and non-completed versions of Seiberg--Witten--Floer cohomology at $r = -2\pi\rho$. First, perturb $\fe_\mu$ by a small element $\fp\in\cP$ such that the Seiberg--Witten--Floer cohomology is well-defined. 
Write 
$\cS_b = (\phi, J, \Gamma, -2\pi\rho, \fe_\mu + \fp)$, $ \cS_0 = (\phi, J, \Gamma, 0, \fe_\mu + \fp).$
By \cite[Theorem $31.1.1$]{monopolesBook}, we have $\TWHMhat^\bullet(\cS_b) \cong  \TWHMhat^\bullet(\cS_0)$, and there are chain maps
$$\hat{i}: \TWCMhat^\bullet(\cS_b) \to \TWCMhat^\bullet(\cS_0), \qquad \hat{j}: \TWCMhat^\bullet(\cS_0) \to \TWCMhat^\bullet(\cS_b)$$
and a chain homotopy $\hat{K}$ from $\hat{j} \circ \hat{i}$ to the identity map defined by counting broken trajectories of solutions to the Seiberg--Witten equations. 
By Proposition \ref{prop:checkVanishing}, this implies $\TWHMhat^\bullet(\cS_b) = 0$ when $d$ is sufficiently large. 

We will show that any cochain $\wt\sigma$ in $\TWCMhat^*(\cS_b)$ (the non-completed version) representing $\sigma$ contains a reducible generator. This immediately implies the desired inequality by the definition of the max-min action. 

Suppose for the sake of contradiction that there is a cochain $\wt\sigma = \sum_{i=1}^n [\fc_i]^\circ$ representing $\sigma$ such that all of the generators $[\fc_i]^\circ$ are irreducible. Write $\partial_*$ and $\partial_\bullet$ for the differentials on $\TWCMhat^*(\cS_b)$ and $\TWCMhat^\bullet(\cS_b)$ respectively. Then we have
$\hat{K}\partial_\bullet - \partial_\bullet\hat{K} = \text{id}.$

Since $\partial_*\wt\sigma = 0$, we have $\partial_\bullet\wt\sigma = 0$ as well. Therefore $-\partial_\bullet\hat{K}\wt\sigma = \wt\sigma.$
It follows that $-\hat{K}\wt\sigma$ is a primitive for $\wt\sigma$ in $\TWCMhat^\bullet(\cS_b)$. However, 
the argument of \cite[Proposition $31.2.6$]{monopolesBook} implies that the map $\hat K$ takes an \emph{irreducible} generator of $\TWCMhat^\bullet(\cS_b)$ to a sum of \emph{finitely} many generators of $\TWCMhat^\bullet(\cS_b)$. (In fact, in the proof of \cite[Proposition $31.2.6$]{monopolesBook}, it is shown that if $\hat K([\fc])$ is an infinite sum of generators, then there must be infinitely many reducible instantons $[\fd_n]$ which limit to $[\fc]$ as $s \to -\infty$, which implies that $[\fc]$ is reducible.)
Since we have assumed $\wt\sigma$ is solely composed of irreducibles, we find that $-\hat{K}\wt\sigma$ is a sum of finitely many generators, which shows that $\wt\sigma$ is a coboundary in ${\TWCMhat}^*(\cS_b)$. This is a contradiction since we assumed $\wt\sigma$ represented a nonzero cohomology class. 

\textbf{Step 3:} The third step shows
$\fa_{-2\pi\rho, \fe_\mu}(\fc_{\text{red}}, \fc_\Gamma) = \fa_{-2\pi\rho, \fe_\mu}(\fc_{\text{red}}', \fc_\Gamma)$. This follows from either explicit computation or the following formal argument. Both $\fc_{\text{red}}$ and $\fc_{\text{red}}'$ are critical points of $\fa_{-2\pi\rho, \fe_\mu}$, and the set of reducible critical points of $\fa_{-2\pi\rho, \fe_\mu}$ is an affine space, which is connected.  So we must have $\fa_{-2\pi\rho, \fe_\mu}(\fc_{\text{red}}, \fc_\Gamma) = \fa_{-2\pi\rho, \fe_\mu}(\fc_{\text{red}}', \fc_\Gamma).$ 
\end{proof}

\subsubsection{Differential equations on $r$}
Fix $\phi, J, \Gamma, \fe_\mu$, consider the 1-parameter family of SW parameter set
 $$\cS_r = (\phi, J, \Gamma, r, \fe_\mu),$$
and let $\fc(r) = (B(r), \Psi(r) = r^{1/2}(\alpha(r), \beta(r)))$ be a $(\mu, \sigma)$ max-min family as in Lemma \ref{lem:maxMin5} (dropping the $(\mu, \sigma)$ from the notation). Fix a base configuration $\fc_\Gamma$ and a homogeneous nonzero class $\sigma \in \TWHM^*(M_\phi, \Gamma, m_-)$. For any $r > -2\pi\rho$ outside of a discrete subset, the family $\fc(r)$ is smooth. For simplicity write $\fa(r) := \fa_{r, \fe_\mu}(\fc(r), \fc_\Gamma)$, $E_\phi(r) := E_\phi(B(r), B_\Gamma)$, $\fcs(r) := \fcs(B(r), B_\Gamma)$, and $\fe_\mu(r) := \fe_\mu(B(r), B_\Gamma)$. As a consequence of the fact that $\fc(r)$ satisfies the $\cS_r$--Seiberg--Witten equations (\ref{eq:sw}),  at any smooth point, we have
\begin{equation}
	\label{eq:actionDerivative1} 
		\frac{\partial}{\partial r} \fa(r) = -\frac{1}{2}E_{\phi}(r).
\end{equation}
The proof of \eqref{eq:actionDerivative1} is identical to \cite[(4.6)]{TaubesWeinstein1} and we refer the reader to \cite{TaubesWeinstein1} for details. As a consequence, 
\begin{equation}
	\label{eq:actionDerivative2}
		\frac{\partial}{\partial r} c^{\HM}_{\sigma}(\cS_r; \fc_\Gamma) = \frac{\partial}{\partial r}(r^{-1}\fa(r)) = -r^{-2}\fa(r) - \frac{r^{-1}}{2}E_{\phi}(r) = -r^{-2}(\fcs(r) + \fe_\mu(r)).
\end{equation}

The following lemma uses the differential equations above to prove a bound on the difference between the limit of the spectral invariants as $r \to \infty$ and the spectral invariant near $r = -2\pi\rho$. It also makes use of spectral flow estimates from Propositions \ref{prop:spectralFlow1} and \ref{prop:spectralFlow2}; these propositions require the curvature of the base connection to satisfy suitable $C^3$ bounds. 

\begin{lem}
\label{lem:maxMinLimit}
Let $\cS_r = (\phi, J, \Gamma, r, \fe_\mu)$, $\sigma$, $\fc_\Gamma = (B_\Gamma, \Psi_\Gamma)$ be fixed as above. Fix any constant $\epsilon > 0$. Fix a constant $\Lambda \geq 1$ such that $\|F_{B_\Gamma}\|_{C^3} \leq \Lambda d$. Then there is a constant $\kappa_{\ref{lem:maxMinLimit}} = \kappa_{\ref{lem:maxMinLimit}}(\epsilon) \geq 1$ and a constant $d_{\ref{lem:maxMinLimit}} \geq 1$ depending on $\epsilon$, $\Lambda$, and geometric constants such that the following holds as long as $d \geq d_{\ref{lem:maxMinLimit}}$. The spectral invariants $c^{\HM}_{\sigma}(\cS_r; \fc_\Gamma)$ converge as $r \to \infty$ and moreover, for every $r_* > -2\pi\rho$, we have the bound
	$$|\lim_{r \to \infty} c^{\HM}_{\sigma}(\cS_r; \fc_\Gamma) - c^{\HM}_{\sigma}(\cS_{r_*}; \fc_\Gamma) + \frac{2\pi^2}{r_*}\SF(\sigma, \fc_\Gamma)| \leq \kappa_{\ref{lem:maxMinLimit}} d^{4/5 + \epsilon}.$$ 
\end{lem}

\begin{proof}
Fix any $r_2 > r_1 > -2\pi\rho$. In what follows we will take $\kappa(\Lambda) \geq 1$ to be a constant depending only on geometric constants and $\Lambda$, which can be assumed to increase between successive appearances. Stokes' theorem and Proposition \ref{prop:3dEstimates1} indicates for any $r$ that we have a bound of the form
\[
|\fe_\mu(r)| \leq \|F_{B(r)} - F_{B_\Gamma}\|_{L^1} \leq 2\pi d + \kappa + \|F_{B_\Gamma}\|_{L^1}.
\]

Recall that $\SF(\fc(r), \fc_\Gamma) = \SF(\sigma, \fc_\Gamma)$ for every $r$. We integrate the identity (\ref{eq:actionDerivative2}) from $r_1$ to $r_2$ to deduce that
\begin{equation} \label{eq:maxMinLimit3}
\begin{split}
& |c^{\HM}_{\sigma}(\cS_{r_2}; \fc_\Gamma) - c^{\HM}_{\sigma}(\cS_{r_1}; \fc_\Gamma) - 2\pi^2(r_2^{-1} - r_1^{-1})\SF(\sigma, \fc_\Gamma)| \\
&\qquad  \leq \frac{1}{2}\int_{r_1}^{r_2} r^{-2}(|\fcs(r) + 4\pi^2\SF(\sigma, \fc_\Gamma)| + |\fe_\mu(r)|)dr \\
&\qquad \leq \frac{1}{2}\int_{r_1}^{r_2} r^{-2}|\fcs(r) + 4\pi^2\SF(\sigma, \fc_\Gamma)| dr + (2\pi d + \kappa + \|F_{B_\Gamma}\|_{L^1})(r_1^{-1} - r_2^{-1}).
\end{split}
\end{equation}

Next, we examine the integral
$$\int_{r_1}^{r_2} r^{-2}|\fcs(r) + 4\pi^2\SF(\sigma, \fc_\Gamma)| dr.$$

By Propositions \ref{prop:spectralFlow1} and \ref{prop:spectralFlow2} and the fact that the base connection has $C^3$ norm bounded by $\Lambda d$, we have the bounds
\begin{equation} \label{eq:maxMinLimit1} |\fcs(r) + 4\pi^2\SF(\sigma, \fc_\Gamma)| \leq \kappa(\Lambda) r^{3/2} \end{equation}
and
$$|\fcs(r) + 4\pi^2\SF(\sigma, \fc_\Gamma)| \leq \kappa(\Lambda) d^{4/3}r^{2/3}\log(r)^{\kappa(\Lambda)}.$$

Fix some constant $\delta_1 \in (0, 1)$. If $d$ is larger than a constant depending only on $\kappa(\Lambda)$ and $\delta_1$, then we can use the fact that $r_2 > r_1 > -2\pi\rho$ to deduce that 
\begin{equation} \label{eq:maxMinLimit2} |\fcs(r) + 4\pi^2\SF(\sigma, \fc_\Gamma)| \leq \kappa(\Lambda)d^{4/3}r^{2/3 + \delta_1}. \end{equation}

For any $r_2 > r_1 > -2\pi\rho$, \eqref{eq:maxMinLimit2} shows 
\begin{equation*}\int_{r_1}^{r_2} r^{-2}|\fcs(r) + 4\pi^2\SF(\sigma, \fc_\Gamma)| dr \leq \kappa(\Lambda) d^{4/3}\int_{r_1}^{r_2} r^{-4/3 + \delta_1} \leq \kappa(\Lambda) d^{4/3} r_1^{-1/6}. \end{equation*}

This along with \eqref{eq:maxMinLimit3} shows that the family $c_\sigma^{\HM}(\cS_r; \fc_\Gamma)$ is Cauchy, which shows that the spectral invariants converge. This proves the first assertion of the lemma. 

We will now prove the second assertion of the lemma. This will use both (\ref{eq:maxMinLimit1}) and (\ref{eq:maxMinLimit2}) to prove a suitable bound for 
$$\int_{r_1}^{r_2} r^{-2}|\fcs(r) + 4\pi^2\SF(\sigma, \fc_\Gamma)| dr$$
for any $r_2 > r_1 > -2\pi\rho$. Fix another constant $\delta_2 = (6\delta_1 + 3)/(5 - 6\delta_1)$. 

Assume without loss of generality that $r_1 \leq d^{1 + \delta_2}$ and $r_2 \geq d^{1 + \delta_2}$. Then by (\ref{eq:maxMinLimit1}),
\begin{equation*}
\int_{r_1}^{d^{1 + \delta_2}} r^{-2}|\fcs(r) + 4\pi^2\SF(\sigma, \fc_\Gamma)| dr \leq \kappa(\Lambda) \int_{r_1}^{d^{1 + \delta_2}} r^{-1/2} dr \leq \kappa(\Lambda) d^{(1 + \delta_2)/2}.
\end{equation*}
By (\ref{eq:maxMinLimit2}),
\begin{equation*}
\int_{d^{1 + \delta_2}}^{r_2} r^{-2}|\fcs(r) + 4\pi^2\SF(\sigma, \fc_\Gamma)| dr \leq \kappa(\Lambda) d^{4/3} \int_{d^{1 + \delta_2}}^{r_2} r^{-4/3 + \delta_1} dr \leq \kappa(\Lambda) d^{4/3 + (1 + \delta_2)(\delta_1 - 1/3)}. 
\end{equation*}

The strange choice of $\delta_2$ is now justified by the fact that
$$\frac{1}{2}(1 + \delta_2) = 4/3 + (1 + \delta_2)(\delta_1 - 1/3) = \frac{4}{5 - 6\delta_1}.$$

It follows that
$$\int_{r_1}^{r_2} r^{-2}|\fcs(r) + 4\pi^2\SF(\sigma, \fc_\Gamma)| dr \leq \kappa(\Lambda) d^{\frac{4}{5 - 6\delta_1}}$$
for any $r_2 > r_1 > -2\pi\rho$. Taking $\delta_1$ sufficiently close to $0$ and combining this bound with \eqref{eq:maxMinLimit3} proves the second assertion of the lemma. 
\end{proof}

\subsection{Proof of Theorem \ref{thm:hutchingsConjectureFullintro}}

We are now ready to prove our Weyl law for PFH spectral invariants. The proof will be carried out in $7$ steps. 

\textbf{Step 1:} The first step reduces proving the theorem to the case where the chosen data satisfy some additional simplifying assumptions. 

First, by approximating a degenerate map with a non-degenerate one, 
we can reduce to the case where $\phi$ and $\phi'$ are both nondegenerate.

Next, we reduce to the case where the reference cycles $\Theta_m$ and $\Theta_m^H$ are separated for every $m$. 
Assume for now that the theorem holds when the reference cycles are separated for every $m$. We will prove the theorem for an arbitrary choice of reference cycles $\Theta_m$ and $\Theta_m^H$. We choose for every $m$ a reference cycle $\Xi_{m}$ so that both $\Xi_m$ and $\Xi_m^H$ are separated reference cycles. We also choose for every $m$ a cycle $W_m \in H_2(M_\phi, \Theta_m, \Xi_m; \bZ)$ and set $W_m^H = (M_H)_*(W_m) \in H_2(M_{\phi'}, \Theta_m^H, \Xi_m^H; \bZ)$. This induces for any $m$ an identification $$\TWPFH_*(\phi, \Theta_m) \simeq \TWPFH_*(\phi, \Xi_m)$$
by the map
$(\Theta, W) \mapsto (\Theta, W + W_m)$
on generators. This isomorphism may not preserve the $\bZ$--grading, but it shifts the $\bZ$-grading by a constant depending only on $W_m$. Likewise, there is an identification of $\TWPFH_*(\phi', \Theta_m^H)$ with $\TWPFH_*(\phi', \Xi_m^H)$, and the exact same shift in the $\bZ$-grading, so in particular the grading difference $I(\tau_m) - I(\sigma_m)$ will not change when we shift the reference cycles. We also deduce the following two identities for the spectral invariants:
$$c_{\sigma_m}(\phi, \Theta_m) = c_{\sigma_m}(\phi, \Xi_m) - \int_{W_m} \omega_\phi, \quad c_{\tau_m^H}(\phi', \Theta_m^H) = c_{\tau_m^H}(\phi', \Xi_m^H) - \int_{W_m} (\omega_{\phi} + dH \wedge dt).$$

This computation, along with the fact that $I(\sigma_m') - I(\sigma_m)$ does not change, implies that the left-hand side of the expression in the statement of the theorem does not change when $\Theta_m$ and $\Theta_m^H$ are changed to $\Xi_m$ and $\Xi_m^H$, and the desired claim is verified.

\textbf{Step 2:} The second step sets up some relevant definitions and notation. 

Choose generic almost-complex structures $J \in \cJ^\circ(dt, \omega_\phi)$ and $J' \in \cJ^\circ(dt, \omega_{\phi'})$. Write $g$ and $g'$ for the Riemannian metrics on $M_\phi$ and $M_{\phi'}$ associated to $(\phi, J)$ and $(\phi', J')$, respectively. Next, the classes $\Gamma_m$ define spin-c structures $S_m = E_m \oplus E_m \otimes V$ on $M_\phi$. Push these forward by $M_H$ to define spin-c structures $S_m' = E_m' \oplus E_m' \otimes V'$ on $M_{\phi'}$ corresponding to the classes $\Gamma'_m$. 

We now define two families of unitary reference connections on $E_m$, starting with the ``harmonic connections''. Write $B^h_m$ for the connection such that the curvature $F^h_m$ is $g$-harmonic and Poincar\'e dual to $\frac{i}{2\pi}\Gamma_m$. Push forward $B^h_m$ by $M_H$ to define a corresponding sequence $B^{h'}_m$ of unitary connections on $E_m'$. Elliptic regularity shows that the $C^3$ norms of the curvatures $\|F^h_m\|_{C^3(g)}$, $\|F^{h'}_m\|_{C^3(g')}$ are both $\lesssim d_m$ (and as a consequence $\lesssim |\rho_m|$). Choose two sequences of spinors $\Psi_m^h$ and $\Psi_m^{h'}$ that are $C^3$-small to define two sequence of reference configurations $\fc_m^h = (B_m^h, \Psi_m^h)$ and $\fc_m^{h'} = (B_m^h, \Psi_m^{h'})$. 

We fix for every $m$ and every $r > -2\pi\rho_m$ the PFH parameter sets
$\bfS_m = (\phi, \Theta_m, J)$ and  $\bfS_m^H = (\phi', \Theta_m^H, J')$
and the corresponding SW parameter sets
$\cS_{m,r} = (\phi, J, \Gamma_m, r, \fe_{\mu_m})$, $\cS_{m,r}^H = (\phi', J', \Gamma_m^H, r, \fe_{\mu_m'}).$

For every $m$ and every $r \gg -2\pi\rho_m$, we fix the data that defines the Lee--Taubes isomorphism between the PFH of $\bfS_m$ and the SWF of $\cS_{m,r}$. 
We will use this isomorphism to identify every $\sigma_m$ with a Seiberg--Witten--Floer cohomology class. Recall that the isomorphism data was defined in \S\ref{subsec:twistedIsoDefinition}. 
One of the items of data was a \emph{$\Theta_m$--concentrated family} of base configurations, which we write as
$\fc^c_m(r) = (B^c_m(r), \Psi^c_m(r) = (\alpha^c_m(r), \beta^c_m(r))).$

The main formal property of the configurations $\fc^c_m(r)$ that we need is that, as $r \to \infty$, the curvatures $F^c_m(r)$ of $B^c_m(r)$ converge weakly as currents to $-2\pi i \Theta_m$. We push forward the data by $M_H$ (see, for example the setup in \S\ref{subsec:twistedIsoGradings}) to define the Lee--Taubes isomorphism between the PFH of $\bfS_m^H$ and the SWF of $\cS_{m,r}^H$. 
We will use this isomorphism to identify $\tau_m^H$ with a Seiberg--Witten--Floer cohomology class for every $m$. The pushforward of the data by $M_H$ defines a \emph{$\Theta_m^H$--concentrated family} of base configurations, which we write as
$\fc^{c'}_m(r) = (B^{c'}_m(r), \Psi^{c'}_m(r) = (\alpha^{c'}_m(r), \beta^{c'}_m(r))).$
The curvatures $F^{c'}_m(r)$ of $B^{c'}_m(r)$ converge weakly as currents as $r \to \infty$ to $-2\pi i \Theta_m^H$. 

Finally, to tie up a loose end, write $\varpi_V$ and $\varpi_{V'}$ for the harmonic two-forms representing $\pi c_1(V)$ with respect to the metrics $g$ and $g'$.

\textbf{Step 3:} In this step, we write down the relationship of the twisted PFH and twisted SWF spectral invariants from \S\ref{sec:twistedIso} in a suitable form. Recall from Proposition \ref{prop:spectralInvariantIdentity} that we have the identities
\begin{equation}
\label{eq:hutchingsProof1}
\begin{split}
\lim_{r \to \infty} c^{\HM}_{\sigma_m}(\cS_{m,r}; \fc^h_m) + \frac{i}{2} \int_{M_{\phi}} (B^c_m(r) - B^h_m) \wedge \omega_{\phi} = -\pi c_{\sigma_m}(\bfS_m), \\
\lim_{r \to \infty} c^{\HM}_{\tau_m^H}(\cS_{m,r}^H; \fc^{h'}_m) + \frac{i}{2} \int_{M_{\phi'}} (B^{c'}_m(r) - B^{h'}_m) \wedge \omega_{\phi'}  = -\pi c_{\tau_m^H}(\bfS_m^H).
\end{split}
\end{equation}
 Taking the difference of the two equations in (\ref{eq:hutchingsProof1}) yields
\begin{equation}
\label{eq:hutchingsProof4}
\begin{split}
&c_{\tau_m^H}(\bfS_m^H) - c_{\sigma_m}(\bfS_m) \\
= & \lim_{r \to \infty} \pi^{-1}\big[c^{\HM}_{\sigma_m}(\cS_{m,r}; \fc^h_m) - c^{\HM}_{\tau_m^H}(\cS_{m,r}^H; \fc^{h'}_m) -\frac{i}{2\pi} \int_{M_\phi} (B^c_m(r) - B^h_m) \wedge dH \wedge dt\big]\\
= & \big[\lim_{r \to \infty} \pi^{-1}(c^{\HM}_{\sigma_m}(\cS_{m,r}; \fc^h_m) - c^{\HM}_{\tau_m^H}(\cS_{m,r}^H; \fc^{h'}_m))\big] 
 + \frac{i}{2\pi}\int_{M_\phi} F^h_m \wedge H dt - \int_{\Theta_m} H dt.
\end{split}
\end{equation}

The last line in (\ref{eq:hutchingsProof4}) follows from Stokes' theorem and the fact that the curvatures $F_m^c(r)$ converge weakly as currents to $-2\pi i \Theta_m$ as $r \to \infty$. It also uses the fact that the configurations $\fc_m^{h'}$ pull back for every $m$ to configurations that are equal to $\fc_m^h$ up to some negligible error, and the same is true for $\fc_m^{c'}$. Also recall that
$(M_H)^*\omega_{\phi'} = \omega_\phi + dH \wedge dt.$

\textbf{Step 4:} The fourth step applies the results of \S\ref{subsec:maxMin}. Apply Proposition \ref{prop:maxMinReducibles} to deduce the following. For any $m$, there are configurations $\fc_{m,\text{red}} = (B_{m,\text{red}}, 0)$ and $\fc_{m,\text{red}}' = (B_{m,\text{red}}', 0)$ satisfying the following properties:
\begin{itemize}
\item $\fc_{m,\text{red}}$ and $\fc_{m,\text{red}}'$ solve the $\cS_{m,-2\pi\rho_m}$-- and $\cS_{m,-2\pi\rho_m}^H$--Seiberg--Witten equations, respectively.
\item The actions of the configurations recover the respective spectral invariants at $r = -2\pi\rho_m$: 
$$c^{\HM}_{\sigma_m}(\cS_{m,-2\pi\rho_m}; \fc_m^h) = (-2\pi\rho_m)^{-1}\fa_{-2\pi\rho_m, \fe_{\mu_m}}(\fc_{m,\text{red}}, \fc_m^h),$$  $$c^{\HM}_{\tau_m^H}(\cS_{m,-2\pi\rho_m}^H; \fc_m^{h'}) = (-2\pi\rho_m)^{-1}\fa_{-2\pi\rho_m, \fe_{\mu_m'}}(\fc_{m, \text{red}}', \fc_m^{h'}).$$
\end{itemize}

Lemma \ref{lem:maxMinLimit} implies that there is a constant $\kappa(\epsilon) > 0$ depending only on the metric $g$, the Hamiltonian $H$, and a choice of constant $\epsilon \in (0,1)$ such that the following two bounds hold:
\begin{equation}
\label{eq:hutchingsProof6}
\begin{split}
|\lim_{r \to \infty} c^{\HM}_{\sigma_m}(\cS_{m,r}; \fc^h_m) + (2\pi\rho_m)^{-1}\fa_{-2\pi\rho_m, \fe_{\mu_m}}(\fc_{m,\text{red}}, \fc_m^h) - \pi\rho_m^{-1}\SF(\sigma_m, \fc_m^h)| &\leq \kappa(\epsilon)d_m^{4/5 + \epsilon} \\
|\lim_{r \to \infty} c^{\HM}_{\tau_m^H}(\cS_{m,r}^H; \fc^{h'}_m) + (2\pi\rho_m)^{-1}\fa_{-2\pi\rho_m, \fe_{\mu_m'}}(\fc_{m,\text{red}}', \fc_m^{h'}) - \pi\rho_m^{-1}\SF(\tau_m^H, \fc_m^{h'})| &\leq \kappa(\epsilon)d_m^{4/5 + \epsilon}.
\end{split}
\end{equation}

Note that the application of Lemma \ref{lem:maxMinLimit} requires the fact that the curvatures $F_m^h$ of the chosen base connections have $C^3$ norms bounded by multiples of $d_m$ depending only on geometric constants. 
Combining (\ref{eq:hutchingsProof4}) with (\ref{eq:hutchingsProof6}) then yields the estimate
\begin{equation}
\label{eq:hutchingsProof7}
\begin{split}
c_{\tau_m^H}(\bfS_m^H) - c_{\sigma_m}(\bfS_m) &= \frac{1}{2}\pi^{-2}\rho_m^{-1}(\fa_{-2\pi\rho_m, \fe_{\mu_m'}}(\fc_{m,\text{red}}', \fc_m^{h'}) - \fa_{-2\pi\rho_m, \fe_{\mu_m}}(\fc_{m,\text{red}}, \fc_m^{h})) \\
&\; + \rho_m^{-1}(\SF(\tau_m^H, \fc_m^{h'}) - \SF(\sigma_m, \fc_m^{h})) 
 + \frac{i}{2\pi} \int_{M_\phi} F_m^h \wedge H dt - \int_{\Theta_m} H dt + O_\epsilon(d_m^{4/5 + \epsilon}).
\end{split}
\end{equation}

\textbf{Step 5:} The fifth step is to compute the difference of the Chern--Simons--Dirac functionals at $\fc_{m,\text{red}}$ and $\fc_{m,\text{red}}'$ for every $m$. We will then plug this into (\ref{eq:hutchingsProof7}) to conclude the conjecture. 

We write down the Seiberg--Witten equations for the reducibles. For brevity, write the curvatures of $B_{m,\text{red}}$ and $B_{m,\text{red}}'$ as $F_{m,\text{red}}$ and $F_{m,\text{red}}'$. Then we find
\begin{equation}
\label{eq:hutchingsProof5}
F_{m,\text{red}} = i\pi\rho_m\omega_\phi + i \varpi_V + id\mu_m,\qquad F_{m,\text{red}}' = i\pi\rho_m\omega_{\phi'} + i\varpi_{V'} + id\mu_m'.
\end{equation}

Using the Seiberg--Witten equations (\ref{eq:hutchingsProof5}) we find
$$(M_H)^*(B_{m,\text{red}}') - B_{m,\text{red}} = b_m + i\pi\rho_m H dt + i\xi_V + i( (M_H)^*(\mu_m') - \mu_m)$$
where $b_m$ is a closed one-form and $\xi_V$ is a fixed primitive for the exact two-form $(M_H)^*(\varpi_{V'}) - \varpi_V$. 

A key observation is that the functional $\fa_{-2\pi\rho_m, \fg}$ for any $\fg \in \cP$ is gauge-invariant. Therefore, we can apply gauge transformations to $B_{m,\text{red}}$ and assume without loss of generality that $b_m$ is harmonic and has $C^0$ norm bounded by a metric-dependent constant $\kappa$ that is independent of $m$. 

We now proceed with the computation. Using the equations in (\ref{eq:hutchingsProof5}) we can write
\begin{align*}
\fa_{-2\pi\rho_m, \fe_{\mu_m}}(\fc_{m,\text{red}}, \fc_m^{h}) &= \frac{1}{2}(\fcs(B_{m,\text{red}}, B_m^{h}) + 2\pi\rho_m E_{\phi}(B_{m,\text{red}}, B_m^{h})) + \fe_{\mu_m}(B_{m,\text{red}}, B_m^{h}) \\
&= \int_{M_{\phi}} (B_{m,\text{red}} - B_m^{h}) \wedge (i\pi\rho_m \omega_{\phi} - \frac{1}{2}F_{m,\text{red}} - \frac{1}{2}F_m^{h} + i\varpi_{V} + id\mu_m) \\
&= \int_{M_{\phi}} (B_{m,\text{red}} - B_m^{h}) \wedge (\frac{i\pi\rho_m}{2}\omega_{\phi} + \frac{i}{2}\varpi_{V} - \frac{1}{2}F_m^{h})
\end{align*}
and similarly
$$\fa_{-2\pi\rho_m, \fe_{\mu_m'}}(\fc_{m,\text{red}}', \fc_m^{h'}) = \int_{M_{\phi'}} (B_{m,\text{red}}' - B_m^{h'}) \wedge (\frac{i\pi\rho_m}{2}\omega_{\phi'} + \frac{i}{2}\varpi_{V'} - \frac{1}{2}F_m^{h'}).$$

We can then pull back the latter of the two identities by $M_H$ and compute the difference of the actions:
\begin{align}
&\quad \fa_{-2\pi\rho_m, \fe_{\mu_m'}}(\fc_{m,\text{red}}', \fc_m^{h'}) - \fa_{-2\pi\rho_m, \fe_{\mu_m}}(\fc_{m,\text{red}}, \fc_m^{h}) 
\label{eq:hutchingsProof9}
\\
& = \int_{M_\phi} ((M_H)^*B_{m,\text{red}}' - B_m^h) \wedge (\frac{i\pi\rho_m}{2}(\omega_{\phi} + dH \wedge dt) + \frac{i}{2}(M_H)^*\varpi_{V'} - \frac{1}{2}F_m^h) 
\nonumber
\\
&\qquad - \int_{M_\phi} (B_{m,\text{red}} - B_m^h) \wedge (\frac{i\pi\rho_m}{2}\omega_\phi + \frac{i}{2}\varpi_V - \frac{1}{2}F_m^h) 
\nonumber
\\
& = \int_{M_\phi} ((M_H)^*B_{m,\text{red}}'-B_{m,\text{red}}) \wedge (\frac{i\pi\rho_m}{2}(\omega_{\phi} + dH \wedge dt)+ \frac{i}{2}(M_H)^*\varpi_{V'} - \frac{1}{2}F_m^h) 
\nonumber
\\
&\qquad + \int_{M_\phi} ( B_{m,\text{red}} - B_m^h) \wedge (\frac{i\pi\rho_m}{2}dH \wedge dt + \frac{i}{2}d\xi_V) 
\nonumber
\\
& = \int_{M_\phi} (b_m + i\pi\rho_m H dt + i\xi_V + i((M_H)^*(\mu_m') - \mu_m)) 
\wedge (\frac{i\pi\rho_m}{2}(\omega_{\phi} + dH \wedge dt) + \frac{i}{2}(M_H)^*\varpi_{V'} - \frac{1}{2}F_m^h) 
\nonumber
\\
&\qquad + \int_{M_\phi} ( B_{m,\text{red}} - B_m^h) \wedge (\frac{i\pi\rho_m}{2}dH \wedge dt + \frac{i}{2}d\xi_V) 
\nonumber
\\
& = \int_{M_\phi}(i\pi\rho_m H dt) \wedge (\frac{i\pi\rho_m}{2}(\omega_{\phi} + dH \wedge dt)- \frac{1}{2}F_m^h) 
+ \int_{M_\phi}(F_{m,\text{red}} - F_m^h)\wedge (\frac{i\pi\rho_m}{2}H \wedge dt)
+ O(|\rho_m|) 
\nonumber
\\
& = -\frac{1}{2}(\pi\rho_m)^2\int_{M_\phi} H \omega_\phi \wedge dt - i\pi\rho_m \int_{M_\phi} F_m^h \wedge (H dt) 
+ \frac{i\pi\rho_m}{2}\int_{M_\phi} F_{m,\text{red}} \wedge (H dt)+ O(|\rho_m|) 
\nonumber
\\
& = -(\pi\rho_m)^2\int_{M_\phi} H \omega_\phi \wedge dt - i\pi\rho_m \int_{M_\phi} F_m^h \wedge (H dt) + O(|\rho_m|).
\nonumber
\end{align}

The errors of $O(|\rho_m|)$ arise from the fact that the norms of $b_m$, $\mu_m$, $\mu_{m'}$, $\varpi_V$, $\varpi_{V'}$, and $\xi_V$ are bounded by geometric constants and the fact that the norms of $F_{m,\text{red}}$, $F_{m,\text{red}}'$, and $F^h_m$ are bounded by multiples of $|\rho_m|$ by geometric constants. 

\textbf{Step 6:} Plug (\ref{eq:hutchingsProof9}) into (\ref{eq:hutchingsProof7}) and rearrange appropriately to deduce the equation
\begin{equation*}
\begin{split}
&(c_{\tau_m^H}(\phi', \Theta_m^H) - c_{\sigma_m}(\phi, \Theta_m) + \int_{\Theta_m} H dt) - \rho_m^{-1}(\SF(\tau_m^H, \fc_m^{h'} ) - \SF(\sigma_m, \fc_m^{h})) \\
&\qquad = -\frac{1}{2}\rho_m\int_{M_\phi} H \omega_\phi \wedge dt + O_\epsilon(d_m^{4/5 + \epsilon}).
\end{split}
\end{equation*}

Note in particular the integrals of $F^h_m \wedge (H dt)$ appearing in \eqref{eq:hutchingsProof7} and \eqref{eq:hutchingsProof9} cancel out. If we multiply through by $-2\rho_m^{-1}$ and expand the previous equation using the identity
$$\rho_m = -2(d_m + 1 - G),$$
we get
\begin{equation}
\label{eq:hutchingsProof10}
\begin{split}
&\frac{c_{\tau_m^H}(\phi', \Theta_m^H) - c_{\sigma_m}(\phi, \Theta_m) + \int_{\Theta_m} H dt}{d_m + 1 - G} + \frac{\SF(\tau_m^H, \fc_m^{h'}) - \SF(\sigma_m, \fc_m^{h})}{2(d_m + 1 - G)^2} \\
&\qquad = \int_{M_\phi} H \omega_\phi \wedge dt + O_\epsilon(d_m^{-1/5 + \epsilon}).
\end{split}
\end{equation}

\textbf{Step 7:} The seventh step rewrites the Seiberg--Witten grading term in (\ref{eq:hutchingsProof10}) in terms of PFH gradings. By Proposition \ref{prop:twistedIsoPreservesGradings} and the bound $\|F^h_m\|_{C^3} \lesssim d_m$, 
\begin{equation}
\label{eq:hutchingsProof11}
|( I(\tau_m) - I(\sigma_m)) + (\SF(\tau_m^H, \fc_m^{h'}) - \SF(\sigma_m, \fc_m^{h}))| = O(d_m^{3/2}).
\end{equation}

Plug (\ref{eq:hutchingsProof11}) into (\ref{eq:hutchingsProof10}) to conclude Theorem \ref{thm:hutchingsConjectureFullintro}.

\section{Non-vanishing}
\label{sec:nonvanish}

The aim of this section is to prove Theorem~\ref{thm:nonvan}. 

Let $R$ be a fixed commutative ring. 
Although Theorem~\ref{thm:nonvan} requires that the coefficient ring has a finite global dimension, we do \emph{not} assume this property on $R$ for now. 
Since we need to consider Floer homology groups with different coefficient rings in our proof, we will include the coefficient ring in the notation of Floer homology throughout this section.

We first prove a nonvanishing result for the \emph{untwisted} Seiberg--Witten Floer homology. 

\begin{thm}
	\label{thm_HMbar_nonvanish}
Suppose $Y$ is a closed oriented 3-manifold, and let $\mathfrak{s}$ be a non-torsion spin-c structure on $Y$. Let 
$\HMbar^*(Y, \mathfrak{s}, c_b;R)$ be the bar-version of monopole Floer cohomology of $(Y,\mathfrak{s})$ with the balanced perturbation in $R$ coefficients. Then 
	$\HMbar^*(Y, \mathfrak{s}, c_b;R)\neq 0.$
\end{thm}

\begin{proof}

Since $\mathfrak{s}$ is non-torsion, we have $b_1(Y)\ge 1$. Let $\mathbb{T}$ be the Picard torus $H^1(Y; \bR) / H^1(Y; \bZ)$. 

We first show that $\HMbar^*(Y, \mathfrak{s}, c_b;\mathbb{Z})$ is finitely generated as a $\mathbb{Z}$--module. By \cite[Theorem $35.1.6$]{monopolesBook}, there is a spectral sequence with $E^2$ page equal to $H_*(\mathbb{T};\Gamma_{\xi_1})$ that abuts to  $\HMbar^*(Y, \mathfrak{s}, c_b;\mathbb{Z})$. Here, $\Gamma_{\xi_1}$ is the local system on $\mathbb{T}$ defined as follows.   Let $\xi_1 \in H^1(\mathbb{T};\mathbb{Z}) = \text{Hom}(H^1(Y;\mathbb{Z}),\mathbb{Z})$ be given by
\begin{equation}
\label{eqn_def_xi_1}
	a \mapsto \frac{1}{2} a \cup c_1(\mathfrak{s}) [-Y].
\end{equation} 
Define $\Gamma_{\xi_1}$ to be the local system on $\mathbb{T}$ with fiber $\mathbb{Z}[T,T^{-1}]$ where the holonomy around a loop $\gamma$ is multiplication by $T^k$ for $k = \xi_1[\gamma]$. The extra negative sign in \eqref{eqn_def_xi_1} comes from the fact that we are working with monopole Floer cohomology instead of homology.

We compute the homology group $H_*(\mathbb{T};\Gamma_{\xi_1})$ following \cite[Page $688$]{monopolesBook}.
Let $2d$ be the divisibility of $c_1(\mathfrak{s})$. Choose a basis $\gamma_i$ for $\mathbb{T}$ with $\xi_1(\gamma_1) = d$ and $\xi_i(\gamma_i) = 0$ for all other $i$. Let $\mathbb{T}'$ be a torus with dimension $(\dim \mathbb{T}-1)$. Then $\mathbb{T} \cong \mathbb{T}_1 \times \mathbb{T}',$ with the first factor generated by $\gamma_1$ and the others generated by $\gamma_i$, and the homology with coefficients in $\Gamma_{\xi_1}$ splits.  On $\mathbb{T}_1$, there are two trajectories from the maximum to the minimum, and so we can arrange so that the differential on $\mathbb{T}_1$ is
\[ \mathbb{Z}[T,T^{-1}] \mapsto \mathbb{Z}[T, T^{-1}], \quad p \to (T^d-1)p,\]
and hence we have
\begin{equation}
\label{eqn:localmorse}
H_*(\mathbb{T};\Gamma_{\xi_1}) = (\mathbb{Z}[T,T^{-1}]/(T^d - 1)) \otimes_\bZ H_*(\mathbb{T}';\mathbb{Z}).
\end{equation}
By \eqref{eqn:localmorse}, the group $H_*(\mathbb{T};\Gamma_{\xi_1})$ is a finitely generated $\bZ$--module, therefore the group $\HMbar^*(Y, \mathfrak{s}, c_b;\mathbb{Z})$ is also finitely generated.

By the universal coefficient theorem, in order to show the non-vanishing of $\HMbar^*(Y, \mathfrak{s}, c_b;R)$, we only need to show that 
$\rank_{\mathbb{Z}}\HMbar^*(Y, \mathfrak{s}, c_b;\mathbb{Z})\neq 0$.

By \cite[Corollary $35.1.7$]{monopolesBook}, there is a filtration on $\HMbar^*(Y, \mathfrak{s}, c_b;\mathbb{Q})$ whose associated graded spaces are given by
\[ \frac{ \text{ker}(\beta_s) }{\text{im}(\beta_{s+3})} \otimes \mathbb{Q},\]
where
$ \beta_s: H_s(\mathbb{T};\Gamma_{\xi_1}) \to H_{s-3}(\mathbb{T}; \Gamma_{\xi_1})$
is the cap product with a class $\xi_3\in H^3(\mathbb{T};\bZ)$.

We now use the trick from \cite[Corollary $35.1.3$]{monopolesBook} and \eqref{eqn:localmorse} to prove the non-vanishing of $\HMbar^*(Y, \mathfrak{s}, c_b;\mathbb{Q})$.  Let $\zeta = e^{2 \pi i /6}$ and  consider
\[ \sum_s \rank_\bZ\left( \frac{\text{ker}(\beta_s)}{\text{im}(\beta_{s+3})}\right)\zeta^s.\]
 Because the Euler characteristic of a bounded complex is equal to the Euler characteristic of its homology, we have that this sum is equal to
\[ d \sum_s \rank_\bZ H_s(\mathbb{T}';\bZ)\cdot  \zeta^s,\]
 This latter sum is just the product of $d$ with the Poincare polynomial of $\mathbb{T}'$, evaluated at $\zeta$.  Since the Poincare polynomial of $\mathbb{T}'$ is $(1 - X)^{b_1(Y) - 1}$, we conclude that 
 $\HMbar^*(Y, \mathfrak{s}, c_b;\mathbb{Q})$ does not vanish, therefore
 $$
 \rank_{\mathbb{Z}}\HMbar^*(Y, \mathfrak{s}, c_b;\mathbb{Z}) = \dim_{\mathbb{Q}}\HMbar^*(Y, \mathfrak{s}, c_b;\mathbb{Q}) >0,
 $$
 and the desired result is proved.
\end{proof}

\begin{thm} \label{thm:nonVanishing}
Suppose $R$ is a ring with finite global dimension.
Let $Y$ be a closed oriented 3-manifold, let $\mathfrak{s}$ be a non-torsion spin-c structure on $Y$. Let 
$\TWHMbar^*(Y, \mathfrak{s}, c_b;R)$ be the bar-version of twisted monopole Floer homology of $(Y,\mathfrak{s})$ with the balanced perturbation in $R$ coefficients. Then 
	$\TWHMbar^*(Y, \mathfrak{s}, c_b;R)\neq 0.$
\end{thm}

\begin{proof}
Recall that the twisted cohomology
$\TWHMbar^*(Y, \mathfrak{s}, c_b)$
can be interpreted as the bar version of Seiberg--Witten--Floer cohomology with coefficients in a local system $\mathbb{B}^\circ$ of $\mathcal{R} = R[H^1(Y; \bZ)]$--modules. The fibers of the local system are themselves isomorphic to $\mathcal{R}$. The chain complex $\TWCMbar^*(Y, \mathfrak{s}, c_b)$ is relatively graded over $\bZ/2d$ where $2d$ is the divisibility of $c_1(\fs)$. When we ignore the $\mathcal{R}$--module structure and view $\TWCMbar^*(Y, \mathfrak{s}, c_b)$ as a $\bZ$--module, the relative $\bZ/2d$--grading can be lifted to a relative $\bZ$--grading. In this proof, we will keep the $\mathcal{R}$--module structure and view $\TWCMbar^*(Y, \fs, c_b)$ as a   chain complex relatively graded over $\bZ/2d$.

Let $\mathcal{I}$ be the ideal in $\mathcal{R} = R[H^1(Y; \bZ)]$ generated by the elements $\{\gamma_i - 1\}_{i = 1}^{b_1(Y)}$, where $\{\gamma_i\}$ is a basis for $H^1(Y;\bZ)$. Then by definition, we have
$$\CMbar^*(Y, \fs, c_b) \simeq \TWCMbar^*(Y, \fs, c_b) \otimes_\mathcal{R} {\mathcal{R}/\mathcal{I}}.$$

Notice that $\mathcal{R}$ is isomorphic to the Laurant polynomial ring over $R$ with $b_1(Y)$ variables. Since $R$ has finite global dimension, $\mathcal{R}$ also has finite global dimension, so all $\mathcal{R}$--modules have finite projective resolutions. The chain complex $\TWCMbar^*(Y, \fs, c_b)$ is a free $\mathcal{R}$--module, therefore it is flat. By Theorem \ref{thm_HMbar_nonvanish}, the chain complex $\CMbar^*(Y, \fs, c_b)$ is not acyclic. Therefore by Lemma \ref{lem_homological_algebra_Kunneth} below, $\TWCMbar^*(Y, \fs, c_b)$ cannot be acyclic.
\end{proof}

\begin{lem}
\label{lem_homological_algebra_Kunneth}
	Suppose $\mathcal{R}$ is a commutative ring, and suppose $C$ is a flat, acyclic, chain complex of $\mathcal{R}$--modules. The chain complex $C$ may be unbounded or cyclically graded. If $M$ is an $\mathcal{R}$--module that has a finite projective resolution, then $C\otimes_R M$ is acyclic.  
\end{lem}

\begin{proof}
	We first prove that if 
	$0\to C_1\to C_2\to \cdots \to C_n \to 0	$
	is an exact sequence of chain complexes and $C_1,\dots,C_{n-1}$ are all acyclic, then $C_n$ is also acyclic. Here, the complexes $C_i$ can be unbounded or cyclically graded. This is proved by induction on $n$.  The cases for $n=1,2$ are obvious, and the case for $n=3$ follows from the long exact sequence of homology. Now suppose $n>3$ and the statement holds for $n-1$. Let $f$ be the chain map from $C_1$ to $C_2$.   Applying the induction hypothesis to 
	$0 \to C_1 \to C_2 \to {\rm coker}(f) \to 0$
	shows that ${\rm coker}(f)$ is acyclic, and we have the following exact sequence of  chain complexes:
	$$
	0 \to {\rm coker}(f) \to C_3 \to \cdots \to C_n \to 0.
	$$
	The induction hypothesis then implies that $C_n$ is acyclic.
	
	Now let $C, M$ be as in the statement of the lemma. Let 
	$0\to P_n\to P_{n-1} \to\cdots \to P_1 \to M \to 0$
	be a finite projective resolution of $M$. Consider the following sequence of chain complexes
	\begin{equation}\label{eqn_chain_complex_exact_sequence_2}
	0 \to C\otimes P_n \to C\otimes P_{n-1} \to \cdots C\otimes P_1\to C\otimes M \to 0.
	\end{equation}
	Since $C$ is flat, the sequence \eqref{eqn_chain_complex_exact_sequence_2} is an exact sequence of chain complexes.  Since $C$ is acyclic and the $P_i$'s  are projective, the chain complex $C\otimes P_i$ is acyclic for every $i$. Therefore $C\otimes M$ is acyclic.
\end{proof}

The proof of Theorem~\ref{thm:nonvan}
is now an easy corollary of the above results:

\begin{proof}[Proof of Theorem~\ref{thm:nonvan}]
As the statement of the theorem only concerns nonvanishing, the reference cycle is irrelevant and we can drop it from the notation.  By the Lee--Taubes isomorphism in \S\ref{sec:twistedIso}, we have
\[ \TWPFH(\phi,\Theta;R) \simeq  \TWHM^*(M_\phi, \Gamma, m_-;R).\]
By \cite[Theorem $31.5.1$]{monopolesBook}, the right hand side of the above isomorphism is itself isomorphic to the hat-version of the group $\TWHMhat^*(M_\phi, \Gamma, m_b;R).$  Now by Corollary~\ref{cor:barEqualsHat}, we have
$\TWHMbar^*(M_\phi, \Gamma, m_b;R) \cong \TWHMhat^*(M_\phi, \Gamma, m_b;R)$
in sufficiently high degree.  The theorem then follows from Theorem~\ref{thm:nonVanishing}. 
\end{proof}

\section{The closing lemma}

We now prove the Closing Lemma as stated by Theorem \ref{thm:closingLemma}.

\subsection{Properties of PFH spectral invariants} \label{subsec:pfhSpectralInvariantProperties} 

The following properties of the PFH spectral invariants will be useful.  The proofs overlap significantly with the proofs in \cite[\S 3]{simplicity20}, so we will be brief.

\begin{prop} \label{prop:hoferContinuity1}
Fix $\phi_0 \in \Diff(\Sigma, \omega)$ and fix Hamiltonians $H_\pm \in C^{\infty}(\mathbb{R}/\mathbb{Z} \times \Sigma)$. Fix a PFH parameter set $\bfS = (\phi_0, \Theta_{\text{ref}}, J)$ where $\Gamma$ has degree $d$ and corresponding PFH parameter sets $\bfS_\pm = (\phi_\pm = \phi_0 \circ \phi^1_{H_\pm}, \Theta_{\text{ref}}^{H_{\pm}}, J_\pm)$. Then for any class $\sigma \in \TWPFH_*(\bfS)$, we have the inequality 
\begin{align*}
d\int_{\bR/\bZ} \min (H_+ - H_-)(t, -) dt \leq c_{\sigma}(\bfS_+) - c_{\sigma}(\bfS_-) + \int_{\Theta_{\text{ref}}}(H_+ - H_-) dt \leq d\int_{\bR/\bZ} \max (H_+ - H_-)(t, -)dt,
\end{align*}
where $\sigma$ is regarded as a class in $\TWPFH_*(\bfS_\pm)$ using the natural identifications from \S \ref{subsec:pfhCobordismMaps}. 
\end{prop}

Before proving Proposition \ref{prop:hoferContinuity1}, we write down two corollaries. The first corollary is that the PFH spectral invariants are independent of $J$. This follows from applying Proposition \ref{prop:hoferContinuity1} to the case where $H_+ = H_- = 0$. 

\begin{cor} \label{cor:pfhSpectralJIndependence}
The PFH spectral invariants are independent of the almost-complex structure:
$$c_{\sigma}(\phi, \Theta_{\text{ref}}, J) = c_{\sigma}(\phi, \Theta_{\text{ref}}, J')$$
for any pair $J$ and $J'$ in $\cJ^\circ(dt,\omega_\phi)$ and any class $\sigma \in \TWPFH_*(\phi, \Gamma, \Theta_{\text{ref}}).$
\end{cor}

The second corollary lists some properties of the PFH spectral invariants, analogous to their namesakes in \cite[Proposition $3.2$]{largeScale}. 

\begin{cor}\label{cor:shiftAndMonotonicity}
Let $\phi_0$, $H_\pm$, $\bfS = (\phi_0, \Theta_{\text{ref}}, J)$, $d$, $\bfS_\pm$, $\sigma$ be as in Proposition \ref{prop:hoferContinuity1}. Then the following two properties hold:
\begin{enumerate}
\item (Monotonicity) If $H_+ \geq H_-$, then 
$$c_\sigma(\bfS_+) + \int_{\Theta_{\text{ref}}} H_+ dt \geq c_\sigma(\bfS_-) + \int_{\Theta_{\text{ref}}} H_- dt.$$ 
\item (Shift) If $H_+ - H_- = h(t)$ is a function of $t \in \bR/\bZ$ only, then 
$$c_\sigma(\bfS_+) - c_\sigma(\bfS_-) = d\int_{\bR/\bZ} h(t) dt - \int_{\Theta_{\text{ref}}} h(t)dt.$$
\end{enumerate}
\end{cor}

The proof of Proposition \ref{prop:hoferContinuity1} relies on the so-called \emph{pseudoholomorphic curve axiom}.
It is an interesting open problem to define PFH cobordism maps on the chain level by counting pseudoholomorphic currents. This remains open, but Chen \cite{Chen21} has shown that PFH cobordism maps as currently defined still detect the existence of certain pseudoholomorphic curves, which in itself has interesting applications. We review the statement of Chen's  pseudoholomorphic curve axiom for PFH cobordism maps here.

Fix a PFH parameter set
$\bfS = (\phi_0, \Theta_{\text{ref}}, J),$
Hamiltonians $H_\pm \in C^{\infty}_c( (0,1) \times \Sigma )$, and PFH parameter sets
$\bfS_\pm = (\phi_{\pm} = \phi \circ \phi^1_{H_{\pm}}, \Theta_{\text{ref}}^{H_{\pm}}, J_\pm).$
It is essential for applications that the almost-complex structures $J_\pm$ can be arbitrarily chosen, namely they do not need to be related to $J$ by the maps $M_{H_{\pm}}$. Fix a smooth homotopy $K: [0,1] \to C^\infty(\bR/\bZ \times \Sigma)$ from $H_-$ to $H_+$ which is constant near $0$ and $1$. 

Fix $r \gg 1$ and an SW continuation parameter set
$\cS_s = (K, J_s, r, \fg_s)$
between two associated SW parameter sets
$\cS_\pm = (\phi_\pm, J_\pm, \Gamma^{H_{\pm}}, r, \fg_\pm).$

The choice of $K$ defines a symplectic $4$-manifold given by the pair
\begin{equation*}
X = (\bR \times M_{\phi_0}, ds \wedge dt + \frac{1}{2}(\omega_{\phi_0} + d(K(s)dt)).
\end{equation*}

The path of almost-complex structures $\{J_s\}_{s \in \bR}$ in $\cS_s$ defines an almost-complex structure $\bar{J}$ on $\bR \times M_{\phi_0}$ satisfying the following conditions: 
\begin{enumerate}
\item On any slice $\{s\} \times M_{\phi_0}$ for $s \in \bR$, $\bar{J}$ preserves the vertical tangent bundle $V$ and restricts to the almost-complex structure $J_s$.
\item $\bar{J}$ sends $\partial_s$ to $R + X_{K(s)}$, where $R$ is the Reeb vector field on $M_{\phi_0}$ associated to $(dt, \omega_\phi)$.
\end{enumerate}

Pick a finite sum
$\wt\sigma_+ = \sum_i (\Theta_i^+, W_i^+)$
of generators of $\TWPFC_*(\bfS_+)$.
Suppose the cochain
$\wt\sigma_- = T^{\PFH}(\bfS_-, \bfS_+; \cS_s)(\wt\sigma_+) \in \TWPFC_*(\bfS_-)$
expands as a finite sum
$\wt\sigma_- = \sum_j (\Theta_{j}^-, W_{j}^-)$
of generators of $\TWPFC_*(\bfS_-)$.
The ``pseudoholomorphic curve axiom'' proved by Chen \cite{Chen21} states the existence of an ``ECH index zero pseudoholomorphic building'' in the symplectic manifold $X$ 
between the two cochains $\wt\sigma_+$ and $\wt\sigma_-$. 
We will say more shortly in the proof of Proposition \ref{prop:hoferContinuity1}.
Note that a consequence of this axiom is that the cobordism map preserves the $\bZ$--gradings on the two complexes.

We now proceed to prove Proposition \ref{prop:hoferContinuity1}.

\begin{proof}[Proof of Proposition \ref{prop:hoferContinuity1}]
Our proof is essentially identical to the proof of \cite[Theorem $3.6(2)$]{simplicity20}.  First, note that by approximating an arbitrary $H(t,x)$ with suitable reparametrizations $\rho'(t)H(\rho(t),x)$, it suffices to prove the proposition with $H\in C^\infty_c((0,1)\times \Sigma)$. Since the spectral invariant for degenerate $\phi$ is defined by a continuous extension from the non-degenerate case, we may assume without loss of generality that $\phi$ is non-degenerate.

\textbf{Step 1:} The first step performs some required additional setup. Fix any cochain $\wt\sigma_+ \in \TWPFC_*(\bfS_+)$ representing $\sigma$ and write it as a finite sum
$\wt\sigma_+ = \sum_i (\Theta_i^+, W_i^+)$
of generators of $\TWPFC_*(\bfS_+)$. We write $\wt\sigma_- = T(\bfS_-, \bfS_+; \cS_s)(\wt\sigma_+)$ as a sum of generators 
$\wt\sigma_- = \sum_j (\Theta_j^-, W_j^-).$
of $\TWPFC_*(\bfS_-)$. Here $\cS_s = (K, J_s, r, \fg_s)$ denotes an auxiliary SW parameter set, where $K$ is a homotopy in $C^\infty(\bR/\bZ \times \Sigma)$ from $H_-$ to $H_+$. Fix a single generator $(\Theta_-, W_-)$ which has maximal action among all of the generators $(\Theta_j^-, W_j^-)$.

Recall that $K$ defines a symplectic $4$-manifold
$(\bR \times M_{\phi_0}, ds \wedge dt + \frac{1}{2}(\omega_{\phi_0} + d(K(s)dt))$
and the path $\{J_s\}$ in $\cS_s$ defines an almost-complex structure $\bar{J}$ on $\bR \times M_{\phi_0}$ satisfying the following properties:
\begin{enumerate}
\item On any slice $\{s\} \times M_{\phi_0}$ for $s \in \bR$, $\bar{J}$ preserves the vertical tangent bundle $V$ and restricts to the almost-complex structure $J_s$. 
\item $\bar{J}$ sends $\partial_s$ to $R + X_{K(s)}$, where $R$ is the Reeb vector field on $M_{\phi_0}$ associated to $(dt, \omega_{\phi_0})$.
\end{enumerate}

\textbf{Step 2:} The second step discusses the pseudoholomorphic curve axiom. As noted in \cite[Remark $3.7$]{simplicity20}, we can assume without loss of generality that there is a $\bar{J}$-holomorphic curve $C \subset \bR \times M_{\phi_0}$ from the generator $(\Theta_-, W_-)$ of maximal action chosen earlier in the sum for $\wt\sigma_-$ to a generator $(\Theta_+, W_+)$ appearing with nonzero coefficient in the cochain $\wt\sigma_+$. Moreover we have the identity
$[C] = W_+ - W_-.$

\textbf{Step 3:} The third step computes the difference in the actions of $(\Theta_\pm, W_\pm)$ and uses this to prove the proposition. Given what was said in Step $2$, we make the following computation:
\begin{align}
\label{eq:hoferContinuity1}
\bfA(\Theta_+, W_+) - \bfA(\Theta_-, W_-) &= \int_{W_+} \omega_{\phi_0} + dH_+ \wedge dt - \int_{W_-} \omega_\phi + dH_- \wedge dt \\
&= \int_{W_+ - W_-} (\omega_{\phi_0} + dH_- \wedge dt) + \int_{W_+} d((H_+ - H_-)dt) 
\nonumber \\
&= \int_C (\omega_\phi + dH_- \wedge dt) + \int_{\Theta_+} (H_+ - H_-) dt - \int_{\Theta_{\text{ref}}} (H_+ - H_-) dt 
\nonumber 
\end{align}

Next, we observe by Stokes' theorem that we have the identity
\begin{equation}
\label{eq:hoferContinuity2}
\int_{\Theta_+} (H_+ - H_-) dt = \int_C d( (K(s) - H_-)dt)  \int_C dK(s) \wedge dt - \int_C dH_- \wedge dt + \int_C K'(s)ds \wedge dt.
\end{equation}

Now observe by definition of the almost-complex structure $\bar{J}$ that the two-form
$\omega_\phi + dK(s) \wedge dt$
is non-negative on any $\bar{J}$-invariant two-plane in the tangent bundle of $\bR \times M_{\phi}$, which shows
$\int_C \omega_\phi + dK(s) \wedge dt \geq 0.$

Combining (\ref{eq:hoferContinuity1}) and (\ref{eq:hoferContinuity2}) with the above shows
\begin{align}
\label{eq:hoferContinuity3}
\bfA(\Theta_+, W_+) - \bfA(\Theta_-, W_-) &= \int_C (\omega_\phi + dK(s) \wedge dt) + \int_C K'(s)ds \wedge dt - \int_{\Theta_{\text{ref}}} (H_+ - H_-) dt \\
&\geq \int_C K'(s)ds \wedge dt - \int_{\Theta_{\text{ref}}} (H_+ - H_-)dt.
\nonumber
\end{align}

It is now convenient to take $K$ to be the homotopy 
$K(s) = (1 - \tau(s))H_- + \tau(s)H_+$
where $\tau: [0,1] \to [0,1]$ is a smooth, increasing function which is equal to $0$ and $1$ near $s = 0$ and $s = 1$, respectively. 
Then the triangle inequality, $\tau'(s) \geq 0$, and the fact that $ds \wedge dt \geq 0$ on the tangent planes of $C$ implies the following:
$$\int_C \tau'(s)\min (H_+ - H_-)(t, -) ds \wedge dt \leq \int_C K'(s)ds \wedge dt \leq \int_C \tau'(s)\max (H_+ - H_-)(t, -) ds \wedge dt.$$

We simplify the above by projecting $C$ onto the $(s, t)$-plane, which has degree $d$, and using the fact that $\int_\bR \tau'(s) ds = \tau(1) - \tau(0) = 1$ to conclude
\begin{equation} \label{eq:hoferContinuity4} d\int_{\bR/\bZ} \min (H_+ - H_-)(t, -) dt \leq \int_C K'(s)ds \wedge dt \leq d\int_{\bR/\bZ} \max (H_+ - H_-)(t, -) dt.\end{equation}

We conclude from the lower bound of (\ref{eq:hoferContinuity4}) the following lower bound:
\begin{equation}
\label{eq:hoferContinuity5}
\begin{split}
\bfA(\Theta_+, W_+) - \bfA(\Theta_-, W_-) + \int_{\Theta_{\text{ref}}} (H_+ - H_-)dt \geq d\int_{\bR/\bZ} \min (H_+ - H_-)(t, -) dt.
\end{split}
\end{equation}

Note by definition of the PFH spectral invariants that $c_{\sigma}(\bfS_-) \leq \bfA(\Theta_-, W_-)$. Taking the inequality (\ref{eq:hoferContinuity5}) over all chains $\wt\sigma_+ \in \TWPFC_*(\bfS_+)$ representing the class $\sigma$ shows
\begin{equation}
\label{eq:hoferContinuity6}
\begin{split}
c_{\sigma}(\bfS_+) - c_{\sigma}(\bfS_-) + \int_{\Theta_{\text{ref}}} (H_+ - H_-) dt \geq d\int_{\bR/\bZ} \min (H_+ - H_-)(t, -)dt.
\end{split}
\end{equation}

Applying the same argument in the reverse direction (using the upper bound of (\ref{eq:hoferContinuity4}) instead) shows the upper bound asserted by the proposition. 
\end{proof}

\subsection{Periodic points are generically dense} In this section, we give a proof of Theorem \ref{thm:closingLemma}, the closing lemma. 

Let $\phi_0$ be any area-preserving diffeomorphism of a surface $(\Sigma, \omega)$. We will assume without loss of generality that the area integral
$A_\Sigma = \int_\Sigma \omega$
is equal to $1$. 
Let $U \subset \Sigma$ and  $V$ be a neighborhood of $\phi_0$ in $\text{Diff}(\Sigma, \omega)$, we show that there is a diffeomorphism $\phi' \in V$ with a periodic point in $U$. The proof proceeds in three steps.

\textbf{Step 1:} The first step is to show that we can find a monotone area-preserving diffeomorphism $\phi$ in the neighborhood $V$. 

\begin{lem}
\label{lem:genericallyMonotone} The set of all monotone area-preserving diffeomorphisms is dense in $\Diff(\Sigma, \omega)$ in the $C^\infty$--topology.  
\end{lem}

\begin{proof}
	It is an immediate consequence of the definition that an area preserving map $\phi:\Sigma\to\Sigma$ is monotone if and only if $[\omega_\phi]\in H^2(M_\phi;\mathbb{R})$ is a real multiple of an integral homology class. Since we assume $\langle [\Sigma],[\omega_\phi]\rangle =1$, the map $\phi$ is monotone if and only if $[\omega_\phi]$ is rational. 
	
	Let $K\subset H_1(\Sigma;\bZ)$ be the kernel of 
	$\phi_*-{\rm id}:H_1(\Sigma;\bZ)\to H_1(\Sigma;\bZ).$
	Let $c_1,c_2,\cdots,c_m$ be oriented closed curves on $\Sigma$ whose fundamental classes form a basis of $K$.  For each $c_i$, let $S_i$ be a smooth 2-chain on $\Sigma$ with integer coefficients such that $\partial S_i = \phi(c_i) - c_i$. 
	Let $\bar{S}_i$ be the closed 2-chain in $M_\phi$ 
	obtained by taking the sum of the image of $S_i$ in 
	$\{0\}\times \Sigma \subset M_\phi$
	 with the image of $[0,1]\times c_i\subset [0,1]\times M$ in $M_\phi$ after suitable triangulations. Then the closed chain $\bar{S}_i$ defines an element in $H_2(M_\phi;\bZ)$.  
	It follows from a straightforward calculation using the Mayer-Vietoris sequence that $H_2(M_\phi;\bZ)$ is generated by $[\Sigma]$ and $[\bar{S}_1],\dots,[\bar{S}_m]$. Therefore, $[\omega_\phi]$ is rational if and only if its pairing with every $[\bar{S}_m]$ is rational.  By the definition of $\bar{S}_i$, we have
	$\langle [\bar{S}_i],[\omega_\phi]\rangle = \int_{\bar{S}_i} \omega_\phi.$
	
	Now fix $\phi_0 \in \Diff(\Sigma,\omega)$. Define $K$, $c_1,\dots,c_m$ as above with respect to $\phi_0$, and fix a choice of $S_1,\dots,S_m$. For  a closed 1-form $\lambda$ on $\Sigma$, define $X_\lambda$ to be the unique vector field on $\Sigma$ such that $\omega(X_\lambda,-) = \lambda(-)$.  Since 
	$\lambda$ is closed, the diffeomorphisms generated by $X_\lambda$ is area-preserving. Let $\phi_\lambda^s$ ($s\in[0,1]$) be the 1-parameter family of diffeomorphisms generated by $X_\lambda$. Let $\phi' = \phi_\lambda^1 \circ \phi_0$. 
	Define $A_i'$ to be the chain on $\Sigma$ given by 
    $$
	[0,1]\times c_i  \to \Sigma , \quad
		(s,p) \mapsto \phi_\lambda^s\circ \phi_0(p)
	$$
	up to triangulations, 
	let $S_i'$ be the sum of $A_i'$ and $S_i$. Then we have  
	 $\int_{S_i'} \omega_{\phi'} = \int_{S_i} \omega_{\phi_0} + \int_{c_i} \lambda.$
	 Since $[c_1],\dots,[c_m]$ are linearly independent in $H_1(\Sigma;\bZ)$,  one can take the $C^\infty$--norm of $\lambda$ to be arbitrarily small such that $\int_{S_i'} \omega_{\phi'}$ are all rational.
	\end{proof}

\textbf{Step 2:} 
Let $\phi \in V$ be the monotone area-preserving diffeomorphism produced in the prior step. After a further Hamiltonian perturbation if necessary, we may  assume that $\phi$ is non-degenerate. Fix a sequence of reference cycles $\Theta_m$ in $M_\phi$ such that the classes $[\Theta_m] \in H_1(M_\phi; \bZ)$ are monotone with positive degrees $d_m$ monotonically increasing as $m \to \infty$. representing $\Gamma_m$ for every $m$. 

As a consequence of the non-vanishing result Theorem \ref{thm:nonVanishing}, the isomorphism of twisted PFH and Seiberg--Witten--Floer cohomology by Proposition \ref{prop_iso_twisted_PFH_SW}, along with the isomorphisms of twisted Seiberg--Witten--Floer homology given by \cite[Theorem $31.5.1$]{monopolesBook} and Corollary \ref{cor:barEqualsHat}, we conclude that for all sufficiently large $m$, the group 
$\TWPFH_*(\phi, \Theta_m)$
does not vanish. We can then choose for each $m$ a nonzero homogeneous class $\sigma_m \in \TWPFH_*(\phi, \Theta_m).$

Choose some embedded disk $\bD \subset U$. Choose a Hamiltonian $H \in C^{\infty}_c( (0,1) \times \Sigma ),$ compactly supported in $\bR/\bZ \times \bD$, such that the Calabi invariant
$\text{CAL}(H) = \int_{\bR/\bZ \times \Sigma} H dt \wedge \omega$
is nonzero. Fix a smooth, non-decreasing function $f: [0,1] \to [0,1]$ such that $f(s) = 0$ for $s$ near $0$ and $f(s) = 1$ for $s$ near $1$. 

For every $s \in [0,1]$, write $\phi^s = \phi \circ \phi^1_{f(s)H}$ and write $M_\phi^s$ for its mapping torus for every $s \in [0,1]$. For every $s \in [0,1]$, the Hamiltonian $f(s)H$ defines a diffeomorphism
$M_{f(s)H}: M_\phi \to M_\phi^s.$

For every $m$ and every $s \in [0,1]$, write $\Theta_m^s = M_{f(s)H}(\Theta_m)$. Recall from what was said in \S\ref{subsubsec:swfCobordismMaps} and \S\ref{subsec:pfhCobordismMaps} that for any $s \in [0,1]$ and any $m$, there is a canonical identification
$\TWPFH_*(\phi, \Theta_m) \simeq \TWPFH_*(\phi^s, \Theta_m^s)$
of $\bZ$--graded modules. We therefore find for every $s \in [0,1]$ a sequence of classes $\sigma^s_m$, obtained as the image of the sequence of classes $\sigma_m$. Moreover, for every $s \in [0,1]$ and any $m$, the $\bZ$-grading of $\sigma_m$ and the $\bZ$-grading of $\sigma_m^s$ coincides.

Now the asymptotic result in Theorem \ref{thm:hutchingsConjectureFullintro} implies the following for any $s \in [0,1]$:
\begin{equation}
\label{eq:closingLemmaProof1} \lim_{m \to \infty} \frac{c_{\sigma_m^s}(\phi^s, \Theta_m^s) - c_{\sigma_m}(\phi, \Theta_m) + \int_{\Theta_m} H^s dt}{d_m} = \text{CAL}(H).
\end{equation}
Define for any $m$ and $s \in [0,1]$ the function
$$c_m(s) = c_{\sigma_m^s}(\phi^s, \Theta_m^s) + \int_{\Theta_m} H^s dt.$$ 
The Hofer continuity result in Proposition \ref{prop:hoferContinuity1} implies that $c_m$ is continuous.

For any $s \in [0,1]$ and any $m$, we write
$\mathcal{T}(s, m) \subset \bR$
for the set of actions
$\int_W \omega_{\phi^s} + \int_{\Theta_m} H^s dt$
across all generators $(\Theta, W) \in \TWPFC_*(\phi^s, \Theta_m^s)$. We also write $\mathcal{T}(s) = \cup_m \mathcal{T}(s, m)$. 

Since the cohomology class of $\omega_{\phi^s}$ is by assumption a real multiple of an integral class, it follows that the set $\mathcal{T}(s, m)$ is discrete for any $s \in [0,1]$ and any $m$. From this we conclude that the set $\mathcal{T}(s) \subset \bR$, as a countable union of discrete subset of $\bR$, has Lebesgue measure zero. The spectrality result in Proposition \ref{prop:pfhSpectrality} implies that for every $s \in [0,1]$ and every $m$, we have $c_m(s) \in \mathcal{T}(s, m)$. 

\textbf{Step 3:} We will now show using a proof by contradiction that there exists some $s \in [0,1]$ such that $\phi^s$ has a periodic point in $\bD \subset U$. Suppose for the sake of contradiction that $\phi^s$ does not have a periodic point in $\bD$ for any $s \in [0,1]$. 

The diffeomorphisms $\{\phi^s\}_{s \in [0,1]}$ only differ from each other on the disk $\bD$ itself, so it follows that under this assumption, they all have the exact same set of periodic points. From this Stokes' theorem implies that the sets of actions $\mathcal{T}(s)$ are all equal to a single fixed, Lebesgue measure zero subset $\mathcal{T} \subset \bR$. 

From what was said in the previous step, we find for any $m$ that $c_m(s)$ is a continuous function from the interval $[0,1]_s$ to $\bR$, taking values in the Lebesgue measure zero set $\mathcal{T} \subset \bR$. It follows that, under our assumptions, $c_m(s)$ is a constant function for any $m$. 
It follows that for every $s$,
$$\lim_{m \to \infty} \frac{c_{\sigma_m^s}(\phi^s, \Theta_m^s) - c_{\sigma_m}(\phi, \Theta_m) + \int_{\Theta_m} H^s dt}{d_m} = 0.$$

However, since we assumed that the Calabi invariant of the Hamiltonian $H$ is nonzero, this contradicts the asymptotic identity from (\ref{eq:closingLemmaProof1}) whenever $s \in [0,1]$ is such that $f(s) \neq 0$. We therefore arrive at a contradiction and there must be some $s \in [0,1]$ such that $\phi^s$ has a periodic point in $\bD$. This proves the smooth closing lemma for area-preserving surface diffeomorphisms.

\section{Supporting estimates} \label{sec:estimates}

\subsection{Conventions and notation} \label{subsec:conventions}

Many estimates include in their statements a choice of a base configuration
$\fc_\Gamma = (B_\Gamma, \Psi_\Gamma) \in \text{Conn}(E_\Gamma) \times C^\infty(S_\Gamma)$. We will always assume that any chosen base configuration is such that $\|\Psi_\Gamma\|_{L^2} + \|\Psi_\Gamma\|_{C^1} \leq 1.$ We will denote the Hodge decomposition of a $1$-form $b$ on a closed Riemannian $3$-manifold $M$ by
$b = b^{\text{co-exact}} + b^{\text{harm}} + b^{\text{exact}}$, where the first, second, and third terms on the right-hand side denote the co-exact, harmonic, and exact parts of $b$. We always assume that the degree $d$ of $\Gamma$ is at least $1$. 

Recall also that a constant is called ``geometric'' if it depends only on the underlying Riemannian metric, and if $x, y > 0$ are positive constants we write $x = O(y)$ or $x \lesssim y$ if $x \leq \kappa y$ where $\kappa$ is a geometric constant.

\subsection{$L^2$ bounds for tame perturbations} \label{subsec:tamePerturbationBounds}

Recall that $\cP$ denotes the Banach space of abstract perturbations for Seiberg--Witten--Floer homology defined in \cite{monopolesBook}. We first recall the following property of the space $\cP$.

\begin{lem}
\label{lem:tamePerturbationEstimate1}
Let $\fg \in \cP$ be a tame perturbation. Then for any pair $(B, \Psi) \in \text{Conn}(E_\Gamma) \times C^\infty(S_\Gamma)$, there is a constant $\kappa_{\ref{lem:tamePerturbationEstimate1}} > 0$ depending on the metric $g$ such that the following estimate holds:
$$\int_{M_\phi} \big(|\fC_\fg(B, \Psi)|^2 + |\fS_\fg(B, \Psi)|^2\big) \leq \kappa_{\ref{lem:tamePerturbationEstimate1}} \|\fg\|_\cP^2 \big(\int_{M_\phi} |\Psi|^2 + 1\big) .$$
\end{lem}

\begin{proof}
The estimate follows from the definition of the $\cP$--norm and the construction of the Banach space $\cP$ in proof of \cite[Theorem $11.6.1$]{monopolesBook}. Our constant $\kappa$ corresponds to the constant $m_2$ in the fourth property of this theorem, and in the fourth part of the definition \cite[Definition $10.5.1$]{monopolesBook} of tame perturbations.
\end{proof}

The following lemma is a consequence of Lemma \ref{lem:tamePerturbationEstimate1}. 

\begin{lem}
\label{lem:tamePerturbationEstimate2}
Let $\fg \in \cP$ be a tame perturbation. Then for any pair $(B, \Psi)$ in $\text{Conn}(E_\Gamma) \times C^\infty(S_\Gamma)$ and choice of base configuration $(B_\Gamma, \Psi_\Gamma)$, there is a constant $\kappa_{\ref{lem:tamePerturbationEstimate2}} \geq 1$, depending only on the metric $g$ such that the following estimate holds:

$$|\fg(B, \Psi) - \fg(B_\Gamma, \Psi_\Gamma)| \leq \kappa_{\ref{lem:tamePerturbationEstimate2}} \|\fg\|_\cP\big(\|F_B - F_{B_\Gamma}\|_{L^2}^2 + \|\Psi\|_{L^2}^2 + 1\big).$$
\end{lem}

\begin{proof}
	Recall that $\fg$ is gauge invariant.
There is a gauge transformation $u$ so that the pair
$(\widetilde{B} = B - u^{-1}du, \widetilde{\Psi} = u \cdot \Psi)$
has the following properties:
$$(\widetilde{B} - B_\Gamma)^{\text{exact}} = 0\quad\text{and}\quad\|(\widetilde{B} - B_\Gamma)^{\text{harm}}\|_{C^0} \lesssim 1.$$

Also notice that $\|\widetilde{\Psi}\|_{L^2} = \|\Psi\|_{L^2}$. Let $(B(t),\Psi(t))$, $t\in[0,1]$ be the linear interpolation from $(B,\Psi)$ to $(\widetilde{B},\widetilde{\Psi})$. 
Applying Lemma \ref{lem:tamePerturbationEstimate1} yields the following estimates: 
\begin{align*}
& \fg(\widetilde{B}, \widetilde\Psi) - \fg(B_\Gamma, \Psi_\Gamma) \\
 \lesssim  & \sup_t \big(\|\fC_\fg(B(t), \Psi(t))\|_{L^2} + \|\fS_\fg(B(t), \Psi(t))\|_{L^2}\big)\big(\|\widetilde{B} - B_\Gamma\|_{L^2} + \|\widetilde\Psi - \Psi_\Gamma\|_{L^2}\big) \\
\lesssim & \|\fg_\cP\| \big(\|\widetilde{\Psi}\|_{L^2} + \|\Psi_\Gamma\|^2_{L^2} + 1\big)\big(\|\widetilde{B} - B_\Gamma\|_{L^2} + \|\widetilde{\Psi} - \Psi_\Gamma\|_{L^2}\big)\\
\lesssim & \|\fg_\cP\|\big(\|\widetilde{B} - B_\Gamma\|^2_{L^2} + \|\widetilde\Psi\|^2_{L^2} + 1\big),
\end{align*}
where the last inequality uses Cauchy--Schwarz and  the assumption that $\|\Psi_\Gamma\|_{L^2} + \|\Psi_\Gamma\|_{C^1} \leq 1.$
Now recall that $\|(\widetilde{B} - B_\Gamma)^{\text{harm}}\|_{C^0(M_\phi)} \lesssim 1$. To bound the co-exact part, observe that $d( (\widetilde{B} - B_\Gamma)^{\text{co-exact}}) = F_B - F_{B_\Gamma}$. The operator $\star^{g_*}d$ on the space of co-exact one-forms is self-adjoint with respect to the $L^2$ inner product and has discrete spectrum not containing zero. It follows that $\|\widetilde{B} - B_\Gamma\|_{L^2}^2 \lesssim \|F_B - F_{B_\Gamma}\|_{L^2}^2 + 1$. This and the prior inequality imply the lemma. 
\end{proof}

\subsection{Solutions to the Seiberg--Witten equations} \label{subsec:swEstimates} In this section, we derive various pointwise and integral bounds for solutions to the Seiberg-Witten equations. 

\subsubsection{The three-dimensional equations} \label{subsubsec:3dEstimates} We begin with some basic estimates on the three-dimensional $\cS$-Seiberg-Witten equations which are essential for much of the analysis in this paper. Proposition \ref{prop:3dEstimates1} restates pointwise estimates proved by Taubes. 

\begin{prop} \label{prop:3dEstimates1}
Let $\fc = (B, \Psi = ( r^{1/2}\alpha, r^{1/2}\beta))$ be a solution of the $(\phi, J, \Gamma, r, \fe_\mu)$-Seiberg-Witten equations (\ref{eq:sw}). Suppose that $\mu$ has $C^3$ norm bounded above by $1$. Then there is a geometric constant $r_{\ref{prop:3dEstimates1}} \geq 1$ and a geometric constant $\kappa_{\ref{prop:3dEstimates1}} \geq 1$ such that if $r \geq r_{\ref{prop:3dEstimates1}}$, the following estimates hold:
\begin{enumerate}
\item $|\alpha|^2 + \kappa_{\ref{prop:3dEstimates1}}^{-1} r|\beta|^2 \leq 1 + \kappa_{\ref{prop:3dEstimates1}} r^{-1}$.
\item $|F_B| \leq \kappa_{\ref{prop:3dEstimates1}} r(1 - |\alpha|^2 + \kappa_{\ref{prop:3dEstimates1}} r^{-1})$.
\item $|\widehat{\nabla}_B\alpha|^2 + r|\widehat{\nabla}_B\beta|^2 \leq \kappa_{\ref{prop:3dEstimates1}} r(1 - |\alpha|^2 + \kappa_{\ref{prop:3dEstimates1}}r^{-1})$.
\item $| \int_{M_\phi} r(1 - |\alpha|^2) - 2\pi d| \leq \kappa_{\ref{prop:3dEstimates1}}$.
\end{enumerate}
\end{prop}

\begin{proof} 
The first item follows by copying the proof of Lemma $2.2$ in \cite{TaubesWeinstein1}, given the rescaling from Remark \ref{rem:conventions}. The second item follows from the first item and the equations (\ref{eq:sw}). The third item follows by copying the proof of Lemma $2.3$ in \cite[Paper IV]{ECHSWF}, given the rescaling from Remark \ref{rem:conventions}.

For the fourth item, observe from the first equation in (\ref{eq:sw}) that
$$F_B = -i(r(1 - |\alpha|^2 + |\beta|^2)\omega_{\phi} + e_1$$
where $e_1$ is a two-form such that $dt \wedge e_1 \equiv 0$. Recall that $\frac{i}{2\pi}F_B$ represents $\text{PD}(\Gamma)$ and the volume form of the metric $g$ is $dt \wedge \omega_{\phi}$. We conclude that \begin{equation*}
\int_{M_\phi} r(1 - |\alpha|^2) = \int_{M_\phi} dt \wedge (iF_B + r|\beta|^2\omega_{\phi}) = 2\pi d + r\int_{M_\phi}|\beta|^2.
\end{equation*}

An application of the first item then yields the desired bound.
\end{proof}

The statement of Proposition \ref{prop:3dEstimates1} assumes that the perturbation term has the form $\fg=\fe_\mu$.
Proposition \ref{prop:3dEstimates2} below deduces some energy bounds for the $\cS$-Seiberg-Witten equations in the presence of an \emph{arbitrary} abstract perturbation $\fg \in \cP$.

\begin{prop} \label{prop:3dEstimates2}
Let $\fc = (B, \Psi = (r^{1/2}\alpha, r^{1/2}\beta))$ be a solution of the $(\phi, J, \Gamma, r, \fg)$-Seiberg-Witten equations (\ref{eq:sw}). Suppose that $\|\fg\|_\cP \leq 1$. Then there is a geometric constant $\kappa_{\ref{prop:3dEstimates2}} \geq 1$ such that the following estimates hold: 
$$\|\nabla_B\Psi\|_{L^2}^2 + \|\Psi\|_{L^4}^4 \leq \kappa_{\ref{prop:3dEstimates2}}(r+1)^2\quad\text{and}\quad\|F_B\|_{L^2} \leq \kappa_{\ref{prop:3dEstimates2}}(r+1).$$
\end{prop}

\begin{proof}
The starting point is the Weitzenbock formula for the Dirac operator $D_B$:
$$D_B^2\Psi = \nabla_B^*\nabla_B\Psi + \frac{R_{g}}{4}\Psi + \cl(\star^{g}F_B)\Psi.$$

Taking the inner product with $\Psi$, integrating over $M_\phi$, and plugging in (\ref{eq:sw}) yields an inequality of the form
$$\int_{M_\phi} |\fS_\fg(B, \Psi)|^2 \geq \int_{M_\phi} |\nabla_B\Psi|^2 + |\Psi|^2(|\Psi|^2 - \kappa (r + 1)).$$
Plug in the bound from Lemma \ref{lem:tamePerturbationEstimate1} to bound the left-hand side from above by a multiple of $\|\Psi\|_{L^2}^2 + 1$, and rearrange to conclude
$$\int_{M_\phi} |\nabla_B \Psi|^2 + |\Psi|^4 \leq \kappa + \kappa(r+1)\|\Psi\|_{L^2}^2.$$
H\"older's inequality and Young's inequality imply
$$\int_{M_\phi} |\Psi|^2 \leq \kappa (r+1) + (2\kappa(r+1))^{-1}\int_{M_\phi} |\Psi|^4.$$
Combining the above two inequalities and rearranging prove the first estimate. The second estimate follows from the first estimate, the first equation in (\ref{eq:sw}) and Lemma \ref{lem:tamePerturbationEstimate1} as follows:
\begin{equation*}
\|F_B\|_{L^2}^2 \lesssim \|\Psi\|_{L^4}^4 + \|\fC_\fg(B, \Psi)\|_{L^2}^2 + r^2 + 1\lesssim \|\Psi\|_{L^4}^4 + r^2 + 1.
\phantom\qedhere\makeatletter\displaymath@qed
\end{equation*}
\end{proof}

Proposition \ref{prop:3dActionEstimates} shows some convenient estimates hold after applying a gauge transformation. 

\begin{prop} \label{prop:3dActionEstimates}
Let $\fc = (B, \Psi = (r^{1/2}\alpha, r^{1/2}\beta))$ be a solution of the $(\phi, J, \Gamma, r, \fe_\mu)$-Seiberg-Witten equations (\ref{eq:sw}). Suppose that $\mu$ has $C^3$ norm bounded above by $1$. Then there is a gauge transformation $u$ such that the following holds.  There is a geometric constant $r_{\ref{prop:3dActionEstimates}} \geq 1$ and a geometric constant $\kappa_{\ref{prop:3dActionEstimates}} \geq 1$ such that if $r \geq r_{\ref{prop:3dActionEstimates}}$, then the following estimates hold:
\begin{enumerate}
\item $|B - B_\Gamma - u^{-1}du| \leq \kappa_{\ref{prop:3dActionEstimates}}(d^{1/3}r^{2/3} + d^{1/3}r^{-1/3}\|F_{B_\Gamma}\|_{C^0} + d^{-2/3}r^{2/3}\|F_{B_\Gamma}\|_{L^1})$.
\item $|\fcs(B-u^{-1}du, B_\Gamma)| \leq \kappa_{\ref{prop:3dActionEstimates}}(d + \|F_{B_\Gamma}\|_{L^1})(d^{1/3}r^{2/3} + d^{1/3}r^{-1/3}\|F_{B_\Gamma}\|_{C^0} + d^{-2/3}r^{2/3}\|F_{B_\Gamma}\|_{L^1})$.
\item $|E_{\phi}(B - u^{-1}du,B_\Gamma)| \leq \kappa_{\ref{prop:3dActionEstimates}}$.
\item $|\fa_{r, \fe_\mu}(u \cdot \fc,\fc_\Gamma)| \leq \kappa_{\ref{prop:3dActionEstimates}}(r + d  + \|F_{B_\Gamma}\|_{L^1} + |\fcs(u \cdot \fc,\fc_\Gamma)|)$.
\end{enumerate}
\end{prop}

\begin{proof}
The proof amounts to repeating the proof of \cite[Lemma $5.1$]{LeeTaubes12}, but tracking more precisely the contribution of the base configuration to the estimates.

Let $u$ be a gauge transformation such that
$\widetilde{b} = B - B_\Gamma - u^{-1}du$
satisfies the property that $\widetilde{b}^{\text{exact}}=0$
and $\|\widetilde{b}^{\text{harm}}\|_{C^0} \lesssim 1$. Let $G(x,y)$ denote the Green's form for the operator $d + d^*$ on the space of one-forms. We can choose $G(x,y)$ so that
$|G(x,y)| \lesssim \text{dist}(x, y)^{-2}.$ Note that
$(d + d^*)\widetilde{b}^{\text{co-exact}} = F_B - F_{B_\Gamma}.$

Therefore, at any point $x \in M_\phi$, we have
\begin{align*}
|\widetilde{b}^{\text{co-exact}}(x)| &=\Big |\int_{M_\phi} G(x,y) \wedge (F_B - F_{B_\Gamma})(y)\Big| \lesssim \int_{M_\phi} \text{dist}(x,y)^{-2}|F_B - F_{B_\Gamma}|(y).
\end{align*}

By the second and fourth points of Proposition \ref{prop:3dEstimates1}, we observe that, as long as $r$ is sufficiently large,
$\|F_B\|_{L^1} \lesssim 1 + r\int_{M_\phi} (1 - |\alpha|^2) \lesssim d + 1\lesssim d.$
Fix some positive $\delta \ll 1$. Then we have
$$
\int_{M_\phi} \text{dist}(x,y)^{-2}|F_B - F_{B_\Gamma}|(y) 
\lesssim \|F_B - F_{B_\Gamma}\|_{C^0} \int_{B_\delta(x)} \text{dist}(x, y)^{-2} + \delta^{-2} \int_{M_\phi} |F_B - F_{B_\Gamma}|.
$$

By the second point of Proposition \ref{prop:3dEstimates1}, the first term in the above is $\lesssim \delta (r + \|F_{B_\Gamma}\|_{C^0})$. By our bound for the $L^1$ norm of $F_B$ established above, the second term is $\lesssim \delta^{-2}(d + \|F_{B_\Gamma}\|_{L^1})$. It follows that
$$\|\widetilde{b}^{\text{co-exact}}\|_{C^0} \lesssim \delta(r + \|F_{B_\Gamma}\|_{C^0}) + \delta^{-2}(d + \|F_{B_\Gamma}\|_{L^1}).$$

Selecting $\delta = d^{1/3}r^{-1/3}$ proves the first item. The second item follows from the first item:
\begin{align*} |\fcs(B - u^{-1}du,B_\Gamma)| &\lesssim \|\widetilde{b}\|_{C^0}(\|F_B\|_{L^1} + \|F_{B_\Gamma}\|_{L^1} + 1) \\
&\lesssim (d + \|F_{B_\Gamma}\|_{L^1})(d^{1/3}r^{2/3} + d^{1/3}r^{-1/3}\|F_{B_\Gamma}\|_{C^0} + d^{-2/3}r^{2/3}\|F_{B_\Gamma}\|_{L^1}).
\end{align*}

For the third item, observe that
$$|E_{\phi}(B - u^{-1}du,B_\Gamma)| = \Big|\int_{M_\phi} \widetilde{b} \wedge \omega_{\phi}\Big| = \Big|\int_{M_\phi} \widetilde{b}^{\text{harm}} \wedge \omega_{\phi}\Big| \lesssim 1.$$

The fourth item is an immediate consequence of the first three items, along with the fact that
$$
\fe_\mu(B - u^{-1}du, B_\Gamma) \leq \|\mu\|_{C^0}\|F_B - F_{B_\Gamma}\|_{L^1} \lesssim d + \|F_{B_\Gamma}\|_{L^1}.
\phantom\qedhere\makeatletter\displaymath@qed
$$
\end{proof}

\subsubsection{The continuation instanton equations} \label{subsubsec:continuationEstimates} We write down here several necessary estimates for solutions of the $\cS_s$-Seiberg-Witten instanton equations (\ref{eq:swContinuation}). Fix SW parameter sets $\cS_\pm = (\phi, J, \Gamma, r_\pm, \fg_\pm)$ and a SW continuation parameter set $\cS_s = (0, J, r_s, \fg_s)$ from $\cS_-$ to $\cS_+$. We will assume that $r_s \geq 1$ and $\|\fg_s\|_{\cP} \leq 1$ for every $s$. 
We also assume that there exists a universal constant $s_*>0$ such that $(r_s,\fg_s)$ is independent of $s$ when $s<-s_*$ and is independent of $s$ when $s>s_*$. 
Write
$$r_{\text{max}} = \sup_{s \in \mathbb{R}} r_s\quad\text{and} \quad\Lambda_\fg = \sup_{s \in \mathbb{R}} \|\frac{\partial}{\partial s} \fg_s\|_{\cP}.$$

We will assume $\Gamma$ is negative monotone with monotonicity constant $\rho < 0$. Fix a solution $\fd = (B_s, \Psi_s = (\alpha_s, \beta_s))$ of the $\cS_s$-Seiberg-Witten instanton equations and let $\fc_\pm = (B_\pm, \Psi_\pm)$ denote its limits as $s \to \pm\infty$. Fix a base configuration $\fc_\Gamma = (B_\Gamma, \Psi_\Gamma)$. Our main objective is to bound the difference of the values of the Chern--Simons--Dirac functional on $\fc_+$ and $\fc_-$. We begin with an explicit computation of this difference in Lemma \ref{lem:csdContinuationGradient1} below. 

\begin{lem}
\label{lem:csdContinuationGradient1}
Let $\cS_\pm = (\phi, J, \Gamma, r_\pm, \fg_\pm)$, $\cS_s = (0, J, r_s, \fg_s)$, $\fd$, $\fc_\pm = (B_\pm, \Psi_\pm)$, and $\fc_\Gamma = (B_\Gamma, \fc_\Gamma)$ be as fixed at the start of \S\ref{subsubsec:continuationEstimates}. Then for any $s_- < s_+$ and any $r$ we have the identity
\begin{equation*}
\begin{split}
&\fa_{r, \fg_{s_-}}(\fd_{s_-}, \fc_\Gamma) - \fa_{r, \fg_{s_+}}(\fd_{s_+}, \fc_\Gamma) \\
&\qquad = \int_{[s_-, s_+] \times M_\phi} \big(|\frac{\partial B_s}{\partial s}|^2 + 2|\frac{\partial \Psi_s}{\partial s}|^2 \big)  + \int_{[s_-, s_+] \times M_\phi} \langle \frac{\partial}{\partial s}B_s, i(r_s - r)dt \rangle \\
& \qquad\qquad - \int_{s_-}^{s_+} \big((\frac{\partial}{\partial s}\fg_s)(B_s, \Psi_s) - (\frac{\partial}{\partial s}\fg_s)(B_\Gamma, \Psi_\Gamma)\big) ds.
\end{split}
\end{equation*}
\end{lem}

\begin{proof}
First, we compute the difference in the Chern--Simons functionals:
\begin{equation} \label{eq:csdContinuationGradient1}
\begin{split}
\fcs(B_{s_-}, B_\Gamma) - \fcs(B_{s_+}, B_\Gamma) = 2\int_{[s_-, s_+] \times M_\phi} |\frac{\partial}{\partial s} B_s|^2 - \langle \frac{\partial}{\partial s} B_s,  \cl^\dagger(\Psi_s) - ir_sdt + \fC_{\fg_s}(B_s, \Psi_s)\rangle.
\end{split}
\end{equation}

This uses Stokes' theorem and the first equation in \eqref{eq:swContinuation}. 
Second, we compute the difference in the energy functionals using Stokes' theorem and the fact that $2 dt$ and $\omega_\phi$ are Hodge-dual:
\begin{equation} \label{eq:csdContinuationGradient2}
rE_\phi(B_{s_+}, B_\Gamma) - r E_{\phi}(B_{s_-}, B_\Gamma) = -2\int_{[s_-, s_+] \times M_\phi} \langle \frac{\partial}{\partial s} B_s, i r dt \rangle.
\end{equation}

Third, we compute the difference in the Dirac terms:
\begin{equation} \label{eq:csdContinuationGradient3}
\begin{split}
&\int_{M_\phi} \langle D_{B_{s_-}}\Psi_{s_-}, \Psi_{s_-} \rangle - \langle D_{B_{s_+}}\Psi_{s_+}, \Psi_{s_+} \rangle \\
&\qquad = \int_{s_-}^{s_+} (\int_{M_\phi} 2|D_{B_s}\Psi_s|^2 + \langle \frac{\partial}{\partial s} B_s, \cl^\dagger(\Psi_s) \rangle) ds - 2\int_{[s_-, s_+] \times M_\phi} \text{Re} \langle D_{B_s}\Psi_s, \fS_{\fg_s}(B_s, \Psi_s) \rangle.
\end{split}
\end{equation}

This follows by using Stokes' theorem, the fact that the Dirac operator is self-adjoint, and the second equation in \eqref{eq:swContinuation} to expand all terms of the form $\frac{\partial}{\partial s}\Psi_s$. Fourth, we compute the differences of the abstract perturbation terms:
\begin{equation} \label{eq:csdContinuationGradient4}
\begin{split}
&\fg_{s_-}(B_{s_-}, \Psi_{s_-}) - \fg_{s_+}(B_{s_+}, \Psi_{s_+}) \\
&\qquad = \int_{s_-}^{s_+} (\int_{M_\phi} \langle \frac{d}{ds} B_s, \fC_{\fg_s}(B_s, \Psi_s) \rangle - \text{Re}\langle D_{B_s}\Psi_s, \fS_{\fg_s}(B_s, \Psi_s) \rangle + |\fS_\fg(B_s, \Psi_s)|^2) ds \\
&\qquad\qquad - \int_{s_-}^{s_+} (\frac{\partial}{\partial s} \fg_s)(B_s, \Psi_s) ds, \\
&\fg_{s_-}(B_\Gamma, \Psi_\Gamma) - \fg_{s_+}(B_\Gamma, \Psi_\Gamma) = -\int_{s_-}^{s_+} (\frac{\partial}{\partial s}\fg_s)(B_\Gamma, \Psi_\Gamma).
\end{split}
\end{equation}

Combining (\ref{eq:csdContinuationGradient1}--\ref{eq:csdContinuationGradient4}) proves the lemma. 
\end{proof}

The following corollary is an immediate consequence of Lemma \ref{lem:csdContinuationGradient1}.
\begin{cor}
	\label{cor_derivative_of_a}
	Let $\cS_\pm = (\phi, J, \Gamma, r_\pm, \fg_\pm)$, $\cS_s = (0, J, r_s, \fg_s)$, $\fd$, $\fc_\pm = (B_\pm, \Psi_\pm)$, and $\fc_\Gamma = (B_\Gamma, \fc_\Gamma)$ be as above. Then for any $r$, we have
	$$
	\frac{\partial}{\partial s} \fa_{r,\fg_s}(\fd_s,\fc_\Gamma) = -\int_{M_\phi} \Big[|\frac{\partial B_s}{\partial s}|^2 + 2|\frac{\partial \Psi_s}{\partial s}|^2 - \langle \frac{\partial B_s}{\partial s}, i(r_s-r)dt\rangle\Big]  + \big((\frac{\partial}{\partial s} \fg_s) (B_s,\Psi_s) - (\frac{\partial}{\partial s} \fg_s) (B_\Gamma,\Psi_\Gamma) \big).
	$$
\end{cor}

\begin{rem} \label{rem:csdGradient}
If $r_s$ and $\fg_s$ are independent of $s$ in $\cS_s$, then $\cS_+ = \cS_- = \cS$ and $\fd$ is a solution of the standard $\cS$-Seiberg-Witten instanton equations. In this case, the second part of Lemma \ref{lem:csdContinuationGradient1} vanishes and we conclude that the Chern--Simons--Dirac functional $\fa_{r, \fg_s}(\fd_s, \fc_\Gamma)$ decreases monotonically in $s$.  
\end{rem}

Proposition \ref{prop:continuationEstimates2} below gives a bound for the change in the Chern--Simons--Dirac functional in the case where $r_+ > r_- > -2\pi\rho$. 

\begin{prop}
\label{prop:continuationEstimates2}
Let $\cS_\pm = (\phi, J, \Gamma, r_\pm, \fg_\pm)$, $\cS_s = (0, J, r_s, \fg_s)$, $r_{\text{max}}$, $\Lambda_\fg$, $\fd = (B_s, \Psi_s)$, $\fc_\pm = (B_\pm, \Psi_\pm)$, $\fc_\Gamma = (B_\Gamma, \Psi_\Gamma)$ be as fixed at the start of \S\ref{subsubsec:continuationEstimates}. Let $z_\pm = 2\pi\rho + r_\pm$. Suppose that $r_+ > r_- > -2\pi\rho$. Then there are constants $\kappa_{\ref{prop:continuationEstimates2}} \geq 1$ and $r_{\ref{prop:continuationEstimates2}} \geq 1$ that depend only on $d$, $\fc_\Gamma$, and geometric constants, such that the following holds. Assume that
$$\sup_{s \in \bR} |r_s - r_-| \leq \kappa_{\ref{prop:continuationEstimates2}}^{-1},\quad\Lambda_\fg \leq \kappa_{\ref{prop:continuationEstimates2}}^{-1},\quad r_\pm \geq r_{\ref{prop:continuationEstimates2}}.$$

Then 
\begin{equation*}
\frac{z_+}{z_-}\fa_{r_+, \fg_+}(\fc_+, \fc_\Gamma) - \fa_{r_-, \fg_-}(\fc_-, \fc_\Gamma) \leq \kappa_{\ref{prop:continuationEstimates2}}\,r_{\text{max}}^2\,\Big((z_-^{-1} + z_+^{-1} + 1)|r_+ - r_-| + \Lambda_\fg\Big).
\end{equation*}
\end{prop}

Proposition \ref{prop:continuationEstimates2} will follow from Lemmas \ref{lem:continuationEstimates3} and \ref{lem:continuationEstimates4} below. Lemma \ref{lem:continuationEstimates3} analyzes the gauge-invariant function
$\cL(s) = z_-\fa_{r_+, \fg_s}(\fd_s, \fc_\Gamma) - z_+\fa_{r_-, \fg_s}(\fd_s, \fc_\Gamma).$
This is similar to \cite[Lemma $31.2.1$]{monopolesBook}, although we need more precise estimates in our situation.

\begin{lem}
\label{lem:continuationEstimates3} Let $\cS_\pm = (\phi, J, \Gamma, r_\pm, \fg_\pm)$, $\cS_s = (0, J, r_s, \fg_s)$,  $\Lambda_\fg$, $z_\pm$, $\fd_s = (B_s, \Psi_s)$, $\fc_\pm = (B_\pm, \Psi_\pm)$, $\fc_\Gamma = (B_\Gamma, \Psi_\Gamma)$ be as above. Suppose that $r_+ > r_- > -2\pi\rho$. There is a geometric constant $\kappa_{\ref{lem:continuationEstimates3}} \geq 1$ such that the following holds. Assume that
$$\sup_{s \in \bR} |r_s - r_-| \leq \kappa_{\ref{lem:continuationEstimates3}}^{-1},\quad\Lambda_\fg \leq \kappa_{\ref{lem:continuationEstimates3}}^{-1}.$$

Then the smooth, real-valued function
$$\cL(s) = z_-\fa_{r_+, \fg_s}(\fd_s, \fc_\Gamma) - z_+\fa_{r_-, \fg_s}(\fd_s, \fc_\Gamma)$$
satisfies that for every $s\in \mathbb{R}$, either $\cL'(s) \geq 0$ or
$|\cL(s)| \leq \kappa_{\ref{lem:continuationEstimates3}}(r_{s}^2+d^2)|r_+ - r_-|$.
\end{lem}

\begin{proof}
By Corollary \ref{cor_derivative_of_a}, we have
\begin{align}
\frac{\partial}{\partial s}(\mathcal{L}(s)) &= \int_{M_\phi} \big((r_+ - r_-)|\frac{\partial}{\partial s} B_s|^2 - (2\pi\rho+r_s) (r_+ - r_-)\langle \frac{\partial}{\partial s}B_s, idt \rangle + 2(r_+ - r_-)|\frac{\partial}{\partial s}\Psi_s|^2\big) 
\label{eqn:continuationEstimates3_proof0}
\\ 
&\qquad - (r_+ - r_-) \big((\frac{\partial}{\partial s} \fg_s)(B_s, \Psi_s) - (\frac{\partial}{\partial s}\fg_s)(B_\Gamma, \Psi_\Gamma)\big).
\nonumber
\end{align}

If $\frac{\partial}{\partial s}\mathcal{L}(s) < 0$, then setting the above to be less than $0$ and dividing by $r_+ - r_-$ shows
\begin{equation*}
\int_{M_\phi} \big(|\frac{\partial}{\partial s}B_s|^2 +2|\frac{\partial}{\partial s} \Psi_s|^2\big) < \int_{M_\phi} (2\pi\rho+r_s)\, \langle \frac{\partial}{\partial s}B_s, i dt \rangle  +  \big((\frac{\partial}{\partial s} \fg_s)(B_s, \Psi_s) - (\frac{\partial}{\partial s}\fg_s)(B_\Gamma, \Psi_\Gamma)\big)\big).
\end{equation*} 

By Young's inequality, the fact that $|dt| \equiv 1$, and Lemma \ref{lem:tamePerturbationEstimate2} we have
\begin{equation}
	\label{eqn:continuationEstimates3_proof1}
\int_{M_\phi} \big(\frac{1}{2}|\frac{\partial}{\partial s}B_s|^2 + |\frac{\partial}{\partial s} \Psi_s|^2\big)  \lesssim r_s^2 + \Lambda_\fg(\|F_{B_s} - F_{B_\Gamma}\|_{L^2}^2 + \|\Psi_s\|_{L^2}^2  + 1).
\end{equation}

We now find lower bounds for the left-hand side of the above inequality using equation \eqref{eq:swContinuation}:
\begin{align}
	\int_{M_\phi} |\frac{\partial}{\partial s} B_s|^2 = & \int_{M_\phi} \Big( |F_{B_s}|^2 + |\cl^\dagger(\Psi_s) - i r_s dt|^2  - 2\langle \cl(\star^{g}F_{B_s})\Psi_s, \Psi_s \rangle
	+ 2 r_s \langle i dt, \star^{g}F_{B_s} \rangle
\label{eqn:continuationEstimates3_proof2} \\
	&\;
	- 2\text{Re}\langle \star^{g}F_{B_s}, i\varpi_V + \fC_{\fg_s}(B_s, \Psi_s) \rangle 
	+ 2\text{Re} \langle \cl^\dagger(\Psi_s) - i r_s dt, i\varpi_V + \fC_{\fg_s}(B_s, \Psi_s) \rangle
	\nonumber \\
	&\;
	+ |i\varpi_V - \fC_{\fg_s}(B_s, \Psi_s)|^2 \Big) 
	\nonumber \\
	\geq & \Big[\int_{M_\phi} \frac{1}{2} |F_{B_s}|^2 + \frac{1}{2}|\cl^\dagger(\Psi_s) - i r_s dt|^2 
	-2\langle \cl(\star^{g}F_{B_s})\Psi_s, \Psi_s \rangle 
	 - 3|i\varpi_V + \fC_{\fg_s}(B_s, \Psi_s)|^2\Big] - 4\pi d r_s 
	\nonumber \\
	\geq & \Big[\int_{M_\phi} \big(\frac{1}{2} |F_{B_s}|^2 + \frac{1}{2}|\cl^\dagger(\Psi_s) - i r_s dt|^2 - 2\langle \cl(\star^{g}F_{B_s})\Psi_s, \Psi_s \rangle \big)
	- \kappa \int_{M_\phi} |\Psi_s|^2 \Big]- (4\pi d r_s + \kappa) 
	\nonumber \\
	\geq & \Big[ \int_{M_\phi} \frac{1}{2} |F_{B_s}|^2 + \frac{1}{2}| |\Psi_s|^2 - \kappa r_s|^2 - 2\langle \cl(\star^{g}F_{B_s})\Psi_s, \Psi_s \rangle \Big]  - \kappa^2 (r_{s}^2+d^2).
	\nonumber
\end{align}

The first inequality follows from Young's inequality and the fact that the integral of $dt \wedge F_{B_s}$ over $M_\phi$ is equal to $-2\pi i d$ for any $s$;
 the second inequality follows from the bound on the $L^2$ norm of $|\fC_{\fg_s}(B_s, \Psi_s)|^2$ for any $s$ given by Lemma \ref{lem:tamePerturbationEstimate1}; the third follows by expanding $|\cl^\dagger(\Psi_s) - i r_s dt|^2$ and absorbing the terms proportional to the integral of $|\Psi_s|^2$. 

As for the integral of $ |\frac{\partial}{\partial s} \Psi_s|^2$, we have:

$$
	2 \int_{M_\phi}  |\frac{\partial}{\partial s} \Psi_s|^2 = 2\int_{M_\phi} \big(|D_{B_s}\Psi_s|^2 + |\fS_{\fg_s}(B_s, \Psi_s)|^2 - 2\text{Re}\langle D_{B_s}\Psi_s, \fS_{\fg_s}(B_s, \Psi_s) \rangle \big) 
$$

Using the Weitzenbock formula for the Dirac operator $D_{B_s}$, we have the identity
$$\int_{M_\phi} |D_{B_s}\Psi_s|^2 = \int_{M_\phi} \big(|\nabla_{B_s}\Psi_s|^2 + \frac{R_{g}}{4}|\Psi_s|^2 + \langle \cl(\star^{g}F_{B_s}\Psi_s, \Psi_s \rangle\big).$$

Paired with the fact that $|D_{B_s}\Psi_s| \leq 3|\nabla_{B_s}\Psi_s|$, this shows 
\begin{align*}
	& \quad 2\int_{M_\phi}  |\frac{\partial}{\partial s} \Psi_s|^2 \\
	&\geq \int_{M_\phi} \big(2|\nabla_{B_s}\Psi_s|^2 - \kappa|\Psi_s|^2 + 2\langle \cl(\star^{g}F_{B_s})\Psi_s, \Psi_s\rangle - 12|\nabla_{B_s}\Psi||\fS_{\fg_s}(B_s,\Psi_s)| + 2|\fS_{\fg_s}(B_s,\Psi_s)|^2\big) \\
	&\geq \int_{M_\phi} \big(\frac{1}{2}|\nabla_{B_s}\Psi_s|^2 - \kappa|\Psi_s|^2 + 2\langle \cl(\star^{g}F_{B_s})\Psi_s, \Psi_s\rangle  - 1000|\fS_{\fg_s}(B_s,\Psi_s)|^2\big).
\end{align*}

Hence by Lemma \ref{lem:tamePerturbationEstimate1},
\begin{equation}
	\label{eqn:continuationEstimates3_proof3}
	 \int_{M_\phi}  |\frac{\partial}{\partial s} \Psi_s|^2  \geq \int_{ M_\phi} \big(\frac{1}{2}|\nabla_{B_s}\Psi_s|^2 - \kappa |\Psi_s|^2 - \kappa + 2\langle \cl(\star^{g_*}F_{B_s})\Psi_s, \Psi_s \rangle \big).
\end{equation}
Combining \eqref{eqn:continuationEstimates3_proof2} and \eqref{eqn:continuationEstimates3_proof3}, we have
\begin{equation}
	\label{eqn:continuationEstimates3_proof4}
\int_{M_\phi} \big(|\frac{\partial}{\partial s}B_s|^2 +2 |\frac{\partial}{\partial s}\Psi_s|^2\big) \geq \frac{1}{2}\int_{M_\phi} \big(|F_{B_s}|^2 + | |\Psi_s|^2 - \kappa r_s|^2 + |\nabla_{B_s}\Psi_s|^2\big)  - \kappa^2 (r_s^2+d^2).
\end{equation}

Combining \eqref{eqn:continuationEstimates3_proof1} and \eqref{eqn:continuationEstimates3_proof4}, assuming $\Lambda_\fg$ is sufficiently small, and rearranging shows that
\begin{equation}
\label{eqn:continuationEstimates3_proof5}
\int_{M_\phi} \big(|F_{B_s}|^2 + ||\Psi_s|^2 - \kappa r_s|^2 + |\nabla_{B_s}\Psi_s|^2\big) \lesssim r_s^2+d^2.
\end{equation}

Now note that, as in the proof of Lemma \ref{lem:tamePerturbationEstimate2},
there is a gauge-transformation $u$ such that the transformed pair $(\widetilde B_s = B_s - u^{-1} du, \widetilde \Psi_s = u \cdot \Psi_s)$ satisfies the estimate
$$\|\widetilde B_s - B_\Gamma\|_{L^2}^2 \lesssim \|F_{B_s} - F_{B_\Gamma}\|_{L^2}^2 + 1.$$
Since $\mathcal{L}(s)$ is gauge-invariant, we may assume without loss of generality that $B_s=\widetilde{B}_s$. Therefore \eqref{eqn:continuationEstimates3_proof0}, \eqref{eqn:continuationEstimates3_proof5}, Lemma \ref{lem:tamePerturbationEstimate2}, and the assumption that $\cL'(s) < 0$ imply that 
$$
|\mathcal{L}(s)| 
\lesssim |r_+ - r_-|(\|F_{B_s} - F_{B_\Gamma}\|_{L^2}^2  + \|\nabla_{B_s}\Psi_s\|_{L^2}^2 + \|\Psi_s\|_{L^2}^2  + 1) 
\lesssim (r_{s}^2+d^2)|r_+ - r_-|.
\phantom\qedhere\makeatletter\displaymath@qed
$$
\end{proof}

Lemma \ref{lem:continuationEstimates4} bounds the change in the Chern--Simons--Dirac functional between $s = s_*$ and $s = -s_*$, with fixed $r = r_-$.

\begin{lem}
\label{lem:continuationEstimates4} Let $\cS_\pm = (\phi, J, \Gamma, r_\pm, \fg_\pm)$, $\cS_s = (0, J, r_s, \fg_s)$, $r_{\text{max}}$, $\Lambda_\fg$, $z_\pm$, $\fd = (B_s, \Psi_s)$, $\fc_\pm = (B_\pm, \Psi_\pm)$, $\fc_\Gamma = (B_\Gamma, \Psi_\Gamma)$ be as above. Suppose that $r_+ > r_- > -2\pi\rho$. There is a geometric constant $\kappa_{\ref{lem:continuationEstimates4}} \geq 1$ such that the following holds. Assume that
$$\sup_{s \in \bR} |r_s - r_-| \leq \kappa_{\ref{lem:continuationEstimates4}}^{-1}, \quad\Lambda_\fg \leq \kappa_{\ref{lem:continuationEstimates4}}^{-1}.$$
Then
$\fa_{r_-, \fg_+}(\fd_{s_*}, \fc_\Gamma) - \fa_{r_-, \fg_{-}}(\fd_{-s_*}, \fc_\Gamma) \leq \kappa_{\ref{lem:continuationEstimates4}} \Lambda_\fg r_{\text{max}}^2.$
\end{lem}

\begin{proof}

We deduce from Lemma \ref{lem:csdContinuationGradient1} the following identity:

\begin{equation} \label{eq:continuationEstimates6}
\begin{split}
& \fa_{r_-, \fg_+}(\fd_{s_*}, \fc_\Gamma) - \fa_{r_-, \fg_-}(\fd_{-s_*}, \fc_\Gamma) \\
= & -\int_{[-s_*,s_*] \times M_{\phi}} \big(|\frac{\partial}{\partial s} B_s|^2 + |\frac{\partial}{\partial s}\Psi_s|^2\big)  
 + \int_{[-s_*,s_*] \times M_{\phi}} \big(\langle \frac{\partial}{\partial s} B_s, (r_s - r_-)dt\rangle \big) 
 \\
& - \int_{-s_*}^{s_*} \big((\frac{\partial}{\partial s} \fg_s)(B_s, \Psi_s) - (\frac{\partial}{\partial s}\fg_s)(B_\Gamma, \Psi_\Gamma)) ds.
\end{split}
\end{equation}

As long as $\sup_{s \in \bR}|r_s - r_-|$ is sufficiently small, we are free to absorb the second integral into the first integral, which is negative. Lemma \ref{lem:tamePerturbationEstimate2} yields the following bound for the third integral:
$$ |\int_{-s_*}^{s_*} \big((\frac{\partial}{\partial s} \fg_s)(B_s, \Psi_s) - (\frac{\partial}{\partial s}\fg_s)(B_\Gamma, \Psi_\Gamma)\big)ds| \lesssim \Lambda_\fg\big(\int_{[-s_*,s_*] \times M_{\phi}} \big(|\frac{\partial}{\partial s} B_s|^2 + |\frac{\partial}{\partial s}\Psi_s|^2\big) + r_s^2\big).$$

As long as $\Lambda_\fg$ is small, the first two terms on the right-hand side can be absorbed into the first line of \eqref{eq:continuationEstimates6}. Since $s_*$ is assumed to be some universal constant, the desired inequality is proved.
\end{proof}

We now prove Proposition \ref{prop:continuationEstimates2} using Lemmas \ref{lem:continuationEstimates3} and \ref{lem:continuationEstimates4}.

\begin{proof}[Proof of Proposition \ref{prop:continuationEstimates2}]
We split up the quantity in the statement of the proposition into the following sum:
\begin{equation*}
\begin{split}
&\frac{z_-}{z_+}\fa_{r_+, \fg_+}(\fc_+, \fc_\Gamma) - \fa_{r_-, \fg_-}(\fc_-, \fc_\Gamma) \\
&\qquad = \underbrace{\frac{z_-}{z_+}\fa_{r_+, \fg_+}(\fc_+, \fc_\Gamma) - \frac{z_-}{z_+}\fa_{r_+, \fg_+}(\fd_{s_*}, \fc_\Gamma)}_{I} + \underbrace{\frac{z_-}{z_+}\fa_{r_+, \fg_+}(\fd_{s_*}, \fc_\Gamma) - \fa_{r_-, \fg_+}(\fd_{s_*}, \fc_\Gamma)}_{II} \\
&\qquad + \underbrace{\fa_{r_-, \fg_+}(\fd_{s_*}, \fc_\Gamma) - \fa_{r_-, \fg_-}(\fd_{-s_*}, \fc_\Gamma)}_{III} + \underbrace{\fa_{r_-, \fg_-}(\fd_{-s_*}, \fc_\Gamma) - \fa_{r_-, \fg_-}(\fc_-, \fc_\Gamma)}_{IV}.
\end{split}
\end{equation*}

First, observe that I and IV are nonpositive by Lemma \ref{lem:csdContinuationGradient1} (see Remark \ref{rem:csdGradient}). Next, the absolute value of III is bounded by $\kappa_{\ref{lem:continuationEstimates4}}\Lambda_\fg r_{\text{max}}^2$ by Lemma \ref{lem:continuationEstimates4}.
By Lemma \ref{lem:continuationEstimates3}, we deduce that 
$$II \lesssim \max\Big\{(z_-)^{-1}\kappa_{\ref{lem:continuationEstimates3}}(r_{\text{max}}^2+d^2)|r_+ - r_-|, \frac{z_-}{z_+}\fa_{r_+, \fg_+}(\fc_+) - \fa_{r_-, \fg_+}(\fc_+)\Big\}.$$

The function $\frac{z_-}{z_+}\fa_{r_+, \fg_+} - \fa_{r_-, \fg_+}$ is gauge-invariant. Then Proposition \ref{prop:3dActionEstimates} implies that for $r_+ > r_- \gg 1$, 
$$|\frac{z_-}{z_+}\fa_{r_+, \fg_+}(\fc_+, \fc_\Gamma) - \fa_{r_-, \fg_-}(\fc_+, \fc_\Gamma)| \le c\cdot  r_{\text{max}}(1 - \frac{z_-}{z_+}) = c\cdot \frac{1}{z_+}r_{\text{max}}(r_+ - r_-)$$
for some constant $c$ that may depend on the geometry of $M_\phi$ as well as $d$ and $\fc_\Gamma$. Putting the bounds for I--IV together yields the desired inequality.
\end{proof}

Proposition \ref{prop:continuationEstimates5} below is a version of Proposition \ref{prop:continuationEstimates2} to hold in the case where $r_- = -2\pi\rho$. The proof method is identical, but we make use of bounds on the gauge-invariant function $\fa_{-2\pi\rho, \fg_s}$ instead of the function $\cL(s)$ from Lemma \ref{lem:continuationEstimates3}.

\begin{prop}
\label{prop:continuationEstimates5} Let $\cS_\pm=(\phi,J,\Gamma,r_\pm,\fg_\pm)$, $\cS_s=(0,J,r_s,\fg_s)$, $r_{\text{max}}$, $\Lambda_\fg$, $\fd=(B_s,\Psi_s)$, $\fc_\pm = (B_\pm,\Psi_\pm)$, $\fc_\Gamma = (B_\Gamma, \Psi_\Gamma)$ be as fixed at the start of \S\ref{subsubsec:continuationEstimates}. Suppose that $r_- = -2\pi\rho$ and $r_+ > r_-$. Then there are constants $\kappa_{\ref{prop:continuationEstimates5}} \geq 1$ and $r_{\ref{prop:continuationEstimates5}} \geq 1$ that depend only on $d$, $\fc_\Gamma$, and geometric constants,  such that the following holds. Assume that
$$\sup_{s \in \bR} |r_s - r_-| \leq \kappa_{\ref{prop:continuationEstimates5}}^{-1},\quad\Lambda_\fg \leq \kappa_{\ref{prop:continuationEstimates5}}^{-1},\quad r_\pm \geq r_{\ref{prop:continuationEstimates5}}.$$
Then
\begin{equation*}
\begin{split}\fa_{r_-, \fg_+}(\fc_+, \fc_\Gamma) - \fa_{r_-, \fg_-}(\fc_-, \fc_\Gamma) &\leq \kappa_{\ref{prop:continuationEstimates5}}r_+^2(|r_+ - r_-| + \Lambda_\fg).
\end{split}
\end{equation*}
If instead $r_- > r_+ = -2\pi\rho$, we have the bound
\begin{equation*}
\begin{split}
\fa_{r_+, \fg_+}(\fc_+, \fc_\Gamma) - \fa_{r_+, \fg_-}(\fc_-, \fc_\Gamma) &\leq \kappa_{\ref{prop:continuationEstimates5}}r_-^2(|r_+ - r_-| + \Lambda_\fg).
\end{split}
\end{equation*}
\end{prop} 

\begin{proof}
Suppose that $r_+ > r_- = -2\pi\rho$; the reverse case $r_- > r_+ = -2\pi\rho$ is identical. Define the gauge-invariant function
$\cL(s) = \fa_{r_-, \fg_s}(\fd_s, \fc_\Gamma).$
Proceeding in the same way as in Lemma \ref{lem:continuationEstimates3} shows that at any $s \in \bR$, either
$\cL'(s) \leq 0$ or $|\cL(s)| \lesssim r_+^2(|r_+ - r_-| + \Lambda_\fg)$, given that $\sup_{s \in \bR}|r_s - r_-|$ and $\Lambda_\fg$ are sufficiently small. This proves the proposition, since either $\fa_{r_-, \fg_+}(\fc_+, \fc_\Gamma) \leq \fa_{r_-, \fg_-}(\fc_-, \fc_\Gamma)$ or both $\fa_{r_-, \fg_+}(\fc_+, \fc_\Gamma)$ and $\fa_{r_-, \fg_-}(\fc_-, \fc_\Gamma)$ have absolute value $\lesssim r_+^2(|r_+ - r_-| + \Lambda_\fg)$.
\end{proof}

\subsection{Spectral flow} \label{subsec:spectralFlow}

In this section, we collect some estimates on the spectral flow of solutions to the Seiberg--Witten equations. The main observation is that, as long as our base connection satisfies suitable estimates, the spectral flow estimates from prior work of Cristofaro-Gardiner--Savale in \cite{CGSavaleSpectralFlow} and Taubes in \cite{TaubesWeinstein2} go through without difficulty. Fix an SW parameter set
$\cS = (\phi, J, \Gamma, r, \fe_\mu)$
with $\Gamma$ monotone with monotonicity constant $\rho < 0$ and degree $d$, $r > -2\pi\rho$, and $\mu$ having $C^3$ and $\cP$-norm less than $1$. Fix a nondegenerate irreducible solution $\fc = (B, \Psi)$ to the $\cS$-Seiberg-Witten equations. 

Let $\Lambda\ge 1$ be a real number such that $\|F_{B_\Gamma}\|_{C^3} \leq \Lambda d$ and $r\ge d/\Lambda.$
All constants in the estimates in \S\ref{subsec:spectralFlow} will only depend on the geometry of $M_\phi$ and the value of $\Lambda$, but not on the spin-c structure or the specific choice of base configuration $\fc_\Gamma$. 

\subsubsection{Cristofaro-Gardiner--Savale's $\eta$-invariant bound} \label{subsubsec:cgSavaleSpectralFlow} Cristofaro-Gardiner--Savale \cite[Equation $4.3$]{CGSavaleSpectralFlow} proved that the spectral flow from $\fc$ to $\fc_\Gamma$ differs from $\fcs(B, B_\Gamma)$ by $cr^{3/2}$ for some $r$-independent constant $c > 0$. We show in the proposition below that the constant $c$ is controlled by geometric constants and the $C^3$ norm of $F_{B_\Gamma}$. 

\begin{prop} \label{prop:spectralFlow1} 
Let $\cS = (\phi, J, \Gamma, r, \fe_\mu)$, $\fc = (B, \Gamma)$, $\fc_\Gamma = (B_\Gamma, \Psi_\Gamma)$, and $\Lambda$ be as fixed at the start of \S\ref{subsec:spectralFlow}. Then there is a constant $\kappa_{\ref{prop:spectralFlow1}} \geq 1$ depending only on $\Lambda$ and geometric constants such that
$$|\fcs(B, B_\Gamma) + 4\pi^2\SF(\fc, \fc_\Gamma)| \leq \kappa_{\ref{prop:spectralFlow1}}r^{3/2}.$$
\end{prop}

The key ingredients in the proof of Proposition \ref{prop:spectralFlow1} are bounds on the \emph{eta invariants} $\eta(D_B)$ and $\eta(D_{B_\Gamma})$ of the connections $B$ and $B_\Gamma$, respectively. See for example \cite[\S3]{CGSavaleSpectralFlow} for a definition of the eta invariant. 

\begin{prop} \label{prop:etaInvariant1}
Let $\cS = (\phi, J, \Gamma, r, \fe_\mu)$, $\fc = (B, \Psi)$, $\fc_\Gamma = (B_\Gamma, \Psi_\Gamma)$ be as fixed at the start of \S\ref{subsec:spectralFlow}. There is a geometric constant $\kappa_{\ref{prop:etaInvariant1}} \geq 1$  such that 
$$|\eta(D_B)| \leq \kappa_{\ref{prop:etaInvariant1}}\quad\text{and}\quad|\eta(D_{B_\Gamma})| \leq \kappa_{\ref{prop:etaInvariant1}}(\Lambda d)^{3/2}.$$
\end{prop}

\begin{proof}
The first bound is proved in \cite[Proposition $9$]{CGSavaleSpectralFlow} and the same proof works for the second bound.  We clarify for the reader why the claimed bounds on the constants hold in our setting; we explain this point in the case of the first bound, with the same considerations holding for the second.  The proof of Proposition $9$ in \cite{CGSavaleSpectralFlow} only uses the fact that $(B, \Psi)$ is a solution of the Seiberg--Witten equations via the pointwise estimates on solutions of the Seiberg--Witten equations, listed in Lemma $7$ of \cite{CGSavaleSpectralFlow}: bounds on the curvature $F_B$ are used in Lemmas $10$ and $11$  of \cite{CGSavaleSpectralFlow} to derive bounds on the coefficients in the small time asymptotic expansion of the Dirac heat kernel and another associated heat kernel, from which the $\eta$-invariant bound follows. These bounds only use estimates on the first three derivatives of the first two terms in the asymptotic expansion, for which a $C^3$ bound on the curvature suffices: in our setting, what is required is that the curvature $F_B$ has, for any $0 \leq k \leq 3$, a $C^k$ norm of $\kappa r^{1 + k/2}$. This bound on the curvature follows from standard bootstrapping of the bounds in Proposition~\ref{prop:3dEstimates1} via elliptic regularity (see, for example, \cite[Lemma 3.7]{LeeTaubes12}).   
\end{proof}

We use Proposition \ref{prop:etaInvariant1} and the Atiyah--Patodi--Singer index theorem \cite{APS1} to prove Proposition \ref{prop:spectralFlow1}. 

\begin{proof}[Proof of Proposition \ref{prop:spectralFlow1}]
 Recall from (\ref{eq:indexEqualsSF1}) that $\SF(\fc_\Gamma, \fc)$ can be computed in the following way. Take a one-parameter family $\fd = (B_s, \Psi_s)$ of connections on $E_\Gamma$ and spinors in $S_\Gamma$ such that $(B_s, \Psi_s) = \fc_\Gamma$ for $s \leq -1$ and $(B_s, \Psi_s) = \fc$ for $s \geq 1$. Then there is an associated Fredholm operator
$$\cL_\fd = \frac{\partial}{\partial s} + \cL_{(B_s, \Psi_s)}$$ from (\ref{eq:extended4DHessian}) and the spectral flow is computed as the index of this operator.

We observe that the operator $\cL_\fd$ is a Dirac operator of Atiyah--Patodi--Singer type on the cylinder $\bR \times M_\phi$ equipped with a cylindrical metric. Write $g$ for the Riemannian metric on $M_\phi$ induced by $(\phi, J)$. The results of \cite{APS1} show that it is equal to the expression
\begin{equation} \label{eq:apsCylindrical} -\frac{1}{4\pi^2}\fcs(B, B_\Gamma) - \SF(\cL_{(B, 0)}, \cL_{(B, \Psi)}) + \SF(\cL_{(B_\Gamma, 0)}, \cL_{(B_\Gamma, \Psi_\Gamma)}) - \eta(D_B) + \eta(D_{B_\Gamma}). \end{equation}

The estimate from \cite[Lemma $5.3$]{TaubesWeinstein1} along with the bounds from Proposition \ref{prop:3dEstimates1} shows that
\begin{equation} \label{eq:spinorSpectralFlow1} |\SF(\cL_{(B, 0)}, \cL_{(B, \Psi)})| \lesssim r^{3/2}, \qquad |\SF(\cL_{(B_\Gamma, 0)}, \cL_{(B_\Gamma, \Psi_\Gamma)})| \lesssim \|F_{B_\Gamma}\|_{C^0}^{3/2} \le (\Lambda d)^{3/2}. \end{equation}
Proposition \ref{prop:etaInvariant1} and the equations (\ref{eq:apsCylindrical}) and (\ref{eq:spinorSpectralFlow1}) together prove the desired spectral flow bound. 
\end{proof}

\subsubsection{Taubes' spectral flow bound} \label{subsubsec:taubesSpectralFlow} The following spectral flow estimate is a version of the one proved by Taubes in \cite{TaubesWeinstein2} with the dependence on the base connection $B_\Gamma$ made more explicit. 

\begin{prop} \label{prop:spectralFlow2} 
Let $\cS = (\phi, J, \Gamma, r, \fe_\mu)$, $\fc = (B, \Psi)$, $\fc_\Gamma = (B_\Gamma, \Psi_\Gamma)$, and $\Lambda$ be as fixed at the start of \S\ref{subsec:spectralFlow}. Then there is a constant $\kappa_{\ref{prop:spectralFlow2}} \geq 1$ depending only on $\Lambda$ and geometric constants such that
$$|\fcs(B, B_\Gamma) + 4\pi^2\SF(\fc, \fc_\Gamma)| \leq \kappa_{\ref{prop:spectralFlow2}}d^{4/3}r^{2/3}\log(r)^{\kappa_{\ref{prop:spectralFlow2}}}.$$
\end{prop}

    The proof of Proposition \ref{prop:spectralFlow2} follows the same arguments as \cite[Proposition 1.9]{TaubesWeinstein2}.      
According to the first paragraph of \cite[Section 3]{TaubesWeinstein2}, the spectral flow estimates in \cite{TaubesWeinstein2} holds for any smooth 1-form $a$ such that $|a|=1$ pointwise.  We take the 1-form $a$ to be $dt$ in our case. By Proposition \ref{prop:3dEstimates1}, we have
 \begin{equation}
 \label{eqn_d_same_order_as_E}
     2\pi d = i\int_{M_\phi}dt\wedge F_B =   \int _{M_\phi} r(1 - |\alpha|^2) + O(1),
 \end{equation}
 where the bound on the $O(1)$ term only depends on the geometry of $M_\phi$.
 Therefore, the term $E$ in \cite[Proposition 1.9]{TaubesWeinstein2} can be replaced by $2\pi d$ in our setting.

The only difference in Proposition \ref{prop:spectralFlow2} is that it considers a family of spin-c structures parametrized by $\Gamma$ and there is a base connection $B_\Gamma$ for each spin-c structure. Therefore, we need to prove that if $\|F_{B_\Gamma}\|_{C^0}\le \Lambda d$ and $r\ge d/\Lambda$, then the constants in \cite[Section 3]{TaubesWeinstein2} only depend on $\Lambda$ and are independent of the spin-c structure or the base connection.

There are exactly three steps in the proof of \cite[Proposition 1.9]{TaubesWeinstein2} that depend on the spin-c structure and the choice of base connection:
\begin{enumerate}
    \item The bound of $|\hat a|$ and $|\mathfrak{cs}(A)|$ in \cite[Lemma 1.8]{TaubesWeinstein2} and \cite[(3.1)]{TaubesWeinstein2}. 
    \item The estimate in Lemma 3.3 of \cite{TaubesWeinstein2}.
    \item The estimate in Lemma 3.4 of \cite{TaubesWeinstein2}.
\end{enumerate}
 The 1-form $\hat a$ and the Chern--Simons functional $\mathfrak{cs}(A)$ are denoted by $\widetilde b$ and $\mathfrak{cs}(\mathfrak{c},\fc_\Gamma)$ in our notation. Since we assume $r\ge d/\Lambda$, the desired uniform bounds for step (1) follow from Proposition \ref{prop:3dActionEstimates}. The rest of this section is devoted to establishing uniform estimates for steps (2) and (3).

As before, we write $\Psi = r^{1/2} \psi$. For $s\in [0,1]$, let $\mathcal{L}_s$ be the Dirac operator on $S_\Gamma$ defined by the connection $B_\Gamma+s(B-B_\Gamma)$, and let 
 $$\mathcal{L}_s': C^\infty(iT^*M_\phi\oplus S_\Gamma \oplus i\mathbb{R}) \to C^\infty(iT^*M_\phi\oplus S_\Gamma \oplus i\mathbb{R})$$
 be the operator defined by  
 \begin{equation}
 \label{eqn_def_L's} 
     \begin{pmatrix}
    b
    \\
    \eta
    \\ 
    f
 \end{pmatrix}
\mapsto
\begin{pmatrix}
    *db  -2^{-1/2}sr^{1/2}(\eta^\dagger\tau\psi+\psi^\dagger\tau\eta)  -df \\
    2^{1/2}sr^{1/2}\cl(b)\psi +   D_B\eta  + 2^{1/2} sr^{1/2}f\psi \\
    -d^*b +  2^{-1/2}sr^{1/2} (-\eta^\dagger \psi + \psi^\dagger \eta)
\end{pmatrix}.
\end{equation}
This is the extended Hessian operator
 in \cite[(1.7)]{TaubesWeinstein2} with $r$ replaced by $s^2r$.

 In the following, we will use $\mathcal{L}$ to denote the operator $\mathcal{L}$ or $\mathcal{L}_s$ for $s\in [0,1]$. We will use $\nabla$ to denote the spin-c connection associated with $B$ if $\mathcal{L} = \mathcal{L}'_s$, and use $\nabla$ to denote the spin-c connection associated with $B_\Gamma + s (B-B_\Gamma)$ if $\mathcal{L} = \mathcal{L}_s$.

 Recall that $S_\Gamma = E_\Gamma \oplus (E_\Gamma\otimes V)$. Let $\widecheck{\nabla}$ be the orthogonal projection of $\nabla$ to $E_\Gamma\otimes V$, let $\widehat{\nabla}$ be the orthogonal projection of $\nabla$ to $E_\Gamma$, let $\mathcal{I} = \nabla - (\widehat{\nabla}\oplus \widecheck{\nabla})$. Then $\mathcal{I}$  is a pointwise linear map given by the tensor product of a map on $\mathbb{C}\oplus V$  with the identity map on $E_\Gamma$, where the map on $\mathbb{C}\oplus V$ only depends on the geometry of $M_\phi$. Moreover, the ``diagonal'' part of $\mathcal{I}$ vanishes; namely, $\mathcal{I}$ takes a section of $E_\Gamma$ to a section of $E_\Gamma\otimes V$ and takes a section of $E_\Gamma\otimes V$ to a section of $E_\Gamma$. Recall that $R$ is defined to be the unique vector field on $M_\phi$ such that $dt(R) \equiv 1$ and $\omega_\phi(R, -) \equiv 0$. The following lemma establishes a uniform bound for the constant in \cite[Lemma 3.3]{TaubesWeinstein2}.

\begin{lem}
\label{lem_global_L2_j_LambdaGamma}
Let
$\cS = (\phi, J, \Gamma, r, \fe_\mu),\fc = (B, \Psi),d,r,\mathcal{L},\nabla$ be as above.
 Let $\Lambda_\Gamma = \|F_{B_\Gamma}\|_{C^0}$.  Let $\lambda> 0$ be an arbitrary constant.  Suppose $j$ is a linear combination of eigenvectors of $\mathcal{L}$ with eigenvalues in $[-\lambda,\lambda]$.  Then
    \begin{equation}
    \label{eqn_lem_global_L2_j_LambdaGamma}
    \|\nabla_R j\|_{L^2}\lesssim (\lambda^2 + \Lambda_\Gamma +d )^{1/2} \|j\|_{L^2}.
    \end{equation}
\end{lem}

\begin{proof}  Without loss of generality, assume $\|j\| = 1$. Recall that by the assumptions of Proposition \ref{prop:spectralFlow2}, we have $d> 0$. We discuss two cases.

\paragraph{Case 1:} $\mathcal{L} = \mathcal{L}_s$ for some $s\in [0,1]$. Let $\mathfrak{q} = \mathcal{L}^2 j$, then $\|\mathfrak{q}\|_{L^2}\le \lambda^2$. By the Weitzenboch formula,
\begin{equation}
\label{eqn_Weitzenboch_Dirac}
    \nabla^*\nabla j + s\cdot\cl( F_B ) + (1-s) \cdot \cl(F_{B_\Gamma}) + \cl(\mathfrak{R}) j = \mathfrak{q},
\end{equation}
where $\mathfrak{R}$ only depends on the geometry of $M$. Write $j=(j_0,j_1)$ and $\mathfrak{q} = (\mathfrak{q}_0,\mathfrak{q}_1)$ with respect to the decomposition $S_\Gamma = E_\Gamma \oplus (E_\Gamma\otimes V)$.  Taking the inner product of \eqref{eqn_Weitzenboch_Dirac} with $(0,j_1)$ yields:
$$
\int \langle \nabla j, \nabla (0,j_1)\rangle + \langle  s\cdot \cl( F_B ) + (1-s) \cdot \cl(F_{B_\Gamma}) + \cl(\mathfrak{R}) j, (0,j_1)\rangle = \int \langle \mathfrak{q_1}, j_1\rangle.
$$
Since
$$
\int \langle \nabla j, \nabla (0,j_1)\rangle = \|\widecheck{\nabla} j_1\|_{L^2}^2 + \| \mathcal{I}j_1 \|_{L^2}^2 +
\int \langle \mathcal{I}j_0, \widecheck{\nabla} j_1 \rangle + \langle \widehat{\nabla} j_0, \mathcal{I} j_1 \rangle
$$
and $\|j_1\|_{L^2} \le \|j\|_{L^2} =1, \|\mathfrak{q}_1\|_{L^2}\le \|\mathfrak{q}\|_{L^2} \le \lambda^2$,
we have
\begin{equation}
\label{eqn_weitzenboch_nabla'j1}
    \|\widecheck{\nabla} j_1\|_{L^2}^2 +
\int \langle \mathcal{I}j_0, \widecheck{\nabla} j_1 \rangle + \langle \widehat{\nabla} j_0, \mathcal{I} j_1 \rangle
\le 
(\lambda^2 + \Lambda_\Gamma + \kappa_1) - s \int \langle\cl(F_B)j, j_1\rangle
\end{equation}
for some constant $\kappa_1$. Recall that
$$
    \cl(F_B)
\begin{pmatrix}
    j_0 \\
    j_1
\end{pmatrix}
=
\begin{pmatrix}
    r(|\alpha|^2-|\beta|^2-1)/2& \alpha\beta^*\\
     \alpha^* \beta & -r(|\alpha|^2-|\beta|^2-1)/2
\end{pmatrix}
\begin{pmatrix}
    j_0 \\
    j_1
\end{pmatrix} + O(1).
$$
By the first item in Proposition \ref{prop:3dEstimates1}, this implies
\begin{equation}
\label{eqn_lower_bound_<cl(F_B)j,j1>}
\int \langle\cl(F_B)j, j_1\rangle \ge - \kappa_2\|j\|_{L^2}^2  = - \kappa_2.
\end{equation}
We also have
\begin{align*}
& \Big| \int \langle \mathcal{I}j_0, \widecheck{\nabla} j_1 \rangle + \langle \widehat{\nabla} j_0, \mathcal{I} j_1 \rangle \Big| = \Big| \int \langle \mathcal{I}j_0, \widecheck{\nabla} j_1 \rangle + \langle  j_0, (\widehat{\nabla})^*(\mathcal{I} j_1) \rangle \Big| \le \frac12 \|\widecheck{\nabla} j\|_{L^2}^2 + \kappa_3.
\end{align*}
Hence by \eqref{eqn_weitzenboch_nabla'j1}, we have
$\|\widecheck{\nabla} j_1\|_{L^2}^2\lesssim \lambda^2 + \Lambda_\Gamma + 1.$
Since
$\|\mathcal{I} j\|_{L^2}\lesssim 1,$
we conclude that
\begin{equation}
\label{eqn_L2_bound_nabla_j1_LambdaGamma}
    \|\nabla j_1\|_{L^2}^2\lesssim \lambda^2 + \Lambda_\Gamma + 1.
\end{equation}

By the definition of $j$, we have $\|\mathcal{L}j\|\le \lambda$. By the definition of the Dirac operator, $\nabla_R j_0$ equals a linear combination of the components of $\nabla j_1$ and $\mathcal{L}j$.  Hence by \eqref{eqn_L2_bound_nabla_j1_LambdaGamma}, we have
\begin{equation}
\label{eqn_L2_bound_nabla_j0_LambdaGamma}
    \|\nabla_R j_0\|_{L^2}^2\lesssim \lambda^2 + \Lambda_\Gamma + 1,
\end{equation}
  and the desired estimate \eqref{eqn_lem_global_L2_j_LambdaGamma} follows from \eqref{eqn_L2_bound_nabla_j1_LambdaGamma} and \eqref{eqn_L2_bound_nabla_j0_LambdaGamma}.

\paragraph{Case 2:} $\mathcal{L} = \mathcal{L}'_s$ for some $s\in [0,1]$.  This case is proved by repeating the argument in \cite[p. 1373-1374]{TaubesWeinstein2}, replacing all instances of the energy bound ``$E$'' with $d$, since none of the estimates therein depend on the choice of base connection $B_\Gamma$.  
\end{proof}

The next two lemmas establish a uniform bound for the constant in \cite[Lemma 3.4]{TaubesWeinstein2} when $\mathcal{L} = \mathcal{L}_s$.

\begin{lem}
\label{lem_holder_estimate_1/2_j}
Suppose $\mathcal{L} = \mathcal{L}_s$ for $s\in[0,1]$. Let $j$ be a section of $S_\Gamma$, let $\fq = \cL^2 j$.
    Suppose $\rho>0$ is less than the injectivity radius of $M$ and $z$ is a point on $M$. Let $e_z$ be the maximum of $|F_B|+|F_{B_\Gamma}|+1$ on $B_\rho(z)$.  Then
    $$
    \big||j(z)| - |j(x)|\big|\lesssim |x-z|^{1/2} [(e_z^2\rho^2 +\rho^{-2})\|j\|_{L^2(B_\rho(z))} + (\rho^{2} e_z + 1)\|\frak{q}\|_{L^2(B_\rho(z))}].
    $$
\end{lem}

\begin{proof}
To simplify notation, we will write $B_\tau = B_{\tau\rho}(z)$ for $\tau\in(0,1]$. Trivialize the line bundle $E_\Gamma$ over $B_1$ by parallel transport via $B_\Gamma + s(B-B_\Gamma)$ along the radial geodesics from $z$. Then the connection $B$ is given by $d + \hat b$ with
$|\hat b| \lesssim \rho \, e_z ,$
$|\partial \hat b| \lesssim  e_z, $
where $\partial$ denotes the partial derivatives in the normal coordinate chart.
Fix an arbitrary trivialization of the bundle $V$ on $B_1$.  This identifies $j$ and $\fq$ with smooth functions on $B_1$ valued in $\mathbb{C}^2$.

The Weitzenbock formula for $\cL^2$ now takes the form
\begin{equation}
\label{eqn_weitzenboch_local_trivialized}
d^*dj = \Gamma_0\cdot j + \Gamma_1\cdot dj + \fq.    
\end{equation}
Here $\Gamma_0$ is an operator coming from several sources: they are multiplication by $\hat b^*\hat b$, Clifford multiplication by $sF_B +(1-s)F_{B_\Gamma}$, and an operator arising from the curvature of $M$.
This implies the pointwise bound
$$
|\Gamma_0|\lesssim e_z + \rho^2\,e_z ^2.
$$
The term $\Gamma_1$ is a linear combination of the components of $\hat b$, which implies the pointwise bound
$$
|\Gamma_1|\lesssim \rho\,e_z.
$$

We first  estimate the $L^2$ norm of $dj$ using \eqref{eqn_weitzenboch_local_trivialized}. Let $\chi$ be a radial cutoff function on $B$ which takes values in $[0,1]$, is equal to $1$ on $B_{1/2}$, is equal to $0$ outside of $B_{2/3}$, and satisfies $|\nabla\chi|\le 10\rho^{-1}$. Take the inner product of \eqref{eqn_weitzenboch_local_trivialized} with $\chi^2j$ and integrate by parts over $B$. We have
$$
\int \langle dj ,d (\chi^2 j) \rangle = \int \langle \Gamma_0\cdot j + \Gamma_1\cdot dj + \fq, \chi^2 j \rangle,
$$
therefore
\begin{align*}
\|\chi dj\|_{L^2}^2 
& \le \int |\langle dj, d(\chi^2)j\rangle| +  \chi^2 (e_z + \rho^2 e_z^2) |j|^2 + \chi^2 \rho e_z  |dj|\cdot |j| + \chi^2|\fq| \cdot |j|
\\  
& \le
 \frac12 \|\chi dj\|_{L^2}^2 + \kappa \bigg(
    \|(d \chi) j\|_{L^2}^2 + (e_z + \rho^2 e_z^2)  \|\chi j \|_{^2}^2 + \rho e_z \|\chi j\|_{L^2}^2 + \|\chi \fq\|_{L^2}\|\chi j \|_{L^2}\bigg).
\end{align*}
Recall that $\rho$ is no greater than the injectivity radius of $M$, thus $\rho\lesssim 1\le e_z$. Hence
$$
\|\chi dj\|_{L^2}^2   \lesssim
    \|(d \chi) j\|_{L^2}^2 + (e_z + \rho^2 e_z^2)  \|\chi j \|_{L^2}^2  + \|\chi \fq\|_{L^2}\|\chi j \|_{L^2}.
$$
Since
$\|\chi \fq\|_{L^2}\|\chi j \|_{L^2}\lesssim \rho^2\|\chi \fq\|_{L^2}^2 + \rho^{-2} \|\chi j \|_{L^2},$
we have
\begin{equation}
\label{eqn_bound_chi_dj_L2}
\|\chi dj\|_{L^2}\lesssim (\rho^{-1} + e_z^{1/2} + \rho e_z)\cdot \|j\|_{L^2(B)} +  \rho\|\fq\|_{L^2(B)}.
\end{equation}

Now for any point $x$ in $B_{1/2}$, let $G_x$ denote the Green's function for $d ^*d$ with Dirichlet boundary condition with singularity at $x$. Then we have
$\|G_x-G_z\|_{L^2} \lesssim |x-z|^{1/2},$
and hence
\begin{equation*}
    \big||j(z)| - |j(x)|\big| \le \Bigg|\int (G_z-G_x) d^*d(\chi^2 j) \Bigg| \lesssim |x-z|^{1/2} \cdot \|d^*d(\chi^2 j)\|_{L^2}.
\end{equation*}
By \eqref{eqn_bound_chi_dj_L2},
\begin{align*}
     \|d^*d(\chi^2 j)\|_{L^2}
    &
    \lesssim \|\chi^2 d^*dj\|_{L^2} + \||d(\chi^2)|\cdot |dj|\|_{L^2}  + \|d^*d(\chi^2)j\|_{L^2}
    \\
    &   \lesssim \|\chi^2\Gamma_0\cdot j\|_{L^2} +
        \|\chi^2 \Gamma_1\cdot dj\|_{L^2} + \|\chi^2 \fq\|_{L^2}
    + \rho^{-1} \|\chi dj\|_{L^2} + \rho^{-2} \|j\|_{L^2(B)} 
    \\
    & \lesssim 
    ( e_z + e_z^2\rho^2)\|j\|_{L^2(B)}
        + \rho e_z[(\rho^{-1} + e_z^{1/2}  + \rho e_z)\cdot \|j\|_{L^2(B)} + \rho \|\fq\|_{L^2(B)}] \\
    & \quad 
        + \|\fq\|_{L^2(B)}
        + \rho^{-1} [(\rho^{-1} + e_z^{1/2}  + \rho e_z)\cdot \|j\|_{L^2(B)} + \rho \|\fq\|_{L^2(B)}]
        +  \rho^{-2} \|j\|_{L^2(B)}
        \\
    & \lesssim  (e_z + e_z^2\rho^2 + \rho\, e_z^{3/2}  + \rho^{-2} + \rho^{-1} e_z^{1/2}) \|j\|_{L^2(B)}
        + (\rho^{2} e_z + 1) \|\fq\|_{L^2(B)}
        \\
    & \lesssim  (e_z^2\rho^2 + \rho^{-2}) \|j\|_{L^2(B)} + (\rho^{2} e_z + 1) \|\fq\|_{L^2(B)}.
\end{align*}  Therefore the lemma is proved.
\end{proof}

As a consequence, we have the following interior bound on the $L^2$ norm of $j$.
\begin{lem}
\label{lem_interior_L2_j_LambdaGamma}
Suppose $\mathcal{L} = \mathcal{L}_s$ for $s\in[0,1]$. Let $j$ be a section of $S_\Gamma$, let $\fq = \cL^2 j$.
    Suppose $\rho>0$ is less than the injectivity radius of $M$ and $z$ is a point on $M$ such that $j(z) = 0$.
     Suppose $\|F_{B_\Gamma}\|_{C^0(B_\rho(z))}\le \Lambda d$, $\rho \le d^{-1/2}$, $\|r(1-|\alpha|^2)\|_{C^0(B_\rho(z))}\le \rho^{-2}$. Then there exists a constant $c$ depending only on $\Lambda$ and the geometry of $M$ such that
    $$
    \|j\|^2_{B_{\sigma\rho}(z)} \le c \sigma^4 (\|j\|_{L^2(B_\rho(z))}^2 + \rho^4 \|\frak{q}\|_{L^2(B_\rho(z))}^2).
    $$
\end{lem}
\begin{proof}
    Let $e_z$ be the maximum of $|F_B|+|F_{B_{\Gamma}}|+1$ on $B_\rho(z)$. By the assumptions and Proposition \ref{prop:3dEstimates1}, there exists a constant $c_1$ depending only on $\Lambda$ and the geometry of $M_\phi$ such that $\rho^2 e_z\le c_1$.
    By Lemma \ref{lem_holder_estimate_1/2_j} and the assumption that $j(z)=0$,
    \begin{align*}
    \|j\|^2_{B_{\sigma\rho}(z)} & \lesssim \sigma^4 [(e_z^4\rho^8 + 1)\|j\|_{L^2(B_\rho(z))}^2 + \rho^4(\rho^4e_z^2 + 1)\|\frak{q}\|_{L^2(B_\rho(z))}^2]
    \\
    &
    \le (c_1^4+c_1^2 + 1)\sigma^4 (\|j\|_{L^2(B_\rho(z))}^2 + \rho^4 \|\frak{q}\|_{L^2(B_\rho(z))}^2),
    \end{align*}
    and the lemma is proved.
\end{proof}

Now we can finish the proof of Proposition \ref{prop:spectralFlow2}.

\begin{proof}[Proof of Proposition \ref{prop:spectralFlow2}]
As discussed after the statement of Proposition  \ref{prop:spectralFlow2}, the desired spectral flow estimate is proved in \cite[Section 3]{TaubesWeinstein2}, and we only need to show that the constants in \cite[Lemma 3.3]{TaubesWeinstein2} and \cite[Lemma 3.4]{TaubesWeinstein2} depend only on $\Lambda$ and the geometry of $M_\phi$ and are independent of the spin-c structure or the base connection $B_\Gamma$.

Recall that by \eqref{eqn_d_same_order_as_E}, we have $E=2\pi d + O(1)$. By the assumptions of Proposition \ref{prop:spectralFlow2}, we have $d\ge 1$. Therefore, taking $\lambda = d^{1/2}$ in Lemma \ref{lem_global_L2_j_LambdaGamma} shows that the constant in \cite[Lemma 3.3]{TaubesWeinstein2} only depends on $\Lambda$ and the geometry of $M_\phi$.  Lemma \ref{lem_interior_L2_j_LambdaGamma} shows that the constant in \cite[Lemma 3.4]{TaubesWeinstein2} only depends on $\Lambda$ and the geometry of $M_\phi$ if $\mathcal{L} = \mathcal{L}_s$.  When $\mathcal{L} = \mathcal{L}_s'$ for $s\in [0,1]$, the statement of \cite[Lemma 3.4]{TaubesWeinstein2} is a local estimate independent of the base connection, and its proof in \cite{TaubesWeinstein2} only makes use of the estimates for $\alpha$ and $\beta$ that are given by Proposition \ref{prop:3dEstimates1} in our setting. Therefore, the constant in \cite[Lemma 3.4]{TaubesWeinstein2} does not depend on the spin-c structure or the base connection $B_\Gamma$ when $\mathcal{L} = \mathcal{L}_s'$.
\end{proof}

\subsection{$(\delta, d)$-approximations} \label{subsec:deltaDEstimates} In this section, we provide details of the construction of $(\delta, d)$-approximations. We then show that the PFH and Seiberg-Witten spectral invariants do not change by much after taking a $(\delta, d)$-approximation.

\subsubsection{Details of the construction} \label{subsubsec:deltaDDetails} Fix a twisted PFH parameter set $\bfS = (\phi, \Theta_{\text{ref}}, J)$ such that the twisted PFH group $\TWPFH_*(\bfS)$ is well-defined and the class $[\Theta_{\text{ref}}]$ has degree $\leq d$. Recall from \S\ref{subsubsec:dFlatApproximations} that a $(\delta, d)$-approximation consists of a pair $(\phi_*, J_*)$ where $\phi_* \in \Diff(\Sigma, \omega)$ and $J_* \in \cJ(dt, \omega_\phi)$, and a one-parameter family $\{H^s\}_{s \in [0,1]}$ of Hamiltonians such that the time-one maps $\wt\phi(s) = \phi^1_{H^s}$ generate a distinguished isotopy from $\phi$ to $\phi_*$. We give details of the construction of $(\phi_*, J_*)$ and $\{H^s\}$ in several steps. 

\textbf{Step 1:}  We begin by summarizing the basic structure of a $(\delta,d)$-approximation.

The map $x \to (0,x) \in \Sigma \times [0,1]$ gives a canonical embedding of the surface $\Sigma$ onto the fiber of the mapping torus $M_\phi$ over the point $0 \in \bR/\bZ$. This defines a natural identification of the vertical tangent bundle $V$ along the image of this embedding with $T\Sigma$. Write $J_\Sigma$ for the almost-complex structure on $T\Sigma$ induced by $J$.  It now follows that there is a constant $\delta_* > 0$ depending only on the area form $\omega$ such that, for any point $p \in \Sigma$, there is a symplectic coordinate chart
$$\tau_p: \mathbb{C} \supset B_{\delta_*}(0) \to \Sigma$$
such that $\tau_p(0) = p$ and $D\tau_p^{-1} \circ J_\Sigma \circ D\tau_p = i$. 

Write $\Lambda_d$ for the set of all periodic points of period less than or equal to $d$ in $\Sigma$. For further reference, as long as $\delta_*$ is sufficiently small, we find that the following is true. For any $\delta \in (0, \delta_*)$, the symplectic coordinate charts fixed above define a disjoint union of embedded disks
$$\bD_\delta = \sqcup_{p \in \Lambda_d} \tau_p(B_\delta(0)) \subset \Sigma.$$

We will now choose a parameter $\delta \ll \delta_*$. Lee--Taubes showed in \cite[Lemma $2.1$]{LeeTaubes12} that, given a fixed $d$, there exists, a map $\phi_* \in \Diff(\Sigma, \omega)$ and a distinguished Hamiltonian isotopy $\wt\phi_*$ from $\phi$ to $\phi_*$ such that the following holds:
\begin{enumerate}
\item For any $s \in [0,1]$, the map $\wt\phi_*(s)$ is $d$-nondegenerate and has the same set of periodic orbits of period less than or equal to $d$ as the diffeomorphism $\phi$. Moreover, a periodic point of $\wt\phi_*(s)$ of period less than or equal to $d$ is positive/negative hyperbolic or elliptic if and only if it is positive/negative hyperbolic or elliptic, respectively as a periodic point of $\wt\phi(s)$. If it is elliptic, it has the same rotation number with respect to $\wt\phi_*(s)$ as with respect to $\phi$.
\item For any $s \in [0,1]$, the map $\wt\phi_*(s)$ agrees with $\phi$ at any point outside of the union of coordinate disks $\bD_\delta \subset \Sigma$.
\item The diffeomorphism $\phi_*$ has a specific normal form inside the union of coordinate disks $\bD_\delta \subset \Sigma$.
\end{enumerate}

The first, second and third items correspond to the third, fifth and fourth points, respectively of \cite[Lemma $2.1$]{LeeTaubes12}; for details regarding the third item, we refer the reader to \cite[Section $2.1$]{LeeTaubes12}. Note that by the second item, as long as $\delta$ is small enough, $\phi_*$ differs from $\phi$ only on $\bD_\delta$. This implies $\phi_*$ is Hamiltonian isotopic to $\phi$, and moreover we can find a Hamiltonian generating the isotopy which is compactly supported in $\bD_\delta$. 

\textbf{Step 2:} To prove what we want to know about these approximations, we will need to recall some facts about the construction of the map $\phi_*$ and the isotopy $\wt\phi_*$ from \cite{LeeTaubes12}.   
 
We first elaborate on the normal form near the relevant periodic points before perturbing.  As long as $\delta$ is sufficiently small, Lemma $2.2$ in \cite{LeeTaubes12} produces, for each embedded Reeb orbit $\gamma$ in $M_\phi$ of period less than or equal to $d$, a tubular neighborhood embedding 
$$\tau_\gamma: \bR/\bZ \times \mathbb{C} \supset \bR/\bZ \times B_{100\delta}(0) \to M_\phi$$ 
of $\gamma$ such that the images of the maps $\tau_\gamma$ are all disjoint.   Let $\theta$ be the coordinate on $\bR/\bZ$ and $z$ be the complex coordinate on $\mathbb{C}$. Then there is a constant $\kappa = \kappa(\phi, J) \geq 1$ such that the following facts hold regarding the tubular neighbhorhood embedding $\tau_\gamma$ for any embedded Reeb orbit $\gamma$ of period $q \leq d$. There are certain smooth functions $\nu_\gamma \in C^\infty(\bR/\bZ, \bR)$ and $\mu_\gamma \in C^\infty(\bR/\bZ, \mathbb{C})$ associated to $\gamma$, defined in \cite[Definition $2.1$]{LeeTaubes12}
and which we make reference to without comment, such that the following holds:
\begin{enumerate}
\item The composition of $\tau_\gamma$ with the projection $M_\phi \to \bR/\bZ$ is the map $(\theta, z) \mapsto q\theta$.
\item For any $\gamma$, there are smooth functions $\nu \in C^\infty(\bR/\bZ, \bR)$ and $\mu \in C^\infty(\bR/\bZ, \mathbb{C})$ such that 
$$(D\tau_\gamma)^{-1}(q^{-1}R) = \partial_\theta - 2i(\nu z + \mu \overline{z} + \fe)\partial_z + 2i(\nu \overline{z} + \overline{\mu} z + \overline{\fe})\partial_{\overline{z}}$$
where $\fe$ is a smooth, complex-valued function on $\bR/\bZ \times B_{100\delta}(0)$ such that 
$$|\fe| + |z||\nabla \fe| \leq \kappa |z|^2.$$
\item For any $\tau_\gamma$ and any point in $\bR/\bZ \times \{0\}$, we have
$(J \circ D\tau_\gamma)(\partial_z) = D\tau_\gamma(i\partial_z).$
\item The restriction of $(\tau_\gamma)^*\omega_\phi$ to $\{p\}\times B_{100\delta}(0)$ is given by $\frac{i}{2}dz \wedge d\overline{z}$  for every $p\in \mathbb{R}/\mathbb{Z}$.
\end{enumerate}

Properties (1)-(3) above are explicitly stated in \cite[Lemma 2.2]{LeeTaubes12}. Property (4) above follows from the construction of the embedding in \cite[Section 2.3]{LeeTaubes12}, which we now explain. Our map $\tau_\gamma$ was denoted by $\varphi_\gamma$ in \cite{LeeTaubes12}. It was stated in \cite[Section 2.3, Part 1]{LeeTaubes12} that the maps $\psi_\tau$ are area-preserving, which implies that the map $\varphi_1$ preserves the area on each slice $\{p\}\times B_{100\delta}(0)$, and hence the map $\varphi_\gamma$ defined by \cite[Definition 2.3]{LeeTaubes12} satisfies property (4) above. 

We next explain the desired form for the element $\phi_*$ of the approximation, in these coordinates.  Use $\tau_\gamma$ to identify the disks $\{ \{k/q\} \times B_{100\delta}(0)\}$ with embedded disks in $\Sigma$, by identifying $\Sigma$ with the fiber over $0$ as above.
Then the construction in \cite{LeeTaubes12} gives a smooth, complex-valued function $\fe'$, equal to $0$ when $z$ is close to zero and equal to $\fe$ when $|z| > \delta$, such that, under this identification, the map $\phi_*$ is the time $1/q$-flow of the vector field
\begin{equation}
\label{eqn:normalform}
v_\gamma = \partial_\theta - 2i(\nu_\gamma z + \mu_\gamma \overline{z} + \fe')\partial_z + 2i(\nu_\gamma \overline{z} + \overline{\mu}_\gamma z + \overline{\fe}')\partial_{\overline z}.
\end{equation}

To prepare for the discussion of what we need to know about the isotopy, we first elaborate on the construction of $\phi_*$.  In Part $2$ of Section $2.3$ of \cite{LeeTaubes12}, the authors construct for each embedded Reeb orbit $\gamma$ a certain family $\{u_{\gamma, \theta}\}_{\theta \in [0,1)}$ of area-preserving maps with domain a small disk $D$ centered at the origin in $\mathbb{C}$ and range $B_{100\delta}(0) \subset \mathbb{C}$.
The family is chosen such that, for a smaller disk $D^* \subset D$, and any $k$, the map $u_{\gamma, k/q}$ agrees with $\tau_\gamma^*\phi$ on $\{k/q\} \times (D \setminus D^*)$.
Moreover, $u_{\gamma, 0}$ acts by the identity map and each map fixes the origin. We now set $\phi_*$ to be equal to $\phi$ outside of any disk in $\Sigma$ of the form $\tau_\gamma(\{k/q\} \times B_{100\delta}(0))$ for an embedded Reeb orbit $\gamma$ of period $q \leq d$. Inside such a disk, we set it to be $(\tau_\gamma^{-1})^*u_{\gamma, k/q}$. We recall as well, from \cite[Eq. 2.7]{LeeTaubes12} and the following discussion in \cite[Section $2.3$]{LeeTaubes12} that the vector field $v_\gamma$ is the sum of $\partial_\theta$ and the time-dependent vector field generating the family of maps $u_{\gamma, \theta}$. 

Now we construct the distinguished Hamiltonian isotopy $\widetilde\phi_*$ from $\phi$ to $\phi_*$. For any $s \in [0,1]$, we fix the map $\widetilde\phi_*(s) \in \Diff(\Sigma, \omega)$ to be equal to $\phi$ outside of any disk in $\Sigma$ of the form $\tau_\gamma(\{k/q\} \times B_{100\delta}(0))$ for an embedded Reeb orbit $\gamma$ of period $q \leq d$. Inside such a disk, we set it to be $(\tau_\gamma^{-1})^*u_{\gamma, s \cdot k/q}$.

We end this step by collecting some key facts about this isotopy.  For any $s \in [0,1]$, the map $\phi^{-1}\widetilde\phi_*(s)$ is supported on a union of embedded disks in $\Sigma$. It follows that $\widetilde\phi_*$ is a Hamiltonian isotopy, generated by a Hamiltonian $H$ compactly supported in this union of embedded disks. We also note that the fact that each map $u_{\gamma, \theta}$ fixes the origin implies that each area-preserving map $\widetilde{\phi}_*(s)$ has the same periodic points of period less than or equal to $d$ as $\phi$ itself.

\textbf{Step 3:} This step discusses the Hamiltonian $H$ generating the isotopy $\widetilde\phi_*(s)$ and its analytic properties. Recall from Step $2$ that, for any $s \in [0,1]$, $\widetilde\phi_*(s)\phi^{-1}$ is supported in the union of embedded coordinate disks $\bD_\delta$ of radius $\delta$ centered on the set $\Lambda_d$ of periodic points of period less than or equal to $d$. We conclude that there is a Hamiltonian $H \in C^\infty(\bR/\bZ \times \Sigma)$, compactly spported in $\bR/\bZ \times \bD_\delta$, generating the isotopy $\widetilde\phi_*(s)$ from $\phi$ to $\phi_*$. We fix a choice of such a Hamiltonian. It will be very important for later arguments to note that the Hamiltonian $H$ satisfies the following bounds.

\begin{lem}
\label{lem:dFlatBounds} Let $(\phi_*, J_*)$ be a $(\delta, d)$-approximation of $(\phi, J)$ with distinguished isotopy $\wt\phi_*$, generated by a Hamiltonian $H \in \cH(\Sigma, \bD_\delta)$ as above. There is a constant $\kappa_{\ref{lem:dFlatBounds}} = \kappa_{\ref{lem:dFlatBounds}}(\phi, J, g, d) \geq 1$ such that 
$$\|H\|_{C^0} + \|dH\|_g \leq \kappa_{\ref{lem:dFlatBounds}}\delta.$$
\end{lem}

\begin{proof}
Fix any embedded Reeb orbit $\gamma$ of period $q \leq d$. By definition, the Hamiltonian vector field $X$ generating the isotopy $\wt\phi_*$ in the tubular neighborhood $\tau_\gamma(\bR/\bZ \times B_{100\delta}(0))$ is equal to
$v_\gamma - (D\tau_\gamma)^{-1}(q^{-1}R).$

It vanishes outside the union of these tubular neighborhoods. Recall that local coordinate expressions for both vector fields were presented in Step $2$. Using these local coordinate expressions, we deduce that
$$X = 2i( (\nu - \nu_\gamma)z + (\mu - \mu_\gamma)\overline{z})\partial_z + 2i( (\nu_\gamma - \nu)\overline{z} + (\overline{\mu}_\gamma - \mu)z)\partial_{\overline{z}} + \fe_X$$
where $\fe_X$ denotes a vector field satisfying a pointwise bound of the form $\|\fe_X\|_g \leq \kappa(\phi, J)|z|^2$.

It follows that there is a constant $\kappa \geq 1$ depending only on the constants $\nu$, $\mu$, $\nu_\gamma$, $\mu_\gamma$ across all periodic orbits $\gamma$ of period less than or equal to $d$, along with the metric $g$, so that we have a bound of the form $$\|X\|_g \leq \kappa \delta.$$

By definition, the exterior derivative $dH$ of the Hamiltonian $H$ satisfies the equation $\omega(X, -) = -dH$. From the bound on $\|X\|_g$ and the fact that $\|\omega\|_g \equiv 2$ we conclude the desired bound on $\|dH\|_g$.  The bound on $\|H\|_{C^0}$ then follows from the prior bound and the fact that $H$ is compactly supported in $\bD_\delta \subset \Sigma$; it can be constructed on each disk in $\bD_\delta$ by integrating $dH$ radially from the center and then subtracting a constant which is $\lesssim \delta$. 
\end{proof}

\textbf{Step $4$:} The fourth step fixes a useful identification of the mapping torus $M_\phi$ with the mapping torus $M_{\wt\phi_*(s)}$ for every $s \in [0,1]$.

Fix a smooth function
$h: [0,1]_s \times [0,1]_t \to [0,1]$
satisfying the following properties:
\begin{enumerate}
\item For every $s$, we have $h(s,t) = 0$ for $t$ near $0$ and $h(s,t) = s$ for $t$ near $1$.
\item For every $(s,t)$, the $t$-derivative $h_t(s,t)$ lies in the interval $[0,4s]$.
\end{enumerate}

The data of the Hamiltonian $H$ from the previous step and the smooth function $h$ produces a smooth family $\{H^s\}_{s \in [0,1]}$ of Hamiltonians supported in $\bD_\delta$, defined by
\begin{equation}
\label{eqn:defnhfamily}
H^s(t, p) = h_t(s,t)H^s(h(s,t), p). 
\end{equation}

Note that $H^1$ is not, in general, equal to the Hamiltonian $H$. However, the isotopy it generates from $\phi$ to $\phi_*$ is equal to the isotopy $\wt\phi_*$ pre-composed with a reparameterization of the domain of the isotopy via the map $h(1, -)$. We observe for any fixed $s \in [0,1]$ that the Hamiltonian $H^s$ generates an isotopy from $\phi$ to $\wt\phi_*(s)$, given by the map
$t \mapsto \wt\phi_*(h(s,t)).$
It consequently defines a diffeomorphism
$$M_{H^s}: M_\phi \to M_{\wt\phi_*(s)}.$$

\textbf{Step 5:} This fifth step collects those facts that we will need about the approximation $J_*$ of $J$ in order to relate the twisted PFH chain complexes.  
\begin{enumerate}
\item $J_*$ agrees with $J$ outside of the union of the images of $\tau_\gamma$ over all embedded Reeb orbits in $M_\phi$ of period less than or equal to $d$.
\item The pushforward $J_*^{H^1}$ lies in the generic subset $\cJ^\circ(dt, \omega_{\phi_*}) \subset \cJ(dt, \omega_\phi).$
\item For any $\tau_\gamma$, $D\tau_\gamma^{-1} \circ J_* \circ D\tau_\gamma$ is equal to $i$ on a small neighborhood of $\bR/\bZ \times \{0\} \subset \bR/\bZ \times B_{100\delta}(0)$.
\item For any pair of PFH generators $\Theta_-$ and $\Theta_+$ with degree $\le d$, defining the same class in $H_1(M_\phi),$ and any class $W \in H_2(M_\phi, \Theta_+, \Theta_-; \mathbb{Z})$, there is a bijection between 
$$\mathbf{M}_1(\Theta_+, \Theta_-, W; J)\quad\text{and}\quad \mathbf{M}_1(M_{H^1}(\Theta_+), M_{H^1}(\Theta_-), (M_{H^1})_*(W); J_*^{H^1}).$$
\end{enumerate}

\begin{rem}
The fourth point is a slight generalization of the analogous statement in Lemma $2.4$ of \cite{LeeTaubes12}, which asserts a bijection between $\mathbf{M}_1(\Theta_+, \Theta_-; J)$ and  $\mathbf{M}_1(M_{H^1}(\Theta_+), M_{H^1}(\Theta_-); J_*^{H^1}).$ We give an explanation below of why this generalization holds as well.

The proof of the bijection in \cite{LeeTaubes12} uses a perturbation-theoretic argument. A sequence $\{J_k\}_{k=1}^N$ of almost-complex structures is chosen such that $J_1 = J$ and $J_N = J_*$, any two consecutive members $J_k$ and $J_{k+1}$ of the sequence are very close in the $C^1$ norm, and their difference has support in progressively smaller neighborhoods of the periodic points of $\phi$ of degree less than or equal to $d$. Then a gluing argument gives for every $k$ a bijection between $\mathbf{M}_1(\Theta_+, \Theta_-; J_k)$ and $\mathbf{M}_1(\Theta_+, \Theta_-; J_{k+1})$ which by nature will identify pseudoholomorphic currents in the same relative homology class, and therefore induce a bijection between $\mathbf{M}_1(\Theta_+, \Theta_-, W; J_k)$ and $\mathbf{M}_1(\Theta_+, \Theta_-, W; J_{k+1})$.  Chaining together all of these bijections (and pushing forward by $M_{H^1}$) and taking a limit as $k\to \infty$ yield the desired bijection between $\mathbf{M}_1(\Theta_+, \Theta_-, W; J)$ and $\mathbf{M}_1(M_{H^1}(\Theta_+), M_{H^1}(\Theta_-), (M_{H^1})_*(W); J_*^{H^1})$.
\end{rem}

We will also fix a distinguished homotopy $\{J_s\} \in \cJ(dt, \omega_\phi)$ from $J$ to $J_*$.

\textbf{Step 6:} We conclude with a discussion of the identification of the twisted PFH groups associated to $(\phi, J)$ and the $(\delta, d)$-approximation $(\phi_*, J_*)$.

Fix a PFH parameter set
$\bfS = (\phi, \Theta_{\text{ref}}, J)$
where $\Gamma$ has degree $\leq d$. Fix a PFH parameter set
$\bfS_* = (\phi_*, \Theta_{\text{ref}}^{H^1}, J_*^{H^1}).$
It is immediate given all of the above considerations that there is a canonical bijection between the generators of $\TWPFC_*(\bfS)$ and $\TWPFC_*(\bfS_*)$. A generator $(\Theta, W)$ of $\TWPFC_*(\bfS)$ is mapped under this bijection to $(M_{H^1}(\Theta), (M_{H^1})_*(W))$. The bijections between the moduli spaces of pseudoholomorphic currents for $J$ and $J_*^{H^1}$ imply that the canonical bijection on the generators induces an isomorphism
$$\TWPFC(\bfS) \to \TWPFC(\bfS_*).$$

It will also be useful to note that there is a bijection, for any $s \in [0,1]$, between twisted PFH generators for $\phi$ with reference cycle $\Theta_{\text{ref}}$ and twisted PFH generators for the diffeomorphism $\wt\phi_*(s)$ with reference cycle $\Theta_{\text{ref}}^{H^s}$, given by the map
$$(\Theta, W) \mapsto ( M_{H^s}(\Theta), (M_{H^s})_*(W)).$$

\subsubsection{PFH spectral invariants} \label{subsubsec:pfhDeltaD} We now give the previously promised proof of Proposition \ref{prop:dFlatSpectralInvariants}, which claims that the PFH spectral invariants change by $O(\delta)$ after taking a $(\delta, d)$-approximation.

\begin{prop}
	\label{prop:pfhActionApproximationInvariance} 
	Let $\bfS = (\phi, \Theta_{\text{ref}}, J),$ $\wt\phi_*(s)$, and $H^s$ be as above.	
	Then for each $s$, the map $M_{H^s}$ defines a bijection between the generators of 
	$\TWPFC_*(\phi, \Theta_{\text{ref}}, J)$
	and 
	$\TWPFC_*(\phi\circ \phi_{H^s}^1,\Theta_{\text{ref}}^{H^s}, J^{H^s}).$
	Moreover, there is a constant $\kappa_{\ref{prop:pfhActionApproximationInvariance}} \geq 1$ depending only on $\phi$, $J$, $d$, $\Theta_{\text{ref}}$ and $g$ such that for every twisted PFH generator $(\Theta, W)$ of $\TWPFC_*(\bfS)$,  we have the bound
	$$|\bfA(\Theta, W) - \bfA_s( (M_{H^s}(\Theta), (M_{H^s})_*(W))| \leq \kappa_{\ref{prop:pfhActionApproximationInvariance}}\delta.$$
\end{prop}

\begin{proof}
Fix any generator $(\Theta, W)$ of $\TWPFC_*(\bfS)$. For the sake of brevity, write
$$(\Theta_s, W_s) = ( M_{H^s}(\Theta), (M_{H^s})_*(W))$$
for the corresponding twisted PFH generator for $\wt\phi_*(s)$.

To show the proposition, we will prove for any $s$ a bound of the form
$$|\bfA_s( M_{H^s}(\Theta), (M_{H^s})_*(W)) - \bfA(\Theta, W)| \lesssim \delta.$$

Expanding the left-hand side and pulling back by $M_{H^s}$, we find
\begin{equation}
\label{eq:pfhActionApproximationInvariance1}
\begin{split} 
\bfA_s(\Theta_s, W_s) &= \int_{W} (M_{H^s})^*\omega_{\wt\phi_*(s)} = \int_W \omega_\phi + \int_W \big((M_{H^s})^*\omega_{\wt\phi(s)} - \omega_\phi\big) \\
&= \int_W \omega_\phi + \int_W dH^s \wedge dt = \int_W \omega_\phi + \int_\Theta H^s dt - \int_{\Theta_{\text{ref}}} H^s dt \\
\end{split}
\end{equation}

The last equality in (\ref{eq:pfhActionApproximationInvariance1}) follows from Stokes' theorem. It is a consequence of the second point in Lemma \ref{lem:dFlatBounds} that the one-form
$$H^s dt = h_t(s,t)H(h(s, t), -) dt$$
has pointwise norm bounded above by $4\kappa_{\ref{lem:dFlatBounds}}\delta$. The bound in the statement of the proposition follows from this bound and (\ref{eq:pfhActionApproximationInvariance1}).
\end{proof}

\begin{proof}[Proof of Proposition \ref{prop:dFlatSpectralInvariants}]
	Let
	$\bfA_*(\Theta, W) = \langle \omega_{\phi_*}, W \rangle$
	denote the version of the functional $\bfA$ for the PFH parameter set $\bfS_*$.
	Proposition \ref{prop:pfhActionApproximationInvariance} shows that for any generator $(\Theta, W)$ of $\TWPFC_*(\bfS)$, we have
	$$|\bfA(\Theta, W) - \bfA_*( M_{H^1}(\Theta), (M_{H^1})_*(W))| \leq \kappa_{\ref{prop:pfhActionApproximationInvariance}}\delta.$$
	
	This immediately implies the proposition since it shows the isomorphism $\TWPFH_*(\bfS) \simeq \TWPFH_*(\bfS_*)$ shifts the PFH action filtration by at most $\kappa_{\ref{prop:pfhActionApproximationInvariance}}\delta$. 
\end{proof}

\subsubsection{Seiberg-Witten spectral invariants} \label{subsubsec:deltaDSW} We now show that the Seiberg-Witten spectral invariants, when the $r$-parameter is sufficiently large, change by $O(\delta)$ upon taking a $(\delta, d)$-approximation.   The precise statement is deferred to Proposition~\ref{prop:dFlatSWSpectral} below, as we first set up the relevant notation.

Fix a family of SW parameter sets
$\cS_r = (\phi, J, \Gamma, r, \fg_r = \fe_\mu + \fp_r)$
where $\|\fp_r\|_{\cP} \leq r^{-1}$ for every $r$. We will assume that the pairs $(\fe_\mu, \fp_r)$ are generic for every $r$ so that the Seiberg--Witten--Floer cohomology is well-defined, and we will also assume tht $\phi$ is nondegenerate. Let $(\phi_*, J_*)$ be a $(\delta, d)$-approximation to $(\phi, J)$ with $\delta > 0$ very small. Let $\{H^s\}_{s \in [0,1]}$ be the associated family of Hamiltonians whose time one maps generate an isotopy from $\phi$ to $\phi_*$. Let $\{J_s\}_{s \in [0,1]}$ be the homotopy from $J$ to $J_*$ fixed in Step $5$ of \S\ref{subsubsec:deltaDDetails}. 
Fix a reference cycle $\Theta_{\text{ref}}$ for $\phi$ such that $[\Theta_{\text{ref}}] = \Gamma$. As always, we take $\delta$ sufficiently small so that the one-form $H^s dt$ is supported away from $\Theta_{\text{ref}}$. 
We also require $\delta > 0$ to be much smaller than the smallest ``gap'' in the action spectrum. Let $\mathcal{T}$ denote the set of actions of all twisted PFH generators for $\phi$ with respect to the reference cycle $\Theta_{\text{ref}}$. The set $\mathcal{T}$ is discrete because $\phi$ is nondegenerate and $\omega_\phi$ is a real multiple of an integral class. Then we require $\delta$ to be small enough so that any two distinct points in $\mathcal{T}$ lie a distance of at least $100\kappa_{\ref{prop:pfhActionApproximationInvariance}}\delta$ away from each other.
Fix a $\Theta_{\text{ref}}$-concentrated family $\fc^c(r) = (B^c(r), \Psi^c(r) = (\alpha^c(r), \beta^c(r))$ as introduced in \S\ref{sec:twistedIso}. Pushing forward by $M_{H^s}$ defines a $\Theta_{\text{ref}}^{H^s}$-concentrated family of configurations for any $s \in [0,1]$. 

Fix a two-parameter family of SW parameter sets
$$\cS_{r,s} = (\wt\phi_*(s), J_s^{H^s}, \Gamma^{H^s}, r, \fg_{r,s} = \fe_{\mu_s} + \fp_{r,s})$$
where $\mu_s$ has $C^3$ and $\cP$-norm less than or equal to $1$ for every $s$ and $\|\fp_{r,s}\|_\cP \leq r^{-10}$ for every $r$ and $s$. We will assume that the pairs $(\fe_{\mu_s}, \fp_{r,s})$ are chosen generically for every $r$ and $s$ (strictly speaking, they should be ``strongly admissible'' in the sense of \cite[\S$3$]{TaubesWeinstein1}). The groups $\TWHM^*(\cS_r)$ and $\TWHM^*(\cS_{r,s})$ are, as long as the $\fp_{r,s}$ are sufficiently small, canonically isomorphic for all $s$ and all $r > -2\pi\rho$. This implies that, given a fixed $\sigma$, the spectral invariant $c_\sigma^{\HM}(\cS_r)$ or $c_\sigma^{\HM}(\cS_{r,s})$ is unambiguously defined. We can now state the desired proposition relating the SW spectral invariants of $\phi$ to those of the approximation $\phi_*$.

\begin{prop} \label{prop:dFlatSWSpectral}
There is a constant $\kappa_{\ref{prop:dFlatSWSpectral}} = \kappa_{\ref{prop:dFlatSWSpectral}}(\phi, J, d, g, \Theta_{\text{ref}}) \geq 1$ depending only on the diffeomorphism $\phi$, the almost-complex structure $J$, the degree $d$, the metric $g$ and the $g$-length of the reference cycle $\Theta_{\text{ref}}$ such that the following bound holds.

Let $\cS_{r,s}$ be the two-parameter family of SW parameter sets fixed prior to the proposition statement. Fix any base configuration $\fc_\Gamma = (B_\Gamma, \Psi_\Gamma)$ and assume $\delta \leq \kappa_{\ref{prop:dFlatSWSpectral}}^{-1}$. Then we have the bound
$$\limsup_{r \to \infty} |c_{\sigma}^{\HM}(\cS_r; \fc_\Gamma)) - c^{\HM}_{\sigma}(\cS_{r,1}; \fc_\Gamma) + \frac{i}{2} \int_{M_\phi} (B^c(r) - B_\Gamma) \wedge dH^1 \wedge dt| \leq \kappa_{\ref{prop:dFlatSWSpectral}}\delta.$$
\end{prop} 

\begin{proof}[Proof of Proposition \ref{prop:dFlatSWSpectral}]
Our argument combines
Proposition \ref{prop:pfhActionApproximationInvariance} and Proposition \ref{prop:actionConvergence} with an argument of Hutchings--Taubes \cite{ArnoldChord2}.  The proof will proceed in three steps. 

\textbf{Step 1:} The first step proves the proposition, assuming that a bound of the form
\begin{equation} \label{eq:dFlatSwSpectral2} \limsup_{r \to \infty} |c_{\sigma}^{\HM}(\cS_r; \fc^c(r)) - c^{\HM}_{\sigma}(\cS_{r,1}; \fc^c(r))| \lesssim \delta\end{equation}
holds.

As long as the abstract perturbations are ``strongly admissible'' in the sense of \cite[\S$3$]{TaubesWeinstein1}, a property that holds generically, we can fix two families of configurations $\fc(r) = (B(r), \Psi(r))$ and $\fc_*(r) = (B_*(r), \Psi_*(r))$ such that the following holds. First, $\fc(r)$ solves the $(\phi, J, \Gamma, r, \fe_\mu)$-Seiberg-Witten equations for every $r$ and $$\lim_{r \to \infty} |r^{-1}\fa_{r, \fe_\mu}(\fc(r); \fc^c(r)) - c_{\sigma}^{\HM}(\cS_r; \fc^c(r))| = 0.$$

Second, $\fc_*(r)$ solves the $(\phi_*, J_*, \Gamma_1, r, \fe_{\mu_1})$-Seiberg-Witten equations for every $r$ and
$$\lim_{r \to \infty} |r^{-1}\fa_{r, \fe_{\mu_1}}(\fc_*(r); \fc^c(r)) - c_{\sigma}^{\HM}(\cS_{r,1}; \fc^c(r))| = 0.$$

We now argue that
$$\lim_{r \to \infty} r^{-1}(\fa_{r, \fe_\mu}(\fc(r); \fc_\Gamma) - \fa_{r, \fe_\mu}(\fc(r); \fc^c(r))) = \frac{i}{2} \int_{M_\phi} (B^c(r) - B_\Gamma) \wedge \omega_\phi$$
and
$$\lim_{r \to \infty} r^{-1}(\fa_{r, \fe_{\mu_1}}(\fc_*(r); \fc_\Gamma) - \fa_{r, \fe_{\mu_1}}(\fc_*(r); \fc^c(r))) = \frac{i}{2} \int_{M_{\phi_*}} (B^c(r) - B_\Gamma) \wedge \omega_{\phi_*}.$$

We can argue as in the proof of the identity (\ref{eq:dFlatSwSpectral1}) in Proposition \ref{prop:actionConvergence} that we have, for any choice of base configuration $\fc = (B, \Psi)$, the identities
$$\lim_{r \to \infty} |r^{-1}\fa_{r, \fe_\mu}(\fc(r); \fc) + \frac{1}{2}E_\phi(B(r), B)| = 0$$
and
$$\lim_{r \to \infty} |r^{-1}(\fa_{r, \fe_{\mu_1}}(\fc_*(r); \fc) + \frac{1}{2}E_{\phi_*}(B_*(r), B)| = 0.$$

Now all that is left is a straightforward computation:
\begin{equation}
\label{eq:dFlatSwSpectral3}
\begin{split}
\lim_{r \to \infty} r^{-1}(\fa_{r, \fe_\mu}(\fc(r); \fc_\Gamma) - \fa_{r, \fe_\mu}(\fc(r); \fc^c(r))) &= \lim_{r \to \infty} \frac{1}{2}(E_\phi(B(r), B_\Gamma) - E_\phi(B(r), B^c(r))) \\
&= \lim_{r \to \infty} \frac{i}{2}\int_{M_\phi} (B^c(r) - B_\Gamma) \wedge \omega_\phi
\end{split}
\end{equation}
\begin{equation}
\label{eq:dFlatSwSpectral4}
\begin{split} 
\lim_{r \to \infty} r^{-1}(\fa_{r, \fe_{\mu_1}}(\fc_*(r); \fc_\Gamma) - \fa_{r, \fe_{\mu_1}}(\fc_*(r); \fc^c(r))) &= \lim_{r \to \infty} \frac{1}{2}(E_{\phi_*}(B_*(r), B_\Gamma) - E_{\phi_*}(B_*(r), B^c(r))) \\
&= \lim_{r \to \infty} \frac{i}{2}\int_{M_{\phi_*}} (B^c(r) - B_\Gamma) \wedge \omega_{\phi_*}.
\end{split}
\end{equation}

Now pull back (\ref{eq:dFlatSwSpectral4}) by $M_{H^1}$, subtract it from (\ref{eq:dFlatSwSpectral3}), and plug the result into (\ref{eq:dFlatSwSpectral2}) to conclude the proposition, using the fact that the pullback of $\omega_{\phi_*}$ by $M_{H^1}$ is equal to $\omega_\phi + dH^1 \wedge dt$.

\textbf{Step 2:} The second step applies the argument of \cite[Lemma $3.4$]{ArnoldChord2} to relate the filtered versions of the Seiberg--Witten--Floer homologies for the parameter sets $\cS_r$ and $\cS_{r,1}$ for sufficiently large $r$. 

More precisely, we will define a subset $\mathcal{O}$ of $\mathbb{R}$, and show that for every $L\in \mathcal{O}$ there exists a constant $r_L>1$, such that for all $r>r_L$, there exists an isomorphism 
$\xi_L:\TWHM^*_L(\cS_r; \fc^c(r))\to \TWHM^*_L(\cS_{r,1}; \fc^c(r))$
so that the following diagram is commutative:
\begin{center}
	\begin{tikzcd}
		\TWHM^*_L(\cS_r; \fc^c(r)) \arrow{r} \arrow{d} & \TWHM^*_L(\cS_{r,1}; \fc^c(r)) \arrow{d} \\
		\TWHM^*(\cS_r) \arrow{r} & \TWHM^*(\cS_{r,1}),
	\end{tikzcd}
\end{center}
where the vertical arrows are the maps $\iota^{\TWHM}_L(\cS_r; \fc^c(r))$ and $\iota^{\TWHM}_L(\cS_{r,1}; \fc^c(r))$, the bottom horizontal arrow is the canonical isomorphism, and the top horizontal arrow is $\xi_L$. 

The set $\mathcal{O}$ is defined as follows. For $s \in [0,1]$, let $\mathcal{T}(s) \subset \bR$ denote the discrete set of twisted PFH actions $\bfA(\Theta, W)$ across all twisted PFH generators $(\Theta, W)$ for the diffeomorphism $\wt\phi_*(s)$ with respect to the reference cycle $\Theta_{\text{ref}}$. Then we define $\mathcal{O}$ to be the set of $L\in \mathbb{R}$ such that $\text{dist}(L, \mathcal{T}(s)) > 10\kappa_{\ref{prop:pfhActionApproximationInvariance}}\delta$ for all $s \in [0,1]$.  

The existence of the above commutative diagram follows from repeating the one-parameter subdivision argument in \cite[Lemma $3.4$]{ArnoldChord2} and \cite[Paper I, Lemma $4.6$]{ECHSWF}. This argument requires $\text{dist}(L, \mathcal{T}(s))$ to be bounded away from $0$ uniformly, and $r$ needs to be larger than a constant depending only on geometric constants and this lower bound, which by our conditions on $L$ depends only on geometric constants and $\delta$. Most of the necessary modifications are cosmetic and reconcile differences in conventions; our continuation maps are analogous to the admissible deformations of \cite{ArnoldChord2}. The sole exception is the following. The argument in \cite{ArnoldChord2} requires the compactness theory of \cite{TaubesWeinstein1}, which implies that the actions of Seiberg--Witten solutions converge to the action of a ECH generator in the limit $r \to \infty$. We use the action convergence result from Proposition \ref{prop:actionConvergence} instead. 

\textbf{Step 3:} The third step shows that for any $s \in [0,1]$, the limit
$$L_s = \lim_{r \to \infty} c^{\HM}_{\sigma}(\cS_{r,s}; \fc^c(r))$$
is well-defined. First, Lemma \ref{lem:maxMinLimit} shows that for any $s \in \bR$ and any \emph{fixed} base configuration $\fc_\Gamma$, the limit $\lim_{r \to \infty} c^{\HM}_{\sigma}(\cS_{r,s}; \fc_\Gamma)$ exists. We use this fact, Proposition \ref{prop:actionConvergence} and the fact that the action spectra $\mathcal{T}(s)$ defined in the previous step are discrete to show that $L_s$ exists. First, we establish the following estimate, which holds for any sufficiently large $r_*$ and $r \gg r_*$: 
\begin{equation} \label{eq:dFlatSwSpectral5} |c^{\HM}_{\sigma}(\cS_{r,s}; \fc^c(r)) - c^{\HM}_{\sigma}(\cS_{r,s}; \fc^c(r_*))| \lesssim r_*^{-1/3}. \end{equation}
We use Proposition \ref{prop:3dActionEstimates} and the fact that $\|F^c(r)\|_{L^1(M_\phi)}$ has an $r$-independent upper bound to control the Chern--Simons functional. To control the functional $E_\phi$, we observe that $\omega_\phi$ is exact on the support of $B^c(r) - B^c(r_*)$, apply Stokes' theorem, and use the fact that $\|F^c(r)\|_{C^0} \lesssim r$. 

The bound (\ref{eq:dFlatSwSpectral5}) implies that $c^{\HM}_{\sigma}(\cS_{r,s}; \fc^c(r))$ is uniformly bounded in $r$. Now by Proposition \ref{prop:actionConvergence} any subsequential limit lies in the discrete set $\mathcal{T}(s)$. 

We assumed prior to the statement of the proposition that any two distinct elements of $\mathcal{T} = \mathcal{T}(0)$ are at least $100\kappa_{\ref{prop:pfhActionApproximationInvariance}}\delta$ away from each other. By Proposition \ref{prop:pfhActionApproximationInvariance}, we also have $\text{dist}(\mathcal{T}(s), \mathcal{T}) \leq \kappa_{\ref{prop:pfhActionApproximationInvariance}}\delta$ for any $s \in [0,1]$. It follows that any two distinct elements of $\mathcal{T}(s)$ are at least $90\kappa_{\ref{prop:pfhActionApproximationInvariance}}\delta$ away from each other. 

Take $r_*$ sufficiently large so that the right-hand side of (\ref{eq:dFlatSwSpectral5}) is bounded above by $\kappa_{\ref{prop:pfhActionApproximationInvariance}}\delta$. Then it follows from (\ref{eq:dFlatSwSpectral5}) that any two subsequential limits of the family $c^{\HM}_{\sigma}(\cS_{r,s}; \fc^c(r))$ as $r \to \infty$ lie a distance of at most $10\kappa_{\ref{prop:pfhActionApproximationInvariance}}\delta$ away from each other. However, they both lie in the discrete set $\mathcal{T}(s)$, and any two distinct elements of $\mathcal{T}(s)$ have a distance of at least $90\kappa_{\ref{prop:pfhActionApproximationInvariance}}\delta$ from each other. We conclude that any two subsequential limits of the family $c^{\HM}_{\sigma}(\cS_{r,s}; \fc^c(r))$ as $r \to \infty$ coincide, and therefore the family has a unique limit as $r \to \infty$.

\textbf{Step 4:} The fourth step completes the proof of (\ref{eq:dFlatSwSpectral2}).

Fix either $L = L_0 - 20\kappa_{\ref{prop:pfhActionApproximationInvariance}}\delta$ or $L = L_0 + 20\kappa_{\ref{prop:pfhActionApproximationInvariance}}\delta$. Since $L_0 \in \mathcal{T}$, we find that $L$ is a distance of at least $10\kappa_{\ref{prop:pfhActionApproximationInvariance}}\delta$ from any point in $\mathcal{T}$. The existence of the respective diagrams in Step $2$ for either value of $L$ implies that
$$L_0 - 20\kappa_{\ref{prop:pfhActionApproximationInvariance}}\delta \leq c^{\HM}_{\sigma}(\cS_{r,1}; \fc^c(r)) \leq L_0 + 20\kappa_{\ref{prop:pfhActionApproximationInvariance}}\delta.$$

Taking $r \to \infty$ in the above shows
$|L_0 - L_1| \leq 40\kappa_{\ref{prop:pfhActionApproximationInvariance}}\delta,$
from which we conclude (\ref{eq:dFlatSwSpectral2}), and therefore the whole proposition.
\end{proof}

\subsection{The PFH and SWF gradings} \label{subsec:twistedIsoGradings}

Recall that both twisted PFH and twisted Seiberg--Witten--Floer cohomology have canonical relative $\bZ$--gradings. The relative $\bZ$--grading on the twisted PFH group sends two generators $(\Theta_+, W_+)$ and $(\Theta_-, W_-)$ to the ECH index
$I(\Theta_+, \Theta_-, W_+ - W_-).$

The relative $\bZ$--grading on the twisted SWF group sends two generators represented by solutions $\fc_+$ and $\fc_-$ to the spectral flow
$\SF(\fc_-, \fc_+)$
between the Hessians $\cL_{\fc_-}$ and $\cL_{\fc_+}$. It is shown by \cite{LeeTaubes12} that their isomorphism reverses the relative $\bZ$--gradings. In this section, we show a technical result, Proposition \ref{prop:twistedIsoPreservesGradings}, estimating the PFH gradings of two distinct, Hamiltonian-isotopic maps $\phi$ and $\phi'$ in terms of corresponding Seiberg--Witten gradings. The result uses a result of Cristofaro--Gardiner \cite{CG13} computing the relative ECH index in a symplectic cobordism in terms of Seiberg--Witten theory. 

\subsubsection{Setup and statement of Proposition \ref{prop:twistedIsoPreservesGradings}}

Fix a nondegenerate map $\phi \in \Diff(\Sigma, \omega)$.  Fix a Hamiltonian $H \in C^\infty_c( (0,1) \times \Sigma)$, and set $\phi' = \phi \circ \phi^1_H$, and furthermore suppose that $\phi'$ is also nondegenerate. Fix a negative monotone class $\Gamma \in H_1(M_\phi; \bZ)$ of degree $d$. Fix a separated trivialized reference cycle $\Theta_{\text{ref}}$ for $\phi$ representing the class $\Gamma$. 

We use these choices to define PFH parameter sets $\bfS = (\phi, \Theta_{\text{ref}}, J)$ and $\bfS^H = (\phi', \Theta_{\text{ref}}^H, J')$. 

Write $g$ and $g'$ for the Riemannian metrics on $M_\phi$ and $M_{\phi'}$ induced by $(\phi, J)$ and $(\phi', J')$, respectively. For $r \gg -2\pi\rho$ choose SW parameter sets $\cS = (\phi, J, \Gamma, r, \fg = \fe_\mu + \fp)$ and $\cS^H = (\phi', J', \Gamma^H, r, \fg' = \fe_\mu' + \fp')$. We assume that $(\fe_\mu, \fp)$ and $(\fe_\mu', \fp')$ are generic, with $\mu$ and $\mu'$ having small $C^3$ and $\cP$-norm and $\fp$ and $\fp'$ having very small $\cP$-norm. Finally, we fix nondegenerate base configurations $\fc_\Gamma = (B_\Gamma, \Psi_\Gamma)$, and $\fc_\Gamma' = (B_\Gamma', \Psi_\Gamma')$ on $M_\phi$ and $M_{\phi'}$, and fix a constant $\Lambda \geq 1$ such that $\|F_{B_\Gamma}\|_{C^3(g)}$ and $\|F_{B_{\Gamma}'}\|_{C^3(g')}$ are $\leq \Lambda d$. We now state the proposition that we will spend the rest of this section proving.  

\begin{prop} \label{prop:twistedIsoPreservesGradings}
Let $\bfS$, $\bfS^H$, $\cS$, $\cS^H$, $\fc_\Gamma$, $\fc_{\Gamma'}$, and $\Lambda$ be as above. Suppose that $\fc_\Gamma$ and $(M_H)^*\fc_{\Gamma'}$ are sufficiently $C^1$-close.  Then for any two classes $\sigma \in \TWPFH(\bfS)$ and $\tau \in \TWPFH(\bfS^H)$, sufficiently large $r$, there is a geometric constant $\kappa_{\ref{prop:twistedIsoPreservesGradings}} \geq 1$ such that, if we set $\sigma_\Phi = T_{\text{Tw}\Phi}^r(\sigma)$ and $\tau_\Phi = T_{\text{Tw}\Phi}^r(\tau)$, we have the inequality
$$|(I(\tau) - I(\sigma)) + (\SF(\tau_\Phi, \fc_\Gamma') - \SF(\sigma_\Phi, \fc_\Gamma))| \leq \kappa_{\ref{prop:twistedIsoPreservesGradings}}(\Lambda d)^{3/2}.$$
\end{prop}

We will prove the proposition in stages.  To start, we review the setup for our argument and fix some notation.  Fix an SW continuation parameter set $\cS_s = (K, J_s, r, \fg_s)$ from $\cS$ to $\cS^H$. Then $\cS_s$ defines a symplectic $4$--manifold
$$(X = \bR \times M_{\phi}, \Omega_X = ds \wedge dt + \frac{1}{2}(\omega_{\phi} + d(K(s)dt)))$$
and counting solutions to the $\cS_s$-Seiberg-Witten instanton equations (\ref{eq:swContinuation}) on $X$ defines a quasi-isomorphism from $\TWCM^{*}(\cS^H)$ to $\TWCM^{*}(\cS)$, with the induced map on cohomology independent of the choice of $K$. 
Choose data as in Section \ref{sec:twistedIso} to define the Lee--Taubes isomorphism
$T_{\text{Tw}\Phi}^r: \TWPFH_*(\bfS) \to \TWHM^{-*}(\cS)$
and push forward this data by $M_H$ define an isomorphism $T_{\text{Tw}\Phi}^r: \TWPFH_*(\bfS^H) \to \TWHM^{-*}(\cS^H).$ 
Fix twisted PFH generators $(\Theta, W)$ and $(\Theta', W')$ for $\phi$ and $\phi'$. Pull back $(\Theta', W')$ by $M_H$ so that $\Theta'$ is considered as a set of loops with multiplicity in $M_\phi$ and $W'$ is a class in $H_2(M_\phi, \Theta', \Theta_{\text{ref}}; \bZ)$. Fix $r \gg 1$ and write\footnote{Strictly speaking this requires taking a $(\delta, d)$-approximation for both $\phi$ and $\phi'$, however since this does not change the grading we omit this from the notation for simplicity.} 
$\fc = (B, \Psi = r^{1/2}(\alpha, \beta)) = \text{Tw}\Phi^r(\Theta, W)$
and
$\fc' = (B', \Psi' = r^{1/2}(\alpha', \beta')) = \text{Tw}\Phi^r(\Theta', W').$

\subsubsection{Applying a computation of Cristofaro-Gardiner} We now begin the proof of the proposition.  Our goal is to compare the difference $I(\Theta', W') - I(\Theta, W)$ of the PFH gradings with the corresponding difference $\SF(\fc, \fc_\Gamma) - \SF(\fc', \fc_\Gamma')$ of the Seiberg--Witten gradings. Recall from \S\ref{subsec:continuationMaps} the \emph{relative index} $\text{ind}(\fc, \fc')$ for the solutions $\fc = \text{Tw}\Phi^r(\Theta, W)$ and $\fc' = \text{Tw}\Phi^r(\Theta', W')$ of the $\cS$- and $\cS^H$-Seiberg-Witten equations. This is the index of the Fredholm operator 
$$\cL_{\fd} = \frac{\partial}{\partial s} + \cL_{\fd_s}$$
where $\fd$ is any path from $\fc$ to $\fc'$. Our first step towards proving Proposition \ref{prop:twistedIsoPreservesGradings} is to show that $I(\Theta', W') - I(\Theta, W)$ is equal to this Seiberg--Witten-theoretic quantity. 

\begin{lem} \label{lem:cg_grading_computation}
Let $\Theta$, $\Theta'$, $W$, $W'$, $\fc$, $\fc'$ be as above. Then
$$I(\Theta', W') - I(\Theta, W) = \text{ind}(\fc, \fc').$$
\end{lem}

\begin{proof}
The proof of Lemma \ref{lem:cg_grading_computation} will be carried out in three steps. The first two steps, after some preparation, use \cite[Theorem $5.1$]{CG13} to show that $\text{ind}(\fc, \fc')$ is equal to a certain ECH index. The last step shows this index is equal to $I(\Theta', W') - I(\Theta, W)$. 

\textbf{Step 1:} The first step uses the Seiberg--Witten solutions $\fc = (B, \Psi = r^{1/2}(\alpha, \beta))$, $\fc' = (B', \Psi' = r^{1/2}(\alpha', \beta'))$, considered as configurations on $M_\phi$, to construct a certain surface $\overline{Z}$ with cylindrical ends which interpolates between $\Theta$ and $\Theta'$. 

Recall by Proposition \ref{prop:untwistedIsomorphism} that $\alpha$ does not vanish outside of a tubular neighborhood of $\Theta$ of radius approximately $r^{-1/2}$, and likewise for $\alpha'$ with respect to the orbit set $\Theta'$. Choose a path $\fd = (B_s, \Psi_s = (\alpha_s, \beta_s))$ from $(B, \Psi)$ to $(B', \Psi')$ in $\text{Conn}(E_\Gamma) \times C^\infty(S_\Gamma)$. We assume for simplicity that for $s \leq -2$ that $(B_s, \Psi_s) = (B, \Psi)$ and for $s \geq 2$ that $(B_s, \Psi_s) = (B', \Psi')$.  

Consider the path of sections $\{\Psi_s\}_{s \in \bR}$ as a section $\Psi_X = (\alpha_X, \beta_X)$ of the pullback of $S_\Gamma$ to $X = \bR \times M_{\phi}$ via the projection to $M_{\phi}$. Make a small deformation of the path $\{\Psi_s\}_{s \in \bR}$, keeping it constant for $|s| \gg 1$, so that $\alpha_X$ is transverse to zero, and the zero set $Z = \alpha_X^{-1}(0)$ is a smooth, properly embedded surface, whose intersections with the sets $\{s \geq 2\}$ and $\{s \leq -2\}$ are cylindrical ends which lie in tubular neighborhoods of radius $\sim r^{-1/2}$ of $\Theta'$ and $\Theta$, respectively.  To cite Cristofaro-Gardiner's computation, we want to assume that $Z$ has no closed components, and we can do this
after making a further deformation to the path, supported in $\{|s| \leq 3\}$.

For the remainder of this construction, fix a large parameter $R \gg 1$. This is required to apply some analysis from \cite[Paper III, \S 2.c]{ECHSWF}, which is a key step in Cristofaro-Gardiner's computation. The intersections of $Z$ with $(-\infty, -R] \times M_\phi$ and $[R, \infty) \times M_\phi$ are disjoint unions of cylinders over embedded loops in $M_\phi$ lying near $\Theta$ and $\Theta'$. Replace these ends with ends exponentially decaying to $\Theta$ as $s \to -\infty$ and $\Theta'$ as $s \to \infty$, as prescribed in \cite[Paper III, \S 2b.1]{ECHSWF}. Denote the resulting surface by $\overline{Z}$. The surface $\overline{Z}$ projects to $M_\phi$ to form a two-chain with boundary $\Theta' - \Theta$, and as a result has a well-defined homology class $[\overline{Z}] \in H_2(M_\phi, \Theta', \Theta; \bZ)$, independent of the exact choice of path $\fd$ and parameter $R \gg 1$, as in \cite{CG13}.

\textbf{Step 2:} We now state the required result from \cite{CG13}. Let $I(\Theta', \Theta, [\overline{Z}])$ denote the standard ECH index in the symplectic cobordism $X$ as defined in \cite[\S$4.2$]{ECHrevisited}. Then \cite[Theorem $5.1$]{CG13} implies the identity
\begin{equation}\label{eq:cg_computation} \text{ind}(\fc, \fc') = I(\Theta', \Theta, [\overline{Z}]).\end{equation}

The theorem in \cite{CG13} is stated for the case where $X$ is a completed symplectic cobordism between contact $3$--manifolds. The argument extends with only cosmetic modifications to our case, where $X$ is instead a completed fibered symplectic cobordism between stable Hamiltonian mapping torii. See also \cite[Remark $5.3$]{CG13} for a summary of some crucial analysis from \cite[Paper III, \S 2.c]{ECHSWF}.

\textbf{Step 3:} The purpose of this step is conclude the proof of the lemma by showing the identity
\begin{equation}\label{eq:surface_homology_computation} [\overline{Z}] = W' - W.\end{equation}

This implies the lemma because $I(\Theta', \Theta, W' - W) = I(\Theta', W) - I(\Theta, W)$ by the properties of the ECH index, so \eqref{eq:cg_computation} and \eqref{eq:surface_homology_computation} combine to prove the identity asserted by the lemma. The computation uses the details of the construction of $\overline{Z}$ and the ``winding number'' interpretation of the twisted Lee--Taubes isomorphism from Remark \ref{rem:relativeWindingNumber}. 

Let $(B^c(r), \Psi^c(r) = (\alpha^c(r), \beta^c(r))$ be the family of base configurations concentrating around $\Theta_{\text{ref}}$ used to define the Lee--Taubes isomorphism. Without loss of generality, assume that the path $\fd = (B_s, \Psi_s)$ fixed above is chosen such that $(B_s, \Psi_s) = (B^c(r), \Psi^c(r) = (\alpha^c(r), \beta^c(r)))$
when $|s| \leq 1$. Recall that $\alpha^c(r)$ is transverse to the zero section and has zero set equal to $\Theta_{\text{ref}}$. 
Let $R \gg 1$ be the large parameter that we used to define the compact surface $Z$ in the previous step, which we completed with cylindrical ends to form $\overline{Z}$. We set 
$$Z_- = \overline{Z} \cap [-R, -1] \times M_{\phi},\qquad Z_{\text{int}} = \overline{Z} \cap [-1, 1] \times M_{\phi}, \qquad Z_+ = \overline{Z} \cap [1, R] \times M_{\phi}.$$

By construction the surface $Z_{\text{int}}$ is the cylinder $[-1, 1] \times \Theta_{\text{ref}}$. By definition, the relative homology class $[\overline{Z}]$ is determined uniquely by the set of algebraic intersection numbers
$(\bR \times \gamma) \cdot \overline{Z} $
across all $\gamma \in \Lambda_\Gamma$, where $\Lambda_\Gamma$ is the set of loops used to define the twisted Lee--Taubes isomorphism. Since the loops in $\Lambda_\Gamma$ are far from $\Theta$ and $\Theta'$, we can assume that there are no intersections outside of $[-R, R] \times M$. The loops in $\Lambda_\Gamma$ are also far from $\Theta_{\text{ref}}$, so the choice of $\fd$ ensures that there are no intersections with the cylinder $Z_{\text{int}} \subset [-1, 1] \times M_{\phi}$. 

The result follows by showing, for any $\gamma \in \Lambda_\Gamma$, the identity
$ (\bR \times \gamma) \cdot \overline{Z} = \langle \gamma, W' - W \rangle.$
By what we have said above, we have
$$(\bR \times \gamma) \cdot \overline{Z} = (\bR \times \gamma) \cdot Z_+ + (\bR \times \gamma) \cdot Z_-.$$
The algebraic intersection numbers on the right-hand side are well-defined since by definition $\bR \times \gamma$ will be far from the boundary of either $Z_+$ or $Z_-$. We will show that
$$(\bR \times \gamma) \cdot Z_+ = \langle \gamma, W' \rangle, \qquad (\bR \times \gamma) \cdot Z_- = -\langle \gamma, W \rangle$$
for every fixed $\gamma \in \Lambda_\gamma$. We will write down the argument for the intersection number with $Z_+$. The argument for $Z_-$ is identical. 

This equality follows from our topological interpretation of the isomorphism between the two twisted theories. Recall from the discussion in Remark \ref{rem:relativeWindingNumber} that the relative winding number 
$\text{wind}(\gamma, \alpha^c(r), \alpha')$
is equal to $\langle \gamma, W' \rangle$. By definition, the section $\alpha_R$ of $E_\Gamma$ is homotopic to $\alpha'$ through sections that do not vanish on $\gamma$. It follows that
$$\text{wind}(\gamma, \alpha^c(r), \alpha_R) = \text{wind}(\gamma, \alpha^c(r), \alpha') = \langle \gamma, W' \rangle.$$

Note also by definition that $\alpha^c(r) = \alpha_{1}$. Therefore, it remains to show
$$\text{wind}(\gamma, \alpha_{1}, \alpha_R) = (\bR \times \gamma) \cdot Z_+.$$

This can be shown explicitly. For any $s \in [1, R]$, we define a smooth, complex-valued function $f_s$ on $\gamma$ by the identity
$\alpha_s = f_s\alpha_{1}/|\alpha_{1}|.$
The zeroes of $f_s$ correspond with the zeros of $\alpha_s$, which by definition are (the projections of) intersection points of $Z_+$ with $\bR \times \gamma$. It follows that, since $Z_+$ intersects $\bR \times \gamma$ transversely, the functions $f_s$ are themselves transverse to zero; $\text{wind}(\gamma, \alpha_1, \alpha_R)$ is precisely equal to the oriented count of zeros of the functions $\{f_s\}$. This in turn is equal to the intersection number of $Z_+$ with $\bR \times \gamma$.
\end{proof}

\subsubsection{Bounding the relative index and the proof of Proposition \ref{prop:twistedIsoPreservesGradings}}

Our next task is to estimate the relative index $\text{ind}(\fc, \fc')$ in terms of the spectral flows $\SF(\fc, \fc_\Gamma)$ and $\SF(\fc', \fc_{\Gamma}')$ where $\fc$, $\fc'_\Gamma$ are the base configurations fixed at the beginning of \S\ref{subsec:twistedIsoGradings}. The following lemma proves such an estimate, assuming $\fc_\Gamma$, $\fc_\Gamma'$ are $C^1$ close and that their curvatures have $C^3$ norms $\lesssim \Lambda d$ for some $\Lambda$. 

\begin{lem}\label{lem:relative_index_spectral_flow}
Let $\fc$, $\fc'$ be as fixed above and $\fc_\Gamma$, $\fc_{\Gamma}'$, $\Lambda$ as fixed in Proposition \ref{prop:twistedIsoPreservesGradings}. Then 
$$|\text{ind}(\fc, \fc') - \SF(\fc, \fc_\Gamma) + \SF(\fc', \fc_\Gamma')| \lesssim (\Lambda d)^{3/2}.$$
\end{lem}

\begin{proof}
The estimate asserted in the lemma follows from the additivity of index under gluing and a computation using the Atiyah--Patodi--Singer index theorem. The details are as follows.  Let $\fd = (B_s, \Psi_s)$ be any smooth path in $\text{Conn}(E_\Gamma) \times C^\infty(S_\Gamma)$ that for $s \ll -1$ is equal to $\fc$ and for $s \gg 1$ is equal to (the pullback of) $\fc'$. Recall from \S\ref{subsec:continuationMaps} that the index $\text{ind}(\fc, \fc')$ is computed by the index of the Fredholm operator
$$\cL_{\fd} = \frac{\partial}{\partial s} + \cL_{(B_s, \Psi_s)}$$
over the cylindrical end manifold $X = \bR \times M_\phi$ with Riemannian metric of the form
$g_X = ds^2 + g_s$
induced by the homotopy of Hamiltonians $K$ fixed prior to the statement of Proposition \ref{prop:twistedIsoPreservesGradings}. It will be convenient to assume that $K(s) \equiv 0$ for $s \leq -1$ and $K(s) \equiv H$ for $s \geq 1$. 

Fix some very large $R \geq 1$. Assume that the path $\fd$ is chosen so that the following four conditions hold:
\begin{itemize}
\item $(B_s, \Psi_s)$ equals $\fc$ for $s \leq -4R$ and $\fc'$ for $s \geq 4R$. 
\item $(B_s, \Psi_s)$ equals $\fc_\Gamma$ for $s \in [-2R, -1]$ and $\fc_\Gamma'$ for $s \in [1, 2R]$. 
\end{itemize}

Now standard gluing arguments for Fredholm operators imply that 
\begin{equation} \label{eq:twistedIsoPreservesGradings3} \text{ind}(\cL_\fd) = \text{ind}(\cL_{\fd_-}) + \text{ind}(\cL_{\fd_\Gamma}) + \text{ind}(\cL_{\fd_+}) \end{equation}
where $\cL_{\fd_-}$, $\cL_{\fd_\Gamma}$ and $\cL_{\fd_+}$ are three Fredholm operators which we define below.

The operator $\cL_{\fd_-}$ is the operator on the cylindrical manifold $\bR \times M_\phi$ associated to a smooth path $\fd_-$ in $\text{Conn}(E_\Gamma) \times C^\infty(S_\Gamma)$ which is equal to $\fc$ for $s \ll -1$ and $\fc_\Gamma$ for $s \gg 1$. We therefore find
\begin{equation} \label{eq:twistedIsoPreservesGradings4} \text{ind}(\cL_{\fd_-}) = \SF(\fc, \fc_\Gamma). \end{equation}

The operator $\cL_{\fd_+}$ is the operator on the cylindrical manifold $\bR \times M_{\phi'}$ associated to a smooth path $\fd_+$ in $\text{Conn}(E_{\Gamma'}) \times C^\infty(S_{\Gamma'})$ which is equal to $\fc_\Gamma'$ for $s \ll -1$ and $\fc'$ for $s \gg 1$. We therefore find
\begin{equation} \label{eq:twistedIsoPreservesGradings5} \text{ind}(\cL_{\fd_+}) = -\SF(\fc', \fc_\Gamma'). \end{equation}

The operator $\cL_{\fd_\Gamma}$ is the operator on the same cylindrical end Riemannian $4$-manifold $X$ diffeomorphic to $\bR \times M_\phi$ that the operator $\cL_{\fd}$ was defined on. It is associated to a smooth path $\fd_\Gamma = (B_{\Gamma,s}, \Psi_{\Gamma,s})$ in $\text{Conn}(E_\Gamma) \times C^\infty(S_\Gamma)$ which is equal to $\fc_\Gamma$ for $s \leq -1$ and (the pullback of) $\fc_\Gamma'$ for $s \geq 1$. We can furthermore assume, since $\fc_\Gamma$ and $\fc_\Gamma'$ are assumed to be very close, that for all $s \in [-1,1]$ the one-form $\frac{\partial}{\partial s} B_{\Gamma,s}$ has very small $C^0$ norm. 

We observe that we can view the Riemannian manifold $X$ as the compact manifold $[-2,2] \times M_\phi$, equipped with the restriction of the Riemannian metric $g_X = ds^2 + g_s$ with cylindrical ends attached, modeled on the Riemannian manifolds $(M_\phi, g)$ and $(M_\phi, g')$. We abuse notation and use $g'$ to denote the pullback of the Riemannian metric on $M_{\phi'}$ induced by the pair $(\phi', J')$. 
Analogously to the proof of Proposition \ref{prop:spectralFlow1}, we can compute the index of $\cL_{\fd_\Gamma}$ using the Atiyah--Patodi--Singer index theorem \cite{APS1}. See \S\ref{subsec:spectralFlow} for discussion of the terminology used here. 

The index of $\cL_{\fd_\Gamma}$ is given by the formula
\begin{equation}
\label{eq:twistedIsoPreservesGradings6}
\begin{split}
&\text{ind}(\cL_{\fd_\Gamma}) = \qquad \int_{[-2,2] \times M_{\phi}} \text{APS}(\fd_\Gamma) \dvol_{g_{X}} + \frac{1}{2}\eta(D_{B_\Gamma'}) - \frac{1}{2}\eta(D_{B_\Gamma}) \\
&\qquad + \SF(\cL_{(B_\Gamma, 0)}, \cL_{(B_\Gamma, \Psi_\Gamma)}) - \SF(\cL_{(B_\Gamma', 0)}, \cL_{(B_\Gamma', \Psi_\Gamma')}) + \frac{1}{2}\eta(g) - \frac{1}{2}\eta(g').
\end{split}
\end{equation}

Here $\text{APS}(\fd_\Gamma)$ denotes the APS integrand associated to the path of Dirac operators $\{D_{B_{\Gamma, s}}\}_{s \in [-2,2]}$. The terms $\eta(g)$ and $\eta(g')$ are the eta invariants of the \textbf{odd signature operators}
$$(b, f) \mapsto (\star d b - df, d^*b)$$
associated to the metrics $g$ and $g'$ corresponding to the pairs $(\phi, J)$ and $(\phi', J')$, respectively. They are bounded by constants depending only on these metrics. 

As in \S\ref{subsec:spectralFlow}, we use \cite[Lemma $5.4$]{TaubesWeinstein1} to show
$$|\SF(\cL_{(B_\Gamma, 0)}, \cL_{(B_\Gamma, \Psi_\Gamma)}) - \SF(\cL_{(B_\Gamma', 0)}, \cL_{(B_\Gamma', \Psi_\Gamma')})| \lesssim \|F_{B_\Gamma}\|_{C^0}^{3/2} + \|F_{B'_\Gamma}\|_{C^0}^{3/2} \lesssim (\Lambda d)^{3/2}.$$

By Proposition \ref{prop:etaInvariant1}, we have $|\eta(D_{B_\Gamma'}) - \eta(D_{B_\Gamma})| \lesssim (\Lambda d)^{3/2}.$
It remains to bound the first term of \eqref{eq:twistedIsoPreservesGradings6}, the integral of $\text{APS}(\fd_\Gamma)$. An explicit formula for the APS integrand is found in the first chapter of \cite{diracHeatKernels}. In dimension four, it splits into a term depending only on the $\hat{A}$-genus of the curvature of the Riemannian metric $g_{X_-}$ and the integral
$$\int_{[-2,2] \times M_\phi} ds \wedge \frac{\partial}{\partial s} B_{\Gamma, s} \wedge F_{B_{\Gamma, s}}.$$

Recall that $\|\partial_s B_{\Gamma, s}\|_{C^0}$ is small, so this latter integral is $\lesssim \|F_{B_\Gamma}\|_{C^0} \leq (\Lambda d)^{3/2}$. It follows from plugging everything into (\ref{eq:twistedIsoPreservesGradings6}) that 
\begin{equation}
\label{eq:twistedIsoPreservesGradings7} |\text{ind}(\cL_{\fd_\Gamma})| \lesssim (\Lambda d)^{3/2}.
\end{equation}

Now plug (\ref{eq:twistedIsoPreservesGradings4}), (\ref{eq:twistedIsoPreservesGradings5}), and (\ref{eq:twistedIsoPreservesGradings7}) into (\ref{eq:twistedIsoPreservesGradings3}) to prove the bound asserted by the lemma. 
\end{proof}

We now prove Proposition \ref{prop:twistedIsoPreservesGradings}. 

\begin{proof}[Proof of Proposition \ref{prop:twistedIsoPreservesGradings}]
Combine Lemmas \ref{lem:cg_grading_computation} and \ref{lem:relative_index_spectral_flow}. 
\end{proof}

\bibliographystyle{abbrv}
\bibliography{main}

\end{document}